\documentclass[12pt]{book}

\usepackage[english]{babel}
\usepackage[a4paper,top=3cm,bottom=3cm,left=3cm,right=2.7cm, marginparwidth=2.5cm, headheight = 14.5pt]{geometry}
\usepackage[utf8]{inputenc} 
\usepackage[T1]{fontenc} 
\usepackage{lmodern} 
\usepackage{amsmath}
\usepackage{amsthm}
\usepackage{graphicx}
\usepackage{tikz-cd}
\usepackage{imakeidx} 
\usepackage[colorlinks=false, unicode]{hyperref}
\usepackage{amssymb}
\usepackage{thmtools}
\usepackage{enumitem}
\usepackage{floatrow}
\usepackage{caption}
\usepackage{mathtools} 
\usepackage{mathrsfs}
\usepackage{mfirstuc} 
\usepackage[labelformat=empty]{subcaption}
\usepackage{fancyhdr}
\usepackage{emptypage} 
\usepackage[style=super,sort=none]{glossaries-extra}
\usepackage{textgreek}
\usepackage{breakcites} 
\usepackage{cleveref} 

\makeindex
\usetikzlibrary{positioning,shapes,fit}

\newcommand\myshade{85}
\definecolor{mycitecolor}{rgb}{0.94, 0.8, 0.0}
\definecolor{mylinkcolor}{rgb}{0.133, 0.0, 0.8}
\definecolor{myurlcolor}{rgb}{0.5, 1.0, 0.83}
\hypersetup{
  linkcolor  = mylinkcolor!\myshade!black,
  citecolor  = mycitecolor!\myshade!black,
  urlcolor   = myurlcolor!\myshade!black,
  colorlinks = true,
}

\newcommand{\N}{\mathbb{N}}
\newcommand{\Np}{{\mathbb{N}^+}}
\newcommand{\Z}{\mathbb{Z}}
\newcommand{\Q}{\mathbb{Q}}
\newcommand{\R}{\mathbb{R}}
\newcommand{\Dyad}{\mathrm{\mathbb{D}}}

\newcommand{\Set}{\cat{Set}} 
\newcommand{\Preord}{\cat{Preo}} 
\newcommand{\Ord}{\cat{Ord}} 
\newcommand{\Top}{\cat{Top}} 
\newcommand{\CompHaus}{\cat{CH}} 
\newcommand{\Pries}{\mathsf{Pries}} 
\newcommand{\Stone}{\cat{Stone}} 
\newcommand{\CompOrd}{\cat{CompOrd}} 
\newcommand{\TopXPreord}{\Top\times_{\Set}\Preord} 
\newcommand{\CompHausXPreord}{\CompHaus\times_\Set\Preord} 
\newcommand{\CompHausPreord}{\cat{CHPreo}} 
\newcommand{\TopPreord}{\cat{TopPreo}} 
\newcommand{\CompHausXOrd}{\cat{\CompHaus \times_\Set \Ord}} 
\newcommand{\TopOrd}{\cat{TopOrd}} 
\newcommand{\TopXOrd}{\cat{\Top \times_\Set \Ord}} 

\newcommand{\MV}{{\mathsf{MV}}} 
\DeclareMathOperator{\CMiso}{\mathsf{OC}} 
\newcommand{\MVM}{\mathsf{MVM}} 
\newcommand{\DMVM}{\mathsf{MVM}_{\rm{dyad}}} 
\newcommand{\DMVMinfty}{\mathsf{MVM}_{\mathsf{dyad}}^{\mathsf{lim}}} 
\newcommand{\TMVM}{\mathsf{MVM}_{\frac{1}{2}}} 
\newcommand{\TMVMinfty}{\mathsf{MVM}_{\frac{1}{2}}^{\mathsf{lim}}} 

\newcommand{\ULG}{{\mathsf{u} \ell \mathsf{G}}} 
\newcommand{\ULM}{{\mathsf{u} \ell \mathsf{M}}} 
\newcommand{\DLM}{{\ell \mathsf{M}_{\mathsf{dyad}}}} 

\newcommand{\Cleq}{\mathrm{C}_{\leq}}
\newcommand{\Cont}{\mathrm{C}}

\newcommand{\Gam}{\Gamma}
\newcommand{\X}{\Xi}

\newcommand{\CoAlg}{\mathsf{CoAlg}}

\newcommand{\Quot}{\mathbf{Q}} 
\newcommand{\QuotEquiv}{\tilde{\mathbf{Q}}}
\newcommand{\Preo}{\mathbf{P}} 

\DeclareMathOperator{\ev}{ev} 

\newcommand{\upset}{\ensuremath{\mathord{\uparrow}\mkern1mu}} 
\newcommand{\downset}{\ensuremath{\mathord{\downarrow}\mkern1mu}} 

\newcommand{\opcat}{\mathsf{op}} 
\newcommand{\opalg}{\mathrm{op}} 
\newcommand{\oprel}{\mathrm{op}} 

\DeclareMathOperator{\ALG}{\cat{Alg}} 

\newcommand{\disthom}[1][]{\operatorname{d}_{\mathrm{hom}}^{#1}}
\newcommand{\distint}[1][]{\operatorname{d}_{\mathrm{int}}^{#1}}
\newcommand{\dist}[1][]{\operatorname{d}^{#1}}

\DeclareMathOperator{\essinf}{essinf}
\DeclareMathOperator{\esssup}{esssup}
\DeclareMathOperator{\ess}{ess}

\newcommand{\0}{\mathbf{0}}
\newcommand{\1}{\mathbf{1}}

\DeclareMathOperator{\opH}{H}
\DeclareMathOperator{\opI}{I}
\DeclareMathOperator{\opS}{S}
\DeclareMathOperator{\opP}{P}

\newcommand{\CompOrdop}{\CompOrd^\opcat} 

\newcommand{\eps}{\varepsilon}

\newcommand{\seq}{\subseteq}
\newcommand{\df}{\coloneqq}

\renewcommand{\leq}{\leqslant}
\renewcommand{\geq}{\geqslant}

\newcommand{\les}{\preccurlyeq}

\newcommand{\epi}{\twoheadrightarrow}
\newcommand{\mono}{\rightarrowtail}
\newcommand{\rmono}{\hookrightarrow}

\newcommand{\half}[1]{\frac{#1}{2}}

\DeclareMathOperator{\ar}{ar}

\newcommand{\Dyadu}{\Dyad\cap [0,1]}

\newcommand{\alge}[1]{\mathbf{#1}}
\newcommand{\trunc}{\tau}

\newcommand{\SignCM}{\Sigma^\mathrm{OC}}
\newcommand{\SignCMomega}{\SignCM_{\leq \omega}}
\newcommand{\SignDMVM}{\Sigma_{\mathrm{dy}}}
\newcommand{\SignDMVMinfty}{\Sigma_{\mathrm{dy}}^{\mathrm{lim}}}

\DeclareMathOperator{\hh}{h}
\DeclareMathOperator{\jj}{j}
\DeclareMathOperator{\dd}{d}
\DeclareMathOperator{\uu}{u}

\newcommand{\limop}{\lambda}
\newcommand{\aprj}{\mu}

\newcommand{\cat}[1]{\mathsf{#1}}
\newcommand{\clalg}[1]{\mathsf{#1}}

\newcommand{\gs}[1]{\mathbf{#1}}

\newcommand{\by}[1]{\text{(#1)}}

\newcommand{\mv}{MV-algebra}
\newcommand{\mvs}{MV-algebras}
\newcommand{\amv}{an MV-algebra}

\newcommand{\mvm}{MV-monoidal algebra}
\newcommand{\mvms}{MV-monoidal algebras}
\newcommand{\amvm}{an MV-monoidal algebra}

\newcommand{\dmvm}{dyadic MV-monoidal algebra}
\newcommand{\dmvms}{dyadic MV-monoidal algebras}
\newcommand{\admvm}{a dyadic MV-monoidal algebra}

\newcommand{\dmvminfty}{limit dyadic MV-monoidal algebra}
\newcommand{\dmvminftys}{limit dyadic MV-monoidal algebras}
\newcommand{\admvminfty}{a limit dyadic MV-monoidal algebra}

\newcommand{\tmvm}{$2$-divisible MV-monoidal algebra}
\newcommand{\tmvms}{$2$-divisible MV-monoidal algebras}
\newcommand{\atmvm}{a $2$-divisible MV-monoidal algebra}

\newcommand{\tmvminfty}{limit $2$-divisible MV-monoidal algebra}
\newcommand{\tmvminftys}{limit $2$-divisible MV-monoidal algebras}
\newcommand{\atmvminfty}{a limit $2$-divisible MV-monoidal algebra}

\renewcommand{\lg}{Abelian $\ell$-group}

\newcommand{\alg}{an Abelian $\ell$-group}

\newcommand{\extlg}{Abelian lattice-ordered group}

\newcommand{\ulg}{unital Abelian $\ell$-group}
\newcommand{\ulgs}{unital Abelian $\ell$-groups}

\newcommand{\extulg}{unital Abelian lattice-ordered group}
\newcommand{\extulgs}{unital Abelian lattice-ordered groups}

\newcommand{\lm}{commutative distributive $\ell$-monoid}
\newcommand{\lms}{commutative distributive $\ell$-monoids}
\newcommand{\alm}{a commutative distributive $\ell$-monoid}

\newcommand{\ulm}{unital commutative distributive $\ell$-monoid}
\newcommand{\ulms}{unital commutative distributive $\ell$-monoids}
\newcommand{\aulm}{a unital commutative distributive $\ell$-monoid}

\newcommand{\extulms}{unital commutative distributive lattice-ordered monoids}

\newcommand{\dlm}{dyadic commutative distributive $\ell$-monoid}
\newcommand{\dlms}{dyadic commutative distributive $\ell$-monoids}
\newcommand{\adlm}{a dyadic commutative distributive $\ell$-monoid}

\newcommand{\extdlms}{dyadic commutative distributive lattice-ordered monoids}

\newcommand{\tulm}{unital $2$-divisible commutative distributive $\ell$-monoid}
\newcommand{\tulms}{unital $2$-divisible commutative distributive $\ell$-monoids}
\newcommand{\atulm}{a unital $2$-divisible commutative distributive $\ell$-monoid}

\def\DD{\ensuremath{\mathcal D}}

\def\FF{\ensuremath{\mathcal F}}
\def\GG{\ensuremath{\mathcal G}}

\declaretheorem[name=Theorem, refname={Theorem,Theorems}, Refname={Theorem,Theorems}, numberwithin=chapter,]{theorem}

\declaretheorem[name=Theorem, numbered = no]{theorem*}

\declaretheorem[name=Proposition, refname={Proposition,Propositions}, Refname={Proposition,Propositions}, sibling=theorem,]{proposition}

\declaretheorem[name=Proposition, numbered = no]{proposition*}

\declaretheorem[name=Lemma, refname={Lemma,Lemmas}, Refname={Lemma,Lemmas}, sibling=theorem,]{lemma}

\declaretheorem[name=Lemma, numbered = no]{lemma*}

\declaretheorem[name=Corollary, refname={Corollary,Corollaries}, Refname={Corollary,Corollaries}, sibling=theorem,]{corollary}

\declaretheorem[name=Corollary, numbered = no]{corollary*}

\declaretheorem[name=Definition, refname={Definition,Definitions}, Refname={Definition,Definitions}, sibling=theorem, style=definition,]{definition}

\declaretheorem[name=Notation, refname={Notation,Notations}, Refname={Notation,Notations}, sibling=theorem, style=definition,]{notation}

\declaretheorem[name=Example, refname={Example,Examples}, Refname={Example,Examples}, sibling=theorem, style=definition,]{example}

\declaretheorem[name=Examples, refname={Examples,Examples}, Refname={Examples,Examples}, sibling=theorem, style=definition,]{examples}

\declaretheorem[name=Remark, refname={Remark,Remarks}, Refname={Remark,Remarks}, sibling=theorem, style=remark,]{remark}

\declaretheorem[name=Remark, style=remark, numbered=no,]{remark*}

\declaretheorem[name=Claim, refname={Claim,Claims}, Refname={Claim,Claims}, style=remark, sibling=theorem]{claim}

\declaretheorem[name=Claim, numbered=no, style=remark]{claim*}

\newenvironment{claimproof}[1][Proof of Claim]{\begin{proof}[#1]}{\end{proof}}

\crefname{chapter}{Chapter}{Chapters}
\Crefname{chapter}{Chapter}{Chapters}

\crefname{section}{Section}{Sections}
\Crefname{section}{Section}{Sections}

\crefname{axiom}{Axiom}{Axioms}
\Crefname{axiom}{Axiom}{Axioms}

\crefname{item}{Item}{Items}
\Crefname{item}{Item}{Items}

\newglossary[tlg]{category}{tld}{tdn}{List of categories}
\newglossary[tlg]{symbol}{tld}{tdn}{List of symbols}
\newglossaryentry{ALG}{
	type = {category},
	name = {\ensuremath{\ALG\Sigma}},
	description = {category of $\Sigma$-algebras and $\Sigma$-homomorphisms between them},
}
\newglossaryentry{CoAlg}{
	type = {category},
	name = {\ensuremath{\CoAlg(V)}},
	description = {category of coalgebras for the endofunctor $V$},
}
\newglossaryentry{CompHaus}{
	type = {category},
	name = {\ensuremath{\CompHaus}},
	description = {category of compact Hausdorff spaces and continuous functions between them},
}
\newglossaryentry{CompHausXOrd}{
	type = {category},
	name = {\ensuremath{\CompHausXOrd}},
	description = {category of compact Hausdorff spaces equipped with a partial order and order-preserving continuous functions between them},
}
\newglossaryentry{CompHausXPreord}{
	type = {category},
	name = {\ensuremath{\CompHausXPreord}},
	description = {category of compact Hausdorff spaces equipped with a preorder and order-preserving continuous functions between them},
}
\newglossaryentry{CompHausPreord}{
	type = {category},
	name = {\ensuremath{\CompHausPreord}},
	description = {category of compact Hausdorff spaces equipped with a closed preorder and order-preserving continuous functions between them},
}
\newglossaryentry{CompOrd}{
	type = {category},
	name = {\ensuremath{\CompOrd}},
	description = {category of compact ordered spaces and order-preserving continuous functions between them},
}
\newglossaryentry{CMiso}{
	type = {category},
	name = {\ensuremath{\CMiso}},
	description = {category of $\SignCM$-algebras in $\opS\opP([0,1])$ and homomorphisms between them},
}
\newglossaryentry{DLM}{
	type = {category},
	name = {\ensuremath{\DLM}},
	description = {category of {\dlms} and homomorphisms between them},
}
\newglossaryentry{MV}{
	type = {category},
	name = {\ensuremath{\MV}},
	description = {category of {\mvs} and homomorphisms between them},
}
\newglossaryentry{MVM}{
	type = {category},
	name = {\ensuremath{\MVM}},
	description = {category of {\mvms} and homomorphisms between them},
}
\newglossaryentry{TMVM}{
	type = {category},
	name = {\ensuremath{\TMVM}},
	description = {category of {\tmvms} and homomorphisms between them},
}	
\newglossaryentry{TMVMinfty}{
	type = {category},
	name = {\ensuremath{\TMVMinfty}},
	description = {category of {\tmvminftys} and homomorphisms between them},
}
\newglossaryentry{DMVM}{
	type = {category},
	name = {\ensuremath{\DMVM}},
	description = {category of {\dmvms} and homomorphisms between them},
}
\newglossaryentry{DMVMinfty}{
	type = {category},
	name = {\ensuremath{\DMVMinfty}},
	description = {category of {\dmvminftys} and homomorphisms between them},
}
\newglossaryentry{Ord}{
	type = {category},
	name = {\ensuremath{\Ord}},
	description = {category of partially ordered sets and order-preserving functions between them},
}
\newglossaryentry{Preord}{
	type = {category},
	name = {\ensuremath{\Preord}},
	description = {category of preordered sets and order-preserving functions between them},
}
\newglossaryentry{Pries}{
	type = {category},
	name = {\ensuremath{\Pries}},
	description = {category of Priestley spaces and order-preserving continuous functions between them},
}
\newglossaryentry{Set}{
	type = {category},
	name = {\ensuremath{\Set}},
	description = {category of sets and functions between them},
}
\newglossaryentry{Stone}{
	type = {category},
	name = {\ensuremath{\Stone}},
	description = {category of Stone spaces and continuous functions between them},
}
\newglossaryentry{Top}{
	type = {category},
	name = {\ensuremath{\Top}},
	description = {category of topological spaces and continuous functions between them},
}
\newglossaryentry{TopXOrd}{
	type = {category},
	name = {\ensuremath{\TopXOrd}},
	description = {category of topological spaces equipped with a partial order and order-preserving continuous functions between them},
}
\newglossaryentry{TopOrd}{
	type = {category},
	name = {\ensuremath{\TopOrd}},
	description = {category of topological spaces equipped with a closed partial order and order-preserving continuous functions between them}
}
\newglossaryentry{TopPreord}{
	type = {category},
	name = {\ensuremath{\TopPreord}},
	description = {category of topological spaces equipped with a closed preorder and order-preserving continuous functions between them},
}
\newglossaryentry{TopXPreord}{
	type = {category},
	name = {\ensuremath{\TopXPreord}},
	description = {category of topological spaces equipped with a preorder and order-preserving continuous functions between them},
}
\newglossaryentry{ULG}{
	type = {category},
	name = {\ensuremath{\ULG}},
	description = {category of {\ulgs} and homomorphisms between them},
}
\newglossaryentry{ULM}{
	type = {category},
	name = {\ensuremath{\ULM}},
	description = {category of {\ulms} and homomorphisms between them},
}

\newglossaryentry{N}{
	type = {symbol},
	name = {\ensuremath{\N}},
	description = {set of nonnegative integers},
}
\newglossaryentry{Np}{
	type = {symbol},
	name = {\ensuremath{\Np}},
	description = {set of positive integers},
}
\newglossaryentry{Z}{
	type = {symbol},
	name = {\ensuremath{\Z}},
	description = {set of integers},
}
\newglossaryentry{Q}{
	type = {symbol},
	name = {\ensuremath{\Q}},
	description = {set of rational numbers},
}
\newglossaryentry{R}{
	type = {symbol},
	name = {\ensuremath{\R}},
	description = {set of real numbers},
}
\newglossaryentry{Dyad}{
	type = {symbol},
	name = {\ensuremath{\Dyad}},
	description = {set of dyadic rationals},
}
	
\newglossaryentry{empty0}{
	type = {symbol},
	name = {\ },
	description = {\ },
}

\newglossaryentry{upset-subset}{
	type = {symbol},
	name = {\ensuremath{\upset A}},
	description = {up-closure of the subset $A$},
}

\newglossaryentry{downset-subset}{
	type = {symbol},
	name = {\ensuremath{\downset A}},
	description = {down-closure of the subset $A$},
}

\newglossaryentry{upset}{
	type = {symbol},
	name = {\ensuremath{\upset x}},
	description = {up-closure of $\{x\}$},
}

\newglossaryentry{downset}{
	type = {symbol},
	name = {\ensuremath{\downset x}},
	description = {down-closure of $\{x\}$},
}

\newglossaryentry{Cont}{
	type = {symbol},
	name = {\ensuremath{\Cont(X,Y)}},
	description = {set of continuous functions from $X$ to $Y$},
}

\newglossaryentry{Cleq}{
	type = {symbol},
	name = {\ensuremath{\Cleq(X,Y)}},
	description = {set of order-preserving continuous functions from $X$ to $Y$},
}

\newglossaryentry{dist}{
	type = {symbol},
	name = {\ensuremath{\dist(x,y)}},
	description={distance between $x$ and $y$},
}

\newglossaryentry{ev}{
	type = {symbol},
	name = {\ensuremath{\ev}},
	description = {evaluation homomorphism},
}

\newglossaryentry{opcat}{
	type = {symbol},
	name = {\ensuremath{\cat{C}^\opcat}},
	description = {dual category of $\cat{C}$},
}
\newglossaryentry{opalg}{
	type = {symbol},
	name = {\ensuremath{\alge{A}^\opalg}},
	description = {dual algebra of $\alge{A}$},
}
\newglossaryentry{oprel}{
	type = {symbol},
	name = {\ensuremath{R^\oprel}},
	description={opposite relation of the binary relation $R$},
}

\newglossaryentry{Quot}{
	type = {symbol},
	name = {\ensuremath{\Quot(X)}},
	description={class of epimorphisms with domain $X$},
}
\newglossaryentry{QuotEquiv}{
	type = {symbol},
	name = {\ensuremath{\QuotEquiv(X)}},
	description = {set of equivalence classes of epimorphisms with domain $X$},
}
\newglossaryentry{Preo}{
	type = {symbol},
	name = {\ensuremath{\Preo(X)}},
	description = {set of closed pre-orders extending the given partial order on $X$},
}

\newglossaryentry{empty1}{
	type = {symbol},
	name = {\ },
	description = {\ },
}

\newglossaryentry{opH}{
	type = {symbol},
	name = {\ensuremath{\opH(\clalg{A})}},
	description = {closure of $\clalg{A}$ under homomorphic images},
}
\newglossaryentry{opI}{
	type = {symbol},
	name = {\ensuremath{\opI(\clalg{A})}},
	description = {closure of $\clalg{A}$ under isomorphisms},
}
\newglossaryentry{opP}{
	type = {symbol},
	name = {\ensuremath{\opP(\clalg{A})}},
	description = {closure of $\clalg{A}$ under products},
}
\newglossaryentry{opS}{
	type = {symbol},
	name = {\ensuremath{\opS(\clalg{A})}},
	description = {closure of $\clalg{A}$ under subalgebras},
}

\newglossaryentry{empty2}{
	type = {symbol},
	name = {\ },
	description = {\ },
}

\newglossaryentry{Gam}{
	type = {symbol},
	name = {\ensuremath{\Gam(M)}},
	description = {unit interval of $M$},
}
\newglossaryentry{X}{
	type = {symbol},
	name = {\ensuremath{\X(A)}},
	description = {set of good $\Z$-sequences in $A$},
}

\newglossaryentry{empty3}{
	type = {symbol},
	name = {\ },
	description = {\ },
}

\newglossaryentry{SignCM}{
	type = {symbol},
	name = {\ensuremath{\SignCM}},
	description = {signature of all order-preserving continuous functions from powers of $[0,1]$ to $[0,1]$},
}

\newglossaryentry{SignCMomega}{
	type = {symbol},
	name = {\ensuremath{\SignCMomega}},
	description = {signature of all order-preserving continuous functions from at most countable powers of $[0,1]$ to $[0,1]$},
}

\newglossaryentry{SignDMVM}{
	type = {symbol},
	name = {\ensuremath{\SignDMVM}},
	description = {signature $\{\oplus, \odot, \lor, \land\} \cup \Dyad$},
}

\newglossaryentry{SignDMVMinfty}{
	type = {symbol},
	name = {\ensuremath{\SignDMVMinfty}},
	description = {signature $\{\oplus, \odot, \lor, \land\} \cup \Dyad \cup \{\limop\}$},
}

\newglossaryentry{empty4}{
	type = {symbol},
	name = {\ },
	description = {\ },
}

\newglossaryentry{oplus}{
	type = {symbol},
	name = {\ensuremath{x \oplus y}},
	description = {$\min\{x + y, 1\}$},
}
\newglossaryentry{odot}{
	type = {symbol},
	name = {\ensuremath{x \odot y}},
	description = {$\max\{x + y - 1, 0\}$},
}
\newglossaryentry{lor}{
	type = {symbol},
	name = {\ensuremath{x \lor y}},
	description = {$\sup\{x,y\}$},
}
\newglossaryentry{land}{
	type = {symbol},
	name = {\ensuremath{x \land y}},
	description = {$\inf\{x,y\}$},
}
\newglossaryentry{zero}{
	type = {symbol},
	name = {\ensuremath{0}},
	description = {the element $0$},
}
\newglossaryentry{one}{
	type = {symbol},
	name = {\ensuremath{1}},
	description = {the element $1$},
}
\newglossaryentry{dyad}{
	type = {symbol},
	name = {\ensuremath{t \in \Dyadu}},
	description = {the element $t$},
}
\newglossaryentry{trunc}{
	type = {symbol},
	name = {\ensuremath{\trunc_n(x,y)}},
	description = {$\max\left\{\min\left\{x, y + \frac{1}{2^n}\right\}, y - \frac{1}{2^n}\right\}$},
}
\newglossaryentry{aprj-one}{
	type = {symbol},
	name = {\ensuremath{\aprj_1(x_1)}},
	description = {$x_1$},
}
\newglossaryentry{aprj-n}{
	type = {symbol},
	name = {\ensuremath{\aprj_n(x_1, \dots, x_n)}},
	description = {$\trunc_{n-1}(x_n,\aprj_{n-1}(x_1, \dots,x_{n-1}))$},
}
\newglossaryentry{limop}{
	type = {symbol},
	name = {\ensuremath{\limop(x_1, x_2, x_3, \dots)}},
	description = {$\lim_{n \to \infty} \aprj(x_1, \dots, x_n)$},
}
\newglossaryentry{dd-n}{
	type = {symbol},
	name = {\ensuremath{\dd_n}},
	description = {the element $\frac{1}{2^n}$},
}
\newglossaryentry{uu-n}{
	type = {symbol},
	name = {\ensuremath{\uu_n}},
	description = {the element $1-\frac{1}{2^n}$},
}
\newglossaryentry{hh}{
	type = {symbol},
	name = {\ensuremath{\hh(x)}},
	description = {$\frac{x}{2}$},
}
\newglossaryentry{jj}{
	type = {symbol},
	name = {\ensuremath{\jj(x)}},
	description = {$\frac{1}{2} + \frac{x}{2}$},
}

\begin{document}

\frontmatter

\thispagestyle{empty}

\centerline {\includegraphics[width=10cm]{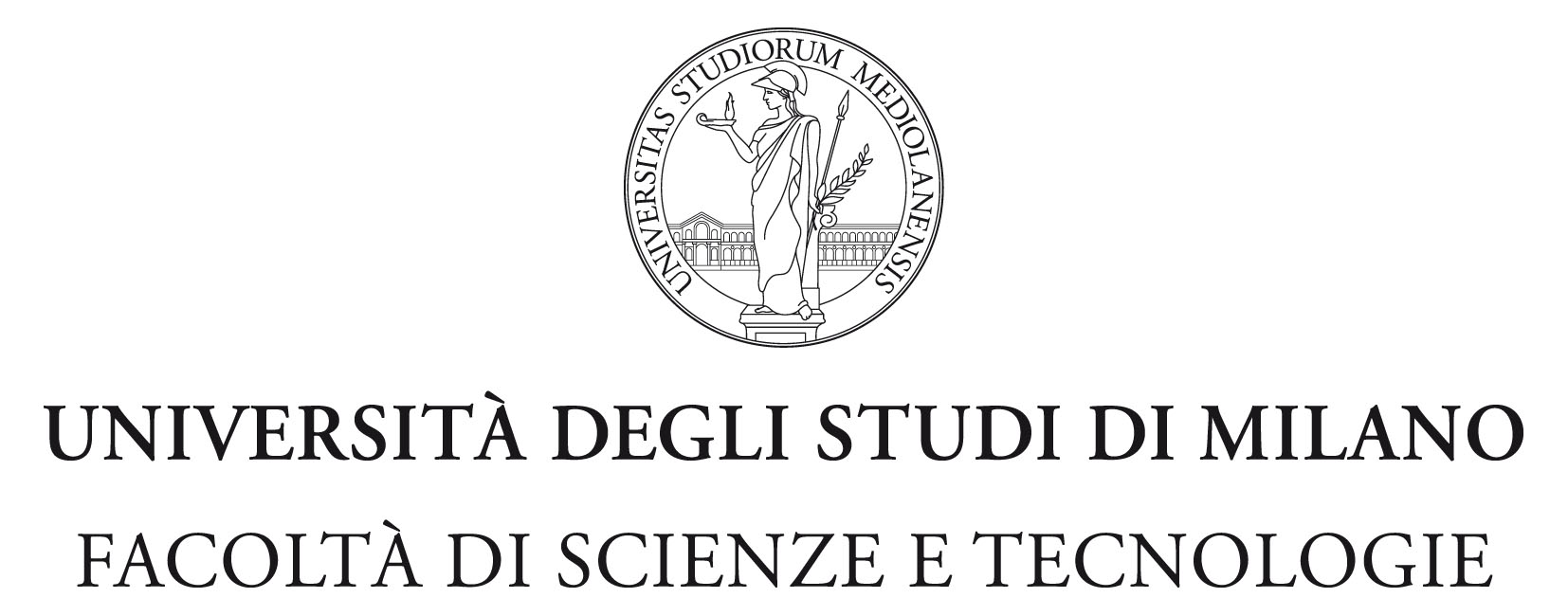}}


\vskip 20 pt
\centerline {\Large{\textsc {}Dipartimento di Matematica Federigo Enriques}}

\vskip 20 pt

\centerline {\Large{\textsc {}Corso di Dottorato di Ricerca in Matematica}}
\vskip 20 pt
\centerline {\Large{\textsc {}XXXIII ciclo}}

\vskip 1.2cm

\centerline {\normalsize {Tesi di Dottorato di Ricerca}} 

\vskip 0.7cm

\begin{center}
\Large {\bf ON THE AXIOMATISABILITY OF\\
THE DUAL OF COMPACT ORDERED SPACES}
\end{center}

\vskip 20 pt
\centerline{MAT01}

\vskip 2cm

\begin{flushleft}
	\noindent\quad\quad\quad Relatore:
	
	\noindent\,\hspace*{1cm} Prof.\ Vincenzo Marra
	
	\vskip 20 pt
	
	\noindent\quad\quad\quad Coordinatore del dottorato:
	
	\noindent\quad\quad\quad Prof.\ Vieri Mastropietro
\end{flushleft}

	\vskip 20 pt

\begin{flushright}
Candidato:\hspace*{1cm}\,

Marco Abbadini\hspace*{1cm}\,
\end{flushright}

\vskip 3.3cm

\centerline{A.\ A.\ 2019-2020}

\cleardoublepage
\phantomsection

\chapter*{Abstract}
We prove that the category of Nachbin's compact ordered spaces and order-preserving continuous maps between them is dually equivalent to a variety of algebras, with operations of at most countable arity.
Furthermore, we show that the countable bound on the arity is the best possible: the category of compact ordered spaces is not dually equivalent to any variety of finitary algebras.
Indeed, the following stronger results hold: the category of compact ordered spaces is not dually equivalent to (i) any finitely accessible category, (ii) any first-order definable class of structures, (iii) any class of finitary algebras closed under products and subalgebras.
An explicit equational axiomatisation of the dual of the category of compact ordered spaces is obtained; in fact, we provide a finite one, meaning that our description uses only finitely many function symbols and finitely many equational axioms.
In preparation for the latter result, we establish a generalisation of a celebrated theorem by D.\ Mundici:
our result asserts that the category of unital commutative distributive lattice-ordered monoids is equivalent to the category of what we call MV-monoidal algebras.
Our proof is independent of Mundici's theorem.


\chapter*{Acknowledgements}
I wish to thank Vincenzo Marra for being my advisor: I cannot overestimate the value his teaching and mentoring have had during my Ph.D.\ studies.
In fact, this work would hardly have been possible without his effort to convey me his organised and deep vision of mathematics, and I am grateful for his many suggestions and useful advice.

I would like to express my gratitude to Luca Reggio for having taken the time to improve a version of a paper of mine, which was the beginning of the whole story contained in this thesis.
Also, Luca explained me an important bit of categorical algebra, which opened me up to a new understanding of the subject, and I enjoyed one brilliant idea of his which led to a nice proof of the main result of this manuscript.

Finally, I would like to thank the two anonymous referees who read this thesis.
I am grateful to the first referee for his or her interesting ideas and encouragement for further research.
I am grateful to the second referee for several comments and suggestions that helped to improve the quality of this manuscript.

\tableofcontents

\setcounter{chapter}{-1}


\chapter*{Introduction}

\addcontentsline{toc}{chapter}{Introduction}
\markboth{Introduction}{Introduction}

In 1936, in his landmark paper \cite{Stone1936}, M.\ H.\ Stone described what is nowadays known as Stone duality for Boolean algebras.
In modern terms, the result states that the category of Boolean algebras and homomorphisms is dually equivalent to the category of totally disconnected compact Hausdorff spaces and continuous maps, now known as Stone or Boolean spaces.
If we drop the assumption of total disconnectedness, we are left with the category $\CompHaus$ of compact Hausdorff spaces and continuous maps.
J.\ Duskin observed in 1969 that the opposite category $\CompHaus^\opcat$ is monadic over the category $\Set$ of sets and functions \cite[{}5.15.3]{Duskin1969}.
In fact, $\CompHaus^{\opcat}$ is equivalent to a variety of algebras with primitive operations of at most countable arity: a finite generating set of operations was exhibited in~\cite{Isbell1982}, while a finite equational axiomatisation was provided in~\cite{MarraReggio2017}.
Therefore, if we allow for infinitary operations, Stone duality can be lifted to compact Hausdorff spaces, retaining the algebraic nature of the category involved.

Shortly after his paper on the duality for Boolean algebras, Stone published a generalisation of this theory to bounded distributive lattices~\cite{Stone1938}.
In his formulation, the dual category consists of what are nowadays called spectral spaces and perfect maps.
In 1970, H.\ A.\ Priestley showed that spectral spaces can be equivalently described as what are now known as Priestley spaces, i.e.\ compact spaces equipped with a partial order satisfying a condition called total order-disconnectedness~\cite{Priestley1970}.
More precisely, Priestley duality states that the category of bounded distributive lattices is dually equivalent to the category of Priestley spaces and order-preserving continuous maps.
The latter category is a full subcategory of the category $\CompOrd$ of compact ordered spaces and order-preserving continuous maps.
Compact ordered spaces, introduced by L.\ Nachbin \cite{Nachbin1948,Nachbin} before Priestley's result on bounded distributive lattices, are compact spaces equipped with a partial order which is closed in the product topology.
Similarly to the case of Boolean algebras, one may ask if Priestley duality can be lifted to the category $\CompOrd$ of compact ordered spaces, retaining the algebraic nature of the opposite category.
In \cite{HNN2018}, D.\ Hofmann, R.\ Neves and P.\ Nora showed that $\CompOrdop$ is equivalent to an $\aleph_1$-ary quasivariety\index{quasivariety of algebras!$\aleph_1$-ary}\index{$\aleph_1$-ary|see{quasivariety of algebras, $\aleph_1$-ary}}, i.e.\ a quasivariety of algebras with operations of at most countable arity, axiomatised by implications with at most countably many premises.
In the same paper, the authors left as open the question whether $\CompOrdop$ is equivalent to a variety of (possibly infinitary) algebras.

The main aims of this manuscript are (i) to show that the dual of the category of compact ordered spaces is in fact equivalent to a variety of (infinitary) algebras---thus providing a positive answer to the open question in \cite{HNN2018}---and (ii) to obtain a finite equational axiomatisation of $\CompOrdop$.

The main results of this thesis are presented in \cref{chap:direct-proof,chap:negative-results,chap:equiv,chap:yosida-like,chap:axiomatisation,chap:finite-axiomatisation}.

In \cref{chap:direct-proof}, we prove that $\CompOrdop$ is equivalent to a variety, with primitive operations of at most countable arity (\cref{t:MAIN});
the proof rests upon a well-known categorical characterisation of varieties.
In \cref{chap:negative-results}, we show that the countable bound on the arity of the primitive operations is the best possible: $\CompOrd$ is not dually equivalent to any variety of finitary algebras (\cref{t:finitely accessible}).
Indeed, the following stronger results hold: $\CompOrd$ is not dually equivalent to (i) any finitely accessible category, (ii) any first-order definable class of structures, (iii) any class of finitary algebras closed under products and subalgebras.

The main goal of the subsequent chapters is to provide explicit equational axiomatisations of $\CompOrdop$.
This is first achieved in \cref{chap:axiomatisation}, where we prove that the category of compact ordered spaces is dually equivalent to the variety consisting of what we call \emph{\dmvminftys} (\cref{t:axiomatisation}).
The axiomatisation builds on what we call \emph{\mvms}.
These generalise {\mvs}, originally introduced by \cite{Chang1958} to serve as algebraic semantics for {\L}ukasiewicz many-valued propositional logic.
In \cref{chap:finite-axiomatisation}, we take a further step by providing a \emph{finite} equational axiomatisation of $\CompOrdop$, meaning that we use only finitely many function symbols and finitely many equational axioms to present the variety: the dual of $\CompOrd$ is there presented as the variety of \emph{\tmvminftys} (\cref{t:axiomatisation-t}).
This finite axiomatisation is a bit more complex than the infinite one in \cref{chap:axiomatisation}.

The intermediate \cref{chap:equiv,chap:yosida-like} are of independent interest.
Even if they are not really necessary to obtain the results reported above, they serve to provide a better intuition on the algebras of the equational axiomatisations of \cref{chap:axiomatisation,chap:finite-axiomatisation}, which otherwise may seem a bit obscure.

In particular, in \cref{chap:equiv}, we show that {\mvms} are exactly the unit intervals of the more intuitive structures that we call \emph{\extulms}.
In fact, we prove that the categories of {\extulms} and {\mvms} are equivalent (\cref{t:G is equiv}).
This equivalence generalises a well-known result by \cite[Theorem~3.9]{Mundici} asserting that the category of Abelian lattice-ordered groups with strong order unit is equivalent to the category of MV-algebras.

Starting from 1940, several descriptions of the dual of the category of compact Hausdorff spaces were obtained: we here cite the works of \cite{KreinKrein1940,Gelfand1941,Kakutani1941,Stone1941,Yosida1941}.
The latter used lattice-ordered vector spaces with strong order unit, and in \cref{chap:yosida-like} we obtain an ordered analogue of this result, that we call \emph{ordered Yosida duality}.
In our formulation, compact Hausdorff spaces are replaced by compact ordered spaces, and lattice-ordered vector spaces with strong order unit are replaced by what we call \emph{\extdlms}.
Even if these structures fail to form a variety, we find them interesting because their axiomatisation is simpler than the equational ones of the following chapters.
This ordered version of Yosida duality fits in the general structure of the manuscript by providing a more accessible intuition to the ideas behind the dualities in \cref{chap:axiomatisation,chap:finite-axiomatisation}.
Also, to prove some of the results in these last two chapters, we will rely on analogous ones obtained in \cref{chap:yosida-like}, for which an easier-to-follow proof will have already been carried out in details.

We conclude this introduction by commenting on the novelty of the results presented here.
It is well-known that the category of compact ordered spaces and order-preserving continuous maps is isomorphic to the category of so-called stably compact spaces and perfect maps \cite[Proposition~VI.6.23]{GierzHofmannKeimelEtAl2003}.
Dualities for stably compact spaces are already known.
For example, the category of stably compact spaces is dually equivalent to the category of stably locally compact frames \cite[Theorem~VI-7.4]{GierzHofmannKeimelEtAl2003}, as well as to the category of strong proximity lattices \cite{JungSuenderhauf1996}.
However, neither of these two categories is (in its usual presentation) a variety of algebras.

To the best of our knowledge, the fact that the category of compact ordered spaces is dually equivalent to a variety of algebras was first proved in \cite{Abbadini2019}, where an explicit (infinite) equational axiomatisation was also presented.
Then, in \cite{AbbadiniReggio2020}, a nicer and shorter proof was obtained, which rested upon a well-known categorical characterisation of varieties.
Additionally, in the latter paper, the aforementioned negative axiomatisability results about $\CompOrdop$ were observed.
The result that $\CompOrd$ is not dually equivalent to any first-order definable class of structures was suggested to us by S.\ Vasey (private communication) as an application of a result of M.\ Lieberman, J.\ Rosick\'y and S.\ Vasey \cite{LieRosVas}, replacing a previous weaker statement.
\Cref{chap:direct-proof,chap:negative-results} follow very closely the lines of \cite{AbbadiniReggio2020}.

We believe that the equational axiomatisation of $\CompOrdop$ in \cref{chap:axiomatisation} is nicer than the one available in \cite{Abbadini2019}, one of the reasons being the self-duality of the primitive operation of countably infinite arity.
The result in \cref{chap:finite-axiomatisation} that also a finite equational axiomatisation exists is new.

\mainmatter


\chapter{Background}\label{chap:background}


We collect here some basic notions and some preliminary results about sets, preordered sets, topological spaces, categories and algebras.
The reader may wish to move to \cref{chap:CompOrd}, skipping the present chapter and referring to it if the need arises.

\section{Foundations}

For the concepts of set\index{set} and class\index{class} we refer to \cite[Chapter~2]{AdamekHerrlichStrecker2006}.

\section{Sets}

We denote with $\N$ the set of natural numbers $\{0, 1, 2, \dots\}$, and with $\Np$ the set $\N \setminus \{0\} = \{1, 2, 3, \dots\}$.

We assume the axiom of choice.

Given a function $f \colon X \to Y$, we let $f[A]$ denote the image of a subset $A \seq X$ under $f$ and $f^{-1}[B]$ the preimage of a subset $B \seq Y$ under $f$.


\subsection{Binary relations} \label{s:binary-relations}

A \emph{binary relation}\index{relation!on a set!binary} on a set $X$ is a subset $R \seq X \times X$.
We write $x \mathrel{R} y$ as an alternative to $(x, y) \in R$.
Given a binary relation $R$, we let $R^\oprel$ denote the relation defined by 
\[
	x \mathrel{R^\oprel} y \ \text{ if, and only if, } \ y \mathrel{R} x.
\]

A binary relation $R$ on a set $X$ is called
\begin{description}
	\item[\emph{reflexive}] \index{relation!on a set!reflexive} provided that, for all $x \in X$, we have $x \mathrel{R} x$;
	\item[\emph{transitive}] \index{relation!on a set!transitive} provided that, for all $x,y,z \in X$, if $x \mathrel{R} y$ and $y \mathrel{R} z$, then $x \mathrel{R} z$;
	\item[\emph{anti-symmetric}] \index{relation!on a set!anti-symmetric} provided that, for all $x,y \in X$, if $x \mathrel{R} y$ and $y \mathrel{R} x$, then $x = y$;
	\item[\emph{symmetric}]\index{relation!on a set!symmetric} provided that, for all $x,y \in X$, if $x \mathrel{R} y$, then $y \mathrel{R} x$.
\end{description}
A \emph{preorder}\index{preorder} on a set $X$ is a reflexive transitive binary relation on $X$.
A \emph{partial order}\index{partial order} on $X$ is an anti-symmetric preorder on $X$.
An \emph{equivalence relation}\index{relation!on a set!equivalence} on $X$ is a reflexive transitive symmetric binary relation on $X$.

Usually, we use the symbol $\les$ for preorders, and the symbol $\leq$ for partial orders.

A \emph{preordered set}\index{preordered set} is a pair $(X, \les)$ where $X$ is a set and $\les$ is a preorder on $X$.
A \emph{partially ordered set}\index{partially ordered set} is a pair $(X, \leq)$ where $X$ is a set and $\leq$ is a preorder on $X$.
To keep notation simple, we shall often write simply $X$ instead of $(X, \les)$ or $(X, \leq)$.

An \emph{up-set}\index{up-set} in a preordered set $(X, {\les})$ is a subset $U$ of $X$ such that, for every $x \in U$ and $y \in X$, if $x \les y$, then $y \in U$.
A \emph{down-set}\index{down-set} is a subset $U$ of $X$ such that, for every $x \in U$ and $y \in X$, if $y \les x$, then $y \in U$.

Analogous concepts are defined for the case of classes instead of sets.


\section{Preordered sets}

Let $(X, \les_X)$ and $(Y, \les_Y)$ be preordered sets.
An \emph{order-preserving}\index{function!order-preserving}\index{order-preserving|see{function, order-preserving}} function from $(X, \les_X)$ to $(Y, \les_Y)$ is a function $f\colon X \to Y$ such that, for all $x, y \in X$, if $x \les_X y$, then $f(x) \les_Y f(y)$.


\subsection{Initial and final preorder} \label{s:init-preorder}

In this subsection we define the notions of initial, final, product, coproduct, discrete, indiscrete, induced and quotient (pre)order.
The motivation for most of this terminology will become clear after the discussion on topological functors below (\cref{s:topological-functors}).

Given a set $X$ and a class-indexed\footnote{We use a class as index to adhere to the definition of topological functors (\cref{d:topological-functor} below).} family $(f_i \colon X \to A_i)_{i \in I}$ of functions from $X$ to the underlying set of a preordered set $(A_i, \les_i)$, we define the \emph{initial preorder}\index{initial!preorder}\index{preorder!initial}\index{order!initial} on $X$ as the greatest preorder $\les$ on $X$ that makes each $f_i$ order-preserving.
It is given by
\[
	x \les y \ \Longleftrightarrow \ \forall i \in I\ f_i(x) \les_i f_i(y).
\]

Given a set $X$ and a class-indexed family $(f_i \colon A_i \to X)_{i \in I}$ of functions from $X$ to the underlying set of a preordered set $(A_i, \les_i)$, we define the \emph{final preorder}\index{final!preorder}\index{preorder!final}\index{order!final} on $X$ as the smallest preorder $\les$ on $X$ that makes each $f_i$ order-preserving.
It is given by the transitive closure of the reflexive closure of the relation
\begin{equation} \label{i:final-preo}
	x \les' y \ \Longleftrightarrow \ \exists i \in I\ \exists x',y' \in A_i:\ x' \les_i y',\ f_i(x')= x,\ f_i(y') = y.
\end{equation}
As one of the anonymous referees pointed out, correcting an error in a preliminary version of this thesis, the reflexive closure of the relation $\les'$ described in \cref{i:final-preo} might fail to be transitive.
For example, let $\{a,b,c,d\}$ be a preordered set on four elements with $a \leq b$ and $c \leq d$, and no other pairs of distinct elements in relation, and consider the map $f \colon \{a,b,c,d\} \to \{x, y, z\}$ that maps $a$ to $x$, $b$ and $c$ to $y$ and $d$ to $z$.

Given a family $(X_i, \les_i)_{i \in I}$ (with $I$ a set) of preordered sets, the \emph{product preorder}\index{product!preorder}\index{preorder!product}\index{order!product} on the set-theoretic product $X \df \prod_{i \in I} X_i$ is the initial preorder $\les$ on $X$ with respect to the family of projections $(\pi_i \colon X \to X_i)_{i \in I}$, i.e.\
\[
	(x_i)_{i \in I} \les (y_i)_{i \in I} \ \Longleftrightarrow \ \forall i \in I\ x_i \les_i y_i.
\]
Unless otherwise specified, it is understood that the set $\prod_{i \in I} X_i$, when regarded as a preordered set, is equipped with the product preorder.

The \emph{coproduct preorder}\index{coproduct!preorder}\index{preorder!coproduct}\index{order!coproduct} on the set-theoretic coproduct $X \df \sum_{i \in I} X_i$ is the final preorder on $X$ with respect to the family of coproduct injections $(\iota_i \colon X_i \hookrightarrow X)_{i \in I}$, i.e.\
\[
	x \les y \ \Longleftrightarrow \ \exists i \in I\ \exists x', y' \in X_i :\ x' \les_i y',\ \iota_i(x') = x,\ \iota_i(y') = y.
\]

Every set $X$ carries two canonical preorders: the \emph{discrete preorder}\index{discrete!preorder}\index{preorder!discrete}\index{order!discrete}, i.e.\ $\{(x,x) \in X \times X \mid x \in X\}$, and the \emph{indiscrete preorder}\index{indiscrete!preorder}\index{preorder!indiscrete}\index{order!indiscrete}, i.e.\ $X \times X$.

Let $(X, \les_X)$ be a preordered set, and let $\iota \colon Y \hookrightarrow X$ be an injective function.
The \emph{induced preorder}\index{induced!preorder}\index{preorder!induced}\index{order!induced} on $Y$ is the initial preorder $\les_Y$ on $Y$ with respect to $\iota$, i.e.\ 
\[
	x \les_Y y \ \Longleftrightarrow \ \iota(x) \les_X \iota(y).
\]

Let $(X, \les_X)$ be a preordered set, and let $q \colon X \twoheadrightarrow Y$ be a surjective function.
The \emph{quotient preorder}\index{quotient!order} on $Y$ is the final preorder $\les_Y$ on $Y$ with respect to $q$, i.e.\ the transitive closure of the relation
\begin{equation} \label{e:quotient-preorder}
	x \les' y \ \Longleftrightarrow \ \exists x', y' \in X :\ x' \les_X y',\ q(x') = x,\ q(y') = y.
\end{equation}
If, for all $x, y \in X$ with $q(x) = q(y)$, we have $x \les_X y$, then the relation on $Y$ defined by \cref{e:quotient-preorder} is transitive, so there is no need to take the transitive closure.

Unless otherwise specified, it is understood that subsets and quotient sets of a preordered sets, when regarded as a preordered set, are equipped with the induced and quotient preorder, respectively.

Initial (resp.\ final, product, coproduct, discrete, indiscrete, induced, quotient) preorder will also be called \emph{initial} (resp.\ \emph{final}, \emph{product}, \emph{coproduct}, \emph{discrete}, \emph{indiscrete}, \emph{induced}, \emph{quotient}) \emph{order}.


\section{Topological spaces}

We assume that the reader has basic knowledge of topology, for which we refer to \cite{Willard1970}.

A \emph{topology}\index{topology} on a set $X$ is a set of subset of $X$ which is closed under arbitrary unions and finite intersections (where by $\emptyset$-indexed union we mean the empty set, and by $\emptyset$-indexed intersection we mean the set $X$).
A \emph{topological space}\index{topological space} is a pair $(X, \tau)$ where $X$ is a set and $\tau$ is a topology on $X$.
To keep notation simple, we shall often write simply $X$ instead of $(X, \tau)$.

Let $(X, \tau_X)$ and $(Y, \tau_Y)$ be topological spaces.
A \emph{continuous function}\index{function!continuous}\index{continuous|see{function, continuous}} from $(X, \tau_X)$ to $(Y, \tau_Y)$ is a function $f\colon X \to Y$ such that, for all $O \in \tau_Y$, we have $f^{-1}[O] \in \tau_X$.


\subsection{Initial and final topology} \label{s:init-topology}

In this subsection we define the notions of initial, final, product, coproduct, discrete, indiscrete, induced and quotient topology.
As for preordered sets, the motivation for most of this terminology will become clear after the discussion on topological functors below (\cref{s:topological-functors}).

Given a set $X$ and a class-indexed family $(f_i \colon X \to A_i)_{i \in I}$ of functions from $X$ to the underlying set of a topological space $(A_i, \tau_i)$, we define the \emph{initial topology}\index{initial!topology}\index{topology!initial} on $X$ as the smallest topology $\tau$ on $X$ that makes each $f_i$ continuous, i.e.\ the topology generated by $\{f_i^{-1}[O] \mid i \in I, O \in \tau_i \}$.

Given a set $X$ and a class-indexed family $(f_i \colon A_i \to X)_{i \in I}$ of functions from $X$ to the underlying set of a topological space $(A_i, \les_i)$, we define the \emph{final topology}\index{final!topology}\index{topology!final} on $X$ as the greatest topology $\tau$ on $X$ that makes each $f_i$ continuous, i.e.\
\[
	S \in \tau \ \Longleftrightarrow \  \forall i \in I \ f_i^{-1}[S] \in \tau_i.
\]

Given a family $(X_i, \tau_i)_{i \in I}$ (with $I$ a set) of topological spaces, the \emph{product topology}\index{product!topology}\index{topology!product} on the set-theoretic product $X \df \prod_{i \in I} X_i$ is the initial topology on $X$ with respect to the family of projections $(\pi_i \colon X \to X_i)_{i \in I}$, i.e.\ the topology generated by $\{\pi_i^{-1}[O] \mid i \in I, O \in \tau_i \}$.
Unless otherwise specified, it is understood that the set $\prod_{i \in I} X_i$, when regarded as a topological space, is equipped with the product topology.

The \emph{coproduct topology}\index{coproduct!topology}\index{topology!coproduct} on the set-theoretic coproduct $X \df \sum_{i \in I} X_i$ is the final topology on $X$ with respect to the family of coproduct injections $(\iota_i \colon X_i \hookrightarrow X)_{i \in I}$, i.e.\
\[
	S \in \tau \ \Longleftrightarrow \ \forall i \in I \ \iota_i^{-1}[S] \in \tau_i.
\]

Every set $X$ carries two canonical topology: the \emph{discrete topology}\index{discrete!topology}\index{topology!discrete}, consisting of all subsets of $X$, and the \emph{indiscrete topology}\index{indiscrete!topology}\index{topology!indiscrete}, consisting of $\emptyset$ and $X$.

Let $(X, \tau_X)$ be a topological space, and let $\iota \colon Y \hookrightarrow X$ be an injective function.
The \emph{induced topology}\index{induced!topology}\index{topology!induced} on $Y$ is the initial topology $\tau_Y$ on $Y$ with respect to $\iota$, i.e.\ 
\[
	S \in \tau_Y \ \Longleftrightarrow \ \exists O \in \tau_X :\ \iota^{-1}[O] = S.
\]

Let $(X, \tau_X)$ be a topological space, and let $q \colon X \twoheadrightarrow Y$ be a surjective function.
The \emph{quotient topology}\index{quotient!topology}\index{topology!quotient} on $Y$ is the final topology $\tau_Y$ on $Y$ with respect to $q$, i.e.\ 
\[
	S \in \tau_Y \ \Longleftrightarrow \ q^{-1}[S] \in \tau_X.
\]

Unless otherwise specified, it is understood that subsets and quotient sets of a topological space, when regarded as a topological space, are equipped with the induced and quotient topology, respectively.


\subsection{Compact Hausdorff spaces}

We recall that a topological space $X$ is said to be \emph{compact}\index{topological space!compact} if each open cover of $X$ has a finite subcover, and \emph{Hausdorff}\index{topological space!Hausdorff} if, for all distinct $x, y \in X$, there exist disjoint open sets $U$ and $V$ in $X$ with $x \in U$ and $y \in V$.
We recall some basic facts about compact Hausdorff spaces.

\begin{proposition} [{\cite[Theorem~13.7]{Willard1970}}] \label{p:diag-closed}
	 A topological space $X$ is Hausdorff if, and only if, its diagonal $\{(x,x) \mid x \in X\}$ is a closed subspace of $X \times X$.
\end{proposition}

\begin{proposition} [{\cite[Theorem~17.7]{Willard1970}}] \label{p:image-of-compact}
	The image of a compact space under a continuous map is compact.
\end{proposition}

\begin{proposition}[{\cite[Theorem~17.5.b]{Willard1970}}] \label{p:compact-in-Haus}
	A compact subset of a Hausdorff space is closed.
\end{proposition}

\begin{proposition}[{\cite[Theorem~17.5.a]{Willard1970}}] \label{p:closed-in-compact}
	Every closed subspace of a compact space is compact.
\end{proposition}

\begin{proposition}\label{p:comp-Haus-closed}
	Every continuous map between compact Hausdorff spaces is closed.
\end{proposition}
\begin{proof}
	Let $f \colon X \to Y$ be a continuous map between compact Hausdorff spaces, and let $K$ be a closed subspace of $X$.
	Then, by \cref{p:closed-in-compact}, $K$ is compact.
	Then, by \cref{p:image-of-compact}, $f[K]$ is compact.
	Then, by \cref{p:compact-in-Haus}, $f[K]$ is closed.
\end{proof}

\begin{proposition} [{\cite[Theorem~17.14]{Willard1970}}]\label{p:comp-Haus-homeo}
	Every continuous bijection between compact Hausdorff spaces is a homeomorphism.
\end{proposition}

\begin{theorem} [{Thychonoff's theorem, \cite[Theorem~17.8]{Willard1970}}]\index{Thyconoff's theorem} \label{t:Thych}
	A product of compact spaces is compact.
\end{theorem}


\section{Categories\index{category}}

We assume that the reader has basic knowledge of categories, functors, and adjunctions, for which we refer to \cite{MacLane1998}.

Throughout this manuscript all categories are assumed to be \emph{locally small}\index{category!locally small}\index{locally small|see{category, locally small}}, i.e.\ given two objects $X$ and $Y$, the morphisms from $X$ to $Y$ form a set.

For infinite products, we use the notation $\prod_{i\in I}{X_i}$, or $X_1 \times \dots \times X_n$ in case of a finite index set.
For infinite coproducts, we use the notation $\sum_{i\in I}{X_i}$, or $X_1 + \dots + X_n$ in case of a finite index set.

\begin{proposition}[{\cite[Theorem~1]{MacLane1998}}] \label{p:adjoint-preservation}
	Left adjoints preserve colimits and right adjoints preserve limits.
\end{proposition}

\begin{definition} \label{d:reflective}
	A \emph{reflective full subcategory}\index{subcategory!full!reflective}\index{reflective|see{subcategory, full, reflective}} of a category $\cat{C}$ is a full subcategory $\cat{D}$ whose inclusion functor $\cat{D} \hookrightarrow \cat{C}$ admits a left adjoint, called the \emph{reflector}\index{reflector}\index{functor!reflector|see{reflector}}.
\end{definition}

\begin{proposition}[{\cite[Propositions~3.5.3 and~3.5.4]{Borceux1994-vol1}}] \label{p:complete-reflective}
	Let $\cat{D}$ be a reflective full subcategory of a category $\cat{C}$.
	If $\cat{C}$ is complete (resp.\ cocomplete), then $\cat{D}$ is complete (resp.\ cocomplete).
\end{proposition}


\subsection{Topological functors} \label{s:topological-functors}

We recall basic notions and results concerning topological functors.
For more details, we refer to \cite[Chapter~21]{AdamekHerrlichStrecker2006}.

\begin{definition}
	A \emph{source}\index{source} is a pair $(A,(f_i)_{i\in I})$ consisting of an object $A$ and a family of arrows $f_i\colon A\to A_i$ with domain $A$, indexed by some class $I$.
	The object $A$ is called the \emph{domain of the source}\index{domain of a source}\index{source!domain of|see{domain of a source}} and the family $(A_i)_{i\in I}$ is called the \emph{codomain of the source}\index{codomain of a source}. Whenever convenient we use the notation $(f_i\colon A\to A_i)_{i\in I}$ instead of $(A,(f_i)_{i\in I})$.
\end{definition}

Given a faithful functor $U \colon \cat{A}\to \cat{X}$, given two objects $A$ and $B$ of $\cat{A}$ and an $\cat{X}$-morphism $f\colon U(A)\to U(B)$, we say that \emph{$f$ is an $\cat{A}$-morphism from $A$ to $B$} if there exists a (necessarily unique) $\cat{A}$-morphism $\overline{f}\colon A \to B$ such that $U(\overline{f})=f$.
When $A$ and $B$ are understood, we simply say that \emph{$f$ is an $\cat{A}$-morphism}.

\begin{definition} \label{d:topological-functor}
	Let $U \colon \cat{A} \to \cat{X}$ be a faitfhul functor.
	\begin{enumerate}[wide]
		
		\item A \emph{$U$-structured source}\index{source!structured}\index{structured!source|see{source, structured}} is a pair $(X,(f_i,A_i)_{i\in I})$ consisting of an object $X$ of $\cat{X}$ and a family of pairs $(f_i,A_i)$, indexed by some class $I$, consisting of an object $A_i$ of $\cat{A}$ and an $\cat{X}$-morphism $f_i\colon X\to U(A_i)$.
		Whenever convenient, we use the notation $(f_i\colon X\to U(A_i))_{i\in I}$ instead of $(X,(f_i,A_i)_{i\in I})$.
		
		\item A source $(f_i\colon A\to A_i)_{i\in I}$ in $\cat{A}$ is said to be \emph{$U$-initial}\index{initial!source} provided that, for each object $C$ of $\cat{A}$, an $\cat{X}$-morphism $h\colon U(C)\to U(A)$ is an $\cat{A}$-morphism if, and only if, for all $i\in I$, the composite $U(C) \xrightarrow{h} U(A) \xrightarrow{G(f_i)} U(A_i)$ is an $\cat{A}$-morphism.
		As a particular case, an $\cat{A}$-morphism $f\colon A\to B$ is \emph{$U$-initial}\index{initial!morphism}\index{morphism!initial} provided that, for each object $C$ of $\cat{A}$, an $\cat{X}$-morphism $h\colon U(C)\to U(A)$ is an $\cat{A}$-morphism if, and only if, the composite $U(C) \xrightarrow{h} U(A) \xrightarrow{G(f)} U(B)$ is an $\cat{A}$-morphism.
		
		\item A \emph{lift}\index{lift of a structured source} of a $U$-structured source $(f_i\colon X\to U(A_i))_{i\in I}$ is a source $(\bar{f}_i\colon A\to A_i)_{i\in I}$ such that $U(A)=X$ and $U(\bar{f}_i)=f_i$.
		
		\item We say that $U$ is \emph{topological}\index{functor!topological} if every $U$-structured source has a unique $U$-initial lift.
	\end{enumerate}
\end{definition}

Topological functors are sometimes introduced with a definition that does not assume faithfulness \cite[Definition~21.1]{AdamekHerrlichStrecker2006}.
However, faithfulness is a consequence \cite[Theorem~21.3]{AdamekHerrlichStrecker2006}.

We let $\Set$ denote the category of sets and functions.

\begin{examples} \label{exs:topological-functors!}
	\begin{description}[wide]
		\item[Topological spaces.] Let $\Top$ denote the category of topological spaces and continuous functions.
		The forgetful functor $\Top \to \Set$ is topological: the unique initial lift of a family of functions $(f_i \colon X \to A_i)$ is obtained by providing $X$ with the initial topology (see \cref{s:init-topology}).
		
		\item[Preordered sets.] Let $\Preord$ denote the category of preordered sets and order-preserving functions.
		The forgetful functor $\Preord \to \Set$ is topological: the unique initial lift of a family of functions $(f_i \colon X \to A_i)$ is obtained by providing $X$ with the initial preorder  (see \cref{s:init-preorder}).
	\end{description}
\end{examples}

\begin{definition}
	Let $U \colon \cat{A} \to \cat{X}$ be a faithful functor.
	Given an object $X$ of $\cat{X}$, we call \emph{fibre of $X$}\index{fibre} the preordered class consisting of all objects $A$ of $\cat{A}$ with $U(A) = X$ ordered by:
	\[
		A \les B \text{ if, and only if, }1_X \colon U(A) \to U(B) \text{ is an $\cat{A}$-morphism.}
	\]
\end{definition}

\begin{examples}
	\begin{description}[wide]
		\item[Topological spaces.]
		The fibre of a set $X$ with respect to the forgetful functor $\Top \to \Set$ is the set of topologies on $X$, ordered by reverse inclusion.
		
		\item[Preordered sets.]
		The fibre of a set $X$ with respect to the forgetful functor $\Preord \to \Set$ is the set of preorders on $X$, ordered by inclusion.
	\end{description}
\end{examples}

\begin{proposition}
	Let $U \colon \cat{A} \to \cat{X}$ be a topological functor.
	Then, the fibre of every object of $\cat{X}$ is a partially ordered class in which every subclass has a join and a meet.
\end{proposition}
\begin{proof}
	Each fibre is a partially ordered class by the implication (1) $\Rightarrow$ (2) in \cite[Proposition~21.5]{AdamekHerrlichStrecker2006}, and every subclass has a join and a meet by \cite[Proposition~21.11]{AdamekHerrlichStrecker2006}.
\end{proof}

The dual notion of source is \emph{sink}\index{sink}\index{sink!structured}\index{structured!sink|see{sink, structured}}, and the dual notion of initial is \emph{final}\index{final!sink}\index{final!morphism}\index{morphism!final}.

\begin{remark}
	Let $U\colon \cat{A} \to \cat{X}$ be a topological functor.
	Then, the unique initial lift of a $U$-structured source $(f_i\colon X\to U(A_i))_{i\in I}$ is the smallest element $A$ in the fibre of $X$ such that, for every $i \in I$, $f_i$ is an $\cat{A}$-morphism from $A$ to $A_i$.
	Similarly, every $U$-structured sink $(f_i\colon U(A_i) \to X)_{i\in I}$ admits a unique $U$-final lift, i.e.\ the greatest element $A$ in the fibre of $X$ such that, for every $i \in I$, $f_i$ is an $\cat{A}$-morphism from $A_i$ to $A$.
	In fact, for any topological functor $U\colon \cat{A} \to \cat{X}$, also the functor $U^\opcat \colon \cat{A}^\opcat \to \cat{X}^\opcat$ is topological {\cite[Topological duality theorem~21.9]{AdamekHerrlichStrecker2006}}.
\end{remark}

Let $U \colon \cat{A} \to \cat{X}$ be a faithful functor, and let $A$ be an object of $\cat{A}$.
The object $A$ is called \emph{discrete}\index{discrete!object}\index{object!discrete} whenever, for each object $B$, every $\cat{X}$-morphism $U(A) \to U(B)$ is an $\cat{A}$-morphism.
The object $A$ is called \emph{indiscrete}\index{indiscrete!object}\index{object!indiscrete} whenever, for each object $B$, every $\cat{X}$-morphism $U(B) \to U(A)$ is an $\cat{A}$-morphism.

\begin{proposition}[{\cite[Proposition~21.11]{AdamekHerrlichStrecker2006}}]
	The smallest (resp.\ largest) element of each fibre of a topological functor is discrete (resp.\ indiscrete).
\end{proposition}

\begin{examples}
	\begin{description}[wide]
		\item[Topological spaces.]
		A topological space $(X, \tau)$ is (in)discrete with respect to the forgetful functor $\Top \to \Set$ if, and only if, the topology $\tau$ is (in)discrete (see \cref{s:init-topology}).
		
		\item[Preordered sets.]
		A preordered set $(X, \les)$ is (in)discrete with respect to the forgetful functor $\Preord \to \Set$ if, and only if, the preorder $\les$ is (in)discrete (see \cref{s:init-preorder}).
	\end{description}
\end{examples}

\begin{proposition}[{\cite[Proposition~21.12]{AdamekHerrlichStrecker2006}}] \label{p:adjoints-topological}
	Let $U\colon \cat{A} \to \cat{X}$ be a topological functor.
	\begin{enumerate}[wide]
		\item The functor $U$ has a left adjoint $F \colon \cat{X} \to \cat{A}$ that maps
		\begin{enumerate}
			\item an object $X$ of $\cat{X}$ to the smallest element in its fibre (the discrete object), and 
			\item an $\cat{X}$-morphism $f \colon X \to Y$ to the unique $\cat{A}$-morphism $\bar{f} \colon F(X) \to F(Y)$ such that $U(\bar{f}) = f$.
		\end{enumerate}

		\item The functor $U$ has a right adjoint $G \colon \cat{X} \to \cat{A}$ that maps
		\begin{enumerate}
			\item an object $X$ of $\cat{X}$ to the greatest element in its fibre (the indiscrete object), and 
			\item an $\cat{X}$-morphism $f \colon X \to Y$ to the unique $\cat{A}$-morphism $\bar{f} \colon G(X) \to G(Y)$ such that $U(\bar{f}) = f$.
		\end{enumerate}
	\end{enumerate}
	The functors $F$ and $G$ are full, faithful and injective on objects.
	Moreover, $U \circ F$ and $U \circ G$ are the identity functor on $\cat{X}$.
\end{proposition}

\begin{proposition}[{\cite[Proposition~21.15]{AdamekHerrlichStrecker2006}}]
	A topological functor uniquely lifts both limits (via initiality) and colimits (via finality), and it preserves both limits and colimits.
\end{proposition}

\begin{theorem}[{\cite[Theorem~21.16(1)]{AdamekHerrlichStrecker2006}}] \label{t:complete-topological}
	Given a topological functor $U \colon \cat{A} \to \cat{X}$, the category $\cat{A}$ is complete (resp.\ cocomplete) if, and only if, the category $\cat{X}$ is complete (resp.\ cocomplete).
\end{theorem}

These last two results provide a description of limits and colimits in $\Top$ and $\Preord$, as explained in the following.

\begin{examples}
	\begin{description}
		\item[Topological spaces.]
		The category $\Top$ is complete and cocomplete.
		The forgetful functor $\Top \to \Set$ uniquely lifts both limits (via initiality) and colimits (via finality), and it preserves both limits and colimits.
		
		\item[Preordered sets.]
		The category $\Preord$ is complete and cocomplete.
		The forgetful functor $\Top \to \Set$ uniquely lifts both limits (via initiality) and colimits (via finality), and it preserves both limits and colimits.
	\end{description}
\end{examples}


\section{Algebras\index{algebra}}

We assume that the reader has basic knowledge of abstract algebras.
We warn the reader that we admit \emph{infinitary} algebras, i.e.\ algebras with operations of infinite arity, and we allow \emph{large} signatures (i.e.\ a class, rather than a set).

In particular, we work with a large signature\index{signature} $\Sigma$ which is the union of the classes of $\kappa$-ary operations, for $\kappa$ cardinal.
We introduce $\Sigma$-algebras and homomorphisms as usual, and we let $\ALG\Sigma$ denote the category of $\Sigma$-algebras and homomorphisms\footnote{Technically speaking, if $\Sigma$ is large, $\Sigma$-algebras and homomorphisms do not form a legitimate category because there is more than a proper class of algebras on a two-elements set.
However, we will always end up restricting to a class contained in $\ALG \Sigma$.}.

With a common abuse of notation, we will often make no notational distinction between an algebraic structure and its underlying set.
When we want to stress the difference between the algebraic structure and its underlying set, we use a letter in bold font for the structure ($\alge{A}$, $\alge{B}$, $\dots$) and the same letter in plain font for the underlying set ($A$, $B$, \ldots).

The interpretation\index{interpretation of a function symbol} of a function symbol $f$ on an algebra $\alge{A}$ is denoted by $f^\alge{A}$, or simply by $f$ when $\alge{A}$ is understood.

When a class of algebras with a common signature is considered as a category, it is understood that the morphisms are the homomorphisms.

By \emph{finitary algebra}\index{algebra!finitary} we mean an algebra with a signature consisting of operations of finite arity.

By \emph{trivial algebra}\index{algebra!trivial} we mean an algebra whose underlying set is a singleton.


\subsection{Products, subalgebras, homomorphic images}

\begin{notation}
Let $\Sigma$ be a signature.
\begin{enumerate}[wide]
	\item An \emph{isomorphic copy}\index{isomorphic copy} of a $\Sigma$-algebra $\alge{A}$ is an algebra which is isomorphic to $\alge{A}$.
	Given a class $\clalg{A}$ of $\Sigma$-algebras, we let $\opI(\clalg{A})$ denote the class of isomorphic copies of algebras in $\clalg{A}$.
	Note that the class $\clalg{A}$ is contained in $\opI(\clalg{A})$.
	
	\item A \emph{(direct) product}\index{product!of algebras} of a family $(\alge{A}_i)_{i \in I}$ of $\Sigma$-algebras is a $\Sigma$-algebra which is isomorphic to the set-theoretic direct product of $(A_i)_{i \in I}$, in which the operation symbols are defined coordinatewise.
	Given a class of $\Sigma$-algebras $\clalg{A}$, we let $\opP(\clalg{A})$ denote the class of direct products of algebras in $\clalg{A}$.
	Note that $\opP(\clalg{A})$ is closed under isomorphisms, the class $\clalg{A}$ is contained in $\opP(\clalg{A})$, and any trivial algebra belongs to $\opP(\clalg{A})$.
	
	\item A \emph{subalgebra}\index{subalgebra} of a $\Sigma$-algebra $\alge{A}$ is an algebra whose underlying set is a subset of $A$ and on which the interpretation of each operation symbol is the restriction of its interpretation on $\alge{A}$.
	Given a class $\clalg{A}$ of $\Sigma$-algebras, we let $\opS(\clalg{A})$ denote the class of subalgebras of algebras in $\clalg{A}$.
	The class $\clalg{A}$ is contained in $\opS(\clalg{A})$.
	The class $\opS(\clalg{A})$ is not guaranteed to be closed under isomorphism.
	However, if $\clalg{A}$ is closed under isomorphic copies, then the class $\opS(\clalg{A})$ is closed under isomorphic copies.
	Thus, for example, $\opS \opP(\clalg{A}) = \opI \opS \opP(\clalg{A})$.

	\item A \emph{homomorphic image}\index{homomorphic image} of a $\Sigma$-algebra $\alge{A}$ is a $\Sigma$-algebra $\alge{B}$ such that there exists a surjective $\Sigma$-homomorphism from $\alge{A}$ to $\alge{B}$.
	Given a class of $\Sigma$-algebras $\clalg{A}$, we let $\opH(\clalg{A})$ denote the homomorphic images of algebras in $\clalg{A}$.
	Note that $\opH(\clalg{A})$ is closed under isomorphisms, and the class $\clalg{A}$ is contained in $\opH(\clalg{A})$.
\end{enumerate}
\end{notation}

\begin{remark} \label{r:IPSH}
	Let $\clalg{A}$ be a class of algebras.
	\begin{enumerate}
		\item The classes $\opI(\clalg{A})$, $\opP(\clalg{A})$, $\opH(\clalg{A})$ are closed under isomorphic images.
		\item The class $\opP(\clalg{A})$ is closed under products.
		\item The class $\opS(\clalg{A})$ is closed under subalgebras.
		\item The class $\opH(\clalg{A})$ is closed under homomorphic images.
		\item If $\clalg{A}$ is closed under isomorphic images, then $\opI(\clalg{A})$, $\opP(\clalg{A})$, $\opS(\clalg{A})$ and $\opH(\clalg{A})$ are closed under isomorphic images.
		\item If $\clalg{A}$ is closed under products, then $\opI(\clalg{A})$, $\opP(\clalg{A})$, $\opS(\clalg{A})$ and $\opH(\clalg{A})$ are closed under products.
		\item If $\clalg{A}$ is closed under subalgebras, then $\opI(\clalg{A})$, $\opS(\clalg{A})$ and $\opH(\clalg{A})$ are closed under subalgebras.
		\item If $\clalg{A}$ is closed under homomorphic images, then $\opI(\clalg{A})$ and $\opH(\clalg{A})$ are closed under homomorphic images.
	\end{enumerate}
\end{remark}

For the sake of clarity, we will use expressions such as `the algebra $A$ is \emph{isomorphic to} a subalgebra of a product of $B$', even if, with our isomorphism-invariant convention about products, the expression `isomorphic to' is redundant. 


\subsection{Birkhoff's subdirect representation theorem}

An algebra $A$ is a \emph{subdirect product}\index{subdirect!product} of an indexed family $(A_i)_{i \in I}$ of algebras if $A$ is a subalgebra of $\prod_{i \in I} A_i$ and, for every $i \in I$, the projection of $A$ on the $i$-th coordinate is $A_i$.
An injective homomorphism $\alpha \colon A \hookrightarrow \prod_{i \in I} A_i$ is \emph{subdirect}\index{subdirect!homomorphism} if the image of $\alpha$ is a subdirect product of $(A_i)_{i \in I}$.
An algebra $A$ is \emph{subdirectly irreducible}\index{subdirectly irreducible} if, for every subdirect injective homomorphism  $\alpha \colon A \hookrightarrow \prod_{i \in I} A_i$, there exists $j \in I$ such that $\pi_j \alpha \colon A \to A_j$ is an isomorphism, where $\pi_j \colon \prod_{i \in I} A_i \to A_j$ is the $j$-th projection.
By taking $I= \emptyset$, we observe that every subdirectly irreducible algebra is non-trivial.

\begin{theorem}[{Birkhoff's subdirect representation theorem \cite[Theorem~2]{Birkhoff1944}}] \label{t:SubRepThe} \index{subdirect!representation theorem}
	Every finitary algebra is isomorphic to a subdirect product of subdirectly irreducible algebras.
\end{theorem}


\subsection{Definition of varieties and quasivarieties}

An \emph{implication}\index{implication} is a (universally quantified) formula
\[ 
\bigwedge_{i\in I}(u_i=v_i)\Longrightarrow (u_0=v_0)
\]
where $I$ is a (possibly infinite) set, and $u_i$, $v_i$ are, for $i\in I\cup \{0\}$, terms over a given set of variables.

\begin{definition}
	We say that a class of algebras $\clalg{A}$ \emph{has free algebras}\index{algebra!free} if, for each cardinal $\kappa$, there exists an algebra $A$ in $\clalg{A}$ and a function $\iota$ from $\kappa$ to (the underlying set of) $A$ such that, for every algebra $B$ in $\clalg{A}$ and every function $f$ from $\kappa$ to (the underlying set of) $B$, there exists a unique homomorphism $\overline{f} \colon A \to B$ such that $f = \overline{f} \iota$.
\end{definition}
\begin{definition}\label{d:quasivariety}
	A class $\clalg{A}$ of $\Sigma$-algebras is called a \emph{quasivariety of $\Sigma$-algebras}\index{quasivariety of algebras} if
	\begin{enumerate}
		\item \label{i:presented-by} the class $\clalg{A}$ can be presented by a class of implications, and
		\item \label{i:free-algs} the class $\clalg{A}$ has free algebras.
	\end{enumerate}
\end{definition}
As shown in \cite[Subsection~3.1]{Adamek2004}, given \cref{i:presented-by}, we have that \cref{i:free-algs} holds if, and only if, for each cardinal $\kappa$, the class $\clalg{A}$ has only a set of isomorphism classes of algebras on $\kappa$ generators.
\begin{definition} \label{d:variety}
	A class $\clalg{A}$ of $\Sigma$-algebras is called a \emph{variety of $\Sigma$-algebras}\index{variety of algebras} if
	\begin{enumerate}
		\item \label{i:presented-by-var} the class $\clalg{A}$ can be presented by a class of equations, and
		\item \label{i:free-algs-var} the class $\clalg{A}$ has free algebras.
	\end{enumerate}
\end{definition}
As for quasivarieties, given \cref{i:presented-by-var}, we have that \cref{i:free-algs-var} holds if, and only if, for each cardinal $\kappa$, the class $\clalg{A}$ has only a set of isomorphism classes of algebras on $\kappa$ generators.

Varieties of algebras in this sense coincide, up to equivalence, with monadic (also known as tripleable\index{triple}\index{functor!tripleable}) categories\index{monad}\index{functor!monadic} over $\Set$; see \cite[Section~4.1]{Borceux1994-vol2} for the related definitions.
Moreover, varieties of algebras coincide, up to equivalence, with varietal categories\index{category!varietal}\index{varietal|see{category, varietal}} in the sense of \cite[Section~1]{Linton1966}: the equivalence between Linton's varietal categories and triplable categories over $\Set$ is asserted at the end of Section~6 in \cite{Linton1966}.

\begin{proposition}[{\cite[Proposition~3.5]{Adamek2004}}] \label{p:char-quasivar-ISP}
	A class of $\Sigma$-algebras having free algebras is a quasivariety if, and only if, it is closed in $\ALG \Sigma$ under products and subalgebras.
\end{proposition}

\begin{lemma} \label{l:ISP}
	Given a $\Sigma$-algebra $A$, the class $\opS\opP(A)$ is a quasivariety.
\end{lemma}
\begin{proof}
	The class $\opS\opP(A)$ has free algebras.
	Indeed, following a standard construction, a free algebra over a set $X$ is given by the subalgebra of $A^{A^X}$ generated by the set of projections $\{ \pi_x \colon A^X \to A \mid x \in X\}$.
	Moreover, by \cref{r:IPSH}, $\opS\opP(A)$ is closed under products and subalgebras.
	By \cref{p:char-quasivar-ISP}, $\opS\opP(A)$ is a quasivariety.
\end{proof}

As observed in \cite[Section~4.1]{AdamekHerrlichStrecker2006}, we have an analogous version of \cref{p:char-quasivar-ISP} for varieties, which generalises a celebrated theorem by \cite{Birkhoff1935} for varieties in a small signature.

\begin{proposition}\label{p:char-var-HSP}
	A class of $\Sigma$-algebras having free algebras is a variety if, and only if, it is closed in $\ALG \Sigma$ under products, subalgebras and homomorphic images.
\end{proposition}


\subsubsection{Comparison with other definitions}

We warn the reader that the term `variety of algebras' (and `quasivariety of algebras') admits different non-equivalent definitions in the literature, depending on the desired level of generality.
Our convention (see \cref{d:variety}) is listed here as \cref{i:variety}.
\begin{enumerate}[wide]

	\item Classically (but not in this manuscript), one considers a class of algebras in a small signature (i.e., a set), with operations of finite arity, defined by a set of equations.
	Boolean algebras and distributive lattices are an example.
	Any such class has free algebras by a theorem of \cite{Birkhoff1935} (see also \cite[Chapter~4, Section~25, Corollary~2]{Graetzer2008}).
	In our convention, these classes are\footnote{Indeed, the smallness of the signature and of the class of equations is not relevant (in terms of algebraic theories).} the varieties of algebras in a signature whose operations have finite arity.
	
	\item \label{i:Slo} \cite{Slominski1959} considers a class of algebras in a small signature (with operations of possibly infinite arity), defined by a set of equations.
	Boolean $\sigma$-algebras \cite[Chapter~29]{GivantHalmos2009} are an example.
	Any such class has free algebras \cite[{}8.3]{Slominski1959}.
	In our convention, these classes are\footnote{Indeed, the smallness of the class of equations is not relevant (in terms of algebraic theories).} the varieties of algebras in a small signature, also called varieties with rank\index{variety of algebras!with rank}.

	\item \label{i:variety} In this manuscript, the term `variety of algebras' (\cref{d:variety}) denotes a class of algebras (in a possibly large signature), defined by a class of equations, with free algebras. 
	Complete join-semilattices (cf.\ \cite[Examples~3.2]{Adamek2004}) and frames (cf.\ \cite[Theorem II.1.2]{Johnstone1986}) are examples.
	Even if this definition is the one we use for the term `variety of algebras', it should be noted that the varieties of algebras that we deal with in this manuscript are also varieties of algebras in the more restrictive sense of \cite{Slominski1959}, described in \cref{i:Slo} above.

	\item One may consider a class of algebras (in a possibly large signature) that can be presented by a class of equations (but possibly lacking some free algebras).
	Complete Boolean algebras are an example;
	it was proved independently by \cite{Gaifman1961,Hales1962} that complete Boolean algebras lack free algebras over a countably infinite set (see \cite{Solovay1966} for a shorter proof).
	
\end{enumerate}
%


\chapter{Compact ordered spaces}\label{chap:CompOrd}


\section{Introduction}

What is the correct partially-ordered version of compact Hausdorff spaces?
More to the point, what is the missing piece in the equation
\begin{equation} \label{e:ordered-of-ch}
	\text{Stone spaces are to Priestley spaces as compact Hausdorff spaces are to \ldots,}
\end{equation}
or, equivalently, in the equation
\begin{equation} \label{e:t2-of-p}
	\text{Stone spaces are to compact Hausdorff spaces as Priestley spaces are to \ldots?}
\end{equation}
Our view (which is not new) is that the answer is given by those structures that L.\ Nachbin introduced under the name of \emph{compact ordered spaces}\footnote{These structure appear in the literature also under the name of `compact pospaces', `compact partially ordered spaces', `partially ordered compact spaces', `separated ordered compact Hausdorff spaces', or `Nachbin spaces'.} (\cite{Nachbin1948}, \cite[Section~3]{Nachbin}).
A \emph{compact ordered space} is a compact space $X$ equipped with a partial order $\leq$ that is a closed subspace of $X \times X$ with respect to the product topology.
In this chapter we collect some known background results on compact ordered spaces to motivate this view.
In doing so, we discuss limits and colimits of compact ordered spaces and an ordered version of Urysohn's lemma, which will come handy in the following chapters.

No result in this chapter is new.

The reader who has familiarity with compact ordered spaces may move to \cref{chap:direct-proof}.


\section{Stone is to Priestley as compact Hausdorff is to compact ordered}
\label{s:characterisations}

To motivate the fact that compact ordered spaces are the missing piece in \cref{e:ordered-of-ch}, we compare some characterisations of Stone, Priestley, compact Hausdorff and compact ordered spaces which best show the analogies between them.
In the following sections, we will recall the classical definition of compact ordered spaces (\cref{d:compact-ordered-space}) and we will make sure that the characterisations of compact ordered spaces stated below are correct.
\begin{enumerate}[wide]
	\item A \emph{Stone space}\index{Stone space} (also known as \emph{Boolean space}\index{Boolean space}) is a compact topological space $X$ such that, for all $x, y \in X$ such that $x \neq y$, there exist an open set $U$ containing $x$ and an open set $V$ containing $y$ such that $U \cap V = \emptyset$ and $U \cup V = X$.
	
	\item A \emph{Priestley space}\index{Priestley space} is a compact topological space $X$ equipped with a partial order $\leq$ such that, for all $x, y \in X$ such that $x \not \leq y$, there exist an open up-set $U$ containing $x$ and an open down-set $V$ containing $y$ such that $U \cap V = \emptyset$ and $U \cup V = X$.
	
	\item A \emph{compact Hausdorff space} is a compact topological space $X$ such that, for all $x, y \in X$ such that $x \neq y$, there exist an open set $U$ containing $x$ and an open set $V$ containing $y$ such that $U \cap V = \emptyset$.
	
	\item A \emph{compact ordered space} is a compact topological space $X$ equipped with a partial order $\leq$ such that, for all $x, y \in X$ such that $x \not \leq y$, there exist an open up-set $U$ containing $x$ and an open down-set $V$ containing $y$ such that $U \cap V = \emptyset$ (\cref{l:char-closed} below).
\end{enumerate}
Note that the characterisation of compact Hausdorff spaces is obtained from the one of Stone spaces simply by dropping the requirement that the separating sets $U$ and $V$ cover the whole space.
The characterisation of compact ordered spaces is obtained in an analogous way from the one of Priestley spaces.
Thus, `compact ordered spaces' is a good fit in \cref{e:t2-of-p}.
Furthermore, Stone spaces are precisely the Priestley spaces whose partial order is equality, and
compact Hausdorff spaces are precisely the compact ordered spaces whose partial order is equality.
Thus, `compact ordered spaces' is a good fit in \cref{e:ordered-of-ch}.

Other closely related characterisations, which again show the analogies, are as follows.
\begin{enumerate}[wide]
	\item A \emph{Stone space} is a compact topological space $X$ such that, for all $(x, y) \in (X \times X) \setminus \{(s,t) \in X \times X \mid s = t\}$, there exist an open set $U$ containing $x$ and an open set $V$ containing $y$ such that $U \cup V = X$ and such that $U \times V$ and $\{(s,t) \in X \times X \mid s = t\}$ are disjoint.
	
	\item A \emph{Priestley space} is a compact topological space $X$ such that, for all $(x, y) \in (X \times X) \setminus \{(s,t) \in X \times X \mid s \leq t\}$, there exist an open set $U$ containing $x$ and an open set $V$ containing $y$ such that $U \cup V = X$ and such that $U \times V$ and $\{(s,t) \in X \times X \mid s \leq t\}$ are disjoint.
	
	\item A \emph{compact Hausdorff space} is a compact topological space $X$ such that, for all $(x, y) \in (X \times X) \setminus \{(s,t) \in X \times X \mid s = t\}$, there exist an open set $U$ containing $x$ and an open set $V$ containing $y$ such that $U \times V$ and $\{(s,t) \in X \times X \mid s = t\}$ are disjoint.
	(In other words, a compact Hausdorff space is a compact space with a closed diagonal, see \cref{p:diag-closed}.)
	
	\item A \emph{compact ordered space} is a compact topological space $X$ such that, for all $(x, y) \in (X \times X) \setminus \{(s,t) \in X \times X \mid s \leq t\}$, there exist an open set $U$ containing $x$ and an open set $V$ containing $y$ such that $U \times V$ and $\{(s,t) \in X \times X \mid s \leq t\}$ are disjoint.
	(In other words, a compact ordered space is a compact space equipped with a closed partial order. In fact, this is Nachbin's original definition, to which we will conform.)
\end{enumerate}

To present some additional analogies, let us fix some conventions.
On the set $\{0,1\}$ we consider the discrete topology and the canonical total order ($0 \leq 1$); on the unit interval $[0,1]$ we consider the Euclidean topology and the `less or equal' total order $\leq$.
Powers are set-theoretic powers equipped with the product topology and product order, and subsets are equipped with the induced topology and the induced order (see \cref{s:init-topology,s:init-preorder}).
Isomorphisms of topological structures are simply homeomorphisms, and isomorphisms of ordered-topological structures are homeomorphisms which preserve and reflect the partial order.
We can now state our desired analogies.
\begin{enumerate}[wide]
	\item A \emph{Stone space} is a topological space which is isomorphic to a closed subspace of a power of the topological space $\{0,1\}$.
	\item A \emph{Priestley space} is a topological space equipped with a partial order which is isomorphic to a closed subspace of a power of the ordered-topological space $\{0,1\}$.
	\item A \emph{compact Hausdorff space} is a topological space which is isomorphic to a closed subspace of a power of the topological space $[0,1]$.
	\item A \emph{compact ordered space} is a topological space equipped with a partial order which is isomorphic to a closed subspace of a power of the ordered-topological space $[0,1]$ (\cref{l:closed-sub-power} below).
\end{enumerate}
In the following sections, we make sure that the characterisations of compact ordered spaces stated above are correct.


\section{Compact ordered spaces}

The following concept is an ordered analogue of the Hausdorff property.

\begin{definition} [{See \cite[Chapter~I, Section~1, p.\ 25]{Nachbin}}] \label{d:closed}
	We say that a preorder on a topological space $X$ is \emph{closed}\index{preorder!closed} if it is a closed subset of the topological space $X \times X$.
\end{definition}

The reader who is acquainted with nets\index{net} (\cite[Definition~11.2]{Willard1970}) will notice that a preorder $\les$ on a topological space $X$ is closed if, and only if\footnote{This equivalence holds because (i) a set is closed if, and only if, together with any net it contains all its limits \cite[Theorem~11.7]{Willard1970}, and (ii) the product topology is the topology of pointwise convergence \cite[Theorem~11.9]{Willard1970}.}, for any two converging nets $x_i \to x$ and $y_i \to y$, the property `$x_i \les y_i$ for all $i$' implies $x \les y$.
Note that, replacing $\les$ with $=$ in this condition, we obtain the Hausdorff property \cite[Theorem~13.7]{Willard1970}.

\begin{example}
	By \cref{p:diag-closed}, the discrete order $\{(x, x) \in X \times X \mid x \in X\}$ on a topological space $X$ is closed if, and only if, $X$ is Hausdorff.
\end{example}

The following shows that the first characterisation of compact ordered spaces in \cref{s:characterisations} is correct.

\begin{lemma}[{\cite[Proposition~1, p.\ 26]{Nachbin}}] \label{l:char-closed}
	A preorder on a topological space $X$ is closed if, and only if, for all $x, y \in X$ such that $x \not \leq y$, there exist an open up-set $U$ containing $x$ and an open down-set $V$ containing $y$ such that $U \cap V = \emptyset$.
\end{lemma}

\begin{lemma}[{\cite[Proposition~2, p.\ 27]{Nachbin}}]
	Every topological space $X$ equipped with a closed partial order is a Hausdorff space.
\end{lemma}

\begin{definition} [{See \cite[Chapter I, Section 3, p.\ 44]{Nachbin}}]\label{d:compact-ordered-space}
	A \emph{compact ordered space}\index{compact ordered space} $(X, \tau, \leq)$ consists of a set $X$ with a compact topology $\tau$ and a closed partial order $\leq$.
\end{definition}

(We have already observed that the closure of the order has a natural characterisation in terms of convergence of nets.
The same happens for the compactness and the Hausdorff properties \cite[Theorems~13.7 and 17.4]{Willard1970}.)

To keep the notation simple, we will often write $X$ or $(X, \leq)$ instead of $(X, \tau, \leq)$.

We denote with $\CompOrd$ the category of compact ordered spaces and order-preserving continuous maps.

\begin{examples} \label{ex:comp-ord-spaces}
	\begin{enumerate}[wide]
		\item \label{i:interval}A basic example of compact ordered space is any compact interval $[a,b]\subseteq \R$---let alone the unit interval $[0,1]$---equipped with the Euclidean topology and the usual total order.
		\item Every compact Hausdorff space equipped with the discrete order is a compact ordered space.
		\item Every finite partially ordered set equipped with the discrete topology is a compact ordered space.
		\item Every Priestley space is a compact ordered space.
	\end{enumerate}
\end{examples}


\section{Limits and colimits}

In this section, we show that the category $\CompOrd$ of compact ordered spaces is complete and cocomplete (see also \cite[Corollary 2, p.\ 2153]{Tholen2009}).
Moreover, we characterise limits and finite coproducts: the product of a family of compact ordered spaces consists of the set-theoretic product equipped with the product topology and product order, and the coproduct of a finite family of compact ordered spaces consists of their disjoint union equipped with the coproduct topology and coproduct order.

To prove these facts, we play with reflections and topological functors.
Let us first define some categories.
\begin{center}
	\begin{tabular}{l|ll}
		\hline
		$\Set$					& Objects: 		& Sets.\\
									& Morphisms: 	& Functions.\\
		\hline
		$\Preord$ 				& Objects: 		& Preordered sets.\\
									& Morphisms:	& Order-preserving functions.\\
		\hline
		$\Top$ 					& Objects: 		& Topological spaces.\\
									& Morphisms: 	& Continuous functions.\\
		\hline
		$\CompHaus$ 			& Objects:  	& Compact Hausdorff spaces.\\
									& Morphisms: 	& Continuous functions.\\
		\hline
		$\TopXPreord$ 			& Objects:  	& Topological spaces with a preorder.\\
									& Morphisms:	& Order-preserving continuous functions.\\
		\hline
		$\CompHausXPreord$	& Objects:  	& Compact Hausdorff spaces with a preorder.\\
									& Morphisms: 	& Order-preserving continuous functions.\\
		\hline
		$\CompHausPreord$ 	& Objects: 		& Compact Hausdorff spaces with a closed preorder.\\
									& Morphisms:	& Order-preserving continuous functions.\\
		\hline
		$\CompOrd$				& Objects: 		& Compact ordered spaces.\\
									& Morphisms: 	& Order-preserving continuous functions.\\
		\hline
	\end{tabular}
\end{center}

The following diagram illustrates some of the inclusion and forgetful functors between these categories.
We will observe that the functors labelled `refl.' are inclusions of reflective full subcategories, and the functors labelled `topol.' are topological functors.
\begin{equation} \label[figure]{e:functors}
	\begin{tikzcd}
		\CompOrd	\arrow[hook]{r}{\text{refl.}} & \CompHausPreord	\arrow[hook]{r}{\text{refl.}} \arrow[swap]{rd}{\text{topol.}} & \CompHausXPreord	\arrow{d}{\text{topol.}} \arrow[hook]{r}{\text{refl.}} & \TopXPreord \arrow[swap]{d}{\text{topol.}} \arrow{dr}{\text{topol.}} &\\
		& & \CompHaus \arrow[hook, swap]{r}{\text{refl.}} & \Top	\arrow[swap]{dr}{\text{topol.}} & \Preord \arrow{d}{\text{topol.}}\\
		& & & & \Set
	\end{tikzcd}
\end{equation}

To this end, we recall from \cref{d:reflective} that a \emph{reflective full subcategory} of a category $\cat{C}$ is a full subcategory $\cat{D}$ whose inclusion functor $\cat{D} \hookrightarrow \cat{C}$ admits a left adjoint, called the \emph{reflector}.

Furthermore, we recall from \cref{d:topological-functor} that a faithful functor $G\colon \cat{A}\to \cat{X}$ is called \emph{topological} provided that every family of morphisms $(f_i\colon X\to G(A_i))_{i\in I}$ in $\cat{X}$ has a unique $G$-initial lift.
Spelling out the details, the existence of a unique $G$-initial lift amounts to say that there exists a unique family of morphisms $(\overline{f}_i\colon A\to A_i)_{i\in I}$ in $\cat{A}$ that is 
\begin{enumerate}
	\item 
		a \emph{lift} of $(f_i\colon X\to G(A_i))_{i\in I}$, i.e.\ such that $G(A)=X$ and, for every $i \in I$, $G(\overline{f}_i)=f_i$, and 
	\item
		\emph{$G$-initial}, i.e., for each object $C$ of $\cat{A}$, an $\cat{X}$-morphism $h\colon U(C)\to X$ is (the image of) an $\cat{A}$-morphism from $C$ to $A$ if, and only if, for every $i \in I$, the composite $U(C) \xrightarrow{h} X \xrightarrow{f_i} G(A_i)$ is (the image of) an $\cat{A}$-morphism from $C$ to $A_i$.
\end{enumerate}

\begin{remark} \label{r:fundamental-pleasant}
	We observe the following pleasant facts.
	\begin{enumerate}[wide]
	
		\item \label{i:top-top-topolo}The forgetful functor $\Top \to \Set$ is topological (see \cref{exs:topological-functors!}).
		
		\item \label{i:pre-pre-preordi} The forgetful functor $\Preord \to \Set$ is topological (see \cref{exs:topological-functors!}).
		
		\item\label{i:Stone-Cech} The inclusion functor $\CompHaus \hookrightarrow \Top$ is reflective.
		The reflector is the Stone-\v{C}ech compactification\index{Stone-\v{C}ech compactification}\index{compactification!Stone-\v{C}ech} functor $\beta \colon \Top \to \CompHaus$ \cite[3.3.9.d]{Borceux1994-vol1}.
		
		\item \label{i:symmetrise}The inclusion functor $\Ord \hookrightarrow \Preord$ is reflective.
		The reflector maps a preordered set $(X, \les)$ to the partially ordered set $(X/{\sim}, \leq)$, where $\sim$ is the equivalence relation defined by
		\[
			x \sim y \Longleftrightarrow x \les y \text{ and } y \les x,
		\]
		and $\leq$ is the quotient order with respect to the map $X \epi X / {\sim}$, which, in this case, is defined by
		\[
			[x]_{\sim} \leq [y]_{\sim} \Longleftrightarrow x \les y.
		\]
		
		\item \label{i:chiudi!} The inclusion functor $\TopPreord \hookrightarrow \TopXPreord$ is reflective.
		The reflector maps a topological space $X$ equipped with a preorder $\les$ to the topological space $X$ itself, equipped with the smallest closed preorder on $X$ which extends $\les$.
		Note that the reflector $\TopXPreord \to \TopPreord$ commutes with the forgeftul functors $\TopPreord \to \Top$ and $\TopXPreord \to \Top$.
		
	\end{enumerate}
\end{remark}

The facts above are the ingredients to obtain the following.

\begin{remark} \label{r:figure-is-right}
	In \cref{e:functors}, the functors labelled `refl.' are inclusions of reflective full subcategories, and the functors labelled `topol.' are topological functors.	
	In the following, we carry out some details.
	\begin{description}[wide]
	
		\item[Reflector of $\CompHaus \hookrightarrow \Top$.]
		See \cref{i:Stone-Cech} in \cref{r:fundamental-pleasant}.
		
		\item[Reflector of $\CompHausXPreord \hookrightarrow \TopXPreord$.]
		The reflector maps an object $(X, \les)$ to $(\beta X, \les')$, where $\les'$ is the final preorder with respect to the map $X \to \beta X$, i.e. the smallest preorder that makes this function order-preserving.
		
		\item[Reflector of $\CompHausPreord \hookrightarrow \CompHausXPreord$.]
		The reflector is the restriction of the reflector of the inclusion $\TopPreord \hookrightarrow \TopXPreord$ (see \cref{i:chiudi!} in \cref{r:fundamental-pleasant}).
		
		\item[Reflector of $\CompOrd \hookrightarrow \CompHausPreord$.]
		The reflector is the restriction of the reflector of the inclusion functor $\TopXOrd \hookrightarrow \TopXPreord$.
		The reflector 
		\[
			\TopXPreord \longrightarrow \TopXOrd
		\]
		maps a topological space with a preorder $(X, \les)$ to the set $X/({\les}\cap {\les^\oprel})$ with the quotient order (which is a partial order) as described by the reflector of $\Ord \hookrightarrow \Preord$, equipped with the quotient topology.
		This reflector restricts to a functor $\CompHausPreord \to \CompOrd$.
		Indeed, if $(X, \les)$ is compact, then also its continuous image $X/({\les}\cap {\les^\oprel})$ is compact, by \cref{p:image-of-compact}.
		Since any continuous map between compact Hausdorff spaces is closed (\cref{p:comp-Haus-closed}), and the continuous map 
		\[
			X \times X \longrightarrow X/({\les}\cap {\les^\oprel}) \times X/({\les}\cap {\les^\oprel})
		\]
		maps the subset $\les$ to $\leq$ and the subset ${\les}\cap {\les^\oprel}$ to the diagonal of $X/({\les}\cap {\les^\oprel})$, we have that $X/({\les}\cap {\les^\oprel})$ has a closed partial order and a Hausdorff topology.
		
		Let us now take a look at the topological functors.
		\item[The topological functor $\Top \to \Set$.]
		See \cref{i:top-top-topolo} in \cref{r:fundamental-pleasant}.
		\item[The topological functor $\Preord \to \Set$.]
		See \cref{i:pre-pre-preordi} in \cref{r:fundamental-pleasant}.
		\item[The topological functor $\TopXPreord \to \Preord$.]
		The unique initial lift of a family of order-preserving functions $(f_i \colon X \to A_i)$ is obtained by providing $X$ with the initial topology.
		\item[The topological functor $\TopXPreord \to \Top$.]
		The unique initial lift of a family of continuous functions $(f_i \colon X \to A_i)$ is obtained by providing $X$ with the initial preorder. 
		\item[The topological functor $\CompHausXPreord \to \CompHaus$.]
		The unique initial lift of a family of continuous functions $(f_i \colon X \to A_i)$ is obtained by providing $X$ with the initial preorder.
		\item[The topological functor $\CompHausPreord \to \CompHaus$.]
		Recall from the discussion above that the inclusion functor $\CompHausPreord \hookrightarrow \CompHausXPreord$ has a reflector that commutes with the forgetful functors to $\CompHaus$.
		Therefore, by \cite[Proposition~21.31]{AdamekHerrlichStrecker2006}, denoting with $U$ the forgetful functor $\CompHausXPreord \to \CompHaus$, every $U$-initial source whose codomain is a family of objects in $\CompHausPreord$ has its domain in $\CompHausPreord$.
		Hence, by \cite[Proposition~21.30]{AdamekHerrlichStrecker2006}, the forgetful functor $\CompHausPreord \to \CompHaus$ is topological.
		
	\end{description}
\end{remark}

\begin{proposition}\label{p:PC-complete-cocomplete}
	The category $\CompOrd$ is complete and cocomplete.
\end{proposition}
\begin{proof}
	By \cref{r:figure-is-right}, $\CompOrd$ is a full reflective subcategory of $\CompHausPreord$, which is topological over $\CompHaus$, which is a full reflective subcategory of $\Top$, which is topological over $\Set$.
	The category $\Set$ is complete and cocomplete.
	By \cref{p:complete-reflective,t:complete-topological}, $\CompOrd$ is complete and cocomplete.
\end{proof}
\begin{proposition} \label{p:pres-limits}
	Every functor in \cref{e:functors} preserves limits.
	In particular, the forgetful functors $\CompOrd \to \Set$, $\CompOrd \to \Preord$ and $\CompOrd \to \Top$ preserve limits.
\end{proposition}
\begin{proof}
	By \cref{r:figure-is-right}, each functor in \cref{e:functors} is either an inclusion of a full reflective subcategory or a topological functor.
	In either cases, it is a right adjoint: in the first case by definition of full reflective subcategory, in the second case by \cref{p:adjoints-topological}.
	By \cref{p:adjoint-preservation}, right adjoint preserve limits.
\end{proof}

\begin{lemma}\label{l:limits}
	The product in $\CompOrd$ of a family of compact ordered spaces $(X_i)_{i \in I}$ consists of the set-theoretic product $\prod_{i \in I} X_i$ equipped with the product topology and product order.
\end{lemma}

\begin{proof}
	By \cref{p:pres-limits}.
\end{proof}

We conclude this section with a note on finite coproducts.

\begin{lemma} \label{l:preserves-finite-coproducts}
	Every functor in \cref{e:functors} preserves finite coproducts.
	In particular, a coproduct in $\CompOrd$ of two compact ordered spaces $X$ and $Y$ is given by the disjoint union of $X$ and $Y$ equipped with the coproduct topology and coproduct order.
\end{lemma}
\begin{proof}
	This is clearly true for the topological functors.
	Let us settle the statement for the inclusion functors.
	The case of binary coproduct follows from the following observations.
	\begin{enumerate}[wide]
		\item If $X$ and $Y$ are Hausdorff (resp. compact) spaces, then the coproduct topology on their disjoint union is Hausdorff (resp. compact).
			
		\item If $X$ and $Y$ are partially ordered sets, then the coproduct preorder on their disjoint union is a partial order.
		
		\item If $X$ and $Y$ are topological spaces equipped with a closed preorder, then the coproduct preorder on their disjoint union is closed with respect to the coproduct topology.
	\end{enumerate}
	The nullary coproduct turns out to be the emptyset (with the only possible topology and preorder) in each category under consideration.
\end{proof}


\section{Ordered Urysohn's lemma}

The following important result, due to L.\ Nachbin, is an ordered version of a celebrated result of P.\ Urysohn \cite[Urysohn's lemma 15.6]{Willard1970}\index{Urysohn's lemma}.

\begin{theorem} [Ordered version of Urysohn's lemma]\label{t:Urysohn}\index{Urysohn's lemma!ordered}
	For any two	disjoint closed subsets $F_0$, $F_1$ of a compact ordered space $X$ where $F_0$ is a down-set and $F_1$ is an up-set, there exists a continuous order-preserving function $f \colon X \to [0,1]$ such that $f(x) = 0$ for $x \in F_0$ and $f(x) = 1$ for $x \in F_1$.
\end{theorem}
\begin{proof}
	The assertion holds by \cite[Theorem~1, p.\ 30]{Nachbin}, which applies to compact ordered spaces in light of \cite[Corollary of Theorem~4, p.\ 48]{Nachbin}.
\end{proof}

For $X$ a preordered set and $x \in X$, we set $\downset x \df \{z \in X \mid z \les x\}$ and $\upset x \df \{z \in X \mid x \les z\}$.

\begin{lemma}[{\cite[Proposition~1, p.\ 26]{Nachbin}}] \label{l:up-set-closed}
	Given a topological space $X$ equipped with a closed preorder, for every $x \in X$ the sets $\downset x$ and $\upset x$ are closed.
\end{lemma}

To illustrate \cref{l:up-set-closed} with an example, note that, applying \cref{l:up-set-closed} to a compact Hausdorff space $X$ equipped with the discrete order, we obtain that the points of $X$ are closed.

\begin{lemma}\label{l:cor of urysohn}
	Let $X$ be a compact ordered space, and let $x, y \in X$ be such that $x \not\geq y$.
	Then there exists a continuous order-preserving function $\psi \colon X \rightarrow [0,1]$ such that $\psi(x) =0$ and $\psi(y) = 1$.
\end{lemma}
\begin{proof} 
	Apply \cref{t:Urysohn} with $F_0 = \downset x$ and $F_1 = \upset y$, both of which are closed by \cref{l:up-set-closed}.
\end{proof} 

\begin{lemma}\label{l:embedding}
	For every compact ordered space $X$, the function
	\begin{align*}
		\ev \colon 	X	& \longrightarrow	\prod_{\hom_{\CompOrd}(X, [0,1])} [0,1]\\
						x	& \longmapsto		\ev_x \colon f \mapsto f(x)
	\end{align*}
	is continuous, order-preserving, injective and order-reflective (with respect to the product order and product topology).
\end{lemma}
\begin{proof}
	It is continuous and order-preserving by the construction of products (see \cref{l:limits}).
	It is injective and order-reflective by \cref{l:cor of urysohn}.
\end{proof}

\begin{lemma} \label{l:subset}
	Let $X$ be a topological space equipped with a closed preorder, and let $Y$ be a subset of $X$.
	Then, the induced preorder on $Y$ is closed with respect to the induced topology on $Y$.
\end{lemma}
\begin{proof}
	By \cref{l:char-closed}.
\end{proof}

We obtain that compact ordered spaces are precisely, up to isomorphisms, the closed subspaces of some power of $[0,1]$ with the induced order, as shown in the following.

\begin{lemma}\label{l:closed-sub-power}
	A topological space $X$ equipped with a preorder is a compact ordered space if and only if there exists an isomorphism (in $\TopXPreord$, i.e.\ an order-preserving order-reflecting homeomorphism) between $X$ and a closed subspace of a power of $[0,1]$.
\end{lemma}
\begin{proof}
	By \cref{l:embedding}, for every compact ordered space $X$, the function 
	\[
		\ev \colon X \to \prod_{\hom_{\CompOrd}(X, [0,1])} [0,1]
	\]
	is continuous, order-preserving, injective and order-reflective.
	Since the image of a compact space under a continuous map is compact (\cref{p:image-of-compact}), the image of $\ev$ is compact.
	Since a compact subset of a Hausdorff space is closed (\cref{p:compact-in-Haus}), the image of $\ev_X$ is closed.
	Therefore, every compact ordered space is isomorphic to a closed subspace of a power of $[0,1]$.
	This settles one direction.
	
	As observed in \cref{i:interval} in \cref{ex:comp-ord-spaces}, the unit interval $[0,1]$ is a compact ordered space.
	By \cref{l:limits}, any power of a compact ordered space is a compact ordered space.
	By \cref{l:subset}, for every topological space $X$ equipped with a closed preorder and every subset $Y$ of $X$, the induced preorder on $Y$ is closed with respect to the induced topology on $Y$.
	By \cref{p:closed-in-compact}, every closed subspace of a compact space is compact.
	Therefore, every closed subspace of a compact ordered space is a compact ordered space.
	In conclusion, every closed subspace of a power of $[0,1]$ (and of its isomorphic copies) is a compact ordered space.
\end{proof}


\section{Conclusions}

We completed our work of showing that the various characterisations of compact ordered spaces in \cref{s:characterisations} are correct.
With these facts at hand, we believe we can finally conclude: Stone spaces are to Priestley spaces as compact Hausdorff spaces are to \emph{compact ordered spaces}.


\chapter{The dual of compact ordered spaces is a variety}\label{chap:direct-proof}


\section{Introduction}

In the previous chapter we presented compact ordered spaces as the correct solution for $X$ in the equation
\begin{equation*}
	\text{Stone spaces are to Priestley spaces as compact Hausdorff spaces are to $X$.}
\end{equation*}

It has been very well known for at least half a century that Stone spaces, Priestley spaces and compact Hausdorff spaces all have an equationally axiomatisable dual.
\begin{enumerate}
	\item \label{i:stone} The category $\Stone$ of Stone spaces and continuous maps is dually equivalent to a variety of finitary algebras---namely, the variety of Boolean algebras \cite{Stone1936}.
	\item \label{i:priestley} The category $\Pries$ of Priestley spaces and order-preserving continuous maps is dually equivalent to a variety of finitary algebras---namely, the variety of bounded distributive lattices \cite{Priestley1970}.
	\item \label{i:comphaus} The category $\CompHaus$ of compact Hausdorff spaces and continuous maps is dually equivalent to a variety of algebras (as observed in \cite[5.15.3]{Duskin1969}; for a proof see \cite[Chapter 9, Theorem 1.11]{BarrWells2005}), with primitive operations of at most countable arity\footnote{Given the fact that the functor $\hom_{\CompHaus}({-},[0,1])\colon \CompHaus^\opcat \to \Set$ is monadic \cite{Duskin1969}, the bound on the arity follows from the fact that every morphism from a power of $[0,1]$ to $[0,1]$ factors through a countable sub-power \cite[Theorem 1]{Mib44}, or, alternatively, from the fact that $[0,1]$ is $\aleph_1$-copresentable \cite[6.5(a)]{GabUlm}. See also \cite{Isbell1982,MarraReggio2017}.}.
\end{enumerate}

The question now arises: 
\begin{quote}
	Is the category of compact ordered spaces dually equivalent to a variety of (possibly infinitary) algebras?
\end{quote}

In fact, this question appears as an open problem in \cite{HNN2018}, and it will be the driving force of this manuscript.
In this chapter, we provide a clear-cut answer: The category of compact ordered spaces is dually equivalent to a variety, with primitive operations of at most countable arity.

Whether a better bound on the arity can be achieved is a question that we will address in the next chapter.

The structure of our proof is the following.
In \cref{s:varieties-as-categories}, we recall a well-known result in category theory, which characterises those categories which are equivalent to some variety of possibly infinitary algebras.
A key property, which characterises varieties among quasivarieties, is the effectiveness of (internal) equivalence relations.
In \cref{s:compact-pospaces} we prove that $\CompOrdop$ is equivalent to a quasivariety.
Then, in \cref{s:equivalence-corelations}, we characterise equivalence relations on a compact ordered space $X$, seen as an object of $\CompOrdop$, as certain preorders on the order-topological coproduct $X+X$.
Then, we rephrase effectiveness into an order-theoretic condition, and show, in \cref{s:proof-of-main-res}, that this condition is satisfied by every preorder arising from an equivalence relation.
This proves the important result stated in \cref{t:effective}: equivalence relations in $\CompOrdop$ are effective.
Finally, we show that this implies that $\CompOrdop$ is equivalent to a variety.

Let us note that the proof of effectiveness of equivalence relations in $\CompOrdop$ is far more involved than the proof of the corresponding fact for $\CompHaus^\opcat$.
To our understanding, this is due to the fact that---as it emerges in the proof of \cite[Chapter 9, Theorem 1.11]{BarrWells2005}---every reflexive relation in $\CompHaus^\opcat$ is an equivalence relation (i.e., $\CompHaus^\opcat$ is a Mal'cev category\index{category!Mal'cev}\index{Mal'cev|see{category, Mal'cev}}), whereas the same does not hold for $\CompOrdop$: the study of symmetry and transitivity seems necessary in our case.

This chapter is based on a joint work with L.\ Reggio \cite{AbbadiniReggio2020}, whose novel results can be found in \cref{s:quotient-objects,s:equivalence-corelations,s:proof-of-main-res} below.
\Cref{s:varieties-as-categories,s:compact-pospaces}, instead, collect some useful results from the literature.


\section{Varieties and quasivarieties as categories}\label{s:varieties-as-categories}

In this section we provide the background needed to state a well-known characterisation of those categories which are equivalent to some (quasi)variety of algebras (\cref{t:char-quasivarieties} below).

Recall, from \cref{d:quasivariety,d:variety}, that by \emph{variety of algebras} (resp.\ \emph{quasivariety of algebras}) we mean a class of algebras in a (possibly large) signature that can be presented by a class of equations (resp.\ implications) and that has free algebras.

The abstract characterisation of varieties and quasivarieties (also called primitive and quasiprimitive classes) has a long history in category theory, starting in the `60s: in particular, we mention \cite{Lawvere1963-book,Isbell1964,Linton1966,Felscher1968,Duskin1969,Vitale1994,Adamek2004}\footnote{To add some detail, \cite{Lawvere1963-book} studies varieties of finitary algebras, \cite{Isbell1964} quasivarieties of finitary algebras, \cite{Linton1966} varieties of possibly infinitary algebras, \cite{Felscher1968} varieties and quasivarieties of possibly infinitary algebras, \cite{Duskin1969} varieties of possibly infinitary algebras, \cite{Vitale1994} varieties of possibly infinitary algebras, and \cite{Adamek2004} varieties and quasivarieties of finitary algebras and varieties and quasivarieties of possibly infinitary algebras.}.


\subsection{Characterisation of quasivarieties}

Each quasivariety of algebras has an object of special interest: the free algebra over one element.
This object possesses certain categorical properties which are used in the characterisation of quasivarieties: it is both a regular generator and a regular projective object.

Recall that an object $G$ of a category $\cat{C}$ with coproducts is a \emph{regular generator}\index{object!regular generator}\index{regular generator|see{object, regular generator}} if, for every object $A$ of $\cat{C}$, the canonical morphism
\[
	\sum_{\hom_{\cat{C}}(G,A)}{G}\to A
\]
is a regular epimorphism.
In a quasivariety, the free algebra over one element is a regular generator\footnote{In fact, the statement is true for any free algebra over a non-empty set.}:
this fact corresponds to the fact that every algebra is the quotient of some free algebra.

Further, an object $G$ of a category $\cat{C}$ is \emph{regular projective}\index{object!regular projective}\index{regular projective|see{object, regular projective}} if, for every morphism $f\colon G\to A$ and every regular epimorphism $g\colon B\twoheadrightarrow A$, there exists a morphism $h \colon G \to B$ such that the following diagram commutes.
\[
	\begin{tikzcd}
		G \arrow[dashed]{r}{h} \arrow[swap]{rd}{f}	& B \arrow[two heads]{d}{g}\\
			& A
	\end{tikzcd}
\]
In a quasivariety, the free object over one element is regular projective\footnote{In fact, the statement is true for any free algebra.}.

\begin{theorem}[{Characterisation of quasivarieties}]\label{t:char-quasivarieties}
	A category is equivalent to a quasivariety if, and only if, it is cocomplete and it has a regular projective regular generator object.
\end{theorem}
\begin{proof}
	By \cite[Theorem~3.6]{Adamek2004}.
\end{proof}


\subsection{Characterisation of varieties}

To obtain a categorical characterisation of varieties, one addresses the question: when is a quasivariety a variety?
By \cref{p:char-quasivar-ISP,p:char-var-HSP}, a quasivariety $\clalg{A}$ is a variety if, and only if, it is closed under homomorphic images.
This happens if, and only if, for every algebra $A$ in $\clalg{A}$ and every congruence $\sim$ on $A$, the quotient $A/{\sim}$ still belongs to $\clalg{A}$, or, equivalently, $\sim$ is the kernel of some morphism.
In categorical terms, congruences are internal equivalence relations, and kernels are kernel pairs: so, a quasivariety is a variety if, and only if, every internal equivalence relation is a kernel pair.
We now recall the related definitions.

\begin{notation}
	Given morphisms $f_0\colon X \to Y_0$ and $f_1\colon X \to Y_1$, the unique morphism induced by the universal property of the product is denoted by $\langle f_0, f_1\rangle\colon X\to Y_0\times Y_1$.
	Similarly, given morphisms $f_0\colon X_0 \to Y$ and $f_1\colon X_1 \to Y$, the coproduct map is denoted by $\binom{g_0}{g_1}\colon X_0+X_1\to Y$.
\end{notation}

Let $\cat{C}$ be a category with finite limits, and $A$ an object of $\cat{C}$.
An \emph{(internal) binary relation}\index{relation!on an object of a category!binary} on $A$ is a subobject $\langle p_0, p_1 \rangle \colon R \mono A \times A$, (or, equivalently, a pair of jointly monic maps $p_0, p_1 \colon R \rightrightarrows A$).
A binary relation $\langle p_0, p_1 \rangle \colon R \mono A \times A$ on $A$ is called
\begin{description}[wide]
	\item[\emph{reflexive}]\index{relation!on an object of a category!reflexive} provided there exists a morphism $d\colon A\to R$ such that the following diagram commutes;
	\[
	\begin{tikzcd}
		A \arrow[swap,tail]{rd}{\langle 1_A,1_A\rangle} \arrow[dashed]{rr}{d}	& 			& R \arrow[tail]{dl}{\langle p_0,p_1\rangle}\\
																				& A\times A &
	\end{tikzcd}
	\]
	\item[\emph{symmetric}]\index{relation!on an object of a category!symmetric} provided there exists a morphism $s\colon R\to R$ such that the following diagram commutes;
	\[\begin{tikzcd}
		R \arrow[dashed]{rr}{s}\arrow[tail,swap]{dr}{\langle p_1,p_0\rangle}	&			& R\arrow[tail]{dl}{\langle p_0,p_1\rangle}\\
																				& A\times A	&
	\end{tikzcd}\]
	\item[\emph{transitive}]\index{relation!on an object of a category!transitive} provided that, if the left-hand diagram below is a pullback square, then there exists a morphism $t\colon P\to R$ such that the right-hand diagram commutes.
	\[
	\begin{tikzcd}
		P \arrow{r}{\pi_1} \arrow[swap]{d}{\pi_0} \arrow[dr, phantom, "\ulcorner", very near start]	& R \arrow{d}{p_0} \\
		R \arrow[swap]{r}{p_1} 																		& A
	\end{tikzcd}
	\ \ \ \ \ \ 
	\begin{tikzcd}
		P \arrow[swap]{rd}{\langle p_0\circ \pi_0,p_1\circ\pi_1\rangle} \arrow[dashed]{rr}{t}	&			& R \arrow[tail]{dl}{\langle p_0,p_1\rangle}\\
																								& A\times A	&
	\end{tikzcd}
	\]
\end{description}
An \emph{(internal) equivalence relation}\index{relation!on an object of a category!equivalence} on $A$ is a reflexive symmetric transitive binary relation on $A$.

\begin{definition}\label{d:effective-exact}
	An equivalence relation $p_0,p_1\colon R\rightrightarrows A$ is \emph{effective}\index{relation!on an object of a category!effective} if it coincides with the kernel pair of its coequaliser.
\end{definition}
For categories of algebras, the definition of equivalence relation given above coincides with the usual notion of congruence\index{congruence}, while the effective equivalence relations in quasivarieties are the so-called \emph{relative} congruences\index{congruence!relative}, i.e.\ congruences that induce a quotient which still belongs to the quasivariety.

We can now state the following folklore result.

\begin{proposition} \label{p:effectiveness}
	A quasivariety $\cat{C}$ is a variety if, and only if, every equivalence relation in $\cat{C}$ is effective.
\end{proposition}
\begin{theorem}[Characterisation of varieties]\label{t:char-varieties}
	A category $\cat{C}$ is equivalent to a variety if, and only if, $\cat{C}$ is cocomplete, $\cat{C}$ has a regular projective regular generator object, and every equivalence relation in $\cat{C}$ is effective.
\end{theorem}
\begin{proof}
	This follows immediately from \cref{t:char-quasivarieties,p:effectiveness} (using the fact that varieties are quasivarieties).
\end{proof}


\subsection{Quasivarieties with a cogenerator}

In \cref{t:char-quasivarieties}, we have seen that a category is equivalent to a quasivariety if, and only if, it is cocomplete and it has a regular projective regular generator object.
Dualising these notions, by \emph{regular injective}\index{object!regular injective}\index{regular injective|see{object, regular injective}} we mean the dual notion of regular projective, and by \emph{regular cogenerator}\index{object!regular cogenerator}\index{regular cogenerator|see{object, regular cogenerator}} we mean the dual notion of regular generator.
It follows that a category $\cat{C}$ is dually equivalent to a quasivariety if, and only if, it is complete and it has a regular injective regular cogenerator object.
In this case, the description of a quasivariety dual to $\cat{C}$ can be obtained by inspection of the proof of \cite[Theorem~3.6]{Adamek2004}.
However, in the case $\cat{C}$ admits a faitfhul representable functor $U \colon \cat{C} \to \Set$, an easier description can be given.
This is explained in \cref{p:Ada-natural} below, for which we need a couple of lemmas.
These results combine the categorical characterisation of quasivarieties and the theory of natural dualities, for an overview of which we refer to \cite{PorstTholen1991}.

In analogy with the definition of regular generator, we recall that an object $G$ of a category $\cat{C}$ with coproducts is a \emph{generator}\index{object!generator}\index{generator|see{object, generator}} if, for every object $A$ of $\cat{C}$, the canonical morphism
\[
	\sum_{\hom_{\cat{C}}(G,A)}{G}\to A
\]
is an epimorphism.
The dual notion is \emph{cogenerator}\index{object!cogenerator}\index{cogenerator|see{object, cogenerator}}.

\begin{lemma} \label{l:cogen-qv}
	An algebra $A$ of a quasivariety $\cat{D}$ is a cogenerator if, and only if, 
	\[
		\cat{D} = \opS\opP(A).
	\]
\end{lemma}
\begin{proof}
	By definition, an object $A$ of $\cat{D}$ is a cogenerator if, and only if, for every object $B$ of $\cat{D}$, the canonical map
	\[
		B \to \prod_{\hom(B, A)} A
	\] 
	is a monomorphism, or, equivalently, there exists a monomorphism from $B$ to some power of $A$.
	Since $\cat{D}$ is a quasivariety, categorical products are classical direct products of algebras and monomorphisms are injective functions.
	Therefore, $A$ is a cogenerator if, and only, $\cat{D} \seq \opS\opP(A)$.
	The inclusion $\cat{D} \supseteq \opS\opP(A)$ holds for every object $A$ because quasivarieties are closed under products and subalgebras.
\end{proof}

We remark that there are examples of varieties which do not have a cogenerator, such as the category of semigroups, the category of groups and the category of rings \cite[Examples~7.18(8)]{AdamekHerrlichStrecker2006}.

\begin{lemma}\label{l:Ada-natural-general-with-cogen-dual}
	Let $\cat{C}$ be a cocomplete category, let $X$ be a regular projective regular generator of $\cat{C}$, and let $X_0$ be a cogenerator of $\cat{C}$.
	Let $\Sigma$ be the signature whose elements of arity $\kappa$ (for each cardinal $\kappa$) are the morphisms from $X$ to $\sum_{i \in \kappa} X$, and let $\bar{X}$ be the $\Sigma$-algebra whose underlying set is $\hom(X, X_0)$ and on which the interpretation of any operation symbol $s$ of arity $\kappa$ maps $(f_i)_{i \in \kappa}$ to the composite $X \xrightarrow{s} \sum_{i \in \kappa} X \xrightarrow{(f_i)_{i \in \kappa}} X_0$.
	Then, $\cat{C}$ is equivalent to $\opS\opP(\bar{X})$.
\end{lemma}

\begin{proof} 
	Following \cite[Theorem~3.6]{Adamek2004}, we have a functor $E \colon \cat{C} \to \ALG \Sigma$,	defined as follows: given an object $Y$, the underlying set of $E(Y)$ is $\hom(X, Y)$, and the interpretation of an operation symbol $s$ of arity $\kappa$ maps $(f_i)_{i \in \kappa}$ to the composite $X \xrightarrow{s} \sum_{i \in \kappa} X \xrightarrow{(f_i)_{i \in \kappa}} Y$; given a morphism $f\colon Y \to Z$, the map $E(f)$ maps $g$ to $f \circ g$.
	Let $\cat{D}$ be the closure of the image of $E$ under isomorphisms.
	By \cite[Theorem~3.6]{Adamek2004}, we have the following: since $X$ is a generator, the functor $E$ is faithful, and since $X$ is a regular projective regular generator, the functor $E$ is full, and $\cat{D}$ is a quasivariety.
	In particular, it follows that $\cat{C}$ is equivalent to $\cat{D}$.
	Since $X_0$ is a cogenerator of $\cat{C}$, the object $E(X_0)$ is a cogenerator in $\cat{D}$.
	Therefore, by \cref{l:cogen-qv}, we have $\cat{D} = \opS\opP(E(X_0))$.
	Note that the object $E(X_0)$ coincides with the object $\bar{X}$ of the statement.
\end{proof}

We remark that the fact that the class $\opS\opP(\bar{X})$ in the statement of \cref{l:Ada-natural-general-with-cogen-dual} is a quasivariety (as attested by the proof) should be self-evident because of \cref{l:ISP}.

We recall that a functor $U \colon \cat{C} \to \Set$ is said to be \emph{representable}\index{functor!representable}\index{representable|see{functor, representable}} if there exists an object $C$ in $\cat{C}$ such that $U$ is naturally isomorphic to $\hom(C, -)$.
For the following result, we recall that representable functors preserve all limits---let alone powers.
We anticipate the fact that we will use the following result with $\cat{C} = \CompOrd$, $X = [0,1]$, and $U \colon \cat{C} \to \Set$ the obvious forgetful functor.

\begin{proposition} \label{p:Ada-natural}
	\
	Let $\cat{C}$ be a complete category, let $X$ be a regular injective regular cogenerator of $\cat{C}$, and let $U \colon \cat{C} \to \Set$ be a faitfhul representable functor.
	Let $\Sigma$ be the signature whose elements of arity $\kappa$ (for each cardinal $\kappa$) are the morphisms from $X^\kappa$ to $X$, and let $\bar{X}$ be the $\Sigma$-algebra whose underlying set is $U(X)$ and on which the interpretation of any operation symbol $f$ is $U(f)$.
	Then, $\cat{C}$ is dually equivalent to $\opS\opP(\bar{X})$.
\end{proposition}

\begin{proof}
	Let $X_0$ be an object such that $U \simeq \hom(-, X_0)$.
	Faithfulness of $U$ is equivalent to the fact that $X_0$ is a generator \cite[Corollary~4.5.9]{Borceux1994-vol1}.
	The result then follows from \cref{l:Ada-natural-general-with-cogen-dual}.
\end{proof}


\section{The dual of compact ordered spaces is a quasivariety}\label{s:compact-pospaces}

\begin{proposition}\label{p:properties-of-PC-morphisms}
	The following statements hold.
	\begin{enumerate}
		\item \label{i:mono}
			A morphism in $\CompOrd$ is a monomorphism if, and only if, it is injective.
		\item \label{i:regmono}
			A morphism in $\CompOrd$ is a regular monomorphisms if, and only if, it is injective and order-reflecting.
		\item \label{i:epi}
			A morphism in $\CompOrd$ is an epimorphisms if, and only if, it is surjective.
		\item \label{i:iso}
			A morphism in $\CompOrd$ is an isomorphism if, and only if, it is bijective and order-reflecting.
	\end{enumerate}
\end{proposition}
\begin{proof}
	Recall that faithful functors reflect monomorphisms and right adjoints preserve monomorphisms.
	The forgetful functor $\CompOrd \to \Set$ is faithful and, by \cref{r:figure-is-right}, right adjoint.
	\Cref{i:mono} follows.
	
	For \cref{i:regmono,i:epi} see, e.g., \cite[Theorem~2.6]{HNN2018}.
	
	It is clear that any isomorphism in $\CompOrd$ is bijective and reflects the order.
	Let $f \colon X \to Y$ be a bijective order-reflective morphism of compact ordered spaces.
	Then,  by \cref{p:comp-Haus-homeo}, $f$ is a homeomorphism, and thus it admits a continuous inverse function $g$; the function $g$ is order-preserving because $f$ is order-reflecting.
\end{proof}
Note that, by \cref{i:regmono} in \cref{p:properties-of-PC-morphisms}, the regular monomorphisms are, up to isomorphisms, the closed subets with the induced topology and order.

\begin{proposition}\label{p:int-is-regular-injective}
	The unit interval $[0,1]$ is a regular cogenerator of $\CompOrd$.
\end{proposition}
\begin{proof}
	By \cref{l:embedding}, for every compact ordered space $X$, the canonical morphism
	\[
		X \to \prod_{\hom_{\CompOrd}(X, [0,1])}[0,1]
	\]
	is injective and order-reflective.
	By \cref{i:regmono} in \cref{p:properties-of-PC-morphisms}, the regular monomorphisms are precisely the injective and order-reflecting monomorphisms.
	It follows that $[0,1]$ is a regular cogenerator.
\end{proof}

\begin{lemma} \label{l:Nac-reg-inj}
	Let $X$ be a compact ordered space equipped with a closed partial order and let $F$ be a closed subset of $X$.
	Then, every continuous order-preserving real-valued function on $F$ can be extended to the entire space in such a way as to remain continuous and order-preserving.
\end{lemma}
\begin{proof}
	The statement holds by \cite[Theorem~6, p.\ 49]{Nachbin}, which applies to compact ordered spaces by \cite[Corollary of Theorem~4, p.\ 48]{Nachbin}.
\end{proof}

\begin{proposition}\label{p:int-is-reg-cogen}
	The unit interval $[0,1]$ is a regular injective object of $\CompOrd$.
\end{proposition}
\begin{proof}
	By \cref{l:Nac-reg-inj}.
\end{proof}

\begin{lemma} \label{l:abstractly-cocountable}
	Every morphism in $\CompOrd$ from a power of $[0,1]$ to $[0,1]$ factors through a countable sub-power.
	\[
		\begin{tikzcd}
			{[0,1]^{X}} \arrow{r}{f} \arrow[swap, two heads]{d}{} & {[0,1]}\\
			{[0,1]^{Y}} \arrow[dashed, swap]{ru}{} &
		\end{tikzcd}
	\]
\end{lemma}
\begin{proof}
	Every continuous map from a power of $[0,1]$ to $[0,1]$ depends on at most countably many coordinates \cite[Theorem 1]{Mib44}.
\end{proof}

We let $\SignCM$ (for `Signature of Order-preserving Continuous function') denote the signature whose operation symbols of arity $\kappa$ are the order-preserving continuous functions from $[0,1]^\kappa$ to $[0,1]$.
Every operation symbol in $\SignCM$ has an obvious interpretation on $[0,1]$.
We let $\SignCM_{\leq \omega}$ denote the sub-signature of $\SignCM$ consisting of the operations symbols in $\SignCM$ of at most countable arity.

\begin{theorem} \label{t:duality-not-explicit}
	The category $\CompOrd$ is dually equivalent to the quasivarieties\footnote{As we will prove, they are actually varieties.}
	\[
		\opS\opP\mathopen{}\left(\left\langle [0,1]; \SignCM \right\rangle\right)\mathclose{},
	\]
	and
	\[
		\opS\opP\mathopen{}\left(\left\langle [0,1]; \SignCM_{\leq \omega} \right\rangle\right)\mathclose{}.
	\]
\end{theorem}
\begin{proof}
	By \cref{p:PC-complete-cocomplete}, $\CompOrd$ is complete.
	The compact ordered space $[0,1]$ is a regular injective regular cogenerator of $\CompOrd$ by \cref{p:int-is-reg-cogen,p:int-is-regular-injective}.
	The forgetful functor from $\CompOrd$ to $\Set$ is faitfhul, and it is represented by the one-element compact ordered space.
	So, \cref{p:Ada-natural} applies and we obtain that $\CompOrd$ is dually equivalent to $\opS\opP\mathopen{}\left(\left\langle [0,1]; \SignCM \right\rangle\right)\mathclose{}$, which,  by \cref{l:ISP}, is a quasivariety.

	By \cref{l:abstractly-cocountable}, we can restrict to operations of at most countable arity.
	(Alternatively, one can use the fact that $[0,1]$ is $\aleph_1$-copresentable in $\CompOrd$ \cite[Proposition~3.7]{HNN2018}; see \cite[Definition~1.16]{AdaRos} for the definition of $\kappa$-presentability for $\kappa$ a regular cardinal.\footnote{The reader might also want to take a look at \cite[Proposition~2.4]{PedicchioVitale2000}, where it is proved that the notions of abstractly finite and finitely generated are equivalent for a regular projective regular generator with copowers.})
\end{proof}

A stronger version of \cref{t:duality-not-explicit} was obtained in~\cite{HNN2018}, where the authors proved that $\CompOrdop$ is an $\aleph_1$-ary quasivariety\index{quasivariety of algebras!$\aleph_1$-ary}, i.e., it can be presented by means of operations of at most countable arity, and of implications with at most countably many premises.
The main contribution of this chapter consists in a proof of the fact that every equivalence relation in $\CompOrdop$ is effective (\cref{t:effective}).
When this result is coupled with the fact that $\CompOrdop$ is equivalent to a quasivariety, we obtain the main result of this chapter: $\CompOrdop$ is equivalent to a variety (\cref{t:MAIN}).


\section{Quotient objects} \label{s:quotient-objects}

We will study equivalence relations in $\CompOrdop$ by looking at their dual in $\CompOrd$.
In order to do so, we need to introduce some notation for the dual concepts.
\begin{notation}\label{n:quotequiv}
	We let $\Quot(X)$ denote the class of epimorphisms of compact ordered spaces with domain $X$.
	We equip $\Quot(X)$ with a relation ${\les}$ as follows: an element $f\colon X \epi Y$ is below an element $g\colon X \epi Z$ if, and only if, there exists a (necessarily unique) morphism $h \colon Y \to Z$ such that the following diagram commutes\footnote{We warn the reader that, in \cite{AbbadiniReggio2020}, the order of $\QuotEquiv$ is the opposite of what we consider here.}.
	\[
		\begin{tikzcd}
			X \arrow[two heads]{r}{f} \arrow[swap, two heads]{rd}{g}	& Y \arrow[dashed]{d}{h} \\
																						& Z
		\end{tikzcd}
	\]
	It is immediate that $\les$ is reflexive and transitive, so that $\Quot(X)$ becomes a preordered class.
	
	There is a standard way in which a partially ordered set is obtained from a preordered set, i.e.\ by identifying elements of an equivalence class (see \cref{i:symmetrise} in \cref{r:fundamental-pleasant}).
	In the same fashion, from $\Quot(X)$, we obtain a partially ordered class (in fact, a set) $\QuotEquiv(X)$.
	Explicitly, $\QuotEquiv(X)$ is the set of equivalence classes of epimorphisms of compact ordered spaces with domain $X$, where two epimorphisms $f\colon X\epi Y$ and $g\colon X\epi Z$ are equivalent if, and only if, there exists an isomorphism $h \colon Y \to Z$ such that $hf = g$; moreover, the equivalence class of $f\colon X \epi Y$ is below the equivalence class of $g\colon X \epi Z$ if, and only if, there exists a morphism $h \colon Y \to Z$ such that $hf = g$.
	Elements of $\QuotEquiv(X)$ are called \emph{quotient objects of $X$}\index{quotient object}\index{object!quotient}.
	We warn the reader that our terminology is non-standard.
	By a quotient object we do not mean a regular epimorphism, but what may be called a \emph{cosubobject}\index{cosubobject}.
	With a little abuse of notation, we take the liberty to refer to an element of $\QuotEquiv(X)$ just with one of its representatives.
\end{notation}
Our next goal is to encode quotient objects on $X$ internally on $X$.
To make a parallelism: by \cite[The Alexandroff Theorem~3.2.11]{Engelking1989}\footnote{The reader is warned that, in \cite{Engelking1989}, by `compact space' is meant what we here call a compact \emph{Hausdorff} space.}, in the category $\CompHaus$ of compact Hausdorff spaces, an epimorphism $f \colon X \epi Y$ (equivalently, a surjective continuous function) is encoded by the equivalence relation ${\sim_f} \df \{(x,y) \in X \times X \mid f(x) = f(y)\}$;
the equivalence relation $\sim_f$ is closed, and, in fact, there is a bijection between equivalence classes of epimorphisms of compact Hausdorff spaces with domain $X$ and closed equivalence relations on $X$.
There is also an analogous version for Stone spaces, namely \emph{Boolean relations}\index{Boolean relation}\footnote{Sometimes called \emph{Boolean equivalences}\index{Boolean equivalence}.} (see \cite[Lemma~1, Chapter 37]{GivantHalmos2009}), and an analogous version for Priestley spaces, introduced under the name of \emph{lattice preorders}\index{lattice preorder}\index{preorder!lattice} in \cite[Definition~2.3]{CignoliLafalcePetrovich1991}\footnote{Lattice preorders are also called \emph{Priestley quasiorders}\index{Priestley quasiorder} (\cite[Definition~3.5]{Schmid2002}), or \emph{compatible quasiorders}\index{compatible quasiorder}\index{quasiorder!compatible}.}.

In the case of compact ordered spaces, we encode a quotient object $f \colon X \epi Y$ via a certain preorder $\les_f$ on $X$, as follows.
\begin{notation}\label{n:les-f}
	Given a morphism $f\colon (X,\leq_X)\to (Y,\leq_Y)$ in $\CompOrd$, we set
	\begin{equation*}
		\les_f=\{(x_1,x_2)\in X\times X\mid f(x_1)\leq_Y f(x_2)\}.
	\end{equation*}
\end{notation}
To illustrate \cref{n:les-f} with an example, consider the following compact ordered space $X =\{a, b ,c\}$ with $a \leq b$.
\[
	\begin{tikzpicture}[node distance= 1.5 cm, auto]
		\node (11) {$\bullet$};
		\node (21) [below of=11]{$\bullet$};
		\node (31) [below of=21]{$\bullet$};
		
		\node [left=0.2 cm] at (11.west) {$c$};
		\node [left=0.2 cm] at (21.west) {$b$};
		\node [left=0.2 cm] at (31.west) {$a$};
		
		\node[rounded corners,draw=black,fit={(11) (21) (31)}] {};

		\draw[->, thick] (31) to node {}  (21);
	
	\end{tikzpicture}
\]
It is understood that the topology is the discrete one (the topology of any finite Stone space is discrete).
Consider a chain $Y$ of two elements $\bot \leq \top$, equipped with the discrete topology, and let $f \colon X \to Y$ be the function that maps $a$ and $b$ to $\bot$ and $c$ to $\top$.
\[
	\begin{tikzpicture}[node distance= 1.5 cm, auto]
		\node (11) {$\bullet$};
		\node (21) [below of=11]{$\bullet$};
		\node (31) [below of=21]{$\bullet$};
		
		\node [above=0.5 cm] at (11.north) {$X$};
				
		\node (12) [right of=11]{};
		\node (22) [right of=21]{};

		\node (13) [right of=12]{};
		\node (23) [right of=22]{};
		
		\node [above=0.1 cm] at (13.north) {$f$};

		\node (14) [right of=13]{};
		\node (24) [right of=23]{};
		
		\node (15) [right of=14]{$\bullet$};
		\node (25) [right of=24]{$\bullet$};
		
		\node [above=0.5 cm] at (15.north) {$Y$};
		
		\node[rounded corners,draw=black,fit={(11) (21) (31)}] {};
		\node[rounded corners,draw=black,fit={(15) (25)}] {};

		\draw[->, thick] (31) to node {}  (21);
		
		\draw[->, thick] (25) to node {}  (15);
		
		\draw[->, thick, dashed] (11) to node {}  (15);
		\draw[->, thick, dashed] (21) to node {}  (25);
		\draw[->, thick, dashed] (31) to node {}  (25);

	\end{tikzpicture}
\]
Then, $\les_f$ looks as the following preorder on $X$.
\[
	\begin{tikzpicture}[node distance= 1.5 cm, auto]
		\node (11) {$\bullet$};
		\node (21) [below of=11]{$\bullet$};
		\node (31) [below of=21]{$\bullet$};
		
		\node[rounded corners,draw=black,fit={(11) (21) (31)}] {};

		\draw[->, thick] (21) to node {}  (11);
		\draw[<->, thick] (31) to node {}  (21);
	
	\end{tikzpicture}
\]
\begin{example}
	If $Y$ is a compact Hausdorff space equipped with the identity partial order, and $f \colon X \to Y$ is a morphism of compact ordered spaces, then ${\les_f} = \{(x, y) \in X \times X \mid f(x) = f(y)\}$.
	So, the specialisation to compact Hausdorff spaces of this approach is precisely the one that we have discussed before \cref{n:les-f}.
\end{example}
\Cref{n:les-f} will be relevant especially for $f$ an epimorphism.
The idea is that, up to an isomorphism, an epimorphism $f$ can be completely recovered from $\les_f$.
In order to establish an inverse for the assignment $f \mapsto {\les_f}$, we shall investigate the properties satisfied by $\les_f$:
these properties are precisely that $\les_f$ is a closed preorder which extends the given partial order, as we now shall see.
\begin{lemma} \label{l:les-f-in-preo}
	If $f\colon (X,\leq_X)\to (Y,\leq_Y)$ is a morphism in $\CompOrd$, then $\les_f$ is a closed preorder on $X$ that extends $\leq_X$.
\end{lemma}
\begin{proof}
	The fact that $\les_f$ is a preorder follows from the fact that $\leq_Y$ is reflexive and transitive.
	The monotonicity of $f$ entails ${\leq_X}\subseteq{\les_f}$.
	The set $\les_f$ is closed in $X\times X$ because $\les_f$ is the preimage of $\leq_Y$ under the continuous map $f\times f\colon X\times X\to Y\times Y$.
\end{proof}
\begin{remark}\label{r:sym}
	We shall now see how one recovers an epimorphism $f$ from $\les_f$.
	Let $(X, \leq_X)$ be a compact ordered space and let ${\les}$ be a closed preoder on $X$ that extends $\leq_X$.
	We equip $X/({\les}\cap{{\les}^\oprel})$ with the quotient topology and the quotient order, defined by
	\[
		[x]_{{\les}\cap{{\les}^\oprel}}\leq_{X/({\les}\cap{{\les}^\oprel})} [y]_{{\les}\cap{{\les}^\oprel}}\ \Longleftrightarrow \ x\les y.
	\]
	Then, $X/({\les}\cap{{\les}^\oprel})$ is a compact ordered space, as proved in the paragraph `Reflector of $\CompOrd \hookrightarrow \CompHausPreord$' in \cref{r:figure-is-right}.
	Moreover, since ${\leq_X}\subseteq{\les}$, the function $X \epi X/({\les}\cap{{\les}^\oprel})$ is order-preserving.
	In conclusion, the function 
	\[
		X \epi X/({\les}\cap{{\les}^\oprel})
	\]
	is an epimorphism in $\CompOrd$.
\end{remark}

For $X$ a compact ordered space, we let $\Preo(X)$ denote the set of closed preorders on $X$ that extend $\leq_X$.
We equip $\Preo(X)$ with the partial order given by inclusion.

Our goal, met in \cref{t:bijection-Q-P}, is to prove that the assignments\\
\begin{minipage}{0.49\textwidth}
	\begin{align*}
		\QuotEquiv(X)			& \longrightarrow \Preo(X)\\
		\big(f \colon X\epi Y\big)	& \longmapsto 		{\les_f}\\
	\end{align*}
\end{minipage}
\begin{minipage}{0.49\textwidth}
	\begin{align*}
		\Preo(X)	& \longrightarrow	\QuotEquiv(X)\\
		{\les}		& \longmapsto 		\big(X\epi X/{({\les} \cap {\les}^\oprel)}\big)\\
	\end{align*}
\end{minipage}\\
establish an isomorphism between the partially ordered sets $\QuotEquiv(X)$ and $\Preo(X)$.
This will allow us to work with $\Preo(X)$ instead of $\QuotEquiv(X)$.


\subsubsection{The adjunction between the coslice category over \texorpdfstring{$X$}{\unichar{"1D44B}} and \texorpdfstring{$\mathbf{P}(X)$}{\unichar{"1D40F}(\unichar{"1D44B})}}

For the rest of this section, we fix a compact ordered space $X$.
In \cref{t:bijection-Q-P} below, we will prove the correspondence between quotients objects on $X$ and closed preorders on $X$ extending the given partial order.
To do so, we start by establishing, in \cref{l:adjunction} below, an adjunction between the coslice category $X\downarrow \CompOrd$ (whose objects are morphisms with domain $X$) and the partially ordered set $\Preo(X)$, regarded as a category.
From this adjunction, we will obtain an equivalence by restricting to the fixed points, and then an isomorphism via a certain quotient.

We let $X\downarrow \CompOrd$ denote the \emph{coslice category}\index{category!coslice}\index{coslice|see{category, coslice}} of $\CompOrd$ over $X$, i.e., the category whose objects are the morphisms in $\CompOrd$ with domain $X$ and whose morphisms from an object $f\colon X\to Y$ to an object $g\colon X\to Z$ are the morphisms $h\colon Y\to Z$ in $\CompOrd$ such that the following triangle commutes.
\[
	\begin{tikzcd}
		X \arrow{r}{f}\arrow[swap]{rd}{g}	& Y \arrow[dashed]{d}{h} \\
														& Z
	\end{tikzcd}
\]
Whenever convenient, we shall regard $\Preo(X)$ as a category, in the way in which it is usually done for partially ordered sets.
\begin{notation}\label{n:F-G}
We let 
	\begin{equation*}
		\begin{split}
			F \colon 	X \downarrow \CompOrd	& \longrightarrow	\Preo(X)	\\
							\big(f\colon X\to Y\big)			& \longmapsto		{\les_f}
		\end{split}
	\end{equation*}
	denote the assignment described in \cref{n:les-f}.
	Note that $\les_f$ belongs to $\Preo(X)$ by \cref{l:les-f-in-preo}.
	This assignment can be extended on morphisms so that $G$ becomes a functor: given $f\colon X\to Y$ and $g\colon X\to Z$ in $X\downarrow \CompOrd$, and given $h\colon Y\to Z$ such that $g=hf$, we set $G(h)$ as the unique morphism in $\Preo(X)$ from $\les_f$ to ${\les}_g$.
	
	We let
	\begin{equation*}
		\begin{split}
			F\colon	\Preo(X)	& \longrightarrow	X\downarrow \CompOrd	\\
						\les		& \longmapsto		\big(X\epi X/{({\les}\cap {\les}^\oprel)}\big)
		\end{split}
	\end{equation*}
	denote the assignment described in \cref{r:sym}.
	This assignment can be extended on morphisms so that $F$ becomes a functor: given ${{\les}_1},{{\les}_{2}} \in \Preo(X)$ such that ${{\les}_1} \seq {{\les}_2}$, $F$ maps the unique morphism from ${\les}_1$ to ${\les}_2$ to the morphism of compact ordered spaces
	\begin{equation*}
		\begin{split}
			X/({\les}_1\cap{\les}_1^\oprel)	&\longrightarrow	X/({\les}_2\cap{\les}_2^\oprel) \\
			[x]_{{\les}_1\cap{\les}_1^\oprel}	&\longmapsto		[x]_{{\les}_2\cap{\les}_2^\oprel}.
		\end{split}
	\end{equation*}
	It is easily seen that the functor $GF \colon \Preo(X) \to \Preo(X)$ is the identity functor.
	We let $\eta$ denote the identity natural transformation from the identity functor $1_{\Preo}$ on $\Preo$ to itself.
	
	Given the adjunction between $\CompOrd$ and $\CompHausPreord$ described in \cref{r:figure-is-right}, for all compact ordered spaces we have a morphism
	\begin{equation}\label{e:eps-f}
		\begin{split}
			\eps_f \colon	X/{({\les_f}\cap {\les_f^\oprel})}	& \longrightarrow	Y \\
								[x]			& \longmapsto		f(x).
		\end{split}
	\end{equation}
	\begin{claim}
		$\eps$ is a natural transformation.
	\end{claim}	
	\begin{claimproof}
		Let $f\colon X\to Y$ and $g\colon X\to Z$ be elements of $X\downarrow \CompOrd$, and let $h \colon Y\to Z$ be such that the $g = hf$.
			We shall prove that the following diagram commutes.
			\[
				\begin{tikzcd}
					{(X\xrightarrow{\pi}X/{\sim_f})} \arrow{r}{\eps_f} \arrow[swap]{d}{FG(h)}	& (X\xrightarrow{f}Y) \arrow{d}{h}\\
					{(X\xrightarrow{\pi}X/{\sim_g})} \arrow{r}{\eps_g}									& (X\xrightarrow{g}Z)
				\end{tikzcd}
			\]
			The commutativity of the diagram above amounts to the commutativity of the following one.
			\[
			\begin{tikzcd}
				X/{\sim_f} \arrow{r}{\eps_f} \arrow[swap]{d}{FG(h)}	& Y \arrow{d}{h}\\
				X/{\sim_g} \arrow{r}{\eps_g}									& Z
			\end{tikzcd}
			\]
			For every $x\in X$ we have 
			\[
				h(\eps_f([x]_{\sim_f}))  =  h(f(x))  =  g(x)  =  \eps_g([x]_{\sim_g})  =  \eps_g(FG(h)(x)).
			\]
			This proves our claim.
	\end{claimproof}
\end{notation}

\begin{lemma}\label{l:adjunction}
	The functor 
	\[
		F\colon \Preo(X)\to X \downarrow \CompOrd
	\]
	is left adjoint to the functor
	\[
		G\colon X\downarrow\CompOrd \to \Preo(X),
	\]
	with unit $\eta$ and counit $\eps$.
\end{lemma}
\begin{proof}
	It remains to prove that the triangle identities hold.
	One triangle identity is trivial because every diagram commutes in a category arising from a partially ordered set.
	We now set the remaining triangle identity.
	Let ${\les}\in \Preo(X)$.
	We shall prove that the following diagram commutes.
	\begin{equation} \label{e:triangle}
		\begin{tikzcd}
				F(\les) \arrow{r}{F(\eta_{\les})} \arrow[swap]{rd}{1_{F(\les)}}	& FGF(\les) \arrow{d}{\eps_{F(\les)}}\\
																										& F(\les)
		\end{tikzcd}
	\end{equation}
	Since $GF$ is the identity functor, and $\eta$ is the identity natural transformation, the commutativity of \cref{e:triangle} amounts to the fact that $\eps_{F(\les)}$ is the identity on $F(\les)$, which is not hard to see.
\end{proof}
We recall that a morphism in a coslice category $X\downarrow \mathsf{C}$ is an isomorphism in $X\downarrow \mathsf{C}$ if, and only if, it is an isomorphism in $\mathsf{C}$.
Then, we have the following.
\begin{lemma}\label{l:unit-iso}
	Given an object $f\colon X\to Y$ of $X\downarrow\CompOrd$, the component of the counit $\eps$ at $f$ is an isomorphism if, and only if, $f$ is an epimorphism.
\end{lemma}
\begin{proof}
	For every $f \colon X \to Y$, the function $\eps_f \colon X/{({\les_f}\cap {\les_f^\oprel})} \to Y$ is injective because, for all $x,y \in X$, we have
	\begin{align*}
		\eps_f([x]) = \eps_f([y])	& \Longleftrightarrow f(x) = f(y)\\
											& \Longleftrightarrow f(x) \leq f(y) \text{ and } f(y) \leq f(x)\\
											& \Longleftrightarrow x \les_f y \text{ and } y \les_f x\\
											& \Longleftrightarrow [x] = [y],
	\end{align*}
	and reflects the order because, for all $x,y \in X$, we have
	\[
		\eps_f(x) \leq \eps_f(y) \ \Longleftrightarrow \ f(x) \leq f(y) \ \Longleftrightarrow \ x \les_f y \ \Longleftrightarrow \ [x] \leq [y].
	\]
	Therefore, by \cref{i:iso} in \cref{p:properties-of-PC-morphisms}, $\eps_f$ is an isomorphism if, and only if, it is surjective, i.e.\ an epimorphism.
\end{proof}
We state the following for future reference.
\begin{lemma}\label{l:correspondence-epi-morphism}
	Let $f \colon X \to Y$ and $g \colon X \to Z$ be morphisms of compact ordered spaces, and suppose $f$ is surjective.
	Then, the condition ${\les_f} \seq {\les_g}$ holds if, and only if, there exists a morphism $h \colon Y \to Z$ of compact ordered spaces such that the following diagram commutes.
	\[
		\begin{tikzcd}
			X \arrow[two heads]{r}{f} \arrow[swap]{rd}{g}	& Y \arrow[dashed]{d}{h} \\
																			& Z 
		\end{tikzcd}
	\]
\end{lemma}
\begin{proof}
	By \cref{l:adjunction,l:unit-iso}.
\end{proof}
We now consider the preordered class $\Quot(X)$ of epimorphisms with domain $X$ as a full subcategory of $X\downarrow\CompOrd$.
\begin{lemma}\label{l:equivalence}
	The restrictions of the functors $F$ and $G$ to $\Preo(X)$ and $\Quot(X)$ are quasi-inverses.
\end{lemma}
\begin{proof}
	By \cref{l:adjunction}, the functor $F\colon \Preo(X)\to (X\downarrow \CompOrd)$ is left adjoint to $G\colon (X\downarrow \CompOrd)\to \Preo(X)$.
	For every ${\les} \in \Preo(X)$, the component of the unit $\eta$ at $\les$ is the identity morphism; in particular, it is an isomorphism.
	As observed in \cref{l:unit-iso}, the component of the counit $\eps$ at an element $f\colon X\to Y$ of $X\downarrow \CompOrd$ is an isomorphism if and only if $f\colon X\to Y$ is an epimorphism (\cref{l:unit-iso}).
\end{proof}

We obtain now the main result of this section.
\begin{theorem}\label{t:bijection-Q-P}
	The assignments\\
	\begin{minipage}{0.49\textwidth}
		\begin{align*}
			\QuotEquiv(X)			& \longrightarrow \Preo(X)\\
			\big(f \colon X\epi Y\big)	& \longmapsto 		{\les_f}\\
		\end{align*}
	\end{minipage}
	\begin{minipage}{0.49\textwidth}
		\begin{align*}
			\Preo(X)	& \longrightarrow	\QuotEquiv(X)\\
			{\les}		& \longmapsto 		\big(X\epi X/{({\les} \cap {\les}^\oprel)}\big)\\
		\end{align*}
	\end{minipage}\\
	establish an isomorphism between the partially ordered sets $\Preo(X)$ and $\QuotEquiv(X)$.
\end{theorem}
\begin{proof}
	By \cref{l:equivalence}.
\end{proof}
%


\section{Equivalence corelations}\label{s:equivalence-corelations}

In this section we provide a description of equivalence relations in the category $\CompOrdop$, which will then be exploited in the next section to prove that equivalence relations in $\CompOrdop$ are effective.

Recall that a binary relation on an object $A$ of a category $\cat{C}$ is a subobject of $A \times A$.
Dualising this definition, given a compact ordered space $X$, we call a \emph{binary corelation}\index{corelation!binary} on $X$ a quotient object $\binom{q_0}{q_1}\colon X+ X \epi S$ of the compact ordered space $X+X$.
We recall from \cref{l:preserves-finite-coproducts} that $X + X$ is the disjoint union of two copies of $X$ equipped with the coproduct topology and coproduct order.
A binary corelation on $X$ is called respectively \emph{reflexive}\index{corelation!reflexive}, \emph{symmetric}\index{corelation!symmetric}, \emph{transitive}\index{corelation!transitive} provided that it satisfies the properties:\\
\begin{minipage}[t]{0.49\textwidth}
	\begin{figure}[H]
		\centering
		\begin{tikzcd}
			{} 							& X+X \arrow[two heads,swap]{dl}{\binom{q_0}{q_1}} \arrow[two heads]{dr}{\binom{1_X}{1_X}}	& \\
			S\arrow[dashed]{rr}{d} 	& 																															& X
		\end{tikzcd}
		{\vspace{-5pt}\caption*{reflexivity}}
	\end{figure}
\end{minipage}
\begin{minipage}[t]{0.49\textwidth}
	\begin{figure}[H]
		\centering
		\begin{tikzcd}
											& X+X \arrow[swap, two heads]{dl}{\binom{q_0}{q_1}} \arrow[two heads]{dr}{\binom{q_1}{q_0}}	& \\
			S \arrow[dashed]{rr}{s}	& 																															& S
		\end{tikzcd}
		{\vspace{-5pt}\caption*{symmetry}}
	\end{figure}
\end{minipage}
\begin{figure}[H]
	\centering
	\begin{tikzcd}
		X \arrow{r}{q_0} \arrow[swap]{d}{q_1}	& S \arrow{d}{\lambda_1} \\
		S \arrow[swap]{r}{\lambda_0}			 	& P \arrow[ul, phantom, "\lrcorner", very near start]
	\end{tikzcd}
	\ \ \ \ $\Longrightarrow$ \ \ \ \
	\begin{tikzcd}
										& X+X \arrow[swap, two heads]{dl}{\binom{q_0}{q_1}} \arrow{dr}{\binom{\lambda_0\circ q_0}{\lambda_1\circ q_1}}	& \\
		S\arrow[dashed]{rr}{t}	& 																																					& P
	\end{tikzcd}
	{\vspace{-5pt}\caption*{transitivity}}
\end{figure}
An \emph{equivalence corelation}\index{corelation!equivalence} on $X$ is a reflexive symmetric transitive binary corelation on $X$.
The key observation is that, since quotient objects on $X + X$ are in bijection with certain preorders on $X + X$, equivalence corelations are more manageable than their duals.
\begin{definition}
	We call \emph{binary corelational structure}\index{corelational structure!binary} on a compact ordered space $X$ an element of $\Preo(X + X)$, i.e.\ a closed preorder on $X+X$ which extends the coproduct order $\leq_{X+X}$ on $X+X$.
\end{definition}
As an example consider a chain $X$ of two elements with discrete topology.
\[ 
	\begin{tikzpicture}[node distance=1.5 cm, auto]
		\node (11) {$\bullet$};
		\node (21) [below of=11]{$\bullet$};
		
		\node[rounded corners,draw=black,fit={(11) (21)}] {};
		
		\draw[->, thick] (21) to node {}  (11);
	\end{tikzpicture}
\]
Then $X+ X$ looks as follows.
\[ 
	\begin{tikzpicture}[node distance= 1.5 cm, auto]
		\node (11) {$\bullet$};
		\node (12) [right of=11]{$\bullet$};
		\node (21) [below of=11]{$\bullet$};
		\node (22) [right of=21]{$\bullet$};
		
		\node[rounded corners,draw=black,fit={(11) (21)}] {};
		\node[rounded corners,draw=black,fit={(12) (22)}] {};
		
		\draw[->, thick] (21) to node {}  (11);
		\draw[->, thick] (22) to node {}  (12);

	\end{tikzpicture}
\]
Here are some examples of binary corelational structures on $X$; the one on the left is the smallest one, the one on the right is the greatest one.

\medskip
\noindent
\begin{minipage}{0.325\textwidth}
	\begin{center}
		\begin{tikzpicture}[node distance= 1.5 cm, auto]
			\node (11) {$\bullet$};
			\node (12) [right of=11]{$\bullet$};
			\node (21) [below of=11]{$\bullet$};
			\node (22) [right of=21]{$\bullet$};
			
			\node[rounded corners,draw=black,fit={(11) (21)}] {};
			\node[rounded corners,draw=black,fit={(12) (22)}] {};
			
			\draw[->, thick] (21) to node {}  (11);
			\draw[->, thick] (22) to node {}  (12);
		\end{tikzpicture}
	\end{center}
\end{minipage}
\begin{minipage}{0.325\textwidth}
	\begin{center}
		\begin{tikzpicture}[node distance= 1.5 cm, auto]
			\node (11) {$\bullet$};
			\node (12) [right of=11]{$\bullet$};
			\node (21) [below of=11]{$\bullet$};
			\node (22) [right of=21]{$\bullet$};
			
			\node[rounded corners,draw=black,fit={(11) (21)}] {};
			\node[rounded corners,draw=black,fit={(12) (22)}] {};
			
			\draw[->, thick] (21) to node {}  (11);
			\draw[->, thick] (22) to node {}  (12);
			\draw[<->, thick] (21) to node {}  (22);
			\draw[->, thick] (21) to node {}  (12);
		\end{tikzpicture}
	\end{center}

\end{minipage}
\begin{minipage}{0.325\textwidth}
	\begin{center}
		\begin{tikzpicture}[node distance= 1.5 cm, auto]
			\node (11) {$\bullet$};
			\node (12) [right of=11]{$\bullet$};
			\node (21) [below of=11]{$\bullet$};
			\node (22) [right of=21]{$\bullet$};
			
			\node[rounded corners,draw=black,fit={(11) (21)}] {};
			\node[rounded corners,draw=black,fit={(12) (22)}] {};
			
			\draw[<->, thick] (21) to node {}  (11);
			\draw[<->, thick] (22) to node {}  (12);
			\draw[<->, thick] (21) to node {}  (22);
			\draw[<->, thick] (21) to node {}  (12);
			\draw[<->, thick] (11) to node {}  (12);
			\draw[<->, thick] (11) to node {}  (22);
		\end{tikzpicture}
	\end{center}
\end{minipage}
\medskip

\Cref{t:bijection-Q-P} establishes a bijective correspondence between binary corelational structures on $X$ (i.e., elements of $\Preo(X+X)$) and binary corelations on $X$ (i.e., elements of $\Quot(X+X)$).
\begin{definition}
	A binary corelational structure on a compact ordered space $X$ is called \emph{reflexive}\index{corelational structure!reflexive} (resp.\ \emph{symmetric}\index{corelational structure!symmetric}, \emph{transitive}\index{corelational structure!transitive}, \emph{equivalence}\index{corelational structure!equivalence}) if the corresponding binary corelation on $X$ is reflexive (resp.\ symmetric, transitive, equivalence).
\end{definition}

\begin{notation}
	We denote the elements of $X+X$ by $(x,i)$, where $x$ varies in $X$ and $i$ varies in $\{0,1\}$.
	Further, $i^*$ stands for $1-i$.
	For example, $(x,1^*)=(x,0)$.
\end{notation}

We anticipate the fact that on the chain of two elements $X= \{\bot, \top\}$ (with $\bot \leq \top$) there are exactly four equivalence corelational structures.

\medskip
\noindent
\begin{minipage}{0.245\textwidth}
	\begin{center}
		\begin{tikzpicture}[node distance= 1.5 cm, auto]
			\node (11) {$\bullet$};
			\node (12) [right of=11]{$\bullet$};
			\node (21) [below of=11]{$\bullet$};
			\node (22) [right of=21]{$\bullet$};
			
			\node[rounded corners,draw=black,fit={(11) (21)}] {};
			\node[rounded corners,draw=black,fit={(12) (22)}] {};
			
			\draw[->, thick] (21) to node {}  (11);
			\draw[->, thick] (22) to node {}  (12);
		\end{tikzpicture}
	\end{center}
\end{minipage}
\begin{minipage}{0.245\textwidth}
	\begin{center}
		\begin{tikzpicture}[node distance= 1.5 cm, auto]
			\node (11) {$\bullet$};
			\node (12) [right of=11]{$\bullet$};
			\node (21) [below of=11]{$\bullet$};
			\node (22) [right of=21]{$\bullet$};
			
			\node[rounded corners,draw=black,fit={(11) (21)}] {};
			\node[rounded corners,draw=black,fit={(12) (22)}] {};
			
			\draw[->, thick] (21) to node {}  (11);
			\draw[->, thick] (22) to node {}  (12);
			\draw[<->, thick] (21) to node {}  (22);
			\draw[->, thick] (21) to node {}  (12);
			\draw[->, thick] (22) to node {}  (11);
		\end{tikzpicture}
	\end{center}
\end{minipage}
\begin{minipage}{0.245\textwidth}
	\begin{center}
		\begin{tikzpicture}[node distance= 1.5 cm, auto]
			\node (11) {$\bullet$};
			\node (12) [right of=11]{$\bullet$};
			\node (21) [below of=11]{$\bullet$};
			\node (22) [right of=21]{$\bullet$};
			
			\node[rounded corners,draw=black,fit={(11) (21)}] {};
			\node[rounded corners,draw=black,fit={(12) (22)}] {};
			
			\draw[->, thick] (21) to node {}  (11);
			\draw[->, thick] (22) to node {}  (12);
			\draw[->, thick] (21) to node {}  (12);
			\draw[->, thick] (22) to node {}  (11);
			\draw[<->, thick] (11) to node {}  (12);
		\end{tikzpicture}
	\end{center}
\end{minipage}
\begin{minipage}{0.245\textwidth}
	\begin{center}
		\begin{tikzpicture}[node distance= 1.5 cm, auto]
			\node (11) {$\bullet$};
			\node (12) [right of=11]{$\bullet$};
			\node (21) [below of=11]{$\bullet$};
			\node (22) [right of=21]{$\bullet$};
			
			\node[rounded corners,draw=black,fit={(11) (21)}] {};
			\node[rounded corners,draw=black,fit={(12) (22)}] {};
			
			\draw[<->, thick] (21) to node {}  (11);
			\draw[<->, thick] (22) to node {}  (12);
			\draw[<->, thick] (21) to node {}  (22);
			\draw[<->, thick] (21) to node {}  (12);
			\draw[<->, thick] (11) to node {}  (12);
			\draw[<->, thick] (11) to node {}  (22);
		\end{tikzpicture}
	\end{center}
\end{minipage}
\medskip

\noindent
In general, as we will prove, every equivalence corelational structures $\les$ on a compact ordered space $X$ is obtained as follows: consider a closed subset $Y$ of $X$ and let $\les$ be the smallest preorder on $X + X$ that extends the coproduct order of $X + X$ and that satisfies $(y,0) \les (y,1)$ and $(y,1) \les (y,0)$ for every $y \in Y$.
For example, the binary corelational structures above are obtained by taking, respectively, $Y = \emptyset$, $Y = \{\bot\}$, $Y = \{\top\}$, and $Y = X$.
In fact, as we will see, proving that an equivalence corelational structure is effective boils down to proving that it arises with the construction above.

\begin{lemma}\label{l:refl}
	A binary corelational structure $\les$ on a compact ordered space $X$ is reflexive if, and only if, for all $x, y \in X$ and $i, j \in \{0,1\}$, we have
	\begin{equation*}
		 (x,i) \les (y,j) \ \Longrightarrow \ x \leq y.
	\end{equation*}
\end{lemma}
\begin{proof}
	Let $\binom{q_0}{q_1}\colon X+X \epi S$ be the binary corelation associated with $\les$.
	By definition of reflexive binary corelational structures, $\les$ is reflexive if, and only if, $\binom{q_0}{q_1}\colon X+X \epi S$ is above $\binom{1_X}{1_X}\colon X+X \epi X$ in the poset $\Quot(X+X)$.
	By \cref{t:bijection-Q-P}, this is equivalent to ${\les} \seq {\les_{\binom{1_X}{1_X}}}$.
	Given $(x,i),(y,j)\in X+X$, we have
	\[
		(x,i)\les_{\binom{1_X}{1_X}}(y,j) \ \Longleftrightarrow \ x\leq y.
	\]
	It follows that the binary corelational structure $\les$ is reflexive if, and only if, $(x,i)\les(y,j)$ entails $x\leq y$.
\end{proof}
For example, the following is a reflexive corelational structure on a two-element chain
\[	
	\begin{tikzpicture}[node distance= 1.5 cm, auto]
		\node (11) {$\bullet$};
		\node (12) [right of=11]{$\bullet$};
		\node (21) [below of=11]{$\bullet$};
		\node (22) [right of=21]{$\bullet$};
		
		\node[rounded corners,draw=black,fit={(11) (21)}] {};
		\node[rounded corners,draw=black,fit={(12) (22)}] {};
		
		\draw[->, thick] (21) to node {}  (11);
		\draw[->, thick] (22) to node {}  (12);
		\draw[->, thick] (21) to node {}  (12);
	\end{tikzpicture}
\]
whereas the following are \emph{not}.

\medskip
\noindent
\begin{minipage}{0.49\textwidth}
	\begin{center}
		\begin{tikzpicture}[node distance= 1.5 cm, auto]
			\node (11) {$\bullet$};
			\node (12) [right of=11]{$\bullet$};
			\node (21) [below of=11]{$\bullet$};
			\node (22) [right of=21]{$\bullet$};
			
			\node[rounded corners,draw=black,fit={(11) (21)}] {};
			\node[rounded corners,draw=black,fit={(12) (22)}] {};
			
			\draw[<->, thick] (21) to node {}  (11);
			\draw[->, thick] (22) to node {}  (12);
		\end{tikzpicture}
	\end{center}
\end{minipage}
\begin{minipage}{0.49\textwidth}
	\begin{center}
		\begin{tikzpicture}[node distance= 1.5 cm, auto]
			\node (11) {$\bullet$};
			\node (12) [right of=11]{$\bullet$};
			\node (21) [below of=11]{$\bullet$};
			\node (22) [right of=21]{$\bullet$};
			
			\node[rounded corners,draw=black,fit={(11) (21)}] {};
			\node[rounded corners,draw=black,fit={(12) (22)}] {};
			
			\draw[->, thick] (21) to node {}  (11);
			\draw[->, thick] (22) to node {}  (12);
			\draw[->, thick] (11) to node {}  (12);
			\draw[->, thick] (11) to node {}  (22);
		\end{tikzpicture}
	\end{center}
\end{minipage}
\medskip

\begin{lemma}\label{l:symm}
	A binary corelational structure $\les$ on a compact ordered space $X$ is symmetric if, and only if, for all $x, y \in X$ and $i, j \in \{0,1\}$, we have
	\begin{equation*}
		(x,i) \les (y,j) \  \Longrightarrow \ (x,i^*) \les (y,j^*).
	\end{equation*}
\end{lemma}
\begin{proof}
	Let $\binom{q_0}{q_1}\colon X+X \epi S$ be the binary corelation associated with $\les$.
	By definition of symmetric corelational structure, $\les$ is symmetric if, and only if, $\binom{q_0}{q_1}\colon X+X \epi S$ is above $\binom{q_1}{q_0}\colon X+X \epi S$ in $\Quot(X+X)$.
	By \cref{t:bijection-Q-P}, this happens exactly when ${\les} \seq {\les_{\binom{q_1}{q_0}}}$.
	For all $(x,i),(y,j)\in X+X$, we have
	\[
		(x,i)\les_{\binom{q_1}{q_0}}(y,j) \ \Longleftrightarrow \ (x,i^*)\les (y,j^*).
	\]
	Therefore, the binary corelational structure $\les$ is symmetric if, and only if, $(x,i)\les(y,j)$ entails $(x,i^*)\les (y,j^*)$.
\end{proof}

For example, the following are symmetric corelational structures on a two-element chain

\medskip
\noindent
\begin{minipage}{0.49\textwidth}
	\begin{center}
		\begin{tikzpicture}[node distance= 1.5 cm, auto]
			\node (11) {$\bullet$};
			\node (12) [right of=11]{$\bullet$};
			\node (21) [below of=11]{$\bullet$};
			\node (22) [right of=21]{$\bullet$};
			
			\node[rounded corners,draw=black,fit={(11) (21)}] {};
			\node[rounded corners,draw=black,fit={(12) (22)}] {};
			
			\draw[->, thick] (21) to node {}  (11);
			\draw[->, thick] (22) to node {}  (12);
			\draw[->, thick] (21) to node {}  (12);
			\draw[->, thick] (22) to node {}  (11);
		\end{tikzpicture}
	\end{center}
\end{minipage}
\begin{minipage}{0.49\textwidth}
	\begin{center}
		\begin{tikzpicture}[node distance= 1.5 cm, auto]
			\node (11) {$\bullet$};
			\node (12) [right of=11]{$\bullet$};
			\node (21) [below of=11]{$\bullet$};
			\node (22) [right of=21]{$\bullet$};
			
			\node[rounded corners,draw=black,fit={(11) (21)}] {};
			\node[rounded corners,draw=black,fit={(12) (22)}] {};
			
			\draw[->, thick] (21) to node {}  (11);
			\draw[->, thick] (22) to node {}  (12);
			\draw[->, thick] (12) to node {}  (21);
			\draw[->, thick] (11) to node {}  (22);
			\draw[<->, thick] (11) to node {}  (12);
		\end{tikzpicture}
	\end{center}
\end{minipage}
\medskip

\noindent
whereas the following is \emph{not}.
\[	
	\begin{tikzpicture}[node distance= 1.5 cm, auto]
		\node (11) {$\bullet$};
		\node (12) [right of=11]{$\bullet$};
		\node (21) [below of=11]{$\bullet$};
		\node (22) [right of=21]{$\bullet$};
		
		\node[rounded corners,draw=black,fit={(11) (21)}] {};
		\node[rounded corners,draw=black,fit={(12) (22)}] {};
		
		\draw[->, thick] (21) to node {}  (11);
		\draw[->, thick] (22) to node {}  (12);
		\draw[->, thick] (21) to node {}  (12);
	\end{tikzpicture}
\]
The last example shows a reflexive corelational structure which is not symmetric: this witnesses the fact that $\CompOrdop$ is not a Mal'cev category\index{category!Mal'cev} (i.e., a finitely complete category where every reflexive relation is an equivalence relation), in contrast to what happens for $\CompHaus^\opcat$.

\begin{lemma}\label{l:pushout-mono}
	Consider regular monomorphisms $f_0\colon X\rmono Y_0$, $f_1\colon X\rmono Y_1$ in $\CompOrd$ and their pushout as displayed below.
	\[
		\begin{tikzcd}
			X \arrow[hookrightarrow]{r}{f_1}											& Y_1	\arrow{d}{\lambda_1} \\
			Y_0 \arrow[hookleftarrow]{u}{f_0} \arrow[swap]{r}{\lambda_0}	& P		\arrow[ul, phantom, "\lrcorner", very near start]
		\end{tikzcd}
	\]
	Then, for every $i\in \{0,1\}$, the following conditions hold.
	\begin{enumerate}
		\item 
			For all $u,v\in Y_i$, the condition $\lambda_{i}(u)\leq \lambda_{i}(v)$ holds if, and only if, $u\leq v$.
		\item 
			For all $u\in Y_i$ and $v\in Y_{i^*}$, the condition $\lambda_{i}(u)\leq\lambda_{i^*}(v)$ holds if, and only if, there exists $x\in X$ such that $u\leq f_{i}(x)$ and $f_{i^*}(x)\leq v$.
	\end{enumerate}
\end{lemma}
\begin{proof}
	Let $q\colon Y_0+Y_1\to P$ be the unique morphism such that the following diagram commutes.
	\[
	\begin{tikzcd}
		Y_0 \arrow{r}{\iota_0} \arrow{rd}[swap]{\lambda_0}	& Y_0+Y_1 \arrow[dashed]{d}{q}	& Y_1 \arrow{l}[swap]{\iota_1} \arrow{ld}{\lambda_1}\\
																			& P										&
	\end{tikzcd}
	\]
	The existence and uniqueness of $q$ is given by the universal property of the coproduct.
	A straightforward argument shows that $q$ is the coequalizer of
	\[
		\iota_0\circ f_0, \iota_1 \circ f_1\colon X\rightrightarrows Y_0+Y_1.
	\]
	By the universal property of coequalizers, and by \cref{t:bijection-Q-P}, $\les_q$ is the smallest preorder $\les$ on $Y_0+Y_1$ such that $\les$ is a closed subspace of $(Y_0+Y_1)\times (Y_0+ Y_1)$, $\les$ extends the coproduct order of $Y_0+Y_1$, and, for all $x\in X$, $\iota_0f_0(x)\les \iota_1f_1(x)$ and $\iota_1f_1(x)\les \iota_0f_0(x)$.
	
	Let $\les_0$ be the relation on $Y_0+Y_1$ defined as follows.
	\begin{enumerate}
		\item
			For all $u,v\in Y_i$, $\iota_i(u)\les_0 \iota_i(v)$ if, and only if, $u\leq v$.
		\item
			For all $u\in Y_i$ and $v\in Y_{i^*}$, $\iota_i(u)\les_0\iota_i(v)$ if, and only if, there exists $x\in X$ such that $u\leq f_{i}(x)$ and $f_{i^*}(x)\leq v$.
	\end{enumerate}
	We shall prove ${\les_q} = {\les_0}$.
	Let us prove ${\les_0} \seq {\les_q}$.
	\begin{enumerate}
		\item
			For all $u,v\in Y_i$, if $u\leq v$, then $\iota_i(u)\leq_{Y_0+Y_1} \iota_i(v)$, which implies $\iota_i(u)\les_q \iota_i(v)$.
		\item
			For all $u\in Y_i$ and $v\in Y_{i^*}$, if there exists $x\in X$ such that $u\leq f_{i}(x)$ and $f_{i^*}(x)\leq v$, then $\iota_i(u)\leq \iota_i f_{i}(x)$ and $\iota_{i^*}f_{i^*}(x)\leq\iota_{i^*}(v)$, which implies $\iota_i(u)\les_q \iota_i f_{i}(x)$ and $\iota_{i^*}f_{i^*}(x)\les_q \iota_{i^*}(v)$, which implies $\iota_i(u)\les_q \iota_i f_{i}(x)\les_q\iota_{i^*}f_{i^*}(x)\les_q \iota_{i^*}(v)$, which implies $\iota_i(u)\les_q\iota_{i^*}(v)$.
	\end{enumerate}
	This proves ${\les_0}\seq {\les_q}$.
	To obtain the converse inclusion it is enough to notice that $\les_0$ is a closed preorder that extends the coproduct order of $Y_0+Y_1$, and that, for all $x\in X$, we have $\iota_0 f_0(x) \les_0 \iota_1 f_1(x)$ and $\iota_1 f_1(x)\les_0 \iota_0 f_0(x)$.
\end{proof}
\begin{lemma}\label{l:transitive}
	A reflexive binary corelational structure $\les$ on a compact ordered space $X$ is transitive if, and only if, for all $x,y \in X$ and all $i \in \{0,1\}$, we have
	\[
		(x,i)\les(y,i^*) \ \Longrightarrow \ \exists z\in X \text{ s.t.\ }(x,i)\les (z,i^*) \text{ and } (z,i)\les (y,i^*).
	\]
\end{lemma}
\begin{proof}
	Let $\binom{q_0}{q_1}\colon X+X \epi S$ be the binary corelation associated with $\les$.
	To improve readability, we write $[x,i]$ instead of $\binom{q_0}{q_1}(x,i)$.
	By definition of transitivity, the binary corelational structure $\les$ is transitive if, and only if, given a pushout square in $\CompOrd$ as in the left-hand diagram below,
	\[\begin{tikzcd}
		X \arrow{r}{q_0} \arrow[swap]{d}{q_1}	& S \arrow{d}{\lambda_1} \\
		S \arrow[swap]{r}{\lambda_0}				& P \arrow[ul, phantom, "\lrcorner", very near start]
	\end{tikzcd}
	\ \ \ \ \ \ \ \ 
	\begin{tikzcd}
		{} 	& X+X \arrow[swap, two heads]{dl}{\binom{q_0}{q_1}} \arrow{dr}{\binom{\lambda_0\circ q_0}{\lambda_1\circ q_1}}	& \\
		S \arrow[dashed]{rr}{t}	& 		& P
	\end{tikzcd}
	\]
	there is a morphism $t\colon S\to P$ such that the right-hand diagram commutes.
	By \cref{l:correspondence-epi-morphism}, such a $t$ exists precisely when, for every $(x,i),(y,j)\in X+X$, the condition $(x,i) \les (y,j)$ implies 
	\[
		\binom{\lambda_0\circ q_0}{\lambda_1\circ q_1}(x,i)\leq \binom{\lambda_0\circ q_0}{\lambda_1\circ q_1}(y,j),
	\]
	i.e., $\lambda_i([x,i])\leq \lambda_j([y,j])$.
	Recall that $\les$ is reflexive provided $q_0$ and $q_1$ are both sections of a morphism $d\colon S\to X$.
	In particular, $q_0$ and $q_1$ are regular monomorphisms in $\CompOrd$.
	Thus, by \cref{l:pushout-mono},  the condition $\lambda_i([x,i])\leq \lambda_j([y,j])$ holds if, and only if,
	\begin{equation}\label{eq:preord-pushout}
		(i = j \text{ and } (x,i) \les (y,j) ) \text{ or } (i \neq j \text{ and } \exists z \in X \text{ s.t.\ } (x,i) \les (z,j) \text{ and } (z,i) \les (y, j)).
	\end{equation}
	We conclude that $\les$ is transitive if, and only if, \cref{eq:preord-pushout} holds whenever $(x,i)\les(y,j)$.
	In turn, this is equivalent to the condition in the statement of the lemma.
\end{proof}
For example, the following are transitive corelational structures on a two-element chain

\medskip
\noindent
\begin{minipage}{0.49\textwidth}
	\begin{center}
		\begin{tikzpicture}[node distance= 1.5 cm, auto]
			\node (11) {$\bullet$};
			\node (12) [right of=11]{$\bullet$};
			\node (21) [below of=11]{$\bullet$};
			\node (22) [right of=21]{$\bullet$};
			
			\node[rounded corners,draw=black,fit={(11) (21)}] {};
			\node[rounded corners,draw=black,fit={(12) (22)}] {};
			
			\draw[->, thick] (21) to node {}  (11);
			\draw[->, thick] (22) to node {}  (12);
			\draw[->, thick] (21) to node {}  (12);
			\draw[->, thick] (11) to node {}  (12);
		\end{tikzpicture}
	\end{center}
\end{minipage}
\begin{minipage}{0.49\textwidth}
	\begin{center}
		\begin{tikzpicture}[node distance= 1.5 cm, auto]
			\node (11) {$\bullet$};
			\node (12) [right of=11]{$\bullet$};
			\node (21) [below of=11]{$\bullet$};
			\node (22) [right of=21]{$\bullet$};
			
			\node[rounded corners,draw=black,fit={(11) (21)}] {};
			\node[rounded corners,draw=black,fit={(12) (22)}] {};
			
			\draw[->, thick] (21) to node {}  (11);
			\draw[->, thick] (22) to node {}  (12);
			\draw[->, thick] (21) to node {}  (12);
			\draw[->, thick] (22) to node {}  (11);
			\draw[<->, thick] (11) to node {}  (12);
		\end{tikzpicture}
	\end{center}
\end{minipage}
\medskip

\noindent
whereas the following is \emph{not}.
\[	
	\begin{tikzpicture}[node distance= 1.5 cm, auto]
		\node (11) {$\bullet$};
		\node (12) [right of=11]{$\bullet$};
		\node (21) [below of=11]{$\bullet$};
		\node (22) [right of=21]{$\bullet$};
		
		\node[rounded corners,draw=black,fit={(11) (21)}] {};
		\node[rounded corners,draw=black,fit={(12) (22)}] {};
		
		\draw[->, thick] (21) to node {}  (11);
		\draw[->, thick] (22) to node {}  (12);
		\draw[->, thick] (21) to node {}  (12);
		\draw[->, thick] (22) to node {}  (11);
	\end{tikzpicture}
\]
This last example shows a reflexive (and symmetric) corelational structure which is not transitive, witnessing again the fact that $\CompOrdop$ is not a Mal'cev category.

Finally, we obtain a characterisation of equivalence corelational structures.
\begin{proposition}\label{p:equivalence}
	A binary corelational structure $\les$ on a compact ordered space $X$ is an equivalence corelational structure if, and only if, for all $x, y \in X$ and all $i,j \in \{0,1\}$ we have
	\[
		(x,i)\les (y,j) \ \Longrightarrow \ x\leq y \text{ and } (x,i^*)\les (y,j^*)
	\]
	and
	\[
		(x,i)\les(y,i^*) \ \Longrightarrow \ \exists z\in X \text{ s.t.\ }(x,i)\les (z,i^*) \text{ and } (z,i)\les (y,i^*).
	\]
\end{proposition}
\begin{proof}
	By \cref{l:refl,l:symm,l:transitive}.
\end{proof}

As an exercise, using \cref{p:equivalence}, the reader may verify that, as anticipated before, the equivalence corelational structures on a two-element chain are precisely the following ones.

\medskip
\noindent
\begin{minipage}{0.245\textwidth}
	\begin{center}
		\begin{tikzpicture}[node distance= 1.5 cm, auto]
			\node (11) {$\bullet$};
			\node (12) [right of=11]{$\bullet$};
			\node (21) [below of=11]{$\bullet$};
			\node (22) [right of=21]{$\bullet$};
			
			\node[rounded corners,draw=black,fit={(11) (21)}] {};
			\node[rounded corners,draw=black,fit={(12) (22)}] {};
			
			\draw[->, thick] (21) to node {}  (11);
			\draw[->, thick] (22) to node {}  (12);
		\end{tikzpicture}
	\end{center}
\end{minipage}
\begin{minipage}{0.245\textwidth}
	\begin{center}
		\begin{tikzpicture}[node distance= 1.5 cm, auto]
			\node (11) {$\bullet$};
			\node (12) [right of=11]{$\bullet$};
			\node (21) [below of=11]{$\bullet$};
			\node (22) [right of=21]{$\bullet$};
			
			\node[rounded corners,draw=black,fit={(11) (21)}] {};
			\node[rounded corners,draw=black,fit={(12) (22)}] {};
			
			\draw[->, thick] (21) to node {}  (11);
			\draw[->, thick] (22) to node {}  (12);
			\draw[<->, thick] (21) to node {}  (22);
			\draw[->, thick] (21) to node {}  (12);
			\draw[->, thick] (22) to node {}  (11);
		\end{tikzpicture}
	\end{center}
\end{minipage}
\begin{minipage}{0.245\textwidth}
	\begin{center}
		\begin{tikzpicture}[node distance= 1.5 cm, auto]
			\node (11) {$\bullet$};
			\node (12) [right of=11]{$\bullet$};
			\node (21) [below of=11]{$\bullet$};
			\node (22) [right of=21]{$\bullet$};
			
			\node[rounded corners,draw=black,fit={(11) (21)}] {};
			\node[rounded corners,draw=black,fit={(12) (22)}] {};
			
			\draw[->, thick] (21) to node {}  (11);
			\draw[->, thick] (22) to node {}  (12);
			\draw[->, thick] (21) to node {}  (12);
			\draw[->, thick] (22) to node {}  (11);
			\draw[<->, thick] (11) to node {}  (12);
		\end{tikzpicture}
	\end{center}
\end{minipage}
\begin{minipage}{0.245\textwidth}
	\begin{center}
		\begin{tikzpicture}[node distance= 1.5 cm, auto]
			\node (11) {$\bullet$};
			\node (12) [right of=11]{$\bullet$};
			\node (21) [below of=11]{$\bullet$};
			\node (22) [right of=21]{$\bullet$};
			
			\node[rounded corners,draw=black,fit={(11) (21)}] {};
			\node[rounded corners,draw=black,fit={(12) (22)}] {};
			
			\draw[<->, thick] (21) to node {}  (11);
			\draw[<->, thick] (22) to node {}  (12);
			\draw[<->, thick] (21) to node {}  (22);
			\draw[<->, thick] (21) to node {}  (12);
			\draw[<->, thick] (11) to node {}  (12);
			\draw[<->, thick] (11) to node {}  (22);
		\end{tikzpicture}
	\end{center}
\end{minipage}
%


\section{Main result: equivalence corelations are effective}\label{s:proof-of-main-res}
Dualising \cref{d:effective-exact}, we say that an equivalence corelation $\binom{q_0}{q_1}\colon X+X \epi S$ on a compact ordered space $X$ (and so the corresponding equivalence corelational structure) is \emph{effective}\index{corelation!effective}\index{corelational structure!effective} provided it coincides with the cokernel pair of its equaliser.
That is, provided the following is a pushout square in $\CompOrd$,
\begin{equation*}
	\begin{tikzcd}
		Y \arrow[hookrightarrow]{r}{k} 							& X \arrow{d}{q_1} \\
		X \arrow[hookleftarrow]{u}{k} \arrow[swap]{r}{q_0}	& S
	\end{tikzcd}
\end{equation*}
where $k \colon Y \to X$ is the equaliser of $q_0,q_1\colon X\rightrightarrows S$ in $\CompOrd$.

\begin{notation}
	Given a compact ordered space $X$ and a closed subspace $Y$ of $X$, we define the relation $\les^{Y}$ on $X+X$ as follows: for all $x, y \in X$ and $i \in \{0,1\}$ we set
	\[
		(x,i) \les^{Y} (y,i) \ \Longleftrightarrow \ x \leq y,
	\]
	and
	\[
		(x,i) \les^{Y} (y,i^*) \ \Longleftrightarrow \ \exists z\in Y \text{ s.t.\ } x \leq z \leq y.
	\]
\end{notation}

\begin{lemma}\label{l:cokernel-inclusion}
	Let $X$ be a compact ordered space, let $Y$ be a closed subspace of $X$, equipped with the induced topology and order.
	The binary corelational structure on $X$ associated with the pushout in $\CompOrd$ of the inclusion $Y \rmono X$ along itself is $\les^Y$.
\end{lemma}
\begin{proof}
	This is an immediate consequence of \cref{l:pushout-mono}.
\end{proof}
\begin{lemma}\label{l:effective-char}
	An equivalence corelational structure $\les$ on a compact ordered space $X$ is effective if, and only if, for all $x, y \in X$ and $i \in \{0,1\}$, we have
	\[(x,i)\les (y,i^*) \ \Longrightarrow \ \exists z\in X \text{ s.t.\ }x \leq z \leq y, (z,i)\les (z,i^*) \text{ and } (z,i^*)\les (z,i).\]
\end{lemma}
\begin{proof}
	Set 
	\[
		Y \df \{x\in X \mid (x,i)\les (x,i^*) \text{ and } (x,i^*)\les (x,i)\},
	\]
	and let us endow $Y$ with the induced topology and induced partial order.
	Denoting by $\binom{q_0}{q_1}\colon X+X \epi S$ the binary corelation on $X$ associated with $\les$, we have
	\[
		Y = \{x \in X \mid q_0(x) \leq q_1(x) \text{ and }q_1(x) \leq q_0(x)\} = \{x \in X \mid q_0(x) = q_1(x)\}.
	\]
	By \cref{p:pres-limits}, the inclusion $Y \rmono X$ is the equaliser of $q_0,q_1\colon X\rightrightarrows S$ in $\CompOrd$.
	Therefore, the binary corelational structure $\les$ is effective if and only if the following diagram is a pushout in $\CompOrd$.
	\begin{equation*}
		\begin{tikzcd}
			Y \arrow[hookrightarrow]{r}{k} 							& X \arrow{d}{q_1} \\
			X \arrow[hookleftarrow]{u}{k} \arrow[swap]{r}{q_0}	& S
		\end{tikzcd}
	\end{equation*}
	In turn, by \cref{l:cokernel-inclusion}, this is equivalent to saying that ${\les}={\les^Y}$.
	By definition of $\les^Y$, we have, for all $x, y \in X$ and $i \in \{0,1\}$,
	\[
		(x,i) \les^{Y} (y,i) \	 	\Longleftrightarrow \ x \leq y,
	\]
	and
	\[
		(x,i) \les^{Y} (y,i^*) \ \Longleftrightarrow \ \exists z\in X \text{ s.t.\ }x \leq z \leq y, (z,i)\les (z,i^*) \text{ and } (z,i^*)\les (z,i).
	\]
	Note that any reflexive binary corelational structure $\les'$ on $X$ satisfies, for all $x, y\in X$ and $i \in \{0,1\}$,
	\begin{equation*} \label{e:refl}
		(x, i) \les' (y, i) \Longleftrightarrow x \leq y.
	\end{equation*}
	The left-to-right implication follows from \cref{l:refl}, while the right-to-left implication holds because $\les$ extends the coproduct order of $X+X$.
	
	Moreover, note that every transitive binary corelational structure $\les'$ on $X$ satisfies
	\[
		\big(\exists z\in X \text{ s.t.\ }x \leq z \leq y, (z,i)\les' (z,i^*) \text{ and } (z,i^*) \les' (z,i) \big)  \ \Longrightarrow \ (x,i) \les' (y,i^*),
	\]
	because $\les'$ extends the partial order of $X+X$.
	
	Therefore, since $\les$ is reflexive and transitive, the condition ${\les} = {\les^Y}$ holds if, and only if, for all $x, y \in X$ and $i \in \{0,1\}$, we have
	\[
		(x,i)\les (y,i^*) \ \Longrightarrow \ \big(\exists z\in X \text{ s.t.\ }x \leq z \leq y, (z,i)\les (z,i^*) \text{ and } (z,i^*)\les (z,i) \big). \qedhere
	\]
\end{proof}
\begin{theorem}\label{t:effective}
	Every equivalence relation in $\CompOrdop$ is effective.
\end{theorem}
\begin{proof}
	Let $\les$ be an equivalence corelational structure on a compact ordered space $X$.
	In view of \cref{l:effective-char}, it is enough to show that, whenever $(x,i)\les (y,i^*)$, there is $z\in X$ such that 
	\[
	x\leq z\leq y,\ (z,i) \les (z,i^*) \text{ and } (z,i^*) \les (z,i).
	\]
	Fix arbitrary $x,y\in X$ and $i\in \{0,1\}$ satisfying $(x,i)\les (y,i^*)$, and set
	\[
	\Omega=\{u\in X\mid (x,i)\les (u,i^*) \text{ and } (u,i)\les (y,i^*) \}.
	\]
	The idea is to apply Zorn's Lemma to show that $\Omega$ has a maximal element $z$ satisfying the desired properties.
	
	Since $(x,i) \les (y, i^*)$ and $\les$ is transitive, by \cref{l:transitive} $\Omega$ is non-empty.
	\begin{claim}
		Every non-empty chain contained in $\Omega$ admits a supremum in $X$ and this element belongs to $\Omega$.
	\end{claim}
	\begin{claimproof}
		First, we show that $\Omega$ is a closed subset of $X$.
		The set $\Omega$ can be written as the intersection of the sets
		\[
			\Omega_1= \{u\in X\mid (x,i)\les (u,i^*) \} \text{ and } \Omega_2=\{u\in X\mid (u,i)\les (y,i^*) \}.
		\]
		The set $\Omega_1$ is the preimage, under the coproduct injection $\iota_{i^*}\colon X \hookrightarrow X+X$, of 
		\[
			\upset(x,i)=\{(w,j)\in X+X\mid (x,i)\les (w,j) \}.
		\]
		Since $\les$ is a closed preorder on $X+X$, the set $\upset(x,i)$ is closed in $X+X$ by \cref{l:up-set-closed}.
		Therefore, its preimage $\Omega_1$ is closed in $X$.
		Analogously, $\Omega_2$ is closed.
		Since $\Omega$ is the union of the closed subsets $\Omega_1$ and $\Omega_2$ of $X$, we conclude that $\Omega$ is a closed subset of $X$.
		
		Let $C$ be a chain contained in $\Omega$.
		By \cite[Proposition~VI.1.3]{ContLattDom}, every directed set in a compact ordered space has a supremum, which coincides with the topological limit of the set regarded as a net.
		Thus, $C$ has a supremum $s$ in $X$, which belongs to the topological closure of $C$ in $X$.
		Since $\Omega$ is a closed subset of $X$, the element $s$ belongs to $\Omega$.
	\end{claimproof}

	Having established that $\Omega$ is non-empty and that every non-empty chain in $\Omega$ admits an upper bound in $\Omega$, we can apply Zorn's Lemma, and obtain that $\Omega$ has a maximal element $z$.
	By \cref{l:refl}, since $\les$ is reflexive, from $(x,i)\les (z,i^*)$ and $(z,i)\les (y,i^*)$ we deduce $x \leq z \leq y$.
	\begin{claim}
		We have $(z,i) \les (z,i^*)$ and $(z, i^*) \les (z,i)$.
	\end{claim}
	\begin{claimproof}
	By \cref{l:transitive}, since $\les$ is transitive, from $(z,i)\les (y,i^*)$ it follows that there is $u \in X$ such that $(z,i)\les (u,i^*)$ and $(u,i)\les (y,i^*)$.
		Also, $(x,i)\les(z,i)$ because $\les$ extends the partial order of $X$.
		Thus, $(x,i)\les (z,i)\les (u,i^*)$, which implies $u\in \Omega$.
		By reflexivity, $(z,i)\les (u,i^*)$ entails $z\leq u$.
		Since $z$ is maximal, we have $z=u$.
		Therefore, $(z,i)\les (z,i^*)$.
		By \cref{l:symm}, since $\les$ is symmetric, from $(z,i)\les (z,i^*)$ we deduce $(z, i^*) \les (z,i)$.
	\end{claimproof}
	We have shown that, if $(x,i) \les (y,i^*)$, then there is $z \in X$ such that $x \leq z \leq y$, $(z,i) \les (z,i^*)$ and $(z, i^*) \les (z,i)$.
	As already pointed out at the beginning of the proof, by \cref{l:effective-char}, this implies that $\les$ is effective.
\end{proof}
\begin{remark}
	In the proof above, the topology plays a relevant role. 
	In fact, the dual of the category $\Ord$ of partially ordered sets does not have effective equivalence relations \cite[Remark~4.18]{HofmannNora2020}.
	To see the difference between $\Ord$ and $\CompOrd$, consider the partially ordered set $[0,1]$, with its canonical total order.
	Consider the relation $\les$ on $[0,1] + [0,1]$ defined as follows: for $i \in \{0,1\}$ and $x,y \in [0,1]$, set
	\[
		(x,i) \les (y,i)\ \Longleftrightarrow \ x \leq y,
	\]
	and
	\[
		(x,i) \les (y, i^*) \ \Longleftrightarrow\ x < y.
	\]
	The relation $\les$ satisfies the condition in \cref{p:equivalence} that characterises equivalence relations, whereas it does not satisfy the condition in \cref{l:effective-char} that characterises effective equivalence relations.
	In this case, the reason why this happens is because the relation $\les$ is not closed: indeed, the sequence $x_n \df \left(1- \frac{1}{n}, 0\right)$ converges to $(1,0)$, the constant sequence $y_n \df (1,1)$ converges to $(1,1)$, for all $n \in \N$ we have $x_n \les y_n$, but $(1,0) \not\les (1,1)$.
\end{remark}

We recall that $\SignCM$ is the signature whose operation symbols of arity $\kappa$ are the order-preserving continuous functions from $[0,1]^\kappa$ to $[0,1]$, and $\SignCM_{\leq \omega}$ is the sub-signature of $\SignCM$ consisting of the operations symbols of at most countable arity.

\begin{corollary} \label{c:equiv-to-varieties}
	The category $\CompOrd$ is dually equivalent to
	\[
		\opS\opP\mathopen{}\left(\left\langle [0,1]; \SignCM \right\rangle\right)\mathclose{}
	\]
	and
	\[
		\opS\opP\mathopen{}\left(\left\langle [0,1]; \SignCM_{\leq \omega} \right\rangle\right)\mathclose{},
	\]
	both of which are varieties.
\end{corollary}
\begin{proof}
	The two classes are quasivarieties by \cref{l:ISP}, and they are equivalent to $\CompOrdop$ by \cref{t:duality-not-explicit}.
	By \cref{t:effective}, every equivalence relation in $\CompOrdop$ is effective.
	Since a quasivariety is a variety if and only if equivalence relations are effective (\cref{p:effectiveness}), the result follows.
\end{proof}

We finally summarise our results.

\begin{theorem}\label{t:MAIN}
	The category $\CompOrd$ of compact ordered spaces is dually equivalent to a variety of algebras, with primitive operations of at most countable arity.
\end{theorem}
\begin{proof}
	By \cref{c:equiv-to-varieties}, the category $\CompOrd$ is dually equivalent to the variety $\opS\opP\mathopen{}\left(\left\langle [0,1]; \SignCM_{\leq \omega} \right\rangle\right)\mathclose{}$, whose primitive operations are of at most countable arity.
\end{proof}
%


\section{Conclusions}

In \cref{chap:CompOrd} we motivated our view that, as Priestley spaces are the partially-ordered generalisation of Stone spaces, the structures introduced by L.\ Nachbin under the name of \emph{compact ordered spaces} are the correct partially-ordered generalisation of compact Hausdorff spaces.

In the present chapter, starting from the observation that the categories of Stone spaces, Priestley spaces and compact Hausdorff spaces all have an equationally definable dual, we investigated whether the same happens for compact ordered spaces.
In fact, this is the case: The category of compact ordered spaces is dually equivalent to a variety of algebras, with primitive operations of at most countable arity.

Exploiting the insights from decades of invastigation of natural dualities and categorical characterisations of (quasi)varieties, we provided a varietal description of the dual of the category of compact ordered spaces: using Linton's language, this is the category of models of the varietal theory of order-preserving continuous functions between powers of the unit interval $[0,1]$, which happens to consists of all the subalgebras of powers of $[0,1]$.

However, some questions still remain unaddressed: Is it necessary to resort to infinitary operations? Does there exist a manageable set of primitive operations and axioms for the dual of the category of compact ordered spaces?
These questions we address in the following chapters.


\chapter{Negative axiomatisability results}\label{chap:negative-results}


\section{Introduction}

In \cref{chap:direct-proof} we proved that the category $\CompOrd$ of compact ordered spaces is dually equivalent to a variety of algebras with operations of at most countable arity.
One may wonder whether it is necessary to resort to infinitary operations.
In this short chapter we show that this is indeed the case: $\CompOrdop$ is not equivalent to any variety of finitary algebras.
In fact, we show the following stronger results.
\begin{enumerate}
	\item The category $\CompOrd$ is not dually equivalent to any finitely accessible category (\cref{t:finitely accessible}).
	\item The category $\CompOrd$ is not dually equivalent to any first-order definable class of structures (\cref{t:not-elementary}).
	\item The category $\CompOrd$ is not dually equivalent to any class of finitary algebras closed under products and subalgebras (\cref{t:SP}).
\end{enumerate}
The second result was suggested by S.\ Vasey (private communication) as an application of a result of M.\ Lieberman, J.\ Rosick\'y and S.\ Vasey \cite{LieRosVas}, replacing a previous weaker statement.

This chapter is based on a joint work with L.\ Reggio \cite{AbbadiniReggio2020}.


\section{Negative results}


\subsection{The dual of \texorpdfstring{$\CompOrd$}{CompOrd} is not a finitely accessible category} \label{subs:finitely-acc}

The first negative result makes use of the concept of finitely accessible category.

Classically, a finitary algebra is called \emph{finitely presentable} if it can be presented by finitely many generators and finitely many equations.
In any variety of finitary algebras, each algebra is the colimit of a directed system of finitely presentable algebras.
We will recall the classical categorical abstraction of finitely presentable algebras and then show that not every object of $\CompOrdop$ is a directed colimit of the objects of this kind, thus proving $\CompOrdop$ not to be equivalent to a variety of finitary algebras.
A categorical abstraction of the notion of finitely presentable algebra has been introduced independently by \cite[Definition~6.1]{GabUlm} and \cite[Expose I, Definition 9.3, p.\ 140]{ArtinGrothendieckVerdier1972}.
To this end, first recall that a partially ordered set is called \emph{directed}\index{partially ordered set!directed}\index{directed!partially ordered set|see{partially ordered set, directed}} provided that every finite subset has an upper bound.
\emph{Directed colimits}\index{colimit!directed}\index{directed!colimit|see{colimit, directed}} (also known as direct limits\index{direct limit} in universal algebra) are colimits of directed systems.
\begin{definition}[{See \cite[Definition~6.1]{GabUlm}, or \cite[Definition~1.1]{AdaRos}}]
	An object $A$ of a category $\cat{C}$ is said to be \emph{finitely presentable}\index{object!finitely presentable} if the covariant hom-functor $\hom_{\cat{C}}(A,-)\colon\cat{C}\to\Set$ preserves directed colimits.
	Explicitly, this means that if $D\colon \cat{I}\to \cat{C}$ is a functor with $\cat{I}$ a directed partially ordered set and $(c_i \colon D(i) \to C)_{i \in \cat{I}}$ is a colimit cocone for $D$, then, for every morphism $f \colon A \to C$ in $\cat{C}$, the following two conditions are satisfied \cite[Proposition~5.1.3]{Borceux1994-vol2}.
	\begin{enumerate}
		\item The morphism $f$ factors through some $c_i$, i.e., there exists $i \in \cat{I}$ and $g \colon A \to D(i)$ such that $f = c_{i} \circ g$.
		\[
			\begin{tikzpicture}
				\matrix(m)[matrix of math nodes, row sep=4em, column sep=5em, text height=1.5ex, text depth=0.25ex]{
				A & C\\ & D(i)\\ 
				};
				\path[->] (m-1-1) edge node[above] {$f$} (m-1-2);
				\path[->, dashed] (m-1-1) edge node[below] {$g$}  (m-2-2);
				\path[->] (m-2-2) edge [bend right=30] node[right] {$c_i$} (m-1-2);
			\end{tikzpicture}
		\]
		\item The factorisation is essentially unique, in the sense that, for all $j,k \in \cat{I}$, for all $g'\colon A \to D(j)$ and $g'' \colon A \to D(k)$ such that $f = c_{j} \circ g' = c_{k} \circ g''$, there exists a common upper bound $l$ of $j$ and $k$ such that $D(j \to l) \circ g' = D(k \to l)\circ g''$.
		\[
			\begin{tikzpicture}
				\matrix(m)[matrix of math nodes, row sep=4em, column sep=1.2em, text height=1.5ex, text depth=0.25ex]{
				A & & C &\\ & & D(l) & \\ &D(j) & &D(k) \\
				};
				\path[->] (m-1-1) edge node[above] {$f$} (m-1-3);
				\path[->] (m-1-1) edge  [bend right=50] node[left] {$g'$} (m-3-2);
				\path[->] (m-1-1) edge  [bend right=30] node[left] {$g''$} (m-3-4);
				\path[->] (m-3-2) edge [bend right=90] node[right] {$c_j$} (m-1-3);
				\path[->] (m-3-4) edge [bend right=60] node[right] {$c_k$} (m-1-3);
				\path[->] (m-2-3) edge [bend right=30] node[right] {$c_l$} (m-1-3);
				\path[->, dashed] (m-3-2) edge node[right] {} (m-2-3);
				\path[->, dashed] (m-3-4) edge node[right] {} (m-2-3);
				\path[->, dashed] (m-1-1) edge node[right] {} (m-2-3);
			\end{tikzpicture}
		\]
	\end{enumerate}
\end{definition}
In every variety of finitary algebras, the finitely presentable objects are precisely the algebras which are finitely presentable in the classical sense \cite[Proposition~3.8.14]{Borceux1994-vol2}.
\begin{definition}[{See \cite[Definition~2.1]{AdaRos}}]
	A category $\cat{C}$ is said to be \emph{finitely accessible}\index{category!finitely accessible}\index{finitely accessible|see{category, finitely accessible}} provided it has directed colimits, and there exists a set $S$ of finitely presentable objects of $\cat{C}$ such that each object of $\cat{C}$ is a directed colimit of objects in $S$.
\end{definition}
For example, varieties and quasivarieties of finitary algebras (with homomorphisms) are finitely accessible categories (cf.\ \cite[Corollary~3.7 and Theorem~3.24]{AdaRos}).

We recall that a \emph{Priestley space} is a compact topological space $X$ equipped with a partial order such that, for all $x,y$ with $x \not\leq y$, there exists a clopen up-set $C$ of $X$ such that $x \in C$ and $y \notin C$.
Priestley spaces were introduced by \cite{Priestley1970} to obtain a duality for bounded distributive lattices.
It is easily seen that the partial order of a Priestley space is closed; hence, every Priestley space is a compact ordered space.
The full subcategory of $\CompOrd$ defined by all Priestley spaces is denoted by $\Pries$.
\begin{lemma}\label{l:pro-completion-posfin}
	A compact ordered space is a Priestley space if, and only if, it is the codirected limit in $\CompOrd$ of finite partially ordered sets equipped with the discrete topologies.
\end{lemma}
\begin{proof}
	Let us denote with $\Ord_{\cat{fin}}$ the category of finite partially ordered sets and order-preserving maps.
	Recall from \cite[Corollary~VI.3.3(ii)]{Johnstone1986}) that the functor $\Ord_{\cat{fin}} \hookrightarrow \Pries$ which equips a finite partially ordered set with the discrete topology provides the pro-completion of $\Ord_{\cat{fin}}$.
	Moreover, it is not difficult to see that the inclusion functor $\Pries \hookrightarrow \CompOrd$ preserves limits.
	The desired result then follows.
\end{proof}
We say that an object in a category $\cat{C}$ is \emph{finitely copresentable}\index{object!finitely copresentable}\index{finitely copresentable|see{object, finitely copresentable}} if it is finitely presentable when regarded as an object of $\cat{C}^\opcat$.
The finitely copresentable objects in $\CompOrd$ are precisely the finite ones \cite[Remark~4.41]{HofmannNora2020}.
For our purposes, we need only one direction, and we here provide a self-contained proof of a slightly more general version of it.

\begin{lemma} \label{l:fin-copres}
	Let $\cat{F}$ be a full subcategory of $\CompOrd$ containing all Priestley spaces.
	Every finitely copresentable object in $\cat{F}$ is finite.
\end{lemma}
\begin{proof}
	Let $(X,\leq)$ be a finitely copresentable object in $\cat{F}$.
	Consider a surjective morphism $\gamma\colon Y\epi X$ in $\CompOrd$ with $Y$ a Priestley space; for example, let $Y=\beta|X|$ be the Stone-\v{C}ech compactification of the underlying set of $X$ equipped with the discrete topology, and let $\gamma\colon (\beta|X|,=)\to (X,\leq)$ be the unique continuous extension of the identity function $|X|\to |X|$.
	By \cref{l:pro-completion-posfin}, $Y$ is the codirected limit in $\CompOrd$ of finite posets $\{Y_i\}_{i\in I}$ with the discrete topologies.
	Denote by $\alpha_i\colon Y \to Y_i$ the $i$-th limit arrow.
	Since $Y$ lies in $\cat{F}$, and the inclusion functor $\cat{F} \hookrightarrow \CompOrd$ reflects limits, $Y$ is in fact the codirected limit of $\{Y_i\}_{i\in I}$ in $\cat{F}$.
	Since the object $X$ is finitely copresentable in $\cat{F}$, there exist $j \in I$ and a morphism $\phi \colon Y_j \to X$ such that $\gamma = \phi \circ \alpha_j$.
	\[
		\begin{tikzcd}
			Y \arrow[twoheadrightarrow]{r}{\gamma} \arrow{d}[swap]{\alpha_j}	& X \\
			Y_j \arrow{ur}[swap]{\phi}														&
		\end{tikzcd}
	\]
	The map $\gamma$ is surjective, hence so is $\phi$: this shows that $X$ is finite.
\end{proof}

The category $\Pries^\opcat$ is equivalent to the category of bounded distributive lattices with homomorphisms \cite{Priestley1970}.
In particular, $\Pries^\opcat$ is a finitely accessible category.
The following result is an adaptation of \cite[Proposition~1.2]{MarraReggio2017} to the ordered case.

\begin{theorem}\label{t:finitely accessible}
	Let $\cat{F}$ be a full subcategory of $\CompOrd$ extending $\Pries$.
	If $\cat{F}^\opcat$ is a finitely accessible category---let alone a variety or quasivariety of finitary algebras---then $\cat{F} = \Pries$.
\end{theorem}
\begin{proof}
	It suffices to show that every object in $\cat{F}$ is a Priestley space.
	Since $\cat{F}^\opcat$ is finitely accessible, every object of $\cat{F}$ is the codirected limit of finitely copresentable objects.
	Using the fact that the inclusion functor $\cat{F} \hookrightarrow \CompOrd$ reflects limits and that finitely copresentable objects in $\cat{F}$ are finite by \cref{l:fin-copres}, we deduce by \cref{l:pro-completion-posfin} that every object of $\cat{F}$ is a Priestley space, as was to be shown.
	
	Finally, we have already observed that finitary varieties and finitary quasivarieties are finitely accessible categories.
\end{proof}
%


\subsection{The dual of \texorpdfstring{$\CompOrd$}{CompOrd} is not a first-order definable class}

A \emph{first-order definable class of structures}\index{first-order definable class} is the class of models of a first-order theory, for which the reader is referred to \cite{ChangKeisler1990}.
When one such class is referred to as a category, it is understood that the morphisms are the homomorphism, i.e., a function that preserves all function symbols and all relation symbols.

\begin{lemma}[\cite{Richter1971}] \label{l:dir-concr}
	The forgetful functor from a first-order definable class of structures to $\Set$ preserves directed colimits.
\end{lemma}

\Cref{l:dir-concr} was used by M.\ Lieberman, J.\ Rosick\'y and S.\ Vasey to prove that the category of compact Hausdorff spaces is not dually equivalent to a first-order definable class of structures \cite[Corollary~12]{LieRosVas}.
In fact, they showed the following fact.
\begin{lemma} \label{l:LieRosVas}
	No faithful functor from $\CompHaus^\opcat$ to $\Set$ preserves directed colimits.
\end{lemma}
\begin{proof}
	See the final section of \cite{LieRosVas}.
\end{proof}
We use this fact in the proof of the following result.

\begin{theorem}\label{t:not-elementary}
	The category $\CompOrd$ is not dually equivalent to any first-order definable class of structures.
\end{theorem}
\begin{proof}
	Let $\Delta' \colon \CompHaus \to \CompHausXPreord$ be the functor that maps a compact Hausdorff space $X$ to the space $X$ itself with the discrete order (this is the left adjoint of the topological forgetful functor $\CompHausXPreord \to \CompHaus$).
	It is not difficult to see that this functor preserves products and equalisers: thus, it preserves limits.
	Moreover, note that the objects in the image of $\Delta$ are compact ordered spaces, so we can restrict $\Delta'$ to a functor $\Delta \colon \CompHaus \to \CompOrd$ that preserves limits.
	Then, the functor $\Delta^\opcat\colon \CompHaus^\opcat\to\CompOrdop$ preserves directed colimits.
	Hence, if there were a faithful functor $F$ from $\CompOrdop$ to $\Set$ preserving directed colimits, then the composition $F\circ \Delta\colon \CompHaus^\opcat\to \Set$ would also be a faithful functor preserving directed colimits, contradicting \cref{l:LieRosVas}.
	Thus, no faithful functor from $\CompOrdop$ to $\Set$ preserves directed colimits.
	By \cref{l:dir-concr}, this shows that $\CompOrd$ cannot be dually equivalent to a first-order definable class of structures.
\end{proof}


\subsection{The dual of \texorpdfstring{$\CompOrd$}{CompOrd} is not an \texorpdfstring{$\opS\opP$}{SP}-class of finitary algebras}

The fact that $\CompOrd$ is not dually equivalent to a variety of finitary algebras, together with the fact that equivalence corelations are effective, allows us to obtain another negative result.

\begin{theorem}\label{t:SP}
	The category $\CompOrd$ is not dually equivalent to any class of finitary algebras closed under products and subalgebras.
\end{theorem}
\begin{proof}
	Let us suppose, by way of contradiction, that $\CompOrdop$ is equivalent to a class of finitary algebras closed under products and subalgebras.
	In~\cite{Ban-var}, it is observed that every class of finitary algebras closed under subalgebras and products in which every equivalence relation is effective is a variety of algebras.
	By \cref{t:effective}, every equivalence relation in $\CompOrdop$ is effective.
	Therefore, $\CompOrdop$ is equivalent to a variety of finitary algebras, but this contradicts \cref{t:finitely accessible} (and \cref{t:not-elementary}).
\end{proof}
%


\section{Conclusions}

In \cref{chap:direct-proof} we proved that the category $\CompOrd$ of compact ordered spaces is dually equivalent to a variety, with operations of at most countable arity.
In the present chapter, we showed that it is indeed necessary to resort to infinitary operations, since $\CompOrd$ is \emph{not} dually equivalent to any variety of finitary algebras.

Having established the best possible bound on the arities in an equational axiomatisation of $\CompOrdop$, we are now left with the question: Can we provide a manageable set of primitive operations and axioms for $\CompOrdop$?
Addressing this question will be our main concern in the following chapters.


\chapter{Equivalence \`a la Mundici for unital lattice-ordered monoids}\label{chap:equiv}


\section{Introduction}

Given a compact ordered space $X$, the set 
\[
	\Cleq(X, [0,1]) \df \{f \colon X \to [0,1] \mid f \text{ is order-preserving and continuous}\}
\]
is closed under pointwise application of each order-preserving continuous function from a power of $[0,1]$ to $[0,1]$.
Thus, recalling that $\SignCM$ denotes the signature whose operation symbols of arity $\kappa$ are the order-preserving continuous functions from $[0,1]^\kappa$ to $[0,1]$, the set $\Cleq(X,[0,1])$ acquires a structure of a $\SignCM$-algebra.

In fact, as described in \cref{chap:direct-proof}, the assignment associating to each compact ordered space $X$ the $\SignCM$-algebra $\Cleq(X,[0,1])$ gives rise to a duality between the category of compact ordered spaces and the variety of algebras
\[
	\opS\opP\mathopen{}\left(\left\langle [0,1]; \SignCM \right\rangle\right)\mathclose{}.
\]

One may wonder whether manageable sets of primitive operations and axioms for this variety exist.
Even if choosing one specific signature seems to us a somewhat arbitrary task, we believe that the choice we will present is natural enough to be worth of consideration.
In particular, as MV-algebras were at the core of the equational axiomatisation of the dual of the category of compact Hausdorff spaces in \cite{MarraReggio2017}, we find it reasonable to base our work on those term-operations of MV-algebras whose interpretation in $[0,1]$ is order-preserving.
By \cite[Section~1]{CabrerJipsenKroupa2019}, such term-operations are generated by $\oplus$, $\odot$, $\lor$, $\land$, $0$ and $1$ (arities $2$, $2$, $2$, $2$, $0$, $0$), with interpretations in $[0,1]$ as follows.
\begin{align*}
	x \oplus y	& = \min\{x + y, 1\};\\
	x \odot y	& = \max\{x + y - 1, 0\};\\
	x \lor y 	& = \max\{x,y\};\\
	x \land y	& = \min\{x,y\};\\
	0				& = \text{the element } 0;\\
	1				& = \text{the element } 1.
\end{align*}

Which reasonable set of equational axioms should we consider for algebras in the signature $\{\oplus, \odot, \lor, \land, 0, 1\}$?
The key insight is that, to gain a better intuition on the subject, we might replace the set $\Cleq(X, [0,1])$ of $[0,1]$-valued order-preserving continuous functions with the set $\Cleq(X,\R)$ of \emph{real-valued} ones.
Then, we should accordingly replace the operations $\oplus$, $\odot$, $\lor$, $\land$, $0$ and $1$ with the operations $+$, $\lor$, $\land$, $0$, $1$ and $-1$.
The set $\Cleq(X,\R)$ is abstracted by what we call {\ulms}, which are algebras in the signature $\{+, \lor, \land, 0, 1, -1\}$ that satisfy certain reasonable axioms.
The set $\Cleq(X,[0,1])$ is abstracted by what we call {\mvms}, which are algebras in the signature $\{\oplus, \odot, \lor, \land, 0, 1\}$ that satisfy the axioms needed for our main result to hold.
Our main success is to have kept these axioms equational and in a finite number: a non-trivial task.
The main result is presented in \cref{t:G is equiv}: we exhibit an equivalence 
\[
	\begin{tikzcd}
		\ULM \arrow[yshift = .45ex]{r}{\Gam}	& \MVM\arrow[yshift = -.45ex]{l}{\X}
	\end{tikzcd}
\]
between the category $\ULM$ of {\ulms} and the category $\MVM$ of {\mvms}.
The functor $\Gam$ maps a {\ulm} $M$ to its unit interval $\Gam(M) \df \{ x \in M \mid 0 \leq x \leq 1\}$ (\cref{s:Gamma}), and the functor $\X$ maps {\amvm} $A$ to the set $\X(A)$ of `good $\Z$-sequences in $A$' (\cref{s:good.def}).

There are both pros and cons in working with {\ulms} or {\mvms}, respectively.
On the one hand, as we mentioned above, it is easier to work with the operations and axioms of {\ulms} rather than those of {\mvms}.
On the other hand, the class of {\mvms} is a variety of finitary algebras, so the tools of universal algebra apply.
The equivalence established in this chapter allows one to transfer the pros of each category to the other one.

As shown in \cref{s:restriction}, our result specialises to (and is inspired by) D.\ Mundici's celebrated result stating that the categories of {\extulgs} and {\mvms} are equivalent \cite[Theorem~3.9]{Mundici} (see also \cite[Section~2]{cdm2000}).
Knowledge about MV-algebras and lattice-ordered groups is assumed but not needed (except for \cref{s:restriction}); the chapter is written so as to maximise the insights for an MV-algebraist, and the reader who does not have such a knowledge might simply disregard the comments about MV-algebras.

We conclude this introduction with a comparison with \cite{AbbadiniEquiv}.
The main result presented here---namely, that the categories of {\ulms} and {\mvms} are equivalent---coincides with the one in \cite{AbbadiniEquiv}.
However, the proofs are different: in the present manuscript, we use Birkhoff's subdirect representation theorem, which simplifies the arguments but relies on the axiom of choice, in contrast with the choice-free proof in \cite{AbbadiniEquiv}.
Moreover, in this document we use $\Z$-indexed sequences instead of $\N$-indexed sequences, which provides some simplifications, and which seems more elegant.


\section{The algebras} \label{s:defn}


\subsection{Unital commutative distributive \texorpdfstring{$\ell$}{\unichar{"02113}}-monoids}

The interplay between lattice and monoid operations is the object of a long-standing interest (see \cite[Chapter~XII]{Fuchs1963}, \cite[Chapter~XIV]{Birkhoff1967}).
The subject emerged with the study of ideals, which has its roots in the work of R.\ Dedekind, and to which W.\ Krull provided some important contributions.
We mention that, in this direction, the introduction of the notion of residuated lattices (a special class of lattice-ordered monoids) by \cite{WardDilworth1939} has opened the way to a research field which is still very active today, also due to its connection to logics; see \cite{GalatosJipsenKowalskiEtAl2007} for a recent account on the subject.

In this subsection, we start by recalling the notions of lattice-ordered semigroup and lattice-ordered monoid; we warn the reader that one may encounter slightly different definitions in the literature, depending on the distributivity laws that the author wants to assume.
However, these notions are only auxiliary: in \cref{d:ulm} we will define the structures we are primarily interested in: {\extulms}.
In our view, these structures are an adequate analogue of Abelian lattice-ordered groups with strong order unit when one wants to replace the group structure with a monoid structure.

We first recall the definition of a lattice.

\begin{definition}
	A \emph{lattice}\index{lattice} is an algebra $\langle A; \lor, \land \rangle$ (arities $2$, $2$, $2$) with the following properties.
	\begin{enumerate}
		\item $x \lor y = y \lor x$.
		\item $x \land y = y \land x$.
		\item $x \lor (y \lor z) = (x \lor y) \lor z$.
		\item $x \land (y \land z) = (x \land y) \lor z$.
		\item $x \lor (x \land y) = x$.
		\item $x \land (x \lor y)=x$.
	\end{enumerate}
	A lattice is \emph{distributive}\index{lattice!distributive} if it satisfies any of the following equivalent conditions.
	\begin{enumerate}
		\item $x \lor (y\land z)=(x\lor y)\land (x\lor z)$.
		\item $x\land (y\lor z)=(x\land y)\lor (x\land z)$.
	\end{enumerate}
\end{definition}

A prototypical example of the algebras in this subsection is $\R$, endowed with the binary operations $+$ (addition), $\lor$ (maximum), $\land$ (minimum), and (for the algebras that require them) the constants $0$, $1$ and $-1$.

\begin{definition} \label{d:l-semigroup}
	A \emph{lattice-ordered semigroup} (\emph{$\ell$-semigroup}\index{lattice-ordered!semigroup}, for short) is an algebra $\langle M; +, \lor, \land\rangle$ (arities $2$, $2$, $2$) with the following properties.
	\begin{enumerate}
		\item  $\langle M; \lor, \land \rangle$ is a lattice.
		\item	$\langle M; +\rangle$ is a semigroup, i.e., $+$ is associative.
		\item \label{i:latt-hom} The operation $+$ distributes over $\lor$ and $\land$ on both sides:
			\begin{enumerate}
				\item $(x \lor y) + z = (x + z) \lor (y + z)$;
				\item $x + (y \lor z) = (x + y) \lor (x + z)$;
				\item $(x \land y) + z = (x + z) \land (y + z)$;
				\item $x + (y \land z) = (x + y) \land (x + z)$.
			\end{enumerate}
	\end{enumerate}
	We say that an $\ell$-semigroup is \emph{commutative}\index{lattice-ordered!semigroup!commutative} if the operation $+$ is commutative, and \emph{distributive}\index{lattice-ordered!semigroup!distributive} if the underlying lattice is distributive.
\end{definition}
Even if we have provided the definition in the general case, we will only be concerned with $\ell$-semigroups that are commutative and distributive.

\begin{remark} \label{r:semigroup-monoidal}
	\Cref{i:latt-hom} in \cref{d:l-semigroup} expresses the fact that $+$ is a lattice homomorphism in both coordinates.
	In fact, as pointed out by one of the referees, an $\ell$-semigroup is an internal semigroup in the monoidal category of lattices and lattice homomorphisms (where the monoidal operation is given by the tensor product).
	For the notion of monoidal category we refer to \cite[Chapter~VII]{MacLane1998}.
\end{remark}

In this manuscript, `$\ell$-'\index{$\ell$-|see{lattice-ordered}} will always be a shorthand for `lattice-ordered '.
\begin{examples}
	\begin{enumerate}[wide]
		\item The algebra $\langle \R; +, \max, \min \rangle$  is an example of a commutative distributive $\ell$-semigroup, as well as any of its subalgebras, such as $\Q$, $\Z$, $\N$, $\Np$, $\N \setminus \{0,1\}$, $\{0,-1, -2, -3, \dots\}$.
		\item Given a distributive lattice $\langle L; \lor, \land \rangle$, the algebras $\langle L; \lor, \lor, \land \rangle$ (i.e.\ $+ \df \lor$) and $\langle L; \land, \lor, \land\rangle$ (i.e.\ $+ \df \land$) are commutative distributive $\ell$-semigroups.
	\end{enumerate}
\end{examples}
\begin{lemma}
	\label{l:sum positive}
	In every $\ell$-semigroup the operation $+$ is order-preserving in both coordinates.
\end{lemma}
\begin{proof}
	Let $M$ be an $\ell$-semigroup.
	For every $y \in M$, the map
	\begin{align*}
			M & \longrightarrow M\\
			x & \longmapsto x + y
	\end{align*}
	is a lattice homomorphism by definition of $\ell$-semigroup.
	Therefore, it is order-preserving.
	Thus, $+$ is order preserving in the first coordinate.
	Analogously for the second coordinate.
\end{proof}
\begin{lemma}\label{l:land + lor=+}
	For all $x$ and $y$ in a commutative $\ell$-semigroup we have
	\[   (x \land y) + (x \lor y) = x + y.   \]
\end{lemma}
\begin{proof}
	We recall the proof, available in \cite[Section~2, p.\ 72]{Choud}, of the two inequalities:
	\begin{gather*}
		(x \land y) + (x \lor y) = ((x \land y) + x) \lor ((x \land y) + y) \leq (y + x) \lor (x + y) = x + y;\\
		(x \land y) + (x \lor y) = (x + (x \lor y)) \land (y + (x \lor y)) \geq (x + y) \land (y + x) = x + y. \qedhere
	\end{gather*}
\end{proof}
\begin{definition} \label{d:lm}
	A \emph{lattice-ordered monoid} (or \emph{$\ell$-monoid}, for short)\index{lattice-ordered!monoid} is an algebra $\langle M; +, \lor, \land, 0\rangle$ (arities $2,2,2,0$) with the following properties.
	\begin{enumerate}[label = N\arabic*.,  ref = N\arabic*, leftmargin=\widthof{LTE4'.} + \labelsep]
		\item \label[axiom]{ax:M1-gen} $\langle M; \lor, \land \rangle$ is a lattice.
		\item	\label[axiom]{ax:M2-gen} $\langle M; +, 0 \rangle$ is a monoid.
		\item \label[axiom]{ax:M3-gen} The operation $+$ distributes over $\lor$ and $\land$ on both sides.
	\end{enumerate}
	We say that an $\ell$-monoid is \emph{commutative}\index{lattice-ordered!monoid!commutative} if the operation $+$ is commutative, and \emph{distributive}\index{lattice-ordered!monoid!distributive} if the underlying lattice is distributive.
\end{definition}

Even if we have provided the definition in the general case, we will only be concerned with $\ell$-monoids that are commutative and distributive.

\begin{remark}
	An $\ell$-monoid is an internal monoid in the monoidal category of lattices and lattice homomorphisms (where the monoidal operation is given by the tensor product, and the unit is given by the one-element lattice).
\end{remark}

\begin{examples}
	\begin{enumerate}[wide]
		\item The set $\R$, with obviously defined operations, is {\alm}, as well as any of its subalgebras, such as $\Q$, $\Z$, $2\Z$, $\N$, $\{0,-1, -2, -3, \dots\}$.
		\item If $\langle L; \lor, \land \rangle$ is a distributive lattice with a bottom element $0$, then the algebra $\langle L; \lor, \lor, \land, 0 \rangle$ (i.e.\ we set $+ \df \lor$) is {\alm}.
		Similarly, if $L$ is a distributive lattice with a top element $0$, then $\langle L; \land, \lor, \land, 0 \rangle$ is {\alm}.
		\item For every topological space $X$ equipped with a preorder, the set of continuous order-preserving functions from $X$ to $\R$ with pointwise defined operations is a {\ulm}.
	\end{enumerate}
\end{examples}

\begin{remark}
	Since $\ell$-monoids are defined by equations, they are closed under products, subalgebras and homomorphic images.
	This allows one to obtain several examples.
\end{remark}

\begin{definition} \label{d:ulm}
	A \emph{unital lattice-ordered monoid} (\emph{unital $\ell$-monoid}, for short)\index{lattice-ordered!monoid!unital} is an algebra $\langle M; +, \lor, \land, 0, 1, -1 \rangle$ (arities $2,2,2,0,0,0$) with the following properties.
	\begin{enumerate}[label = M\arabic*.,  ref = M\arabic*,  start = 0, leftmargin=\widthof{LTE4'.} + \labelsep]
		\item \label[axiom]{ax:U0}
			$\langle M; +, \lor, \land, 0\rangle$ is an $\ell$-monoid.
		\item	\label[axiom]{ax:U1}
			$-1 + 1 = 0$ and $1 + -1 = 0$.
		\item \label[axiom]{ax:U2}
			$-1 \leq 0 \leq 1$.
		\item \label[axiom]{ax:U3}
			For all $x\in M$, there exists $n\in \Np$ such that 
			\[
				\underbrace{-1 + \dots + -1}_{n\ \text{times}}\leq x\leq \underbrace{1 + \dots + 1}_{n\ \text{times}}.
			\]
	\end{enumerate}
	A unital $\ell$-monoid is called \emph{commutative} if the operation $+$ is commutative, and \emph{distributive} if the underlying lattice is distributive\footnote{In \cite{AbbadiniEquiv}, we assumed distributivity of the lattice and called `unital commutative $\ell$-monoids' the algebras that here are referred to as `unital commutative \emph{distributive} $\ell$-monoids'.}.
\end{definition}
We will refer to the element $1$ as the \emph{positive unit}\index{unit!positive}, and to the element $-1$ as the \emph{negative unit}\index{unit!negative}.

In this manuscript we will restrict our attention to those unital $\ell$-monoids which are commutative and distributive.
We denote with $\ULM$ the category of {\ulms} with homomorphisms.

Given $n\in \N$, we write $nx$ for $\underbrace{x + \dots + x}_{n \text{ times}}$, and we write $n$ for $n1$ and $-n$ for $n(-1)$.
Furthermore, we use the shorthand $z - 1$ for $z + (-1)$.

\begin{examples} \label{ex:ulms}
	\begin{enumerate}[wide]
		\item The set $\R$, with obviously defined operations, is {\aulm}, as well as any of its subalgebras, such as $\Q$ and $\Z$.
			An example of a subalgebra of $\R$ which is not a group is, for any irrational element $s$ in $\R$, the algebra $\Z[s] =\{a + bs \mid a \in \Z, b \in \N \}$.
			
		\item For every topological space $X$ equipped with a preorder, the set of bounded continuous order-preserving functions from $X$ to $\R$ with pointwise defined operations is a {\ulm}.
		
		\item \label{i:lexico}For every totally ordered {\ulm} $M$ and every {\lm} $L$, we have {\aulm} $M \stackrel{\to}{\times} L$ defined as follows: the underlying set is $M \times L$, the order is lexicographic, i.e.\ 
			\[
				\big((x,y) \leq (z,w)\big) \Longleftrightarrow \big((x \leq z, x \neq z) \text{ or } (x = z, y \leq w)\big),
			\]
			the operation $+$ is defined componentwise, i.e.\ $(x,y) + (z,w) = (x + z, y + w)$, the positive unit is $(1,0)$ and the negative unit is $(-1,0)$.
	\end{enumerate}
\end{examples}

\begin{remark}
	All the axioms of {\ulms} are equations, except \cref{ax:U3}.
	However, notice that \cref{ax:U3} is preserved by subalgebras, homomorphic images and finite products.
	Thus, {\ulms} are closed under subalgebras, homomorphic images and finite products.
	This is not the case for arbitrary products: for example, $\R^\N$ does not satisfy \cref{ax:U3}.
\end{remark}

\begin{remark} \label{r:te}
	Given {\aulm} $\langle M; +, \lor, \land, 0, 1, -1 \rangle$, one defines the operation
	\[
		x \cdot y \df x - 1 + y.
	\]
	Note that this operation does not coincide on $\R$ with the usual multiplication.
	However, it might deserve to be denoted in this way because the equations $x \cdot 1 = x = 1 \cdot x$ hold\footnote{An additional reason for this notation is the fact that $\oplus$ is to $+$ what $\odot$ is to $\cdot$; this will be clear once the functor $\Gam$ will be defined.}.
	In fact, {\ulms} admit a term-equivalent description in the signature $\{+, \cdot, \lor, \land, 0, 1\}$. (From this signature, the constant $-1$ can be recast as $0 \cdot 0$.)
	In this signature, \cref{ax:U0,ax:U1} are equivalent to:
	\begin{enumerate}[label = S\arabic*.,  ref = S\arabic*,  start = 1, leftmargin=\widthof{LTE4'.} + \labelsep]
		\item \label[axiom]{ax:E1}$\langle M; \lor, \land \rangle$ is a lattice.		
		\item	\label[axiom]{ax:E2}
				$\langle M; +, 0 \rangle$ and $\langle M; \cdot, 1 \rangle$ are monoids.
		\item	\label[axiom]{ax:E3}
				Both the operations $+$ and $\cdot$ distribute over both $\lor$ and $\land$.
		\item \label[axiom]{ax:E4}
				$(x \cdot y) + z = x \cdot (y + z)$.
		\item \label[axiom]{ax:E5}
				$(x + y) \cdot z = x + (y \cdot z)$.
	\end{enumerate}
	(Note that \cref{ax:E5,ax:E4} are equivalent if $+$ and $\cdot$ are commutative.)
	The addition of \cref{ax:U2} equals the addition of the following axiom.
	\begin{enumerate}[label = S\arabic*.,  ref = S\arabic*,  start = 5, leftmargin=\widthof{LTE4'.} + \labelsep]
		\item \label[axiom]{ax:E6}
				$0 \leq 1$.
	\end{enumerate}
	The addition of \cref{ax:U3} equals the addition of the following axiom.
	\begin{enumerate}[label = S\arabic*.,  ref = S\arabic*,  start = 6, leftmargin=\widthof{LTE4'.} + \labelsep]
		\item \label[axiom]{ax:E7}
				For every $x\in M$ there exists $n\in \Np$ such that $\underbrace{0 \cdot \dots \cdot 0}_{n\ \text{times}} \leq x \leq \underbrace{1 + \dots + 1}_{n\ \text{times}}$.
	\end{enumerate}
	Moreover, the distributivity of the lattice in the old signature is equivalent to the distributivity of the lattice in the new signature, and the commutativity of $+$ in the old signature is equivalent to the commutativity of $+$ and $\cdot$ in the new signature.
	The class of algebras satisfying \cref{ax:E1,ax:E2,ax:E3,ax:E4,ax:A5,ax:E6,ax:E7}, distributivity of the underlying lattice, and commutativity of $+$ and $\cdot$, is term-equivalent to the class of {\ulms}.
	In our treatment, however, we shall stick to the signature and axioms of \cref{d:ulm}, and we will use the present remark only to explain the axioms of {\mvms} below.
\end{remark}


\subsection{MV-monoidal algebras} \label{subsec:MV}

The idea that we pursue is that {\aulm} is determined by its unit interval.
For a {\ulm} $M$, we set 
\[
	\Gam(M)\df \{x\in M\mid 0\leq x\leq 1 \}.
\]
On $\Gam(M)$ we define the constants $0$ and $1$ and the binary operations $\lor$ and $\land$ by restriction from $M$, and, for $x, y \in \Gam(M)$, we set 
\[
	x\oplus y\df (x+y)\land 1,
\]
and
\[
	x\odot y\df (x+y-1)\lor 0.
\]
By \cref{l:sum positive}, $\oplus$ and $\odot$ are internal operations on $\Gam(M)$: indeed, we have $x\oplus y\in \Gam(M)$ because $x+y\geq 0+0= 0$, and we have $x\odot y\in \Gam(M)$ because $x+y-1\leq 1+1-1=1$.

The intent behind \cref{d:MVM} below is to provide a finite equational axiomatisation of the algebras which are isomorphic to $\langle \Gam(M); \oplus, \odot, \lor, \land, 0, 1 \rangle$ for some {\ulm} $M$.
We will call the algebras satisfying this axiomatisation \emph{\mvms}.
The main result of this chapter is that the categories of {\ulms} and {\mvms} are equivalent, the equivalence being witnessed by the functor $\Gam$.

On $[0, 1]$, consider the elements $0$ and $1$ and the operations $x \lor y \df \max\{x, y\}$, $x \land y \df \min\{x, y\}$, $x \oplus y \df \min\{x + y, 1\}$, and $x \odot y \df \max\{x + y - 1, 0\}$.
This is a prime example of what we call {\amvm}.
\begin{definition} \label{d:MVM}
	An \emph{\mvm}\index{MV-monoidal algebra} is an algebra $\langle A; \oplus, \odot, \lor, \land, 0, 1 \rangle$ (arities $2$, $2$, $2$, $2$, $0$, $0$) satisfying the following equational axioms. 
	\begin{enumerate}[label = E\arabic*., ref = E\arabic*, leftmargin=\widthof{LTE4'.} + \labelsep]
		\item	\label[axiom]{ax:A1}
				$\langle A; \lor, \land \rangle$ is a distributive lattice.		
		\item	\label[axiom]{ax:A2}
				$\langle A; \oplus, 0 \rangle$ and $\langle A; \odot, 1 \rangle$ are commutative monoids.
		\item	\label[axiom]{ax:A3}
				Both the operations $\oplus$ and $\odot$ distribute over both $\lor$ and $\land$.		
		\item	\label[axiom]{ax:A4}
				$(x \oplus y) \odot ((x \odot y) \oplus z) = (x \odot (y \oplus z)) \oplus (y \odot z)$. 
		\item \label[axiom]{ax:A5}
				$(x \odot y) \oplus ((x \oplus y) \odot z) = (x \oplus (y \odot z)) \odot (y \oplus z)$.
		\item	\label[axiom]{ax:A6}
				$(x \odot y) \oplus z = ((x \oplus y) \odot ((x \odot y) \oplus z)) \lor z$.
		\item	\label[axiom]{ax:A7}
				$(x \oplus y) \odot z = ((x \odot y) \oplus ((x \oplus y) \odot z)) \land z$.
	\end{enumerate}
\end{definition}
\xmakefirstuc{\mvs} are unit intervals of {\ulgs};
the name `{\mvm}' suggests that these algebras play the role of {\mvs} when `group' is replaced by `monoid'.
In fact, as we will prove, {\mvms} are unit intervals of {\ulms}.

We let $\MVM$ denote the category of {\mvms} with homomorphisms.

Notice that \cref{ax:A1,ax:A2,ax:A3} coincide with \cref{ax:E1,ax:E2,ax:E3} in \cref{r:te}, together with the distributivity of the underlying lattice and the commutativity of the monoidal operations.

\Cref{ax:A4} is a sort of associativity, which resembles \cref{ax:E4}, i.e.\ $(x \cdot y) + z = x \cdot (y + z)$.
In particular, one verifies that the interpretation on $[0,1]$ of both the left-hand and right-hand side of \cref{ax:A4} equals
\begin{equation} \label{e:xyz}
	((x + y + z - 1) \lor 0) \land 1.
\end{equation}
Notice that, using the definition of $\cdot$ from \cref{r:te}, the element $x + y + z - 1$ appearing in \eqref{e:xyz} coincides with the interpretation on $\R$ of $(x \cdot y) + z$ and $x \cdot (y + z)$.
In fact, in our view, \cref{ax:A4} is essentially the condition $(x \cdot y) + z = x \cdot (y + z)$ expressed at the unital level, i.e.: 
\begin{equation} \label{e:xyz-01}
	(((x \cdot y) + z) \lor 0) \land 1 = ((x \cdot (y + z)) \lor 0) \land 1.
\end{equation}
In fact, the term $x \cdot y$ in the left-hand side of \eqref{e:xyz-01} corresponds to the terms $x \oplus y$ and $x \odot y$ in the left-hand side of \cref{ax:A4}, and the term $y + z$ in the right-hand side of \eqref{e:xyz-01} corresponds to the terms $y \oplus z$ and $y \odot z$ in the right-hand side of \cref{ax:A4}.

Analogously, \cref{ax:A5} corresponds to \cref{ax:E5}, i.e.\ $(x + y) \cdot z = x + (y \cdot z)$.
Notice that \cref{ax:A4,ax:A5} are equivalent, given the commutativity of $\oplus$ and $\odot$; we have included both so to make it clear that, if $\langle A; \oplus, \odot, \lor, \land, 0, 1 \rangle$ is {\amvm}, then also its order-dual $\langle A; \odot, \oplus, \land, \lor, 1, 0 \rangle$ is {\amvm}.

\Cref{ax:A6} expresses how the term $(x \odot y) \oplus z$ differs from its non-truncated version $(x \cdot y) + z$: essentially, the axiom can be read as 
\[
	(x \odot y) \oplus z = ((x \cdot y) + z) \lor z.
\]
Analogously, \cref{ax:A7} can be read as
\[
	(x \oplus y) \odot z = ((x + y) \cdot z) \land z.
\]

\begin{examples}
	\begin{enumerate}[wide]
		\item The unit interval $[0,1]$ is {\amvm}.
		\item Every bounded distributive lattice $L$ can be made into {\amvm} by setting $x\oplus y\df x\lor y$, and $x\odot y\df x\land y$.
		In fact, the category of bounded distributive lattices is a subvariety of the variety of {\mvms}, obtained by adding the axioms $\forall x, y\ (x \oplus y= x \lor y)$ and $\forall x, y\ (x \odot y= x \land y)$.
		\item For every topological space $X$ equipped with a preorder, the set of continuous order-preserving functions from $X$ to $[0,1]$ with pointwise defined operations is {\amvm}.
	\end{enumerate}
\end{examples}

We remark that $\MVM$ is a variety of algebras whose primitive operations are finitely many and of finite arity, axiomatised by a finite number of equations.


\subsubsection{Basic properties of {\mvms}}

If $\langle A; \oplus, \odot, \lor, \land, 0, 1 \rangle$ is {\amvm}, then also its so-called \emph{dual algebra}\index{dual algebra}\index{algebra!dual|see{dual algebra}}\footnote{Not to be confused with the notion of categorical dual.} $\langle A; \odot, \oplus, \land, \lor, 1, 0 \rangle$---obtained by interchanging the roles of $\odot$ and $\oplus$, the roles of $\land$ and $\lor$ and the roles of $1$ and $0$---is {\amvm}.
We will use this observation to shorten some proofs.

We give a name to the right- and left-hand terms of \cref{ax:A4,ax:A5}; we will then prove that their interpretations in {\amvm} coincide.

\begin{notation}
	We set
	\begin{align*}
		\sigma_1(x,y,z) \df (x \oplus y) \odot ((x \odot y) \oplus z);\\
		\sigma_2(x,y,z) \df (x \odot y) \oplus ((x \oplus y) \odot z);\\
		\sigma_3(x,y,z) \df (x \odot (y \oplus z)) \oplus (y \odot z);\\
		\sigma_4(x,y,z) \df (x \oplus (y \odot z)) \odot (y \oplus z).
	\end{align*}
\end{notation}

In $[0,1]$, the interpretation of any of these terms is
\[
	((x + y + z - 1) \lor 0) \land 1.
\]

Note that \cref{ax:A5,ax:A5,ax:A6,ax:A7} can be written, respectively, as
\begin{align*}
	\sigma_1(x,y,z) 		& = \sigma_3(x,y,z),\\
	\sigma_2(x,y,z) 		& = \sigma_4(x,y,z),\\
	(x \odot y) \oplus z	& = \sigma_1(x, y, z) \lor z,\\
	(x \oplus y) \odot z	& = \sigma_2(x, y, z) \land z.
\end{align*}

\begin{lemma} \label{l:permutations}
	The terms $\sigma_1$, $\sigma_2$, $\sigma_3$, $\sigma_4$ in the theory of {\mvms} are all invariant under any permutation of the variables, and they all coincide.
	In other words, for every {\mvm} $A$, for all $i,j \in \{1,2,3,4\}$, for all permutations $\tau,\rho \colon \{1,2,3\} \to \{1,2,3\}$ and for all $x_1, x_2,x_3 \in A$ we have
	\[
		\sigma_i(x_{\tau(1)},x_{\tau(2)},x_{\tau(3)}) = \sigma_j(x_{\rho(1)},x_{\rho(2)},x_{\rho(3)}).
	\]
\end{lemma}
\begin{proof}
	In the theory of {\mvms}, by commutativity of $\oplus$ and $\odot$, the term $\sigma_1$ is invariant under transposition of the first and the second variable, and the term $\sigma_3$ is invariant under transposition of the second and the third variable.
	Moreover, by \cref{ax:A4}, we have $\sigma_1(x,y,z) = \sigma_3(x,y,z)$.
	Since any two distinct transpositions in the symmetric group on three elements generate the whole symmetric group, it follows that $\sigma_1$ and $\sigma_3$ are invariant under every permutation of the variables.
	By commutativity of $\oplus$ and $\odot$, we have $\sigma_1(x,y,z) = \sigma_4(z, y, x)$.
	Therefore, also $\sigma_4$ is invariant under every permutation, and $\sigma_1(x,y,z) = \sigma_4(x, y, z)$.
	By \cref{ax:A5}, we have $\sigma_2(x,y,z) = \sigma_4(x,y,z)$, and we conclude that $\sigma_1$, $\sigma_2$, $\sigma_3$, $\sigma_4$ are invariant under any permutation of the variables, and coinciding.
\end{proof}

In particular, \cref{l:permutations} guarantees that, for all $x,y,z$ in {\amvm}, we have 
\[
	\sigma_1(x,y,z) = \sigma_2(x,y,z) = \sigma_3(x,y,z) = \sigma_4(x,y,z).
\]

\begin{notation}
	For $x, y, z$ in {\amvm}, we let $\sigma(x,y,z)$ denote the common value of $\sigma_1(x,y,z)$, $\sigma_2(x,y,z)$, $\sigma_3(x,y,z)$ and $\sigma_4(x,y,z)$.
\end{notation}

We recall the interpretation of $\sigma(x,y,z)$ in $[0,1]$:
\[
	\sigma(x,y,z) = ((x + y + z - 1) \lor 0) \land 1.
\]
Loosely speaking, $\sigma(x, y, z)$ is the second layer of the sum of $x$, $y$, and $z$.
In fact, the symbol $\sigma$ should be evocative of the word `\emph{s}um'.

\begin{lemma}\label{l:bounded-lattice}
	For every element $x$ of {\amvm} we have $0 \leq x \leq 1$.
\end{lemma}
\begin{proof}
	We have
	\begin{align*}
		x	
			& = (x \odot 1) \oplus 0
			&& \by{\cref{ax:A2}}\\
			& = \sigma_1(x, 1, 0) \lor 0
			&& \by{\cref{ax:A6}}\\
			& = \sigma_3(x, 1, 0) \lor 0
			&& \by{\cref{ax:A4}}\\
			& = ((x \odot (1 \oplus 0)) \oplus (1 \odot 0)) \lor 0
			&& \by{def.\ of $\sigma_3$}\\
			& = ((x \odot 1) \oplus 0) \lor 0
			&& \by{\cref{ax:A2}}\\
			& = x \lor 0.
			&& \by{\cref{ax:A2}}
	\end{align*}
	Thus $0 \leq x$.
	Dually, $x \leq 1$.
\end{proof}

\begin{lemma} \label{l:absorbing}
	For every element $x$ in {\amvm}, we have $x \oplus 1 = 1$ and $x \odot 0 = 0$.
\end{lemma}
\begin{proof}
	We have
	\begin{align*}
		0	& = 1 \odot 0								&& \text{(\cref{ax:A2})}\\
			& = (1 \lor x) \odot 0					&& \text{(\cref{l:bounded-lattice})}\\
			& = (1 \odot 0) \lor (x \odot 0)		&& \text{(\cref{ax:A3})}\\
			& = 0 \lor (x \odot 0)					&& \text{(\cref{ax:A2})}\\
			& = x \odot 0.								&& \text{(\cref{l:bounded-lattice})}
	\end{align*}
	Dually, $x \oplus 1 = 1$.
\end{proof}

\begin{lemma} \label{l:order-preserving properties}
	In every {\mvm}, the operations $\oplus$ and $\odot$ are order-preserving in both coordinates.
\end{lemma}
\begin{proof}
	For every {\mvm} $\langle A; \oplus, \odot, \lor, \land, 0, 1 \rangle$, the algebras $\langle A; \oplus, \lor, \land \rangle$ and $\langle A; \odot, \lor, \land \rangle$ are $\ell$-semigroups.
	Therefore, by \cref{l:sum positive}, the operations $\oplus$ and $\odot$ are order-preserving in both coordinates.
\end{proof}

\begin{lemma} \label{l:increasing}
	For all $x$ and $y$ in {\amvm} we have $x \leq x \oplus y$ and $x \geq x \odot y$.
\end{lemma}
\begin{proof}
	By \cref{l:bounded-lattice}, we have $0\leq y$.
	Thus, by \cref{l:order-preserving properties}, we have $x = x \oplus 0  \leq  x \oplus y$.
	Dually, $x \geq x \odot y$.
\end{proof}
\begin{lemma}
	\label{l:almost associative}
	For all $x,y,z$ in {\amvm} we have
	\[
		x \odot (y \oplus z) \leq (x \odot y) \oplus z.
	\]
\end{lemma}
\begin{proof}
	Using \cref{ax:A4,ax:A5,ax:A6,ax:A7}, we obtain
	\[
		x \odot (y \oplus z) = x \land \sigma(x, y, z) \leq \sigma(x, y, z) \leq \sigma(x, y, z) \lor z = (x \odot y) \oplus z. \qedhere
	\]
\end{proof}


\section{The unit interval functor \texorpdfstring{$\Gamma$}{\textGamma}}
	\label{s:Gamma}

In this section we define a functor $\Gam\colon\ULM\to\MVM$.
The main goal of the chapter is to show that $\Gam$ is an equivalence.
Recall from the beginning of \cref{subsec:MV} that, for a {\ulm} $M$, we set 
\[
	\Gam(M)\df \{x\in M\mid 0\leq x\leq 1 \},
\]
and we define on $\Gam(M)$ the operations $0$, $1$, $\lor$, $\land$ by restriction, $x \oplus y\df (x+y)\land 1$, and $x\odot y\df (x+y-1)\lor 0$.
Our next goal---met in \cref{t:Gamma is MVM} below---is to show that $\Gam(M)$ is {\amvm}.
We need some lemmas.

\begin{lemma}\label{l:plusdot in mon}
	Let $M$ be a {\ulm}. For all $x,y,z\in \Gam(M)$ we have
	\begin{align*}
		(x \odot y) \oplus z	& = ((x + y + z - 1) \lor z) \land 1,\\
	\intertext{and}
		(x \oplus y) \odot z	& = ((x + y + z - 1) \land z) \lor 0.
	\end{align*}
\end{lemma}

\begin{proof}
	We have
	\begin{align*}
		(x \odot y) \oplus z	& = ((x \odot y) + z) \land 1
									&& \by{def.\ of $\oplus$}\\
									& = (((x + y - 1) \lor 0) + z) \land 1
									&& \by{def.\ of $\odot$}\\
									& = ((x + y + z - 1) \lor z) \land 1
									&& \by{$+$ distr.\ over $\land$}
	\intertext{and}
		(x \oplus y) \odot z	& = ((x \oplus y) + z - 1) \lor 0
									&& \by{def.\ of $\odot$}\\
									& = (((x + y) \land 1) + z - 1) \lor 0
									&& \by{def.\ of $\oplus$}\\
									& = ((x + y + z - 1) \land z) \lor 0.
									&& \by{$+$ distr.\ over $\land$} \qedhere
	\end{align*}
\end{proof}
The following establishes \cref{ax:A1,ax:A2,ax:A3} for $\Gam(M)$.
\begin{lemma}
	\label{l:A1-A4}
	For any {\ulm} $M$ the following properties hold.
	\begin{enumerate}
		\item	$\langle \Gam(M);\lor,\land \rangle$ is a distributive  lattice.		
		\item $\langle \Gam(M);\oplus,0\rangle$ and $\langle \Gam(M);\odot,1\rangle$ are commutative monoids.
		\item In $\Gam(M)$ the operations $\oplus$ and $\odot$ distribute over $\lor$ and $\land$.	
	\end{enumerate}
\end{lemma}
\begin{proof}
	 \begin{enumerate}[wide]
	 
	 	\item
	 			This follows from the fact that $\langle M; \lor, \land \rangle$ is a distributive lattice.
	 			
	 	\item
	 			We show that $\langle \Gam(M);\oplus,0\rangle$ is a commutative monoid.
			 	\begin{enumerate}
			 	
			 		\item 
			 				We have
							\begin{equation*}
					 			\begin{split}
					 				(x \oplus y) \oplus z	& = (((x + y) \land 1) + z) \land 1 \\
							 										& = (x + y + z) \land (1 + z) \land 1 \\
							 										& = (x + y + z) \land 1 \\
							 										& = (x + y + z) \land (x + 1) \land 1 \\
							 										& = (x + ((y + z) \land 1)) \land 1 \\
							 										& = x \oplus (y \oplus z).
					 			\end{split}
					 		\end{equation*}

			 		\item 
			 				We have $x \oplus y = (x + y) \land 1 = (y + x) \land 1 = y \oplus x$.
			 		
			 		\item
			 				We have $x \oplus 0 = (x + 0) \land 1 = x \land 1 = x$.
			 	
			 	\end{enumerate}
				We show that $\langle \Gam(M);\odot,1\rangle$ is a commutative monoid.
				\begin{enumerate}
					\item 
			 				We have
							\begin{equation*}
					 			\begin{split}
					 				(x \odot y) \odot z		& = (((x + y - 1) \lor 0) + z - 1) \lor 0 \\
							 										& = (x + y + z - 2) \lor (z - 1) \lor 0 \\
							 										& = (x + y + z - 2) \lor 0  \\
							 										& = (x + y + z - 2) \lor (x - 1) \lor 0 \\
							 										& = (x + ((y + z - 1) \lor 0) - 1) \lor 0 \\
							 										& = x \odot (y \odot z).
					 			\end{split}
					 		\end{equation*}
					\item 
						We have $ x \odot y  =  (x + y - 1) \lor 0  =  (y + x - 1) \lor 0  =  y \odot x $.
						
					\item 
						We have $x \odot 1 = (x + 1 - 1) \lor 0 - 1 = x \lor 0 = x$.
				\end{enumerate}

	 	\item 
	 			We show that $\oplus$ distributes over $\lor$ and $\land$: we have
		 		\begin{align*}
	 				(x \lor y) \oplus z	& = ((x \lor y) + z) \land 1 \\
			 									& = ((x + z) \lor (y + z)) \land 1 \\
			 									& = ((x + z) \land 1) \lor ((y + z) \land 1) \\
			 									& = (x \oplus z) \lor (y \oplus z).
				\intertext{and}
	 				(x \land y) \oplus z	& = ((x \land y) + z) \land 1 \\
		 										& = ((x + z) \land (y + z)) \land 1 \\
		 										& = ((x + z) \land 1) \land ((y + z) \land 1) \\
		 										& = (x \oplus z) \land (y \oplus z).
				\end{align*}
	 			We show that $\odot$ distributes over $\lor$ and $\land$: we have
	 			\begin{align*}
	 				(x \lor y) \odot z	& = ((x \lor y) + z - 1) \lor 0\\
			 									& = ((x + z - 1) \lor (y + z - 1)) \lor 0 \\
			 									& = ((x + z - 1) \lor 0) \lor (((y + z - 1) \lor 0)) \\
			 									& = (x \odot z) \lor (y \odot z).
	 			\intertext{and}
	 				(x \land y) \odot z	& = ((x \land y) + z - 1) \lor 0\\
			 									& = ((x + z - 1) \land (y + z - 1)) \lor 0 \\
			 									& = ((x + z - 1) \lor 0) \land (((y + z - 1) \lor 0)) \\
			 									& = (x \odot z) \land (y \odot z). \qedhere
		 		\end{align*}
	\end{enumerate}
\end{proof}
\begin{lemma}
	\label{l:oplus odot is good}
	Let $M$ be {\aulm}. For all $x, y \in \Gam(M)$ we have
	\[
		(x \oplus y) + (x \odot y) = x + y.
	\]
\end{lemma}
\begin{proof}
	We have
	\begin{align*}
		(x \oplus y) + (x \odot y)	& = ((x + y) \land 1) + ((x + y - 1) \lor 0)	&& \by{def.\ of $\oplus$ and $\odot$}\\
											& = ((x + y) \land 1) + ((x + y) \lor 1) - 1 && \by{$+$ distr.\ over $\lor$}\\
											& = x + y + 1 - 1										&& \by{\cref{l:land + lor=+}}\\
											& = x + y.												&&
		\qedhere
	\end{align*}
\end{proof}
\begin{lemma} \label{l:sigma in monoid}
	Let $M$ be a {\ulm}.
	For all $x,y,z\in \Gam(M)$, the elements 
	\[  (x \oplus y) \odot ((x \odot y) \oplus z),  \]
	\[  (x \odot y) \oplus ((x \oplus y) \odot z),  \]
	\[  (x \odot (y \oplus z)) \oplus (y \odot z),  \]
	\[  (x \oplus (y \odot z)) \odot (y \oplus z)   \]
	coincide with 
	\[  (x + y + z - 1) \lor 0) \land 1.  \]
\end{lemma}

\begin{proof}
	We have
	\begin{align*}
		& (x \oplus y) \odot ((x \odot y) \oplus z)								&& \\
		& = (((x \oplus y) + (x \odot y) + z) \land (x \oplus y)) \lor 0 	&& \by{\cref{l:plusdot in mon}} \\
		& = ((x + y + z - 1) \land (x \oplus y)) \lor 0							&& \by{\cref{l:oplus odot is good}} \\
		& = ((x + y + z - 1) \land (x + y) \land 1) \lor 0						&& \by{def.\ of $\oplus$}\\
		& = ((x + y + z - 1) \land 1) \lor 0										&& \by{$x + y \leq x + y + z - 1$}\\
		& = ((x + y + z - 1) \lor 0) \land 1.										&&				
	\intertext{and}
		& (x \odot (y \oplus z)) \oplus (y \odot z)								&& \\
		& = ((x + (y \oplus z) + (y \odot z)) \lor (y \odot z)) \land 1	&& \by{\cref{l:plusdot in mon}}\\
		& = ((x + y + z - 1) \lor (y \odot z)) \land 1							&& \by{\cref{l:oplus odot is good}}\\
		& = ((x + y + z - 1) \lor (y + z - 1) \lor 0) \land 1					&& \by{def.\ of $\odot$}\\
		& = ((x + y + z - 1) \lor 0) \land 1.										&& \by{$y + z - 1 \leq x + y + z - 1$}
	\end{align*}
	The fact that also $(x \odot y) \oplus ((x \oplus y) \odot z)$ and $(x \oplus (y \odot z)) \odot (y \oplus z)$ coincide with $(x + y + z - 1) \lor 0) \land 1$ follows from the commutativity of $\oplus$ and $\odot$ (which is easily seen to hold) and the commutativity of $+$.
\end{proof}

\begin{theorem}
	\label{t:Gamma is MVM}
	If $M$ is a {\ulm}, then $\Gam(M)$ is {\amvm}.
\end{theorem}
\begin{proof}
	\Cref{ax:A1,ax:A2,ax:A3} in \cref{d:MVM} hold by \cref{l:A1-A4}.
	\Cref{ax:A4,ax:A5} hold by \cref{l:sigma in monoid}.
	\Cref{ax:A6,ax:A7} hold by \cref{l:plusdot in mon} and \cref{l:sigma in monoid}.
\end{proof}
Given a morphism of {\ulms} $f\colon M\to N$, we denote with $\Gam(f)$ its restriction $\Gam(f)\colon \Gam(M)\to \Gam(N)$.
This establishes a functor
\[\Gam\colon \ULM\to \MVM.\]
The main goal of this chapter is to show that $\Gam$ is an equivalence of categories.
%


\section{Good \texorpdfstring{$\mathbb{Z}$}{\unichar{"2124}}-sequences: definition and basic properties} \label{s:good.def}

The idea to establish a quasi-inverse of the functor $\Gam \colon \ULM \to \MVM$ is that every element $x$ of {\aulm} is uniquely determined by the function
\begin{align*}
	\zeta_M(x)	\colon 	\Z	& \longrightarrow 	\Gam(M)\\
								n	& \longmapsto			((x-n) \lor 0) \land 1.
\end{align*}
The intent behind the definition of good $\Z$-sequence (\cref{d:good-Z} below) is to describe the properties of the function $\zeta_M(x)$.
Indeed, in \cref{t:exists unique good sequence}, we will prove that {\aulm} $M$ is in bijection with the set of good $\Z$-sequences in $\Gam(M)$.

\begin{definition}
	Let $A$ be {\amvm}.
	A \emph{good pair}\index{good!pair}\index{pair!good} in {\amvm} $A$ is a pair $(x_0,x_1)$ of elements of $A$ such that $x_0 \oplus x_1 = x_0$ and $x_0 \odot x_1 = x_1$.
\end{definition}
A function from $\Z$ to a set $A$ will be called a \emph{$\Z$-sequence} in A\index{$\mathbb{Z}$-sequence}.
\begin{definition} \label{d:good-Z}
	Let $A$ be {\amvm}.
	A \emph{good $\Z$-sequence}\index{good!$\mathbb{Z}$-sequence}\index{$\mathbb{Z}$-sequence!good} in {\amvm} $A$ is a $\Z$-sequence $\gs{x}$ in $A$ which satisfies the following conditions.
	\begin{enumerate}
		\item The value $\gs{x}(k)$ is eventually $1$ for $k \to - \infty$, i.e., there exists $n \in \Z$ such that, for every $k < n$, we have $\gs{x}(k)=1$.
		\item The value $\gs{x}(k)$ is eventually $0$ for $k \to + \infty$, i.e., there exists $m \in \Z$ such that, for every $k\geq m$, we have $\gs{x}(k)=0$.
		\item For every $k\in \Z$, the pair $(\gs{x}(k),\gs{x}(k+1))$ is good.
	\end{enumerate}
\end{definition}
\begin{notation}
	We will write, more concisely, $(x_0,x_1,x_2,\dots)$ for the $\Z$-sequence
	\begin{align*}
		\Z	& \longrightarrow	A \\
		k	& \longmapsto		{
										\begin{cases}
											1		&	\text{if } k < 0; \\
											x_k	&	\text{if } k \geq 0. \\
										\end{cases}
									}
	\end{align*}
	Moreover, instead of $(x_0,\dots,x_n,0,0,0,\dots)$ we will write, more concisely, $(x_0,\dots,x_n)$. 
	In particular, for each $x\in A$, $(x)$ denotes the good $\Z$-sequence
	\begin{align*}
		\Z	& \longrightarrow	A	\\
		k	& \longmapsto		{
										\begin{cases}
											1	&	\text{if } k < 0;	\\
											x	&	\text{if } k = 0;	\\
											0	&	\text{if } k > 0.	\\
										\end{cases}
									}
	\end{align*}
\end{notation}
The use of sequences indexed by $\Z$ instead of $\N$ is not new: in \cite[Section 2.1]{BalEtal2002} it was used for the equivalence between {\ulgs} and MV-algebras.
\begin{remark}
	In our definition of good pair we included both the conditions $x_0\oplus x_1=x_0$ and $x_0\odot x_1=x_1$ because, in general, they are not equivalent.
	For example, let $A = \Gam(\Z\stackrel{\to}{\times}\{0,1\})$, where $\Z\stackrel{\to}{\times}\{0,1\}$ is the lexicographic product of the two {\lms} $\Z$ (with addition) and $\{0,1\}$ (with $+ \df \lor$), with $(1,0)$ as positive unit and $(-1,0)$ as negative unit (see \cref{i:lexico} in \cref{ex:ulms}).
	We have $(0,1) \oplus (0,1) = (0,1)$ and $(0,1) \odot (0,1) = (0,0) \neq (0,1)$.
\end{remark}
%


\subsubsection{Basic properties of good \texorpdfstring{$\mathbb{Z}$}{\unichar{"2124}}-sequences}

\begin{remark}
	Given two elements $x_0$ and $x_1$ in {\amvm} $\alge{A}$, $(x_0, x_1)$ is a good pair in $\alge{A}$ if, and only if $(x_1, x_0)$ is a good pair in the dual of $\alge{A}$.
\end{remark}

\begin{lemma}\label{l:switch of plus and dot if good}
	Let $A$ be an {\mvm}, let $(x_0,x_1)$ be a good pair in $A$, and let $y\in A$.
	Then
	\[
		x_0 \odot (x_1 \oplus y)  =  x_1 \oplus (x_0 \odot y),
	\]
	and the terms on both sides coincide with $\sigma(x_0,x_1,y)$.
\end{lemma}

\begin{proof}
	We have
	\begin{align*}
		x_0 \odot (x_1 \oplus y)
			& = (x_0 \oplus x_1) \odot ((x_0 \odot x_1) \oplus y)
			&& \by{$(x_0,x_1)$ is good}\\
			& = \sigma_1(x_0, x_1, y)
			&& \by{def.\ of $\sigma_1$}\\
			& = \sigma_2(x_0, x_1, y)
			&& \by{\cref{l:permutations}}\\
			& = (x_0 \odot x_1) \oplus ((x_0 \oplus x_1) \odot y)
			&& \by{def.\ of $\sigma_2$}\\
			& = x_1 \oplus (x_0 \odot y).
			&& \by{$(x_0,x_1)$ is good} \qedhere
	\end{align*}
\end{proof}

\begin{lemma}
	\label{l:oplus odot is a good pair}
	Let $A$ be {\amvm}, and let $x,y\in A$.
	Then $(x\oplus y,x\odot y)$ is a good pair.
\end{lemma}
\begin{proof}
	We have
	\begin{align*}
		(x \oplus y) \oplus (x \odot y)	& = (1 \odot (x \oplus y)) \oplus (x \odot y)	&& \by{\cref{ax:A2}}\\
													& = \sigma_3(1,x,y)										&& \by{def.\ of $\sigma_3$}\\
													& = \sigma_4(1,x,y)										&& \by{\cref{l:permutations}}\\
													& = (1 \oplus (x \odot y)) \odot (x \oplus y)	&& \by{def.\ of $\sigma_4$}\\
													& = 1 \odot (x \oplus y)								&& \by{\cref{l:absorbing}}\\
													& = x \oplus y.											&& \by{\cref{ax:A2}}
	\end{align*}
	Dually,
	$(x\oplus y)\odot (x\odot y)=x\odot y$.
\end{proof}

\begin{lemma}
	\label{l:bipartite}
	Let $n,m\in \N$, let $x_0,\dots, x_n, y_0, \dots,y_m$ be elements of {\amvm} and suppose that, for every $i\in \{0, \dots,n\}$, and every $j\in \{0, \dots, m\}$, the pair $(x_i,y_j)$ is good.
	Then,
	\[   (x_0\odot \dots \odot x_n, y_0\oplus \dots \oplus y_m)   \]
	is a good pair.
\end{lemma}
\begin{proof}
	The statement is trivial for $(n,m)=(0,0)$.
	The statement is true for $(n,m)=(1,0)$ because
	\[
		(x_0\odot x_1) \oplus y_0\stackrel{\text{\cref{l:switch of plus and dot if good}}}{=} (x_0 \oplus y_0) \odot x_1 = x_0 \odot x_1,
	\]
	and
	\[
		(x_0\odot x_1)\odot y_0=x_0\odot (x_1\odot y_0)=x_0\odot y_0=y_0.
	\]
	The case $(n,m)=(0,1)$ is analogous.
	Let $(n,m)\in \N\times \N\setminus \{(0,0), (0,1),(1,0) \}$, and suppose that the statement is true for each $(h,k)\in \N\times \N$ such that $(h,k)\neq (n,m)$, $h\leq n$ and $k\leq m$.
	We prove that the statement holds for $(n,m)$.
	At least one of the two conditions $n\neq 0$ and $m\neq 0$ holds.
	Suppose, for example, $n\neq 0$.
	Then, by inductive hypothesis, the pairs $(x_0\odot \dots \odot x_{n-1}, y_0\oplus \dots \oplus y_m)$ and $(x_{n}, y_0\oplus \dots \oplus y_m)$ are good.
	Now we apply the statement for the case $(1,0)$, and we obtain that $(x_0\odot \dots \odot x_n, y_0\oplus \dots \oplus y_m)$ is a good pair.
	The case $m\neq 0$ is analogous.
\end{proof}
%


\section{Subdirectly irreducible {\mvms}}

In this section we prove that, for every subdirectly irreducible {\mvm} $A$, we have: (i) $A$ is totally ordered and (ii) for every good pair $(x_0, x_1)$ in $A$ we have either $x_0 = 1$ or $x_1 = 0$.
These two conditions are of interest because, as one may prove, for {\amvm} $A$, the conjunction of (i) and (ii) is equivalent to the enveloping {\ulm} of $A$ being totally ordered.


\subsection{Subdirectly irreducible algebras are totally ordered}
	\label{s:sub-irr_tot-ord}

In this subsection we prove that every subdirectly irreducible {\mvm} is totally ordered (\cref{t:sub irr is totally ordered}).
Our source of inspiration is \cite[Section~1]{Repn}, in analogy to which we proceed.

Given {\amvm} $A$ and a lattice congruence $\theta$ on $A$ such that $\lvert A/\theta\rvert=2$, we set
\[   \theta^* \df \{(a,b) \in A \times A \mid \forall x \in A\ (a \oplus x, b \oplus x) \in \theta\text{ and } (a \odot x, b \odot x) \in \theta \},   \]
and we let $\0(\theta)$ and $\1(\theta)$ denote the classes of the lattice congruence $\theta$ corresponding to the smallest and greatest element of the lattice $A/\theta$, respectively.
\begin{lemma}
	\label{l:char-theta*}
	Let $A$ be {\amvm}, and let $\theta$ be a lattice congruence on $A$ such that $\lvert A/\theta\rvert=2$.
	Then, $\theta^*$ is the greatest $\{\oplus, \odot, \lor, \land, 0, 1\}$-congruence contained in $\theta$.
\end{lemma}
\begin{proof}
	We have $\theta^*\seq \theta$ because, for every $(a,b)\in \theta^*$, we have $(a,b)=(a\oplus 0,b\oplus 0)\in \theta$. 
	\begin{claim}
		The relation $\theta^*$ contains every $\{\oplus, \odot, \lor, \land, 0, 1\}$-congruence contained in $\theta$.
	\end{claim}
	\begin{claimproof}
		Let $\rho$ be a $\{\oplus, \odot, \lor, \land, 0, 1\}$-congruence contained in $\theta$.
		Let $(a,b)\in \rho$, and let $x\in A$.
		Since $(x,x)\in \rho$, and $\rho$ is a congruence, we have $(a\oplus x,b\oplus x)\in \rho\seq \theta$, and $(a\odot x,b\odot x)\in \rho\seq \theta$.
		Thus, $(a,b)\in \theta^*$.
	\end{claimproof}
	\begin{claim}
		The relation $\theta^*$ is a $\{\oplus, \odot, \lor, \land, 0, 1\}$-congruence.
	\end{claim}
	\begin{claimproof}
		The relation $\theta^*$ is an equivalence relation because $\theta$ is so.
		In the following, let $a,a',b,b'\in A$, and suppose $(a,a')\in \theta^*$ and $(b,b')\in \theta^*$: for all $x\in A$, we have $(a\oplus x,a'\oplus x)\in \theta$, $(a\odot x,a'\odot x)\in \theta$, $(b\oplus x,b'\oplus x)\in \theta$, and $(b\odot x,b'\odot x)\in \theta$.
			
		Let us prove $(a\lor b, a'\lor b')\in \theta^*$.
		Let $x\in A$.
		Since $(a\oplus x,a'\oplus x)\in \theta$, $(b\oplus x,b'\oplus x)\in \theta$, and $\theta$ is a lattice congruence, we have
		\[
			\big((a\oplus x)\lor (b\oplus x),(a'\oplus x)\lor (b'\oplus x)\big)\in \theta,
		\]
		i.e., 
		\[
			\big((a\lor b)\oplus x, (a'\lor b')\oplus x\big)\in \theta.
		\]
		Analogously, 
		\[
			\big((a\lor b)\odot x, (a'\lor b')\odot x\big)\in \theta.
		\]
		This proves $(a\lor b, a'\lor b')\in \theta^*$.	Analogously, $(a\land b, a'\land b')\in \theta^*$.
		
		Let us prove $(a\oplus b, a'\oplus b')\in \theta^*$.
		Let $x\in A$.
		We shall prove 
		\begin{equation}
		\label{e:oplus oplus}
			(a\oplus b\oplus x,a'\oplus b'\oplus x)\in \theta
		\end{equation}
		and
		\begin{equation}
		\label{e:in theta 0}
			\big((a\oplus b)\odot x,(a'\oplus b')\odot x\big)\in \theta.
		\end{equation}
		Since $(a,a') \in \theta^*$, we have 
		\[
			\big(a \oplus (b \oplus x), a' \oplus (b \oplus x)\big) \in \theta.
		\]
		Since $(b,b') \in \theta^*$, we have
		\[
			\big(b\oplus (a'\oplus x),b'\oplus (a'\oplus x)\big)\in \theta.
		\]
		Hence, by transitivity of $\theta$, we have $(a\oplus b\oplus x,a'\oplus b'\oplus x)\in \theta$, and so \cref{e:oplus oplus} is proved.
		
		Let us prove \cref{e:in theta 0}.
		By transitivity of $\theta$, it is enough to prove 
		\begin{equation}\label{e:in theta 1}
			\big((a \oplus b) \odot x, (a' \oplus b) \odot x\big) \in \theta,
		\end{equation}
		and
		\begin{equation}\label{e:in theta 2}
			\big((a'\oplus b)\odot x,(a'\oplus b')\odot x\big)\in \theta.
		\end{equation}
		Let us prove \cref{e:in theta 1}.
		Suppose, by way of contradiction, $((a\oplus b)\odot x,(a'\oplus b)\odot x)\notin \theta$.
		Then, without loss of generality, we may assume $(a\oplus b)\odot x\in \0(\theta)$ and $(a'\oplus b)\odot x\in \1(\theta)$.	We have	$\1(\theta)\ni(a'\oplus b)\odot x\leq x$; thus $x\in \1(\theta)$.
		We have
		\begin{equation*}
			\underbrace{(a\oplus b)\odot x}_{\in \0(\theta)}=\sigma(a,b,x)\land \underbrace{x}_{\in \1(\theta)},
		\end{equation*}
		and thus $\sigma(a,b,x)\in \0(\theta)$.
		We have
		\begin{equation*}
		\underbrace{(a'\oplus b)\odot x}_{\in \1(\theta)}=\sigma(a',b,x)\land \underbrace{x}_{\in \1(\theta)},
		\end{equation*}
		and thus $\sigma(a',b,x)\in \1(\theta)$.
		We have
		\[
			\0(\theta) \ni \sigma(a,b,x)  =  (a \odot (b \oplus x)) \oplus (b \odot x)  \geq  a \odot (b \oplus x),
		\]
		and thus $a\odot (b\oplus x)\in \0(\theta)$.
		Since $(a,a')\in \theta^*$, it follows that $a'\odot (b\oplus x)\in \0(\theta)$.
		We have
		\[  \underbrace{a'\odot(b\oplus x)}_{\in \0(\theta)}= a' \land \underbrace{\sigma(a',b,x)}_{\in \1(\theta)}.  \]
		Therefore, $a'\in \0(\theta)$.
		We have
		\[
			\1(\theta)\ni\sigma(a',b,x)=(a' \oplus (b\odot x)) \odot(b\oplus x)\leq a' \oplus (b\odot x),
		\]
		and thus $a'\oplus (b\odot x)\in \1(\theta)$.
		Since $(a,a')\in \theta^*$, it follows that $a\oplus (b\odot x)\in \1(\theta)$.
		We have
		\[\underbrace{a\oplus(b\odot x)}_{\in \1(\theta)}= a \lor \underbrace{\sigma(a,b,x)}_{\in \0(\theta)}.\]
		Therefore, $a\in \1(\theta)$.
		Thus, $a\in \1(\theta)$ and $a'\in \0(\theta)$, and this contradicts $(a,a')\in \theta^*\seq \theta$.
		In conclusion, \cref{e:in theta 1} holds, and analogously for \cref{e:in theta 2}.
		By transitivity of $\theta$, \cref{e:in theta 0} holds.
		This proves $(a\oplus b, a'\oplus b')\in \theta^*$.
		Analogously, $(a\odot b, a'\odot b')\in \theta^*$.	
		Therefore, $\theta^*$ is a $\{\oplus, \odot, \lor, \land, 0, 1\}$-congruence.	
	\end{claimproof}
\end{proof}
Given a set $A$, we let $\Delta_A$ (or simply $\Delta$, when $A$ is understood) denote the identity relation $\{(s,s) \mid s \in A\}$ on $A$.
\begin{lemma}
	\label{l:existence of lattice congruence}
	For every subdirectly irreducible {\mvm} $A$ there exists a lattice congruence $\theta$ on $A$ such that $\lvert A/\theta\rvert=2$ and $\theta^*=\Delta$.
\end{lemma}
\begin{proof}
	Since $A$ is distributive as a lattice, it can be decomposed into a subdirect product of two-element lattices.
	Let $\{\theta_i \}_{i\in I}$ be the set of lattice congruences of $A$ corresponding with such a decomposition.
	Then $\bigcap_{i\in I}\theta_i=\Delta$.
	By \cref{l:char-theta*}, each $\theta_i^*$ is a $\{\oplus, \odot, \lor, \land, 0, 1\}$-congruence, and $\Delta\seq\theta^*_i\seq \theta_i$.
	Therefore we have $\bigcap_{i\in I}\theta_i^*=\Delta$, and the fact that $A$ is subdirectly irreducible implies $\theta^*_j=\Delta$ for some $j\in I$.
\end{proof}
\begin{theorem}
	\label{t:sub irr is totally ordered}
	Every subdirectly irreducible {\mvm} is totally ordered.
\end{theorem}
\begin{proof}
	Let $A$ be a subdirectly irreducible {\mvm}.
	By \cref{l:existence of lattice congruence}, there exists a lattice congruence $\theta$ on $A$ such that $\lvert A/\theta\rvert=2$ and such that $\theta^*=\Delta$, i.e., for all distinct $a,b\in A$, there exists $x\in A$ such that $(a\oplus x,b\oplus x)\notin \theta$, or $(a\odot x,b\odot x)\notin \theta$.
	
	Let $a,b\in A$.
	We shall prove that either $a\leq b$ or $b\leq a$ holds.
	Suppose, by way of contradiction, that this is not the case, i.e., $a\land b\neq a$ and $a\land b\neq b$.
	Since $a\land b\neq a$, there exists $x\in A$ such that $((a\land b)\oplus x,a\oplus x)\notin \theta$ or $((a\land b)\odot x,a\odot x)\notin \theta$.
	Since $a\land b\neq b$, there exists $y\in A$ such that $((a\land b)\oplus y,b\oplus y)\notin \theta$ or $((a\land b)\odot y,b\odot y)\notin \theta$.
	We have four cases.
	\begin{enumerate}[wide]
		\item
		\label{i:case oplus oplus} Suppose $((a\land b)\oplus x,a\oplus x)\notin \theta$ and $((a\land b)\oplus y,b\oplus y)\notin \theta$.		
		Then, since $a\land b\leq a$, and $a\land b\leq b$, we have $(a\land b)\oplus x\in \0(\theta)$, $a\oplus x\in \1(\theta)$, $(a\land b)\oplus y\in \0(\theta)$, and $b\oplus y\in \1(\theta)$. Then, we have
		\begin{align*}
			\0(\theta)	& \ni		((a \land b) \oplus x) \lor ((a \land b) \oplus y) \\
							& =		(a \land b) \oplus (x \lor y)								&& \by{$\oplus$ distr.\ over $\lor$}\\
							& =		(a \oplus (x \lor y)) \land (b \oplus(x \lor y))	&& \by{$\oplus$ distr.\ over $\land$}\\
							& \geq	(a \oplus x) \land (b \oplus y)							&& \by{\cref{l:order-preserving properties}}\\
							& \in \1(\theta),
		\end{align*}
		which is a contradiction.
		\item The case $((a \land b) \odot x, a \odot x) \notin \theta$ and $((a \land b) \odot y, b \odot y) \notin \theta$ is analogous to \cref{i:case oplus oplus}.
		
		\item
		\label{i:case oplus odot} Suppose $((a\land b)\oplus x,a\oplus x)\notin \theta$ and $((a\land b)\odot y,b\odot y)\notin \theta$.
		Then, since $a\land b\leq a$, we have $(a\land b)\oplus x\in \0(\theta)$, and $a\oplus x\in \1(\theta)$. Therefore,
		\[   \0(\theta)  \ni  (a \land b) \oplus x  =  \underbrace{(a \oplus x)}_{\in \1(\theta)} \land (b \oplus x).   \]
		Hence, $b \oplus x\in \0(\theta)$, which implies $b \in \0(\theta)$, which implies $(a\land b)\odot y\in \0(\theta)$ and $b\odot y\in \0(\theta)$, which contradicts $((a\land b)\odot y,b\odot y)\notin \theta$.
		
		\item The case $((a\land b)\odot x,a\odot x)\notin \theta$ and $((a\land b)\oplus y,b\oplus y)\notin \theta$ is analogous to \cref{i:case oplus odot}.
	\end{enumerate}
	In each case, we are led to a contradiction.
\end{proof}
%


\subsection{Good pairs in subdirectly irreducible algebras}
	\label{s:good-pairs_in_sub-irr}

The goal of this subsection---met in \cref{l:good-pairs-sub-irr}---is to prove that, for every good pair $(x_0,x_1)$ in a subdirectly irreducible {\mvm} $A$, we have either $x_0 = 1$ or $x_1 = 0$.
This implies that good $\Z$-sequences in  $A$ are of the form
\begin{align*}
	\Z	&	\longrightarrow	A	\\
	k	&	\longmapsto			{
									\begin{cases}
										1	&	\text{if } k < n;	\\
										x	&	\text{if } k = n;	\\
										0	&	\text{if } k > n.	\\
									\end{cases}
									}
\end{align*}
for some $n \in \Z$ and $x \in A$.
\begin{notation} \label{n:sim-bot-top}
	Let $A$ be {\amvm} and let $t,x,y \in A$.
	We write $x \gtrsim_\bot^t y$ if, and only if, there exists $n\in \N$ such that 
	\[
		x \oplus \underbrace{t \oplus \cdots \oplus t}_{n \text{ times}} \geq y.
	\]
	We write $x \sim_\bot^t y$ if, and only if, $x \gtrsim_\bot^t y$ and $y \gtrsim_\bot^t x$.
	
	We write $x \lesssim_t^\top y$ if, and only if, there exists $n \in \N$ such that 
	\[
		x \odot \underbrace{t \odot \cdots \odot t}_{n \text{ times}} \leq y.
	\]
	We write $x \sim_t^\top y$ if, and only if, $x \lesssim_t^\top y$ and $y \lesssim_t^\top x$.
\end{notation}
\begin{lemma}
	\label{l:bot-top-congruences}
	Let $A$ be {\amvm}.
	For every $t\in A$ the following conditions hold.
	\begin{enumerate}
		\item \label{i:t-sim-0} The relation $\sim_\bot ^t$ is the smallest $\{\oplus, \odot, \lor, \land, 0, 1\}$-congruence $\sim$ on $A$ such that $t\sim 0$.
		\item \label{i:t-sim-1} The relation $\sim_t^\top$ is the smallest $\{\oplus, \odot, \lor, \land, 0, 1\}$-congruence $\sim$ on $A$ such that $t\sim 1$.
	\end{enumerate}
\end{lemma}
\begin{proof}
	We prove \cref{i:t-sim-0}; \cref{i:t-sim-1} is dual.	It is immediate that $t\sim_\bot^t 0$, and that, if $\sim$ is a $\{\oplus, \odot, \lor, \land, 0, 1\}$-congruence on $A$ such that $t\sim 0$, then ${\sim_\bot^t}\seq {\sim}$.
	To prove that $\sim_\bot^t$ is a congruence, we first prove that $\sim_\bot^t$ is an equivalence relation.
	It is enough to prove that $\gtrsim_\bot^t$ is reflexive and transitive.
	Reflexivity of $\gtrsim_\bot^t$ is trivial.
	To prove transitivity, suppose $x\gtrsim_\bot^t y\gtrsim_\bot^t z$.
	Then there exist $n, n' \in \N$ such that 
	\[
		x \oplus \underbrace{t \oplus \cdots \oplus t}_{n \text{ times}} \geq y
	\]
	and
	\[
		y \oplus \underbrace{t \oplus \cdots \oplus t}_{n' \text{ times}} \geq z.
	\]
	Therefore,
	\[
		x\oplus(\underbrace{t\oplus\cdots\oplus t}_{n\text{ times}})\oplus(\underbrace{t\oplus\cdots\oplus t}_{n'\text{ times}})\geq y\oplus \underbrace{t\oplus\cdots\oplus t}_{n'\text{ times}}\geq z.
	\]
	It follows that $x \gtrsim_\bot^t y$.
	Thus, $\gtrsim_\bot^t$ is transitive.
	This proves that $\sim_\bot^t$ is an equivalence relation.
	
	Let us prove that $\sim_\bot^t$ is a congruence.
	Suppose $x\sim_\bot^t x'$ and $y\sim_\bot^t y'$.
	Then, there exist $n, n', m, m' \in \N$ such that
	\[
		x \oplus \underbrace{t \oplus \cdots \oplus t}_{n \text{ times}} \geq x',
	\]
	\[ 
		x' \oplus \underbrace{t \oplus \cdots \oplus t}_{m \text{ times}} \geq x,
	\]
	\[
		y \oplus \underbrace{t \oplus \cdots \oplus t}_{n' \text{ times}} \geq y',
	\]
	and
	\[
		y' \oplus \underbrace{t \oplus \cdots \oplus t}_{m' \text{ times}} \geq y.
	\]

	We have $x \land y \gtrsim_\bot^t x'\land y'$ because
	\[ 
		(x \land y) \oplus \underbrace{t \oplus \cdots \oplus t}_{\max\{n, n'\} \text{ times}} =
		(x \oplus \underbrace{t \oplus \cdots \oplus t}_{\max\{n, n'\} \text{ times}}) \land (y \oplus \underbrace{t \oplus \cdots \oplus t}_{\max\{n, n'\} \text{ times}}) \geq
		x'\land y'. 
	\]
	Analogously, we have $x' \land y' \gtrsim_\bot^t x \land y$.
	Hence, $x \land y \sim_\bot^t x' \land y'$, and, analogously, $x \lor y\sim_\bot^t x'\lor y'$.
	
	We have $x \oplus y \gtrsim_\bot^t x' \oplus y'$ because
	\[
		x \oplus y \oplus \underbrace{t \oplus \cdots \oplus t}_{(n+n') \text{ times}}=
		(x \oplus \underbrace{t \oplus \cdots \oplus t}_{n\text{ times}}) \oplus (y \oplus \underbrace{t \oplus \cdots \oplus t}_{n'\text{ times}}) \geq
		x' \oplus y'.
	\]
	Analogously, $x' \oplus y' \gtrsim_\bot^t x \oplus y$.
	Hence, $x \oplus y \sim_\bot^t x' \oplus y'$.
	
	We have $x \odot y \gtrsim_\bot^t x'\odot y'$ because
	\begin{align*}
		(x \odot y) \oplus \underbrace{t \oplus \cdots \oplus t}_{(n + n') \text{ times}}	
			& = 		((x \odot y) \oplus \underbrace{t \oplus \cdots \oplus t}_{n \text{ times}})
						\oplus \underbrace{t \oplus \cdots \oplus t}_{n' \text{ times}} 						&& \\
			& \geq	((x \oplus \underbrace{t \oplus \dots \oplus t}_{n \text{ times}}) \odot y)
						\oplus \underbrace{t \oplus \dots \oplus t}_{n' \text{ times}}							&& \text{(\cref{l:almost associative})}\\
			& \geq	(x \oplus \underbrace{t \oplus \cdots \oplus t}_{n \text{ times}}) 
						\odot (y \oplus \underbrace{t \oplus \cdots \oplus t}_{n' \text{ times}})			&& \text{(\cref{l:almost associative})}\\
			& \geq	x'\odot y'.																								&&
	\end{align*}
	Analogously, we have $x' \odot y' \gtrsim_\bot^t x \odot y$.
	Hence, $x \odot y \sim_\bot^t x' \odot y'$, and \cref{i:t-sim-0} is proved.	
\end{proof}
\begin{lemma}
	\label{l:good approximation for leq}
	Let $A$ be {\amvm}, let $(x_0,x_1)$ be a good pair in $A$, and let $a,b\in A$ be such that $a\leq b\oplus x_1$ and $a\odot x_0\leq b$.
	Then $a\leq b$.
\end{lemma}
\begin{proof}
	By Birkhoff's subdirect representation theorem, it is enough to prove it for $A$ a subdirectly irreducible algebra.
	By \cref{t:sub irr is totally ordered}, the algebra $A$ is totally ordered.
	Therefore, we have either $a \leq b$ or $b \leq a$.
	In the first case, the desired statement is proved.
	So, let us assume $b \leq a$.
	We have
	\[
		a \land \sigma(a, x_0, x_1) \stackrel{\text{\cref{ax:A7}}}{=} a \odot (x_0 \oplus x_1) = a \odot x_0.
	\]
	Since $A$ is totally ordered, we have either $a = a \odot x_0$ or $\sigma(a, x_0, x_1) = a \odot x_0$.
	In the first case, we have $a = a \odot x_0 \leq b$, so the desired statement holds.
	So, we can assume $\sigma(a, x_0, x_1) = a \odot x_0$.
	Dually, $\sigma(b, x_0, x_1) = b \oplus x_1$.
	Since $b \leq a$, and since every operation of {\mvms} is order-preserving (\cref{l:order-preserving properties}), we have $\sigma(b, x_0, x_1) \leq \sigma(a, x_0, x_1)$.
	Therefore,
	\[
		a \leq b \oplus x_1 = \sigma(b, x_0, x_1) \leq \sigma(a, x_0, x_1) = a \odot x_0 \leq b. \qedhere
	\]
\end{proof}
\begin{lemma}
	\label{l:intersection-of-congruences}
	Let $A$ be {\amvm}.
	For all $x, y \in A$, the intersection of $\sim_\bot^{x\odot y}$ and $\sim_{x\oplus y}^\top$ is the identity relation on $A$.
\end{lemma}
\begin{proof}
	Set $u\df x\oplus y$, and $v\df x\odot y$.
	Let us take $a,b\in A$ such that $a\sim_\bot^v b$ and $a\sim_u^\top b$.
	Then, there exists $n,m\in \N$ such that
	\[
		a \leq b \oplus \underbrace{v \oplus \cdots \oplus v}_{m\text{ times}}
	\]
	and
	\[
		a \odot \underbrace{u \odot \cdots \odot u}_{n\text{ times}} \leq b.
	\]
	Since $(u,v)$ is a good pair by \cref{l:oplus odot is a good pair}, also $(u\odot \dots\odot u,v\oplus\cdots\oplus v)$ is so, by \cref{l:bipartite}.
	By \cref{l:good approximation for leq}, $a\leq b$; analogously, $b\leq a$, and therefore $a=b$.
\end{proof}
\begin{theorem}
	\label{t:good pairs in sub irr}
	Let $A$ be a subdirectly irreducible {\mvm}.
	Then, for all $x,y\in A$, we have either $x\oplus y=1$ or $x\odot y=0$.
\end{theorem}
\begin{proof}
	By \cref{l:intersection-of-congruences}, the intersection of $\sim_\bot^{x\odot y}$ and $\sim_{x\oplus y}^\top$ is the identity relation $\Delta$ on $A$.
	By \cref{l:bot-top-congruences}, $\sim_\bot^{x\odot y}$ and $\sim_{x\oplus y}^\top$ are $\{\oplus, \odot, \lor, \land, 0, 1\}$-congruences, $x\odot y\sim_\bot^{x\odot y}0$ and $x\oplus y\sim_{x\oplus y}^\top 1$.
	Since $A$ is subdirectly irreducible, either $\sim_\bot^{x\odot y}=\Delta$ or $\sim_{x\oplus y}^\top=\Delta$.
	In the former case we have $x \odot y=0$;
	in the latter one we have $x \oplus y=1$.
\end{proof}
\begin{corollary} \label{l:good-pairs-sub-irr}
	Let $(x_0,x_1)$ be a good pair in a subdirectly irreducible {\mvm}.
	Then, either $x_0 = 1$ or $x_1 = 0$.
\end{corollary}
\begin{corollary}
	\label{c:good sequence in sub irr}
	Let $\gs{x}$ be a good $\Z$-sequence in a subdirectly irreducible {\mvm} $A$.
	Then, there exists $k \in \Z$ and $x\in A$ such that $\gs{x}$ is the following function.
	\begin{align*}
		\Z	& \longrightarrow	A\\
		n	& \longmapsto			{
										\begin{cases}
											1	&	\text{if } n < k	;	\\
											x	&	\text{if } n = k	;	\\
											0	&	\text{if } n > k	.
										\end{cases}
										}
	\end{align*}
\end{corollary}
%


\section{Operations on the set \texorpdfstring{$\X(A)$}{\textXi(\unichar{"1D434})} of good \texorpdfstring{$\mathbb{Z}$}{\unichar{"2124}}-sequences in \texorpdfstring{$A$}{\unichar{"1D434}}}
	\label{s:operations on Good}

We denote with $\X(A)$ the set of good $\Z$-sequences in {\amvm} $A$.
We will endow $\X(A)$ with a structure of {\aulm}.


\subsection{The constants}

We denote with $\gs{0}$ the good $\Z$-sequence 
\begin{align*}
	\Z	&	\longrightarrow	A	\\
	n	&	\longmapsto			{
									\begin{cases}
										1	&	\text{if } n < 0;		\\
										0	&	\text{if } n \geq 0.	\\
									\end{cases}
									}
\end{align*}
We denote with $\gs{1}$ the good $\Z$-sequence 
\begin{align*}
	\Z	&	\longrightarrow	A	\\
	n	&	\longmapsto			{
									\begin{cases}
										1	&	\text{if } n < 1;		\\
										0	&	\text{if } n \geq 1.	\\
									\end{cases}
									}
\end{align*}
We denote with $\gs{-1}$ the good $\Z$-sequence 
\begin{align*}
	\Z	&	\longrightarrow	A	\\
	n	&	\longmapsto		{
								\begin{cases}
									1	&	\text{if } n < -1;		\\
									0	&	\text{if } n \geq -1.	\\
								\end{cases}
								}
\end{align*}
%


\subsection{The lattice operations}

For good $\Z$-sequences $\gs{a}$ and $\gs{b}$ in $A$, we let $\gs{a} \lor \gs{b}$ denote the function
\begin{align*}
	\Z	& \longrightarrow	A \\
	n	& \longmapsto		\gs{a}(n) \lor \gs{b}(n),
\end{align*}
and we let $\gs{a} \land \gs{b}$ denote the function
\begin{align*}
	\Z	& \longrightarrow	A\\
	n	& \longmapsto 		\gs{a}(n) \land \gs{b}(n).
\end{align*}

\begin{proposition} \label{p:join is good}
	For all good $\Z$-sequences $\gs{a}$ and $\gs{b}$ in {\amvm}, the $\Z$-sequences $\gs{a}\lor \gs{b}$ and $\gs{a}\land \gs{b}$ are good.
\end{proposition}
\begin{proof}
	By Birkhoff's subdirect representation theorem, we can safely suppose the {\mvm} to be subdirectly irreducible.
	Then, by \cref{t:sub irr is totally ordered,c:good sequence in sub irr}, the $\Z$-sequence $\gs{a} \lor \gs{b}$ is either $\gs{a}$ or $\gs{b}$, and the same holds for $\gs{a}\land \gs{b}$.
\end{proof}
\begin{proposition}
	\label{p:distributive lattice}
	Let $A$ be {\amvm}.
	Then, $\langle\X(A);\lor,\land\rangle$ is a distributive lattice.
\end{proposition}
\begin{proof}
	The statement holds because $\lor$ and $\land$ are applied componentwise on $\X(A)$, and $\langle A; \lor, \land \rangle$ is a distributive lattice.
\end{proof}
\noindent For $A$ an {\mvm}, we have a partial order $\leq$ on $\X(A)$, induced by the lattice operations.
Since the lattice operations are defined componentwise, we have the following.

\begin{lemma}\label{l:order is pointwise}
	For all good $\Z$-sequences $\gs{a}$ and $\gs{b}$ in {\amvm} we have $\gs{a} \leq \gs{b}$ if, and only if, for all $n \in \Z$, we have $a_n \leq b_n$.
\end{lemma}
%


\subsection{The addition}

We now want to define sum of good $\Z$-sequences.
Let $\gs{a}$ and $\gs{b}$ be good $\Z$-sequences in {\amvm}.
There are two natural ways to define $\gs{a}+\gs{b}$.
The first one is
\begin{equation}
	\label{eq:sum-bigodot}
	(\gs{a}+\gs{b})(n)\df \bigodot_{k\in \Z}\gs{a}(k)\oplus \gs{b}(n-k)
\end{equation}
and the second one is
\begin{equation}
	\label{eq:sum-bigoplus}
	(\gs{a}+\gs{b})(n)\df \bigoplus_{k\in \Z}\gs{a}(k)\odot \gs{b}(n-k-1).
\end{equation}
Note that the right-hand side of \cref{eq:sum-bigodot} is well-defined because all but finitely many terms equal $1$; analogously, the right-hand side of \cref{eq:sum-bigodot} is well-defined because all but finitely many terms equal $0$. 

In fact, these two ways coincide, as shown in the following.
\begin{lemma}
	\label{l:associative}
	Let $\gs{a}$ and $\gs{b}$ be good $\Z$-sequences in {\amvm}.
	Then, for every $n\in \Z$, we have
	\[ \bigodot_{k\in \Z}\gs{a}(k)\oplus \gs{b}(n-k)=\bigoplus_{k\in \Z}\gs{a}(k)\odot \gs{b}(n-k-1). \]
\end{lemma}
\begin{proof}
	By Birkhoff's subdirect representation theorem, it is enough to prove the statement for a subdirectly irreducible {\mvm} $A$.
	By \cref{c:good sequence in sub irr}, up to a translation of $\gs{a}$ and $\gs{b}$, we can assume $\gs{a}=(a)$ and $\gs{b}=(b)$ for some $a,b\in A$.
	Then
	\begin{equation*}
		\begin{split}
			\bigodot_{k\in \Z} \gs{a}(k) \oplus \gs{b}(n-k)	
				& =	\left(\bigodot_{k\in \Z, k<0} 1 \oplus \gs{b}(n-k)\right) \odot (a \oplus \gs{b}(n)) \odot \left(\bigodot_{k\in \Z, k>0} 0 \oplus \gs{b}(n-k)\right)\\
				& =	(a \oplus\gs{b}(n)) \odot \left(\bigodot_{k\in \Z, k>0} \gs{b}(n-k))\right)\\
				& =	{
						\begin{cases}
							1				& \text{if }n<0;\\			
							a\oplus b	& \text{if }n=0;\\
							a\odot b		& \text{if }n=1;\\			
							0				& \text{if }n>1.		
						\end{cases}
						}
		\end{split}
	\end{equation*}
	Moreover,
	\begin{align*}
		\bigoplus_{k\in \Z}(\gs{a}(k)\odot \gs{b}(n-k-1))	& =	\left(\bigoplus_{k\in \Z, k<0}1\odot \gs{b}(n-k-1)\right)\oplus (a\odot \gs{b}(n-1))\oplus 0\\
																		& =	\left(\bigoplus_{k\in \Z, k<0}\gs{b}(n-k-1)\right)\oplus (a\odot \gs{b}(n-1))\\
																		& =	{
																				\begin{cases}
																					1				& \text{if }n<0;\\			
																					a \oplus b	& \text{if }n=0;\\
																					a \odot b	& \text{if }n=1;\\			
																					0				& \text{if }n>1. 		
																				\end{cases}
																				}
	\end{align*}
\end{proof}
Given good $\Z$-sequences $\gs{a}$ and $\gs{b}$ in {\amvm}, we set, for every $n \in \Z$,
\begin{equation}
	\label{eq:sum-bigodot1}
	(\gs{a}+\gs{b})(n)\df \bigodot_{k\in \Z}\gs{a}(k)\oplus \gs{b}(n-k),
\end{equation}
or, equivalently (by \cref{l:associative}),
\begin{equation}
	\label{eq:sum-bigoplus1}
	(\gs{a}+\gs{b})(n)\df \bigoplus_{k\in \Z}\gs{a}(k)\odot \gs{b}(n-k-1).
\end{equation}
\begin{proposition}
	\label{p:sum is good}
	For all good $\Z$-sequences $\gs{a}$ and $\gs{b}$ in {\amvm}, the $\Z$-sequence $\gs{a}+\gs{b}$ is good.
\end{proposition}
\begin{proof}
	By Birkhoff's subdirect representation theorem, it is enough to prove the statement for a subdirectly irreducible {\mvm} $A$.
	Then, up to a translation of $\gs{a}$ and $\gs{b}$, we can assume $\gs{a}=(a)$ and $\gs{b}=(b)$ for some $a,b\in A$.
	Then we have $\gs{a} + \gs{b} = (a \oplus b, a \odot b)$.
	The pair $(a \oplus b, a \odot b)$ is good by \cref{l:oplus odot is a good pair}.
	Therefore, $\gs{a}+\gs{b}$ is good.
\end{proof}
%


\subsection{The algebra \texorpdfstring{$\X(A)$}{\textXi(\unichar{"1D434})} is a unital commutative distributive \texorpdfstring{$\ell$}{\unichar{"02113}}-monoid}

\begin{proposition}
	\label{p:sum is commutative}
	Addition of good $\Z$-sequences in {\amvm} is commutative.
\end{proposition}
\begin{proof}
	By commutativity of $\oplus$ and $\odot$.
\end{proof}

\begin{proposition}
	\label{p:0 is neutral}
	For every good $\Z$-sequence $\gs{a}$ in {\amvm} we have $\gs{a} + \gs{0} = \gs{a}$.
\end{proposition}
\begin{proof}
	The proof is straightforward.
\end{proof}
\begin{proposition}
	\label{p:associativity full}
	Addition of good $\Z$-sequences in {\amvm} is associative.
\end{proposition}
\begin{proof}
	By Birkhoff's subdirect representation theorem, it is enough to prove the statement for a subdirectly irreducible {\mvm} $A$.
	So, let $\gs{a}$, $\gs{b}$ and $\gs{c}$ be good $\Z$-sequences in $A$.
	By \cref{c:good sequence in sub irr}, up to a translation for each of $\gs{a}$, $\gs{b}$ and $\gs{c}$, we can suppose $\gs{a}=(a)$, $\gs{b}=(b)$ and $\gs{c}=(c)$, for some $a,b,c\in A$.
	Then, $\gs{a}+\gs{b}=(a\oplus b,a\odot b)$, and $\gs{b} +\gs{c}=(b\oplus c,b\odot c)$.
	We have
	\begin{align*}
		(\gs{a} + \gs{b}) + \gs{c}	& = (a \oplus b, a \odot b) + (c) \\
									& = \big((a \oplus b) \oplus c, (a \oplus b) \odot((a \odot b) \oplus c), (a \odot b) \odot c\big) \\
									& = \big(a \oplus (b \oplus c), (a \oplus (b \odot c)) \odot(b \oplus c), a \odot(b \odot c)\big)	&& \by{\cref{l:permutations}}\\
									& = (a) + (b \oplus c, b \odot c)\\
									& = \gs{a} + (\gs{b} + \gs{c}).	&& \qedhere
	\end{align*}
\end{proof}
\begin{proposition}
	\label{p:distributivity full}
	For all good $\Z$-sequences $\gs{a}$, $\gs{b}$, $\gs{c}$ in {\amvm} we have
	\begin{align}
		\label{eq:distr-lor}		\gs{a} + (\gs{b} \lor \gs{c})		& = (\gs{a} + \gs{b}) \lor (\gs{a} + \gs{c})
	\intertext{and}
		\label{eq:distr-land}	\gs{a} + (\gs{b} \land \gs{c})	& = (\gs{a} + \gs{b}) \land (\gs{a} + \gs{c}).
	\end{align}
\end{proposition}
\begin{proof}
	Let us prove \cref{eq:distr-lor}.
	By Birkhoff's subdirect representation theorem, it is enough to prove the statement for a subdirectly irreducible algebra.
	In this case, by \cref{t:sub irr is totally ordered}, we have either $\gs{b}\leq \gs{c}$ or $\gs{c}\leq \gs{b}$.
	Without loss of generality, we can suppose $\gs{b}\leq \gs{c}$.
	Then, by the definition of $+$ and monotonicity of $\oplus$ and $\odot$, we have $\gs{a} + \gs{b} \leq \gs{a} + \gs{c}$, and thus $\gs{a} + (\gs{b} \lor \gs{c}) = (\gs{a} + \gs{b}) \lor (\gs{a} + \gs{c})$.
	\Cref{eq:distr-land} is analogous.
\end{proof}
\begin{theorem}
	For {\amvm} $A$, the algebra $\X(A)$ is {\aulm}.
\end{theorem}
\begin{proof}
	By \cref{p:distributive lattice}, $\X(A)$ is a distributive lattice.
	By \cref{p:sum is commutative,p:associativity full,p:0 is neutral}, $\X(A)$ is a commutative monoid.
	By \cref{p:distributivity full}, $+$ distributes over $\land$ and $\lor$.
	Thus, $\X(A)$ is {\alm}.
	It is easily verified that $\gs{-1} + \gs{1} = \gs{0}$.
	Since the order in $\X(A)$ is pointwise (\cref{l:order is pointwise}), and $0 \leq 1$ in $A$, we have $\gs{-1} \leq \gs{0} \leq \gs{1}$.
	By induction, one proves that, for every $n \in \N$, the $\Z$-sequence $n\gs{1}$ is the function
	\begin{align*}
		\Z	&	\longrightarrow	A	\\
		k	&	\longmapsto			{
										\begin{cases}
											1	&	\text{if } k < n;		\\
											0	&	\text{if } k \geq n,	\\
										\end{cases}
										}
	\end{align*}
	and $n(\gs{-1})$ is the function
	\begin{align*}
		\Z	&	\longrightarrow	A	\\
		k	&	\longmapsto		{
									\begin{cases}
										1	&	\text{if } k < -n;		\\
										0	&	\text{if } k \geq -n.	\\
									\end{cases}
									}
	\end{align*}
	Since $1$ is the maximum of $A$ and $0$ is the minimum of $A$, we have \cref{ax:U3}, i.e., for all ${\gs{a}} \in \X(A)$, there exists $n \in \Np$ such that $n(\gs{-1})\leq{\gs{a}} \leq n\gs{1}$.
\end{proof}
Every morphism of {\mvms} $f\colon A\to B$ preserves $0$, $1$, and good pairs; thus, we are allowed to define the function
\begin{equation*}
	\begin{split}
		\X(f) \colon	\X(A)		& \longrightarrow	\X(B)\\
							\gs{x}	& \longmapsto		\big(\Z \to B; n \mapsto f(\gs{x}(n))\big).
	\end{split}
\end{equation*}
More concisely: $\X(f)(\gs{x})(n)=f(\gs{x}(n))$.
\begin{lemma}
	For every morphism $f\colon A\to B$ of {\mvms}, the function $\X(f)$ is a morphism of {\ulms}.
\end{lemma}
\begin{proof}
	Let us prove that $\X(f)$ preserves $+$.
	Let $\gs{a},\gs{b}\in \X(A)$, and let $n\in \Z$.
	Then, 
	\begin{equation*}
		\begin{split}
			\big(\X(f)(\gs{a} + \gs{b})\big)(n)	& = f((\gs{a} + \gs{b})(n))\\
											& = f\mathopen{}\left(\bigodot_{k\in \Z} \gs{a}(k) \oplus \gs{b}(n - k)\right)\mathclose{}\\
											& = \bigodot_{k\in \Z} f(\gs{a}(k)) \oplus f(\gs{b}(n - k))\\
											& = \bigodot_{k\in \Z} \X(f)(\gs{a})(k) \oplus \X(f)(\gs{b})(n - k)\\
											& = \big(\X(f)(\gs{a}) + \X(f)(\gs{b})\big)(n).
		\end{split}
	\end{equation*}
	Therefore, $\X(f)$ preserves $+$.
	Straightforward computations show that $\X(f)$ preserves also $0$, $1$, $-1$, $\lor$ and $\land$.
\end{proof}
It is easy to see that $\X\colon \MVM\to \ULM$ is a functor.


\section{The equivalence} 
	\label{s:MVM and positive-unital}

The aim of the present section is to prove that the functors $\Gam\colon \ULM\to \MVM$ and $\X\colon \MVM\to \ULM$ defined in \cref{s:Gamma,s:operations on Good} are quasi-inverses.


\subsection{Natural isomorphism for {\mvms}}

For each {\mvm} $A$, define the function 
\begin{equation*}
	\begin{split}
		\eta_A \colon	A	& \longrightarrow 	\Gam\X(A)\\
							x	& \longmapsto 			(x).
	\end{split}
\end{equation*}
\begin{proposition}
	\label{p:unit mvm}
	For every {\mvm} $A$, the function $\eta_A \colon A \to \Gam\X(A)$ is an isomorphism of {\mvms}.
\end{proposition}
\begin{proof}
	The facts that $\eta_A$ is a bijection and that it preserves $0,1,\lor,\land$ are immediate.
	Let $x,y\in A$.
	Then $(x) + (y)  =  (x \oplus y, x \odot y)$.
	Therefore
	\begin{align*}
		\eta_A(x) \oplus \eta_A(y)	& = (x) \oplus (y)\\
											& = ((x) + (y)) \land {\gs{1}}\\
											& = (x \oplus y, x \odot y) \land {\gs{1}}\\
											& = (x \oplus y)\\
											& = \eta_A(x \oplus y)
	\end{align*}
	and
	\begin{align*}
		\eta_A(x) \odot \eta_A(y)	& = (x) \odot(y)\\
											& = (((x) + (y)) \lor {\gs{1}}){\gs{-1}}\\
											& = ((x \oplus y, x \odot y) \lor {\gs{1}}){\gs{-1}}\\
											& = (1, x \odot y){\gs{-1}}\\
											& = (x \odot y)\\
											& = \eta_A(x \odot y). \qedhere
	\end{align*}
\end{proof}
\begin{proposition}
	\label{p:eta1 is natural}
	For every morphism of {\mvms} $f\colon A\to B$, the following diagram commutes.
	\[
		\begin{tikzcd}
			A \arrow{r}{\eta_A} \arrow[swap]{d}{f}	& \Gam\X(A) \arrow{d}{\Gam\X(f)}\\
			B \arrow[swap]{r}{\eta_B}					& \Gam\X(B),\\
		\end{tikzcd}
	\]
\end{proposition}
\begin{proof}
	For every $x\in A$ we have
	\[
		\Gam\X(f)(\eta_A(x))=\Gam\X(f)((x))=\X(f)((x))=(f(x))=\eta_B(f(x)). \qedhere
	\]
\end{proof}
%


\subsection{Natural isomorphism for unital lattice-ordered monoids}

\begin{lemma} \label{l:up-and-down}
	For every $x$ in {\aulm} we have
	\[
		((x \lor -1) \land 0) + ((x\lor 0) \land 1) = (x \lor -1) \land 1.
	\]
\end{lemma}
\begin{proof}
	We have
	\begin{align*}
		((x \lor -1) \land 0)+ ((x\lor 0) \land 1)	& = (((x \lor -1) \land 1) \land 0) + (((x \lor -1) \land 1) \lor 0)\\
																	& = (x \lor -1) \land 1. \qedhere
	\end{align*}
\end{proof}

\begin{proposition}
	\label{p:trunc is good}
	Let $M$ be {\aulm}, let $x\in M$ and set, for each $n\in \Z$,
	\[ \gs{x}(n) \df ((x- n)\lor 0)\land 1. \]
	Then $\gs{x}$ is a good $\Z$-sequence in $\Gam(M)$.
\end{proposition}
\begin{proof}
	Clearly, for each $n\in \Z$, we have $\gs{x}(n)\in \Gam(M)$.
	
	Since there exists $n \in \Np$ such that $-n  \leq  x  \leq  n$, we have, for every $k< -n$, $\gs{x}(k)  =  1$, and, for every $k \geq n$, $\gs{x}(k)=0$.
	
	Let us prove that, for every $n\in \Z$, $(\gs{x}(n), \gs{x}(n+1))$ is a good pair.
	We have
	\begin{align*}
		& \gs{x}(n) + \gs{x}(n+ 1)\\
		& = (((x- n) \lor 0)\land 1) + (((x- n -1)\lor 0)\land 1)\\
		& = (((x- n-1) \lor -1)\land 0) + (((x- n -1)\lor 0)\land 1) + 1\\
		& = (((x - n - 1) \lor -1) \land 1) + 1	&& \by{\cref{l:up-and-down}}\\
		& = ((x - n) \lor 0) \land 2.\\
	\end{align*}
	Then, we have 
	\begin{equation*}
		\begin{split}
			(\gs{x}(n) \oplus \gs{x}(n+1))	& = \gs{x}(n) + \gs{x}(n+1)) \land 1\\
														& = ((x - n) \lor 0) \land 2 \land 1\\
														& = ((x - n) \lor 0) \land 1\\
														& = \gs{x}(n),
		\end{split}
	\end{equation*}
	and 
	\begin{equation*}
		\begin{split}
			(\gs{x}(n) \odot \gs{x}(n+1))	& = (\gs{x}(n) + \gs{x}(n+1) - 1) \lor 0\\
													& = ((((x - n) \lor 0) \land 2) - 1) \lor 0\\
													& = (((x - n - 1) \lor -1) \land 1) \lor 0\\
													& = ((x - n - 1) \lor 0) \land 1\\
													& = \gs{x}(n + 1). \qedhere
		\end{split}
	\end{equation*}
\end{proof}

\begin{lemma}
	\label{l:sum of good sequences in mon}
	Let $M$ be {\aulm}, let $x \in M$, and let $y\in \Gam(M)$.
	Set $x_0 \df (x \lor 0) \land 1$ and $x_{-1} \df ((x + 1) \lor 0) \land 1$.
	Then, we have
	\[  ((x + y) \lor 0) \land 1 =x_{-1} \odot (x_0\oplus y).\]
\end{lemma}
\begin{proof}
	First, note that
	\begin{align*}
		x_{-1} + x_0	& = (((x + 1) \lor 0) \land 1) + ((x \lor 0) \land 1)\\
							& = ((x \lor -1) \land 0) + ((x \lor 0) \land 1) + 1\\
							& = ((x \lor -1) \land 1) + 1.	&& \by{\cref{l:up-and-down}}
	\end{align*}
	We have
	\begin{align*}
		x_{-1} \odot (x_0 \oplus y) 	& = ((x_{-1} + x_0 + y -1) \lor 0) \land 1	&& \by{\cref{l:sigma in monoid}}\\
												& = ((((x \lor -1) \land 1) + 1 + y -1) \lor 0) \land 1\\
												& = (((x \lor (1 + y)) \land (1 + y)) \lor 0) \land 1\\
												& = ((x + y) \lor 0) \land 1.	&& \qedhere
	\end{align*}
\end{proof}

\begin{notation}
	For every {\ulm} $M$, and every $x \in M$, we define the $\Z$-sequence
	\begin{align*}
		\zeta_M(x)	\colon 	\Z	& \longrightarrow	\Gam(M)\\
									n	& \longmapsto		((x- n)\lor 0)\land 1.
	\end{align*}
	This defines a function $\zeta_M \colon M \to \X(\Gam(M))$ that maps an element $x \in M$ to $\zeta_M(x)$.
\end{notation}

\begin{proposition} \label{p:eps1 is morphism}
	For every {\ulm} $M$, the function $\zeta_M \colon M \to \X\Gam(M)$ is a morphism of {\ulms}, i.e.\ it preserves $+$, $\lor$, $\land$, $0$, $1$ and $-1$.
\end{proposition}
\begin{proof}
	It is easily seen that $\zeta_M$ preserves $0$, $1$, $-1$, $\lor$ and $\land$.
	Let us prove that $\zeta_M$ preserves $+$.
	Let $x, y \in M$.
	Let $k \in \Z$.
	We shall prove 
	\[
		\big(\zeta_M(x) + \zeta_M(y)\big)(k) = \zeta_M(x + y) (k).
	\]
	We settle the case $y \geq 0$; the case of not necessarily positive $y$ is obtained via a translation.
	We prove the statement by induction on $n \in \Np$ such that $0 \leq y \leq n$.
	For every $k \in \Z$, we write $x_k$ for $\zeta_M(x)(k) = ((x-k) \lor 0) \land 1$.
	Let us prove the base case $n=1$.
	For every $k \in \N$, we have
	\[
		\zeta_M(x + y) (k) = ((x + y - k) \lor 0) \land 1 \stackrel{\text{\cref{l:sum of good sequences in mon}}}{=} x_{k-1} \odot (x_k\oplus y) = \big(\zeta_M(x) + \zeta_M(y)\big)(k),
	\]
	so the base case is settled.	
	Let us suppose that the case $n$ holds for a fixed $n\in\Np$, and let us prove the case $n+1$.
	So, let us suppose $0 \leq y \leq n+1$.
	Therefore,
	\begin{align*}
		\zeta_M(x + y)	& = \zeta_M\big(x + (y \land n) + ((y-n) \lor 0)\big)\\
							& = \zeta_M\big(x + (y \land n)\big) + \zeta_M\big((y-n) \lor 0\big)			&& \by{base case}\\
							& = \zeta_M(x) + \zeta_M(y \land n) + \zeta_M\big((y-n) \lor 0\big)	&& \by{ind.\ case}\\
							& = \zeta_M(x) + \zeta_M\big((y \land n) + ((y-n) \lor 0)\big)		&& \by{base case}\\
							& = \zeta_M(x) + \zeta_M(y).											&& \qedhere
	\end{align*}
\end{proof}

\begin{proposition}	\label{p:trunc sums up}
	Let $M$ be {\aulm}, let $x\in M$ and let $n\leq m\in \Z$ be such that $n\leq x\leq m$.
	Then,
	\[
		x = n + \sum_{i=n}^{m-1} ((x -i) \lor 0) \land 1.
	\]
\end{proposition}
\begin{proof}
	Let $M$ be {\aulm}, let $x\in M$ and let $n\in \Z$ be such that $n\leq x$.
	We prove, by induction on $m\geq n$, that 
	\begin{equation} \label{e:many-slices}
		x \land m = n + \sum_{i=n}^{m-1} ((x -i) \lor 0) \land 1.
	\end{equation}
	The case $m=n$ is trivial.
	Let us suppose that \cref{e:many-slices} holds for a certain $m \geq n$, and let us prove that the statement holds for $m + 1$.
	We have
	\begin{align*}
		x \land (m + 1)	& = \big((x \land (m + 1)) \land m\big) + \big((x \land (m + 1)) \lor m\big) - m 	&& \\
								& = n + \left(\sum_{i=n}^{m-1} ((x -i) \lor 0) \land 1\right) + \big(((x-m) \lor 0) \land 1\big)	&& \text{(ind. hyp.)} \\
								& = n + \sum_{i=n}^{m} ((x -i) \lor 0) \land 1.	&& \qedhere		
	\end{align*}
\end{proof}

\begin{lemma}
	\label{l:base case}
	Let $M$ be {\aulm}, and let $m\in \N$.
	Then, for every good $\Z$-sequence $(x_0,\dots,x_m)$ in $\Gam(M)$, we have
	\begin{equation*}
	\label{e:land}
		(x_0+\dots +x_m)\land 1=x_0.
	\end{equation*}
\end{lemma}
\begin{proof}
	We prove the statement by induction on $m\in \N$.
	The case $m=0$ is trivial.
	Suppose the statement holds for a fixed $m\in \N$, and let us prove that it holds for $m+1$:
	\begin{align*}
		(x_0 + \dots + x_{m+1}) \land 1
			& = (x_0 + \dots + x_{m + 1}) \land (x_0 + \dots + x_{m - 1} + 1) \land 1	&&	\\
			& = \big(x_0 + \dots + x_{m - 1} + ((x_m + x_{m + 1}) \land 1)\big) \land 1	&&	\\
			& = (x_0 + \dots + x_{m-1} + x_m) \land 1\\
			& = x_0.	&& \text{(ind. hyp.)} \qedhere
	\end{align*}
\end{proof}
\begin{lemma}
	\label{l:lor of goods 1}
	Let $M$ be {\aulm}.
	Then, for every $k\in \N$ and every good $\Z$-sequence $(x_0,\dots, x_{k})$ in $\Gam(M)$, we have
	\begin{equation}
	\label{e:lor ominus 1}
		(x_0+\dots+x_{k})\lor 1=1+x_1+\dots+x_{k}.	
	\end{equation}
\end{lemma}
\begin{proof}	
	We prove this statement by induction on $k\in \N$.
	The case $k=0$ is trivial.
	Let us suppose that the statement holds for a fixed $k\in \N$, and let us prove that it holds for $k+1$.	
	We have
	\begin{align*}
		1 + x_1 + \dots + x_{k+1}	& = (1 + x_1 + \dots + x_{k}) + x_{k+1}							&& \\
											& = ((x_0 + \dots + x_{k}) \lor 1) + x_{k+1}						&& \text{(ind. hyp.)}	\\
											& = (x_0 + \dots + x_{k} + x_{k+1}) \lor (1 + x_{k+1})		&& \\
											& = (x_0 + \dots + x_{k+1}) \lor \big((x_{k} + x_{k+1}) \lor 1\big)	&& \\
											& = \big((x_0 + \dots + x_{k+1}) \lor(x_{k} + x_{k+1})\big) \lor 1	&& \\
											& = (x_0 + \dots + x_{k+1}) \lor 1.									&& \qedhere
	\end{align*}
\end{proof}

\begin{proposition}
	\label{p:unique}
	Let $M$ be {\aulm}, let $\gs{x}$ and $\gs{y}$ be good $\Z$-sequences in $\Gam(M)$.
	Let $n,m \in \Z$ with $n \leq m$ be such that $\gs{x}(k)=\gs{y}(k) = 1$ for all $k < n$, and $\gs{x}(j) = \gs{y}(j) = 0$ for all $j \geq m$.
	If $n+\sum_{i=n}^{m-1} \gs{x}(i)=n+\sum_{i=n}^{m-1} \gs{y}(i)$, then $\gs{x} = \gs{y}$.
\end{proposition}
\begin{proof}
	Without loss of generality, we may suppose $n=0$.
	We prove the statement by induction on $m$.
	The case $m=0$ is trivial.
	Suppose that the statement holds for a fixed $m\in \N$, and let us prove it for $m+1$.
	By \cref{l:base case}, we have
	\[
		\gs{x}(0)=(\gs{x}(0)+\dots+\gs{x}({m+1}))\land 1=(\gs{y}(0)+\dots+\gs{y}({m+1}))\land 1=\gs{y}(0).
	\]
	By \cref{l:lor of goods 1},
	\begin{equation*}
		\begin{split}
			\gs{x}(1) + \dots + \gs{x}({m+1})	& = \big((\gs{x}(0) + \gs{x}(1) + \dots + \gs{x}({m+1})) \lor 1\big) - 1 \\
												& = \big((\gs{y}(0) + \gs{y}(1) + \dots + \gs{y}({m+1})) \lor 1\big) -1 \\
												& = \gs{y}(1) + \dots + \gs{y}({m+1}).
		\end{split}
	\end{equation*}
	By inductive hypothesis, for all $i\in \{1,\dots,m+1\}$, we have $\gs{x}(i)=\gs{y}(i)$.
\end{proof}
\begin{theorem}
	\label{t:exists unique good sequence}
	Let $M$ be {\aulm}.
	The map $\zeta_M \colon M \to \X\Gam(M)$ is bijective, i.e., for every good $\Z$-sequence $\gs{x}$ in $\Gam(M)$ there exists exactly one element $x \in M$ such that, for every $n \in \Z$, we have
	\[
		\gs{x}(n)=((x-n)\lor 0)\land 1.
	\]
\end{theorem}
\begin{proof}
	We construct an inverse $\theta_M$ of $\zeta_M \colon \Z \to \X\Gam(M)$.
	Given a good sequence $\Z$-sequence $\gs{x}$ in $\Gam(M)$, let $n,m \in \Z$ with $n \leq m$ be such that $\gs{x}(k)= 1$ for all $k < n$, and $\gs{x}(j) = 0$ for all $j \geq m$.
	Elements $n$ and $m$ with these properties exist by the definition of good $\Z$-sequence.
	Define $\theta_M(\gs{x}) = n + \sum_{i=n}^{m-1} \gs{x}(i)$; note that this value does not depend on the choice of $n$ and $m$ with the properties above.
	By \cref{p:trunc sums up}, for every $x \in M$, we have $\theta_M(\zeta_M(x)) =x$, i.e.\ the composite $M \xrightarrow{\zeta_M} \X\Gam(M) \xrightarrow{\theta_M} M$ is the identity on $M$.
	By \cref{p:unique}, the map $\theta_M$ is injective.
	From $\theta_M \circ \zeta_M = 1_M$ and the injectivity of $\theta_M$, it follows that $\theta_M$ is the inverse of $\zeta_M$.
\end{proof}
\begin{proposition} \label{p:eps1 is natural}
	For every morphism of {\ulms} $f\colon M \to N$, the following diagram commutes.
	\[
		\begin{tikzcd}
			M \arrow[swap]{d}{f} \arrow{r}{\zeta_M}	& \X\Gam(M) \arrow{d}{\X\Gam(f)}\\
			N \arrow[swap]{r}{\zeta_N}						& \X\Gam(N)
		\end{tikzcd}
	\]
\end{proposition}
\begin{proof}
	For every $x \in M$ and every $k \in \Z$, we have
	\begin{equation*}
		\zeta_N(f(x))(k) = ((f(x) - k) \lor 0) \land 1 = f(((x - k) \lor 0) \land 1) = f(\zeta_M(x)(k)). \qedhere
	\end{equation*}
\end{proof}
%


\subsection{Main result: the equivalence}

\begin{theorem}
	\label{t:G is equiv}	
	The categories $\ULM$ of {\ulms} (see \cref{d:ulm}) and $\MVM$ of {\mvms} (see \cref{d:MVM}) are equivalent, as witnessed by the quasi-inverse functors
	\[
		\begin{tikzcd}
			\ULM\arrow[yshift = .45ex]{r}{\Gam}	& \MVM.\arrow[yshift = -.45ex]{l}{\X\ }
		\end{tikzcd}
	\]
	
\end{theorem}
\begin{proof}
	The functor $\Gam\X\colon \MVM\to \MVM$ is naturally isomorphic to the identity functor on $\MVM$ by \cref{p:unit mvm,p:eta1 is natural}.
	The functor $\X\Gam\colon \ULM\to \ULM$ is naturally isomorphic to the identity functor on $\ULM$ by \cref{p:eps1 is morphism,p:eps1 is natural,t:exists unique good sequence}.
\end{proof}
%


\section{The equivalence specialises to Mundici's equivalence}
	\label{s:restriction}

We recall that an \emph{\extlg} (\emph{\lg}, for short)\index{lattice-ordered!group!Abelian} is an algebra $\langle G; +, \lor, \land, 0, - \rangle$ (arities $2$, $2$, $2$, $0$, $1$) such that $\langle G; \lor, \land \rangle$ is a lattice, $\langle G; +, 0, - \rangle$ is an Abelian group, and $+$ distributes over $\lor$ and $\land$.
It is well-known that the underlying lattice of any {\lg} is distributive \cite[Proposition~1.2.14]{Goodearl1986}.

Furthermore, we recall that a \emph{\extulg} (\emph{\ulg}, for short)\index{lattice-ordered!group!unital} is an algebra $\langle G; +, \lor, \land, 0, -, 1 \rangle$ (arities $2$, $2$, $2$, $0$, $1$, $0$) such that $\langle G; +, \lor, \land, 0, - \rangle$ is {\alg} and $1$ is a strong order unit, i.e.\ $0 \leq 1$ and, for all $x \in M$, there exists $n \in \N$ such that $x \leq n1$.
We let $\ULG$ denote the category of {\ulgs} and homomorphisms.

For all basic notions and results about lattice-ordered groups, we refer to \cite{BKW}.

In every {\ulg} one defines the constant $-1$ as the additive inverse of $1$.
	
\begin{remark} \label{r:ulg-reduct}
	It is not difficult to prove that the $\{+, \lor, \land, 0, 1, -1\}$-reducts of {\ulgs} are precisely the {\ulms} in which every element has an inverse.
	Moreover, the forgetful functor from $\ULG$ to the category of $\{+, \lor, \land, 0, 1, -1\}$-algebras is full, faithful and injective on objects.
	Thus, the category of {\ulgs} is isomorphic to the full subcategory of $\ULM$ given by those {\ulms} in which every element has an inverse.
\end{remark}

We recall that an \emph{\mv}\index{MV-algebra} $\langle A; \oplus, \lnot, 0 \rangle$ is a set $A$ equipped with a binary operation $\oplus$, a unary operation $\lnot$ and a constant $0$ such that $\langle A; \oplus, 0 \rangle$ is a commutative monoid, $\lnot 0 \oplus x = \lnot 0$, $\lnot\lnot x = x$ and $\lnot(\lnot x \oplus y) \oplus y = \lnot(\lnot y \oplus x) \oplus x$.
We let $\MV$ denote the category of {\mvs} with homomorphisms.
For all basic notions and results about {\mvs} we refer to \cite{cdm2000}.

Via $ \oplus, \lnot, 0$, one defines the operations $1 \df \lnot 0$, $x \odot y\df \lnot(\lnot x \oplus \lnot y)$, $x \lor y\df (x \odot \lnot y) \oplus y$, and $x \land y\df x \odot (\lnot x \oplus y)$.

By a result of \cite[Theorem~3.9]{Mundici}, the categories of {\ulgs} and {\mvs} are equivalent.
In this section, we prove that this equivalence follows from \cref{t:G is equiv}.
\begin{lemma}
	\label{l:MV is MVM}
	Every {\mv} is {\amvm}.
\end{lemma}
\begin{proof}
	Since $[0, 1]$ generates the variety of {\mvs} \cite[Theorem~2.3.5]{cdm2000}, it suffices to check that \cref{ax:A1,ax:A2,ax:A3,ax:A4,ax:A5,ax:A6,ax:A7} hold in $[0, 1]$.
	This is the case because $\R$ is easily seen to be a {\ulm} and thus, by \cref{t:Gamma is MVM}, the unit interval $[0, 1]$ is {\amvm}.	
\end{proof}

\begin{proposition} \label{p:reducts-of-MV}
	The reducts of {\mvs} to the signature $\{\oplus, \odot, \lor, \land, 0, 1\}$ are the {\mvms} $A$ such that, for every $x \in A$, there exists $y \in A$ such that $x \oplus y = 1$ and $x \odot y = 0$.
\end{proposition}
\begin{proof}
	If $A$ is {\amv}, then the $\{\oplus, \odot, \lor, \land, 0, 1\}$-reduct of $A$ is {\amvm} by \cref{l:MV is MVM} and, for every $x\in A$, we have $x\oplus \lnot x=1$ and $x\odot \lnot x=0$.
	This settles one direction.
	
	For the converse direction, let $A$ be an {\mvm}, and suppose that, for every $x\in A$, there exists $y\in A$ such that $x\oplus y=1$ and $x\odot y=0$.
	One such element is unique because, if $y,z\in A$ are such that $x\oplus y=1$, $x\odot y=0$, $x\oplus z=1$ and $x\odot z =0$, then
	\[
		y=0\oplus y=(z\odot x)\oplus y\stackrel{\text{\cref{l:almost associative}}}{\geq} z\odot(x\oplus y)=z\odot 1=z,
	\]
	and, analogously, $z\geq y$.
	
	For $x\in A$, we let $\lnot x$ denote the unique element such that $x\oplus \lnot x=1$ and $x\odot \lnot x=0$.
	
	We have $\lnot 0=1$ because $0\oplus 1=1$ and $0\odot 1=1$;
	hence, $x\oplus \lnot 0=x\oplus 1=1=\lnot 0$.
	We have $\lnot\lnot x=x$ because $\lnot x\oplus x = 1 = x\oplus \lnot x$ and $\lnot x\odot x = 0 = x\odot \lnot x$.
	Moreover, we have $x\odot y=\lnot(\lnot x\oplus \lnot y)$ because 
	\begin{align*}
	 	(x \odot y) \oplus(\lnot x \oplus \lnot y)	& \geq	((x\oplus \lnot x)\odot y)\oplus \lnot y	&& \by{\cref{l:almost associative}}\\
																 	& =		(1 \odot y) \oplus \lnot y\\
																 	& =		y \oplus \lnot y\\
																 	& =		1
	\end{align*}	
	(and hence $(x\odot y)\oplus(\lnot x\oplus \lnot y)=1$), and
	\begin{align*}
		(x\odot y)\odot(\lnot x\oplus \lnot y)	& \leq	x\odot ((y\odot \lnot y)\oplus \lnot x)
															&& \by{\cref{l:almost associative}}\\
															& = x \odot (0 \oplus \lnot x)\\
															& = x \odot \lnot x\\
															& = 0
	\end{align*} 
	(and hence $(x\odot y)\odot(\lnot x\oplus \lnot y)=0$).
	Furthermore, we have
	\[
		\sigma(x,  \lnot y,y) = (x \odot (\lnot y \oplus y)) \oplus (\lnot y \odot y) = (x \odot 1) \oplus 0 = x,
	\]
	and thus
	\[
		\lnot(\lnot x\oplus y)\oplus y=(x\odot \lnot y)\oplus y=\sigma(x,\lnot y,y)\lor y=x\lor y.
	\]
	Analogously, $\lnot(\lnot y\oplus x)\oplus x=y\lor x$.
	Therefore, 
	\begin{equation*}
		\lnot(\lnot x\oplus y)\oplus y=x\lor y= y \lor x = \lnot(\lnot y\oplus x)\oplus x. \qedhere
	\end{equation*}
\end{proof}

\begin{remark} \label{r:mv-reduct}
	It is not difficult to see that the forgetful functor from $\MV$ to the category of $\{\oplus, \odot, \lor, \land, 0, 1\}$-algebras is full, faithful and injective on objects.
	Thus, by \cref{p:reducts-of-MV}, the category of {\mvs} is isomorphic to the full subcategory of $\MVM$ given by those {\mvms} $A$ such that, for every $x \in A$, there exists $y \in A$ such that $x \oplus y = 1$ and $x \odot y = 0$.
\end{remark}

\begin{theorem}
	\label{t:restricts to Mundici}
	The equivalence 
	$
		\begin{tikzcd}
			\ULM \arrow[yshift = .45ex]{r}{\Gam}	& \MVM \arrow[yshift = -.45ex]{l}{\X\ }
		\end{tikzcd}
	$
	restricts to an equivalence between $\ULG$ and $\MV$.
\end{theorem}
\begin{proof}
	By \cref{r:ulg-reduct,r:mv-reduct}, it is enough to prove that, for every $M\in\ULM$, the following conditions are equivalent.
	\begin{enumerate}
		\item \label{i:inv} Every element of $M$ is invertible.
		\item \label{i:compl} For every $x\in \Gam(M)$, there exists $y\in \Gam(M)$ such that $x\oplus y=1$ and $x\odot y=0$.
	\end{enumerate}
	Suppose \cref{i:inv} holds.
	Let $x\in \Gam(M)$.
	It is immediate that $1-x \in \Gam(M)$.
	Moreover, we have
	\[
		x\oplus (1-x)=(x+ 1- x)\land 1=1\land 1=1,
	\]
	and
	\[
		x\odot (1-x)=(x+1-x-1)\lor 0=0\lor 0=0.
	\]
	So \cref{i:compl} holds.
	
	Suppose \cref{i:compl} holds.
	We first prove that every element of $\Gam(M)$ has an inverse.
	For $x\in \Gam(M)$, let $y\in \Gam(M)$ be such that $x\oplus y=1$ and $x\odot y=0$.
	The element $y-1$ is the inverse of $x$, because
	\begin{align*}
		x+(y-1)	& = ((x + y - 1) \lor 0) + ((x + y - 1) \land 0)\\
					& = (x \odot y) + (((x + y) \land 1) - 1)\\
					& = (x \odot y) + (x \oplus y) - 1\\
					& = 0 + 1 - 1\\
					& = 0.
	\end{align*}
	Every element of $G$ is invertible since it may be written as a sum of elements of $\Gam(M)\cup\{-1\}$.
\end{proof}
%


\section{Conclusions}

In the direction of obtaining an equational axiomatisation of the dual of the category of compact ordered spaces, we paid special attention to the operations $\oplus$, $\odot$, $\lor$, $\land$, $0$ and $1$.
In order to arrive at a convenient set of axioms to impose on algebras in this signature, we first considered a reasonable set of axioms using the operations $+$, $\lor$, $\land$, $0$, $1$, and $-1$.
Then, we showed that the unit intervals of {\ulms} are the algebras in the signature $\{\oplus, \odot, \lor, \land, 0, 1\}$ that we called {\mvms}.

This result provides us with a good set of axioms to impose on algebras in the signature $\{\oplus, \odot, \lor, \land, 0, 1\}$.
We can now build on these algebras to obtain a duality for $\CompOrd$.
However, before turning to our equational axiomatisation of $\CompOrdop$, we shall still take advantage of the algebras in the signature $\{+, \lor, \land, 0, 1, -1\}$ to clarify the intuitions behind the dualities that we will encounter in \cref{chap:axiomatisation,chap:finite-axiomatisation}.


\chapter{Ordered Yosida duality} \label{chap:yosida-like}


\section{Introduction}

In this chapter the main character is $\R$: we investigate algebras of continuous order-preserving real-valued functions.
We obtain a duality between the category $\CompOrd$ of compact ordered spaces and a certain class of algebras, in the style of \cite{Yosida1941}.
K.\ Yosida gave the essential elements of a proof of the fact that the category of compact Hausdorff spaces is dually equivalent to the category of vector lattices with a unit that are complete in the metric induced by the unit, along with their unit-preserving linear lattice homomorphisms.
Of the several descriptions of the dual of the category of compact Hausdorff spaces that were obtained at around the same time as \cite{Yosida1941}, it is appropriate to mention here the result due to \cite{Stone1941}, where divisible Archimedean lattice-ordered groups with a unit are used in place of vector lattices.

The class of algebras that we use in this chapter to dualise $\CompOrd$ is not equationally definable (not even first-order definable); an equational axiomatisation of $\CompOrdop$ will be obtained in \cref{chap:axiomatisation}.

We conclude this introduction with an outline of the chapter, which is structured around three points.

\paragraph{Few operations are enough.}
Given two topological spaces $X$ and $Y$ equipped with a preorder, we set
\[
	\Cleq(X, Y) \df \{f \colon X \to Y \mid f \text{ is order-preserving and continuous}\}.
\]
For every cardinal $\kappa$, every order-preserving continuous function $f \colon \R^\kappa \to \R$, and every topological space $X$ equipped with a preorder, the set $\Cleq(X,\R)$ is closed under pointwise application of $f$.
So, we have a contravariant functor
\[
	\Cleq(-, \R) \colon \CompOrd \to \ALG \{+, \lor, \land, 0, 1, -1\}.
\]
We show that this functor is full and faithful (\cref{t:abstract-full,p:faithful}), and we deduce that the category of compact ordered spaces is dually equivalent to the category of $\{+, \lor, \land, 0, 1, -1\}$-algebras which are isomorphic to $\Cleq(X, \R)$ for some compact ordered space $X$ (\cref{t:gelfand-like-abstract}).

\paragraph{A more explicit Yosida-like duality.}
We then proceed to obtain a duality with a class of algebras which is more explicitly defined.
We let $\Dyad$ denote the set of dyadic rationals.
In the same spirit as above, we have a contravariant functor from $\CompOrd$ to $\ALG \{+, \lor, \land\} \cup \Dyad$, still denoted by $\Cleq(-, \R)$.
This functor is full and faithful, and so the category of compact ordered spaces is dually equivalent to the category of $(\{+, \lor, \land\} \cup \Dyad)$-algebras which are isomorphic to $\Cleq(X, \R)$ for some compact ordered space $X$.
Next, we characterise these algebras as the algebras $M$ in the signature $\{+, \lor, \land\} \cup \Dyad$ such that the function
\begin{align*}
	\disthom[M] \colon	M \times M	& \longrightarrow	[0,+ \infty]\\
								(x,y)			& \longmapsto		\sup_{f \in \hom(M, \R)} \lvert f(x) - f(y) \rvert
\end{align*}
is a metric, and $M$ is Cauchy complete with respect to it (\cref{t:char-algebras}).
We conclude that the category of compact ordered spaces is dually equivalent to the category of algebras with these properties (\cref{t:one-char}).

\paragraph{An intrinsic definition of $\disthom$.}
Next, we provide a more intrinsic definition of the function $\disthom$, as follows.
We define {\adlm} as an algebra $\alge{M}$ in the signature $\{ + , \lor, \land\}\cup\Dyad$ such that $\langle M; + , \lor, \land,0\rangle$ is {\alm}, for all $x \in M$ there exist $\alpha, \beta \in \Dyad$ such that $\alpha^\alge{M} \leq x \leq \beta^\alge{M}$, and, for all $\alpha, \beta \in \Dyad$, we have (i) if $\alpha \leq \beta$, then $\alpha^\alge{M} \leq \beta^\alge{M}$, and (ii) $\alpha^\alge{M} +^\alge{M} \beta^\alge{M} = (\alpha +^\R \beta)^\alge{M}$.
With the help of Birkhoff's subdirect representation theorem, we prove that, for every {\dlm} $M$, the function $\disthom[M]$ coincides with the function
\begin{align*}
	\distint[M] \colon	M \times M	& \longrightarrow	[0, +\infty)\\
								(x, y)		& \longmapsto		\inf\left\{t \in \Dyad\cap [0, +\infty) \mid y + (-t)^M \leq x \leq y + t^M\right\}.
\end{align*}
We conclude that the category of compact ordered spaces is dually equivalent to the category of {\dlms} $M$ that satisfy
\[
	\distint[M](x,y) = 0 \Rightarrow x = y
\]
(so that $\distint[M]$ is a metric) and are Cauchy-complete with respect to $\distint[M]$ (\cref{t:Yosida}).


\section{Few operations are enough}

In this section, we show that the category of compact ordered spaces is dually equivalent to the class of algebras in the signature $\{+, \lor, \land, 0, 1, -1\}$ which are isomorphic to $\Cleq(X, \R)$ for some compact ordered space $X$ (\cref{t:gelfand-like-abstract}).
To prove it, we define a contravariant functor $\Cleq(-, \R) \colon \CompOrd \to \ALG \{+, \lor, \land, 0, 1, -1\}$, which maps a compact ordered space $X$ to the algebra $\Cleq(X, \R)$ of order-preserving continuous functions from $X$ to $\R$, and we show that this functor is full and faithful.


\subsection{The contravariant functor \texorpdfstring{$\Cleq({-}, \R)$}{C\unichar{"02264}(\unichar{"02013},\unichar{"0211D})}: definition}

As already mentioned in the outline of the chapter, given two topological spaces $X$ and $Y$ equipped with a preorder, we set
\[
	\Cleq(X, Y) \df \{f \colon X \to Y \mid f \text{ is order-preserving and continuous}\}.
\]
For every cardinal $\kappa$, every order-preserving continuous function $f \colon \R^\kappa \to \R$, and every topological space $X$ equipped with a preorder, the set $\Cleq(X,\R)$ is closed under pointwise application of $f$.

\begin{remark}\label{r:all operations in V are mon and cont}
	The functions $+, \lor, \land \colon \R^2 \to \R$, and all function from $\R^0$ to $\R$ are continuous and order-preserving with respect to the product order and product topology.
\end{remark}

We have a contravariant functor
\[
	\Cleq(-, \R) \colon \CompOrd \to \ALG \{+, \lor, \land, 0, 1, -1\},
\]
which maps a compact ordered space $X$ to the algebra $\Cleq(X, \R)$ of order-preserving continuous functions from $X$ to $\R$ with pointwise defined operations, and which maps a morphism $f \colon X \to Y$ to the homomorphism $- \circ f \colon \Cleq(Y, \R) \to \Cleq(X, \R)$.


\subsection{The contravariant functor \texorpdfstring{$\Cleq({-}, \R)$}{C\unichar{"02264}(\unichar{"02013},\unichar{"0211D})} is faithful}

\begin{lemma} \label{l:cogenerator}
	Let $X$ and $Y$ be compact ordered spaces, and let $f$ and $g$ be order-preserving continuous maps from $X$ to $Y$.
	If $f \neq g$, then there exists an order-preserving continuous function $h \colon Y \to \R$ such that $h \circ f \neq h \circ g$.
\end{lemma}
\begin{proof}
	By hypothesis, there exists $x \in X$ such that $f(x) \neq g(x)$.
	Without loss of generality, we may suppose $f(x) \not\geq g(x)$.
	By \cref{l:cor of urysohn}, there exists an order-preserving continuous function $h \colon Y \to [0,1]$ such that $h(f(x)) = 0$ and $h(g(x)) = 1$.
\end{proof}

\begin{proposition} \label{p:faithful}
	The contravariant functor 
	\[
		\Cleq(-,\R) \colon \CompOrd \to \ALG \{+, \lor, \land\} \cup \Dyad
	\]
	is faithful.
\end{proposition}
\begin{proof}
	By \cref{l:cogenerator}.
\end{proof}


\subsection{The contravariant functor \texorpdfstring{$\Cleq({-}, \R)$}{C\unichar{"02264}(\unichar{"02013},\unichar{"0211D})} is full}

The results in the remaining part of this section are essentially an adaptation of the results in \cite{HN2018}.

\begin{lemma}[{\cite[Proposition~5, p.\ 45]{Nachbin}}] \label{l:compord-sandwich}
	Let $X$ be a compact ordered space, let $F$ be an up-set and let $V$ be an open neighbourhood of $F$.
	Then there exists an open up-set $W$  such that $F \seq W \seq V$.
\end{lemma}

\begin{lemma}\label{l:separation-upsets}
	Let $A$ and $B$ be disjoint closed up-sets of a compact ordered space.
	Then, there exist two disjoint open up-sets that contain $A$ and $B$ respectively.
\end{lemma}

\begin{proof}
	Using the fact that any compact ordered space is compact and Hausdorff and thus normal (see \cite[Theorem~17.10]{Willard1970}), we obtain that there exist two disjoint open neighbourhoods $U'$ and $V'$ of $A$ and $B$ respectively.
	By \cref{l:compord-sandwich}, there exists an open up-set $U$ such that $A \seq U \seq U'$; again by \cref{l:compord-sandwich}, there exists an open up-set $V$ such that $B \seq V \seq V'$.
	The sets $U$ and $V$ satisfy the desired properties.
\end{proof}

\begin{lemma} \label{l:separate-up}
	Let $x$ and $y$ be elements of a compact ordered space with no common upper bound.
	Then there exist two disjoint open up-sets that contain $x$ and $y$ respectively.
\end{lemma}
\begin{proof}
	By \cref{l:up-set-closed}, the sets $\upset x$ and $\upset y$ are closed.
	Then, apply \cref{l:separation-upsets}.
\end{proof}

\begin{notation} \label{n:AA}
	Given a compact ordered space $X$ and a map $\Phi \colon \Cleq(X, \R) \rightarrow \R$,	we set
	\[
		\DD(\Phi)\df\bigcap_{\psi \in \Cleq(X, \R)}\{y \in X \mid \psi(y) \leq \Phi(\psi)\}.
	\]
\end{notation}

Our goal is, given a map $\Phi \colon \Cleq(X, \R) \rightarrow \R$ that preserves $ + , \lor, \land$ and every real number, to find $x \in X$ such that $\Phi$ is the evaluation at $x$, i.e., for every $\psi \in \Cleq(X, \R)$, we have $\Phi(\psi) = \psi(x)$.
To do so, we show that $\DD(\Phi)$ has a maximum element $x$, and we then show that $\Phi(\psi) = \psi(x)$ for every $\psi \in \Cleq(X, \R)$.

\begin{lemma}\label{l:DD-closed}
	Let $X$ be a compact ordered space, and let $\Phi \colon \Cleq(X, \R) \rightarrow \R$ be a map that preserves $+$, $\lor$, $\land$ and every real number.
	Then, the set $\DD(\Phi)$ is a closed down-set of $X$.
\end{lemma}

\begin{proof}
	For every $\psi \in \Cleq(X, \R)$, the set $\{y \in X \mid \psi(y) \leq \Phi(y)\}$ is closed.
	Thus, $\DD(\Phi)$ is an intersection of closed subsets of $X$, and hence $\DD(\Phi)$ is closed.
	To prove that $\DD(\Phi)$ is a down-set of $X$, let $y \in \DD(\Phi)$ and let $x \leq y$.
	Then, for every $\psi \in \Cleq(X, \R)$, we have $\psi(x) \leq \psi(y)$ by monotonicity of $\psi$, and we have $\psi(y) \leq \Phi(\psi)$ by definition of $\DD(\Phi)$; hence, $\psi(x) \leq \Phi(\psi)$, and thus $x \in \DD(\Phi)$.
\end{proof} 

\begin{lemma} \label{l:A=phi}
	Let $X$ be a compact ordered space and let $\Phi \colon \Cleq(X, \R) \rightarrow \R$ be a map that preserves $ + , \lor, \land$ and every real number.
	Then, for every $t \in \R$ we have
	\[
		\DD(\Phi) = \bigcap_{\psi \in \Cleq(X, \R) : \Phi(\psi) = t} \{y \in X \mid \psi(y) \leq t \}.
	\]
\end{lemma}

\begin{proof} 
	\begin{enumerate}[wide]
		
		\item [{[$\seq$]}]
		Let $y \in \DD(\Phi)$.
		For every $\psi \in \Cleq(X, \R)$ such that $\Phi(\psi) = t$, we have $\psi(y) \leq \Phi(\psi) = t$, where the inequality follows from the fact that $y \in \DD(\Phi)$.
		
		\item [{[$\supseteq$]}]
		Let $x \in \bigcap_{\psi \in \Cleq(X, \R) : \Phi(\psi) = t} \{y \in X \mid \psi(y) \leq t \}$.
		Let $\psi \in \Cleq(X, \R)$.
		We shall prove $\psi(x) \leq \Phi(\psi)$.
		Define the function
		\begin{align*}
			\psi'	\colon	X	& \longrightarrow	\R\\
								x	& \longmapsto		\psi(x) - \Phi(\psi) + t.
		\end{align*}
		The function $\psi'$ is order-preserving and continuous because $\psi$ is such.
		Then 
		\begin{align*}
			\Phi(\psi')	& = \Phi(\psi - \Phi(\psi) + t)	&& \text{(by def.\ of $\psi'$)}\\
							& = \Phi(\psi) - \Phi(\psi) + t	&& \text{(by hyp.\ on $\Phi$)}\\
							& = t.									&&
		\end{align*}
		By hypothesis on $x$, it follows that $\psi'(x) \leq t$, i.e.\ $\psi(x) - \Phi(\psi) + t \leq t$, i.e.\ $\psi(x) \leq \Phi(\psi)$, as was to be proved. \qedhere
		
	\end{enumerate}
\end{proof}

The following is inspired by \cite[Proposition 6.12]{HN2018}.

\begin{lemma} \label{l:phi-is-sup}
	Let $X$ be a compact ordered space and let $\Phi \colon \Cleq(X, \R) \rightarrow \R$ be a map that preserves $ + , \lor, \land$ and every real number.
	Then, for all $\psi \in \Cleq(X, \R)$, we have
	\begin{equation} \label{e:sup-DD}
		\Phi(\psi) = \max_{y \in \DD(\Phi)} \psi(y).
	\end{equation}
\end{lemma}

\begin{proof}
	Let $\psi \in \Cleq(X, \R)$.
	We shall prove that $\Phi(\psi)$ is the maximum of the image $\psi[\DD(\Phi)]$ of $\DD(\Phi)$ under $\psi$.
	By \cref{l:DD-closed}, the set $\DD(\Phi)$ is a closed subset of the compact space $X$.
	Hence, $\DD(\Phi)$ is compact.
	Thus, $\psi[\DD(\Phi)]$ is compact, as well.
	Therefore, it is enough to prove
	\[
	\Phi(\psi) = \sup_{y \in \DD(\Phi)} \psi(y).
	\]
	\begin{enumerate}[wide]
		
		\item [{[$\geq$]}]
		By definition of $\DD(\Phi)$, for every $y \in \DD(\Phi)$, we have $\psi(y) \leq \Phi(\psi)$.
		
		\item [{[$\leq$]}]
		Given $z \in \R$ and a subset $Z$ of $\R$, the condition $z \leq \sup Z$ holds if, and only if, for every real upper bound $t$ of $Z$ and every $\eps > 0$ we have $z \leq t + \eps$.
		We apply this observation with $z \df \Phi(\psi)$ and $Z \df \{\psi(y) \mid y \in \DD(\Phi)\}$.
		So, we let $t \in \R$ be such that, for every $y \in \DD(\Phi)$, we have $\psi(y) \leq t$, and we let $\eps > 0$.
		We shall prove $\Phi(\psi) \leq t + \eps$.
		Set
		\[
			U \df \left\{y \in X \mid \psi(y) < t + \eps \right\}.
		\]
		Clearly, $U$ is open and $\DD(\Phi) \seq U$.
		Since $X$ is compact, the image of $\psi$ is a compact subset of $\R$, and so it admits an upper bound $M \in \R$.
		For every $\psi' \in \Cleq(X, \R)$ we set
		\[
			s(\psi')\df \{x \in X \mid \psi'(x) > M\}.
		\]
		\begin{claim*}
			We have 
			\begin{equation*}
				X = U \cup \bigcup_{\psi' \in \Cleq(X, \R) : \Phi(\psi') \leq t} s(\psi').
			\end{equation*}
		\end{claim*}
		\begin{claimproof}
			For every $x \in X \setminus \DD(\Phi)$, there exists, by Lemma \ref{l:A=phi}, an element $\widetilde{\psi} \in \Cleq(X, \R)$ such that $\Phi(\widetilde{\psi}) = t$ and $\widetilde{\psi}(x) > t$.
			Let $n \in \N$ and $k \in \Z$ be such that $n(\Phi(\widetilde{\psi})) + k \leq t$ and $n(\widetilde{\psi}(x)) + k> M$ (such $n$ and $k$ do exist because $\widetilde{\psi}(x) > t$).
			Set $\psi' \df n\widetilde{\psi} + k$.
			Then, $\psi' \in \Cleq(X, \R)$, $\Phi(\psi') \leq t$ and $\psi'(x) > M$.
		\end{claimproof}
		Since $X$ is compact, there exist $\psi_1, \dots, \psi_n \in \Cleq(X, \R)$ with $\Phi(\psi_i) \leq t$ for all $i \in \{1, \dots, n\}$ such that
		\[
			X = U \cup s(\psi_1) \cup \dots \cup s(\psi_n).
		\]
		Therefore, for all $x \in X$, either $x \in U$, i.e.\ $\psi(x) < t + \eps$, or there exists $j \in \{1, \dots, n\}$ such that $x \in s(\psi_j)$, i.e.,\ $\psi_j(x) > M$, and thus $\psi_j(x) \geq \psi(x)$.
		Hence, we have
		\[
			\psi \leq (t + \eps) \lor \psi_1\lor \dots \lor \psi_n.
		\]
		Therefore, we have
		\begin{align*}
			\Phi(\psi)	& \leq (t + \eps) \lor \Phi(\psi_1) \lor \dots \lor \Phi(\psi_n)\\
							& = (t + \eps) \lor t \lor \dots \lor t\\
							& = t + \eps. \qedhere
		\end{align*}
	\end{enumerate}
\end{proof} 

\begin{lemma} \label{l:directed}
	Let $X$ be a compact ordered space, and let $\Phi \colon \Cleq(X, \R) \rightarrow \R$ be a map that preserves $+ , \lor, \land$ and every real number.
	Then, $\DD(\Phi)$ is directed.
\end{lemma}

\begin{proof}
	The set $\DD(\Phi)$ is non-empty by \cref{l:phi-is-sup}.
	
	Let $x$ and $y$ be elements of $\DD(\Phi)$.
	We claim that $x$ and $y$ admit a common upper bound which belongs to $\DD(\Phi)$.
	Suppose, by way of contradiction, that this is not the case.
	By \cref{l:DD-closed}, the set $\DD(\Phi)$ is a closed subspace of $X$; hence, $\DD(\Phi)$ is a compact ordered space.
	By \cref{l:separate-up}, there exist two disjoint open up-sets $U$ and $V$ of $\DD(\Phi)$ that contain $x$ and $y$ respectively.
	By \cref{l:up-set-closed}, the sets $\upset x$ and $\upset y$ are closed.
	Hence, by \cref{t:Urysohn}, there exist an order-preserving continuous function 
	\[
		\psi_1 \colon \DD(\Phi) \to [0,1]
	\]
	such that for all $z \in \DD(\Phi) \setminus V$ we have $\psi_1(z) = 0$, and for all $z \in \upset y$ we have $\psi_1(z) = 1$, and an order-preserving continuous function
	\[
		\psi_2 \colon \DD(\Phi) \to [0,1]
	\]
	such that for all $z \in \DD(\Phi) \setminus U$ we have $\psi_2(z) = 0$, and for all $z \in \upset x$ we have $\psi_2(z) = 1$.
	Using again the fact that $\DD(\Phi)$ is closed (\cref{l:DD-closed}), the functions $\psi_1$ and $\psi_2$ can be extended to order-preserving continuous functions on $X$ by \cref{l:Nac-reg-inj}.
	Then, by \cref{l:phi-is-sup}, we have $\Phi(\psi_1) = \max_{z \in \DD(\Phi)} \psi_1(z) = 1$ and $\Phi(\psi_2) = \max_{z \in \DD(\Phi)} \psi_2(z) = 1$.
	Since $\Phi$ preserves $\land$, we have
	\begin{equation} \label{e:phi-sum-1}
		\Phi(\psi_1 \land \psi_2) = \Phi(\psi_1) \land \Phi(\psi_2) = 1 \land 1 = 1.
	\end{equation}
	By \cref{l:phi-is-sup}, we have also 
	\begin{equation} \label{e:phi-sum-2}
		\Phi(\psi_1 \land \psi_2) = \max_{z \in \DD(\Phi)} \psi_1(z) \land \psi_2 (z).
	\end{equation}
	Since $U$ and $V$ are disjoint subsets of $\DD(\Phi)$, for every $z \in \DD(\Phi)$ we have either $z \in \DD(\Phi) \setminus V$ (and thus $\psi_1(z) = 0$), or $z \in \DD(\Phi) \setminus U$ (and thus $\psi_2(z) = 0$).
	In any case, for $z \in \DD(\Phi)$, we have $\psi_1(z) \land \psi_2 (z) = 0$.
	It follows that the right-hand side of \cref{e:phi-sum-2} equals $0$, and therefore also the left-hand side: $\Phi(\psi_1 \land \psi_2) = 0$.
	This contradicts the equality $\Phi(\psi_1 \land \psi_2) = 1$ obtained in \cref{e:phi-sum-1}.
	This settles our claim that $x$ and $y$ admit a common upper bound which belongs to $\DD(\Phi)$, and this concludes the proof.
\end{proof}

\begin{lemma} \label{l:DD-max}
	Let $X$ be a compact ordered space, and let $\Phi\colon \Cleq(X, \R) \rightarrow \R$ be a map that preserves $\lor, \land, + $ and every element in $\R$.
	Then $\DD(\Phi)$ admits a maximum element.
\end{lemma}

\begin{proof}
	Every directed set in a compact ordered space has a supremum, which coincides with the topological limit of the set regarded as a net \cite[Proposition~VI.1.3]{ContLattDom}.
	Since $\DD(\Phi)$ is directed (\cref{l:directed}), the set $\DD(\Phi)$ admits a supremum $x$ which coincides with the topological limit of the set regarded as a net.
	Since $\DD(\Phi)$ is closed (\cref{l:DD-closed}), the element $x$ belongs to $\DD(\Phi)$.
\end{proof}

\begin{theorem}\label{t:full}
	Let $X$ be a compact ordered space, and let $\Phi\colon \Cleq(X, \R) \rightarrow \R$ be a map that preserves $\lor, \land, +$ and every real number.
	Then, there exists a unique $x_0 \in X$ such that $\Phi$ is the evaluation at $x_0$, i.e.\ such that, for every $\psi \in \Cleq(X, \R)$, we have $\Phi(\psi) = \psi(x_0)$.
\end{theorem}

\begin{proof}
	Uniqueness follows from \cref{l:cor of urysohn}.	
	Let us prove existence.
	By \cref{l:DD-max}, $\DD(\Phi)$ admits a maximum element $x_0$.
	By \cref{l:phi-is-sup}, we have, for all $\psi \in \Cleq(X, \R)$,
	\[
		\Phi(\psi) = \max_{z \in \DD(\Phi)}\psi(z) = \psi(x_0),
	\]
	where the last equality holds because $x_0$ is the maximum of $\DD(\Phi)$ and $\psi$ is order-preserving.
\end{proof}

\begin{theorem}\label{t:morphism-is-evaluation}
	Let $X$ be a compact ordered space, and let $\Phi\colon \Cleq(X, \R)\rightarrow \R$ be a map that preserves $\lor$, $\land$, $+$, $0$, $1$, and $-1$.
	Then, there exists a unique $x_0 \in X$ such that $\Phi$ is the evaluation at $x_0$, i.e.\ such that, for every $\psi \in \Cleq(X, \R)$, we have $\Phi(\psi) = \psi(x_0)$.
\end{theorem}

\begin{proof} 
	By \cref{t:full}, it is enough to prove that $\Phi$ preserves every real number.
	It is immediate that $\Phi$ preserves every element of $\Z$.
	Let $\frac{p}{q}$ be any rational number.
	Then, we have $q\Phi(\frac{p}{q}) = \Phi(q\frac{p}{q}) = \Phi(p) =p$, which implies $\Phi(\frac{p}{q}) = \frac{p}{q}$.
	Thus, $\Phi$ preserves every rational number.
	By monotonicity of $\Phi$ and order-density of $\Q$, the function $\Phi$ preserves every real number.
\end{proof} 

We note that this fact specialises to compact Hausdorff spaces.
To this end, let $\Cont(X, Y)$ denote the set of continuous functions.

\begin{corollary}
	Let $X$ be a compact Hausdorff space, and let $\Phi\colon \Cont(X, \R)\rightarrow \R$ be a map that preserves $\lor$, $\land$, $+$, $0$, $1$, and $-1$.
	Then, there exists a unique $x_0 \in X$ such that $\Phi$ is the evaluation at $x_0$, i.e.\ such that, for every $\psi \in \Cont(X, \R)$, we have $\Phi(\psi) = \psi(x_0)$.
\end{corollary}

\begin{lemma} \label{l:initial-for-morphisms}
	A function $f \colon X \to Y$ between compact ordered spaces is order-preserving and continuous if, and only if, for every order-preserving continuous function $h \colon Y \to \R$, the composite $h \circ f$ is order-preserving and continuous.
\end{lemma}
\begin{proof}
	This follows from the fact that every compact ordered space embeds in a power of $[0,1]$ (see \cref{l:closed-sub-power}).
\end{proof}

\begin{theorem} \label{t:abstract-full}
	The contravariant functor
	\[
		\Cleq(-, \R) \colon \CompOrd \to \ALG \{+, \lor, \land, 0, 1, -1\}
	\]
	is full.
\end{theorem}
\begin{proof}
	Let $X$ and $Y$ be compact ordered spaces and let $s \colon \Cleq(Y, \R) \to \Cleq(X, \R)$ be a function that preserves $+$, $\lor$, $\land$, $0$, $1$ and $-1$.
	We define a function $\overline{s} \colon X \to Y$.
	Given an element $x \in X$, we define the function
	\begin{align*}
		\ev_x \colon	\Cleq(X, \R)	& \longrightarrow	\R\\
							f					& \longmapsto		f(x).
	\end{align*}
	The function $\ev_x$ preserves $+$, $\lor$, $\land$, $0$, $1$ and $-1$.
	Therefore, also the composite $\Cleq(Y, \R) \xrightarrow{s} \Cleq(X, \R) \xrightarrow{\ev_x}$ preserves $+$, $\lor$, $\land$, $0$, $1$ and $-1$.
	Hence, by \cref{t:morphism-is-evaluation}, given an element $x \in X$, there exists a unique element $y \in Y$ such that, for every $\psi \in \Cleq(Y, \R)$, we have $(\ev_x \circ s)(\psi) = \psi(y)$, i.e.\ $s(\psi)(x) = \psi(y)$.
	This defines the unique function $\overline{s}\colon X \to Y$ such that, for every $\psi \in \Cleq(X, \R)$, we have 
	\begin{equation} \label{e:ev}
		s(\psi) = \psi \circ \overline{s}.
	\end{equation}
	Let us prove that $\overline{s}$ is order-preserving and continuous.
	By \cref{l:initial-for-morphisms}, to prove that $\overline{s}$ is a morphism, it is enough to prove that, for every order-preserving continuous function $h \colon Y \to \R$, the composite $h \circ \overline{s}$ is order-preserving and continuous.
	Let $h \colon Y \to \R$ be an order-preserving continuous function.
	Then, for every $x \in X$, we have
	\[
		(h \circ \overline{s})(x)  =  h(\overline{s}(x))  \stackrel{\text{\cref{e:ev}}}{=}  s(h)(x).
	\]
	Hence, $h \circ \overline{s} = s(h)$, and this proves that $h \circ \overline{s}$ is order-preserving and continuous.
	Therefore, $\overline{s}$ is order-preserving and continuous.
	By \cref{e:ev}, the functor $\Cleq(-, \R) \colon \CompOrd \to \ALG \{+, \lor, \land, 0, 1, -1\}$ maps the morphism $\overline{s}$ to $s$.
\end{proof}

\begin{theorem} \label{t:gelfand-like-abstract}
	The category of compact ordered spaces is dually equivalent to the category of algebras in the signature $\{+, \lor, \land, 0, 1, -1\}$ which are isomorphic to $\Cleq(X, \R)$ for some compact ordered space $X$.
\end{theorem}
\begin{proof}
	The contravariant functor
	\[
		\Cleq(-, \R) \colon \CompOrd \to \ALG \{+, \lor, \land, 0, 1, -1\}
	\]
	is faithful by \cref{p:faithful} and full by \cref{t:abstract-full}.
\end{proof}


\section{Ordered Stone-Weierstrass theorem}

In the next section, we will obtain a duality between $\CompOrd$ and a class of algebras whose definition is more intrinsic than the definition of the class of algebras described in \cref{t:gelfand-like-abstract}.
Before moving to such a duality, we obtain an ordered version of the classical Stone-Weierstrass theorem, which will suggest us a signature for the corresponding algebras.

In 1885, K.\ Weierstra{\ss} proved an approximation theorem for continuous functions over the unit interval $[0,1]$ \cite{Weierstrass1885}.
Later, M.\ H.\ Stone provided a simpler proof of Weierstra{\ss}' approximation theorem and proved some additional similar results \cite{Stone1937,Stone1948}, which fall under the name of \emph{Stone-Weierstrass theorem}\index{Stone-Weierstrass theorem}.
We recall that a set of functions $F$ from a set $X$ to $Y$ \emph{separates the elements of $X$}\index{separation} (or is \emph{separating}) if, for all $x, y \in X$ such that $x \neq y$, there exists $f \in F$ such that $f(x) \neq f(y)$.
One version of the classical Stone-Weierstrass theorem asserts: If $X$ is a compact Hausdorff space and $L$ is a real vector subspace of the set of continuous functions from $X$ to $\R$ which is closed under $\lor$ and $\land$, which contains the constant function $1$ and which separates the elements of $X$, then every continuous function from $X$ to $\R$ is the uniform limit of a sequence in $L$ \cite[Theorem~7.29]{HS75}\footnote{As noted in the cited reference, the hypothesis are slightly redundant: if $X$ admits a separating family of continuous real-valued functions, then $X$ has to be a Hausdorff space.}.
We would like to obtain an analogous result, where compact Hausdorff spaces are replaced by compact ordered spaces.
To do so, we rely on the following theorem, which characterises in a simple way the closure under uniform limits of a sublattice of the lattice of real-valued continuous functions over a compact space\footnote{The theorem is essentially due to \cite[Theorem~1]{Stone1948}, who proved the implication \cref{i:unif-limit} $\Leftrightarrow$ \cref{i:Stone}. I would like to thank V.\ Marra and L.\ Spada for a discussion which highlighted the equivalent \cref{i:up-down}.}.

\begin{theorem}\label{t:StWe-figo-2-options}
Let $X$ be a compact space, let $L$ be a set of continuous functions from $X$ to $\R$ that is closed under $\lor$ and $\land$, and suppose that either $X$ or $L$ is non-empty.
For every function $f \colon X \to \R$, the following conditions are equivalent.
\begin{enumerate}
	\item \label{i:unif-limit} The function $f$ is a uniform limit of a sequence in $L$.
	\item \label{i:Stone} The function $f$ is continuous and, for all $x, y \in X$ and all $\eps > 0$, there exists $g \in L$ such that $\lvert f(x) - g(x) \rvert < \eps$ and $\lvert f(y) - g(y) \rvert < \eps$.
	\item \label{i:up-down} The function $f$ is continuous, for all $z \in X$ the value $f(z)$ belongs to the topological closure of $\{h(z) \mid h \in L\}$, and, for all distinct $x, y \in X$ and all $\eps > 0$, there exists $g \in L$ such that $g(x) > f(x) - \eps$ and $g(y) < f(y) + \eps$.
\end{enumerate}
\end{theorem}
\begin{proof}
The equivalence \cref{i:unif-limit} $\Leftrightarrow$ \cref{i:Stone} is due to M.\ H.\ Stone \cite[Theorem~1]{Stone1948}.
The implication \cref{i:Stone} $\Rightarrow$ \cref{i:up-down} is trivial.
Let us prove the implication \cref{i:up-down} $\Rightarrow$ \cref{i:Stone}.
So, let $f$ be a continuous function, suppose that, for all $x, y \in X$ and all $\eps > 0$, there exists $g \in L$ such that $f(x) < g(x) + \eps$ and $f(y) > g(y) - \eps$, and suppose that, for all $z \in X$, $f(z)$ belongs to the closure of $\{h(z) \mid h \in L\}$.
The desired conclusion holds for $x = y$ because in this case, by hypothesis on $f$, there exists $h \in L$ such that $\lvert f(x) - h(x) \rvert < \eps$.
So, we are left with the case $x \neq y$.
By hypothesis on $f$, there exists $h \in L$ such that $\lvert f(x) - h(x) \rvert < \eps$, and there exists $k \in L$ such that $\lvert f(y) - k(y) \rvert < \eps$.
We have either $k(x) \geq f(x)$ or $k(x) \leq f(x)$, and we have either $h(y) \geq f(y)$ or $h(y) \leq f(y)$.
Thus, we have four cases: 
\begin{enumerate}[wide]
	\item Case $k(x) \geq f(x)$ and $h(y) \geq f(y)$.
	In this case we obtain the desired conclusion by taking $g = h \land k$.
	\item Case $k(x) \leq f(x)$ and $h(y) \leq f(y)$.
	In this case we obtain the desired conclusion by taking $g = h \lor k$.
	\item Case $k(x) \leq f(x)$ and $h(y) \geq f(y)$.
	Since $x \neq y$, by hypothesis there exists $l \in L$ such that $l(x) > f(x) - \eps$ and $l(x) < f(x) + \eps$.
	Then, we obtain the desired conclusion by taking $g = (h \land l) \lor k$.
	\item Case $k(x) \geq f(x)$ and $h(y) \leq f(y)$.
	This case is similar to the previous one. \qedhere
\end{enumerate}
\end{proof}


\subsubsection{Consequences in the ordered case}

\begin{proposition} \label{p:with-constants}
	Let $X$ be a compact ordered space, and let $L$ be a subset of $\Cleq(X, \R)$ which is closed under $\lor$ and $\land$ and which contains every constant in a dense subset of $\R$.
	Then, the set of uniform limits of sequences in $L$ is $\Cleq(X, \R)$ if, and only if, for every $x, y \in X$ with $x \not\geq y$, and every $s,t \in \R$ with $s < t$, there exists $g \in L$ such that $g(x) \leq s$ and $g(y) \geq t$.
\end{proposition}
\begin{proof}
	The left-to-right implication is a consequence of the ordered Urysohn's lemma (\cref{t:Urysohn}).
	Let us prove the right-to-left implication.
	Since $L$ is contained in $\Cleq(X, \R)$, and $\Cleq(X, \R)$ is closed under uniform limits, every uniform limit of a sequence in $L$ is in $\Cleq(X, \R)$.
	For the opposite direction, let $f \in \Cleq(X, \R)$: we shall prove that $f$ is a uniform limit of a sequence in $L$.
	Let $x, y \in X$, and let $\eps > 0$.
	\begin{claim}
		There exists a function $g \in L$ such that $g(x) < f(x) + \eps$ and $g(y) > f(y) - \eps$.
	\end{claim}
	\begin{claimproof}
		If $f(x) \geq f(y)$, then, since $L$ contains every constant in a dense subset of $\R$, there exists an element $d \in \R$ such that the function $g$ which is constantly equal to $d$ belongs to $L$, and $ f(y) - \eps < d  < f(x) + \eps$.
		The function $g$ has the desired properties.
		
		We are left with the case $f(x) < f(y)$.
		Since $f$ is order-preserving, we deduce $x \not\geq y$.
		Therefore, by hypothesis, there exists a function $g \in L$ such that $g(x) \leq f(x)$ and $g(y) \geq f(y)$.
		The function $g$ has the desired properties.
	\end{claimproof}
	Therefore, applying the implication \cref{i:up-down} $\Rightarrow$ \cref{i:unif-limit} in \cref{t:StWe-figo-2-options}, we conclude that $f$ is a uniform limit of a sequence in $L$.
\end{proof}

An ordered version of separation is the following.

\begin{definition}
	Let $X$ and $Y$ be preordered sets, and let $F$ be a set of order-preserving functions from $X$ to $Y$.
	We say that $F$ \emph{order-separates the elements of $X$}, or that $F$ is \emph{order-separating}\index{order-separation}\index{separation!order-} if, for all $x, y \in X$ such that $x \not\geq y$, there exists $f \in F$ such that $f(x) \not\geq f(y)$.
\end{definition}

By contraposition, and using the fact that the elements of $F$ are order-preserving, the condition of order-separation above is equivalent to the following one:
\[
	\forall x, y \in X \ \ (x \leq y \Leftrightarrow \forall f \in F \  f(x) \leq f(y)).
\]
This means that the source $(f \colon X \to Y)_{f \in F}$ is initial with respect to the forgetful functor $\Preord \to \Set$ (cf.\ \cref{s:init-preorder}).

When the orders of both $X$ and $Y$ are discrete (i.e., the identity), order-separation is equivalent to separation, because $\not\geq$ coincides with $\neq$.

\begin{lemma} \label{l:dilate}
	Let $X$ be a compact ordered space, and let $L$ be an order-separating subset of $\Cleq(X, \R)$ which is closed under $\lor$ and $\land$ and contains every constant in a dense subset of $\R$.
	Suppose that, for all $a,b,c,d \in \R$ with $a \leq b < c \leq d$, there exists a function $\theta \colon \R \to \R$ such that $\theta(b) \leq a$, $\theta(c) \geq d$, and $L$ is closed under $\theta$.
	Then, the set of uniform limits of sequences in $L$ is $\Cleq(X, \R)$.
\end{lemma}
\begin{proof}
	By \cref{p:with-constants}.
\end{proof}

\begin{lemma} \label{l:expansion}
	For all $a,b,c,d \in \R$ with $a \leq b < c \leq d$, and for every dense subset $D$ of $\R$, there exists $n \in \Np$ and $u \in D$ such that $nb +u \leq a$ and $nc + u \geq d$.
\end{lemma}
\begin{proof}
	This is a well known fact.
\end{proof}

\begin{theorem}[Ordered Stone-Weierstrass theorem]\label{t:StWe}\index{Stone-Weierstrass theorem!ordered}
	Let $X$ be a compact ordered space, let $L$ be an order-separating\footnote{The hypothesis are slightly redundant: if $X$ is a compact space equipped with a partial order, and $X$ admits an order-separating family of continuous order-preserving real-valued functions, then $X$ has to be a compact ordered space.} set of continuous order-preserving functions from $X$ to $\R$ which is closed under $+$, $\lor$, $\land$ and which contains every constant in a dense subset of $\R$.
	Then, the set of uniform limits of sequences in $L$ is the set of continuous order-preserving functions from $X$ to $\R$.
\end{theorem}
\begin{proof}
	By \cref{l:dilate,l:expansion}.
\end{proof}


\section{Description as Cauchy complete algebras}

We now proceed to obtain a duality between the category $\CompOrd$ of compact ordered spaces and a class of algebras which, compared to the one in \cref{t:gelfand-like-abstract}, is more explicitly defined.
The choice of the signature is suggested by the ordered version of the Stone-Weierstrass theorem (\cref{t:StWe}).
We will apply the theorem with the dense subset of $\R$ in the statement being the set of dyadic rationals (i.e.\ the elements of the form $\frac{k}{2^n}$ for $k \in \Z$ and $n \in \N$), denoted by $\Dyad$.

\begin{remark} \label{r:dyad-full}
	By \cref{l:cogenerator}, the contravariant functor 
	\[
		\Cleq(-, \R) \colon \CompOrd \to \ALG \{+, \lor, \land\} \cup \Dyad
	\]
	is faithful.
	From \cref{t:abstract-full} we deduce that it is also full.
	It follows that the category of compact ordered spaces is dually equivalent to the category of algebras in the signature $\{+, \lor, \land\} \cup \Dyad$ which are isomorphic to $\Cleq(X, \R)$ for some compact ordered space $X$.
\end{remark}

We next characterise these algebras.

\begin{notation}
	Given an algebra $M$ in the signature $\{+, \lor, \land\} \cup \Dyad$, we define the function
	\begin{align*}
		\disthom[M] \colon	M \times M	& \longrightarrow	[0,+ \infty]\\
									(x,y)			& \longmapsto		\sup_{f \in \hom(M, \R)} \lvert f(x) - f(y) \rvert.
	\end{align*}
\end{notation}

We recall that, given two algebras $A$ and $B$ in a common signature, we have a canonical homomorphism
\begin{align*}
	\ev \colon	A	& \longrightarrow B^{\hom(A,B)}\\
					x	& \longmapsto		\ev_x,
\end{align*}
where $\ev_x$ is defined by
\begin{align*}
	\ev_x \colon	\hom(A,B)	& \longrightarrow	B\\
						f				& \longmapsto		f(x).
\end{align*}

\begin{remark} \label{r:uniform}
	Let $M$ be an algebra in the signature $\{+, \lor, \land\} \cup \Dyad$.
	Denoting with $\ev$ the canonical homomorphism from $M$ to $\R^{\hom(M,\R)}$, we note that, for all $x,y \in M$, the element $\disthom[M](x,y)$ coincides with the value assumed on the pair $(\ev(x), \ev(y))$ by the uniform possibly $\infty$-valued metric on $\R^{\hom(M,\R)}$.
\end{remark}

\begin{lemma}
	For every algebra $M$ in the signature $\{+, \lor, \land\} \cup \Dyad$, the function $\disthom[M]$ satisfies positivity, symmetry and triangle inequality.
\end{lemma}
\begin{proof}
	The uniform possibly $\infty$-valued metric on $\R^{\hom(M,\R)}$ satisfies positivity, symmetry and triangle inequality.
	By \cref{r:uniform}, the function $\disthom[M]$ possesses these properties, as well.
\end{proof}

\begin{remark}
	The function $\disthom[M]$ is not guaranteed to be a metric: indeed, it might fail to be finite-valued and to satisfy the identity of indiscernibles property\index{identity of indiscernibles}, i.e.\ 
	\[
		\forall x \ \forall y\ \disthom[M](x,y) = 0 \Longrightarrow x = y.
	\]
	For example, the function $\disthom[M]$ is not finite-valued for $M = \R^\N$.
	Moreover, the function $\disthom[M]$ does not satisfy the identity of indiscernibles property for $M = \R \stackrel{\to}{\times} \R$, i.e., the lexicographic product of $\R$ and $\R$ (see \cref{i:lexico} in \cref{ex:ulms}) on which the interpretation of a dyadic rational $t$ is $(t, 0)$.
\end{remark}

\begin{lemma} \label{l:dist-finite}
	For every compact ordered space $X$, the function 
	\[
		\disthom[\Cleq(X, \R)] \colon \Cleq(X, \R) \times \Cleq(X, \R) \to \R
	\]
	is the uniform metric.
\end{lemma}
\begin{proof}
	By \cref{t:morphism-is-evaluation}, any function $f \colon \Cleq(X, \R) \to \R$ that preserves $+$, $\lor$, $\land$, $0$, $1$ and $-1$ is the evaluation at an element $z \in X$.
	Moreover, evaluations at elements of $X$ are homomorphisms with respect to the signature $\{+, \lor, \land\} \cup \Dyad$.
	It follows that, for every $x, y \in \Cleq(X, \R)$, we have 
	\[
		\disthom[\Cleq(X, \R)] = \sup_{f \in \hom(M, \R)} \lvert f(x) - f(y) \rvert = \sup_{z \in X} \lvert x(z) - y(z) \rvert,
	\]
	and this last term is the value of the uniform metric at $(x,y)$.
	The fact that $\disthom[\Cleq(X, \R)]$ is finite-valued is guaranteed by the fact that $X$ is compact, and that continuous functions map compact sets to compact sets (\cref{p:image-of-compact}).
\end{proof}

\begin{lemma}\label{l:injective}
	The following conditions are equivalent for an algebra $M$ in the signature $\{+, \lor, \land\} \cup \Dyad$.
	\begin{enumerate}
		\item The canonical homomorphism $\ev \colon M \to \R^{\hom(M, \R)}$ is injective.
		\item The function $\disthom[M]$ satisfies the identity of indiscernibles property.
	\end{enumerate}
\end{lemma}
\begin{proof}
	Injectivity of $\ev$ is expressed by the condition 
	\[
		\forall x\ \forall y\  \ev(x) = \ev(y) \Longrightarrow x = y.
	\]
	The identity of indiscernibles property of $\disthom[M]$ is expressed by the condition
	\[
		\forall x\ \forall y\  \disthom[M](x,y) = 0 \Longrightarrow x = y.
	\]
	For every $x, y \in M$, the condition $\ev(x) = \ev(y)$ holds if, and only if, for every homomorphism $f$ from $M$ to $\R$ we have $\ev(x)(f) = \ev(y)(f)$, i.e.\ $f(x) = f(y)$; in turn, the latter condition is equivalent to
	$\sup_{f \in \hom(M, \R)} \lvert f(x) - f(y) \rvert = 0$, i.e.\ $\disthom[M] = 0$.
	It follows that injectivity of $\ev$ and the identity of indiscernibles property of $\disthom$ are equivalent.
\end{proof}

We recall that $\TopXPreord$ denotes the category whose objects are topological spaces equipped with a preorder, and whose morphisms are the continuous order-preserving maps.
With a bit of theory on natural dualities (see \cite[Proposition~5.4]{HN2018}), one shows that we have two adjoint contravariant functors 
\[
	\begin{tikzcd}
		\TopXPreord \arrow[yshift = .45ex]{r}{\Cleq(-,\R)}	& \ALG \{+, \lor, \land\} \cup \Dyad, \arrow[yshift = -.45ex]{l}{\hom(-, \R)}
	\end{tikzcd}
\]
with the evaluation functions as units.
For an algebra $M$ in the signature $\{+, \lor, \land\} \cup \Dyad$, the set $\hom(M, \R)$ is endowed with the initial topology and order with respect to the structured source of evaluation maps $(\ev_x \colon \hom(M, \R) \to \R)_{x \in M}$, or, equivalently, the induced topology and order with respect to the inclusion $\hom(M, \R) \seq \R^M$.

\begin{lemma} \label{l:order-closed}
	The order on $\hom(M, \R)$ is closed and the topology is Hausdorff.
\end{lemma}
\begin{proof}
	Every subset of a power of $\R$ has a closed order (by \cref{l:char-closed}) and a Hausdorff topology.
\end{proof}

\begin{lemma} \label{l:closed}
	The set $\hom(M, \R)$ is a closed subset of $\R^M$.
\end{lemma}
\begin{proof}
	The idea is that $\hom(M, \R)$ is closed because it is defined by the equations that express the preservation of primitive operation symbols of $\DLM$.
	
	To make this precise, set $\mathcal{L} = \{+, \lor, \land\} \cup \Dyad$.
	For each $h \in \mathcal{L}$, we let $\ar h$ denote the arity of $h$; moreover, we let $h^M$ denote the interpretation of $h$ in $M$, and we let $h_{\R}$ denote the interpretation of $h$ in $\R$.
	For $a \in M$, we let $\pi_a\colon \R^M\rightarrow \R$ denote the projection onto the $a$-th coordinate (which is continuous).
	We have
	\begin{align*}
		& \hom(M, \R)\\
		& = \bigcap_{h \in \mathcal{L}} \bigcap_{a_1, \dots,a_{\ar h} \in M} \left\{x \colon M \rightarrow \R \mid x\mathopen{}\left(h^M(a_1, \dots,a_{\ar h})\right)\mathclose{}  =  h^\R(x(a_1), \dots, x(a_{\ar h}))\right\}\\
		& = \bigcap_{h \in \mathcal{L}} \bigcap_{a_1, \dots, a_{\ar h} \in M} \left\{x \in \R^M \mid \pi_{h^M(a_1, \dots, a_{\ar h})}(x)  =  h^\R\mathopen{}\left(\pi_{a_1}(x), \dots, \pi_{a_{\ar h}}(x)\right)\mathclose{}\right\}.
	\end{align*}
	By \cref{r:all operations in V are mon and cont}, the function $h^\R$ is continuous; therefore, the function from $\R^M$ to $\R$ which maps $x$ to $ h^\R(\pi_{a_1}(x), \dots, \pi_{a_{\ar h}}(x))$ is continuous.
	Since $\R$ is Hausdorff, the set
	\[
		\left\{x \in \R^M\mid \pi_{h^M(a_1, \dots, a_{\ar h})}(x) = h^\R\mathopen{}\left(\pi_{a_1}(x), \dots, \pi_{a_{\ar h}}(x)\right)\mathclose{}\right\}
	\]
	is closed.
\end{proof}

Concerning the proof of \cref{l:closed} above, the reader can compare \cite[Proposition~2.3(a)]{LambekRattray1979}, whose application shows that $\hom(M, \R)$ is an equaliser of a power of $\R$.

\begin{lemma} \label{l:char-compact-dist}
	For every algebra $M$ in the signature $\{+, \lor, \land\} \cup \Dyad$, the topology on $\hom(M, \R)$ is compact (or, equivalently, $\hom(M, \R)$ is a compact ordered space) if, and only if, the function $\disthom[M]$ is finite-valued.
\end{lemma}
\begin{proof}
	Suppose $\hom(M, \R)$ is compact.
	Then, for every $x \in M$, the projection 
	\begin{align*}
		\hom(M, \R)	& \longrightarrow \R\\
		f				& \longmapsto 		f(x)
	\end{align*}
	onto the $x$-th coordinate has a compact image, since the image of a compact set under a continuous map is compact by \cref{p:image-of-compact}.
	Therefore, for every $x \in M$, there exists $t \in \R$ such that, for every homomorphism $f \colon M \to \R$, we have $\lvert f(x) \rvert \leq t$.
	Let $x, y \in M$.
	Let $t, s \in \R$ be such that, for every homomorphism $f \colon M \to \R$, we have $\lvert f(x) \rvert \leq t$ and $\lvert f(y) \rvert \leq s$.
	By the previous discussion, $t$ and $s$ exist with these properties.
	Then, for every homomorphism $f \colon M \to \R$, we have
	\[
		\lvert f(x) - f(y) \rvert \leq \lvert f(x) \rvert + \lvert f(y) \rvert \leq t + s.
	\]
	Thus, 
	\[
		\disthom[M](x,y) = \sup_{f \in \hom(M, \R)} \lvert f(x) - f(y) \rvert \leq t+s < \infty.
	\]
	Hence, $\disthom[M]$ is finite-valued.
	
	For the converse direction, let us suppose that $\disthom[M]$ is finite-valued.
	Then, for every $x \in M$, we have 
	\[
		\infty > \disthom[M](x,0) = \sup_{f \in \hom(M, \R)} \lvert f(x) - f(0) \rvert = \sup_{f \in \hom(M, \R)} \lvert f(x)\rvert.
	\]
	Therefore, for every $x \in M$, the image of the projection $\hom(M, \R)$ onto the $x$-th coordinate is bounded.
	Therefore, there exists a family $(t_x)_{x \in M}$ of real numbers such that $\hom(M, \R) \seq \prod_{x \in M}[-t_x, t_x]$.
	By Thychonoff's theorem, the space $\prod_{x \in M}[-t_x, t_x]$ is compact.
	By \cref{l:closed}, $\hom(M, \R)$ is a closed subspace of $\R^M$, and thus it is a closed subspace of $\prod_{x \in M}[-t_x, t_x]$. 
	Since a closed subspace of a compact space is compact (\cref{p:closed-in-compact}), $\hom(M, \R)$ is compact.
	
	By \cref{l:order-closed}, the order of $\hom(M, \R)$ is closed.
	Therefore, $\hom(M, \R)$ is a compact ordered space if, and only if, it is compact.
\end{proof}

\begin{remark} \label{r:cont-mon-ev}
	Every element in the image of the canonical homomorphism 
	\[
		\ev \colon M \to \R^{\hom(M, \R)}
	\]
	is an order-preserving continuous function.
\end{remark}

\begin{lemma} \label{l:order-separating}
	Let $M$ be an algebra in the signature $\{+, \lor, \land\} \cup \Dyad$.
	The image of the canonical homomorphism $\ev \colon M \to \R^{\hom(M, \R)}$ is order-separating.
\end{lemma}
\begin{proof}
	Let $f, g \in \hom(M, \R)$, and suppose $f \not \leq g$.
	Then, there exists $x \in M$ such that $f(x) \not\leq g(x)$, i.e.\ $\ev_x(f) \not \leq \ev_x(g)$.
\end{proof}

\begin{theorem}\label{t:char-algebras}
	The following conditions are equivalent for an algebra $M$ in the signature $\{+, \lor, \land\} \cup \Dyad$.
	\begin{enumerate}
		\item \label{i:some} The algebra $M$ is isomorphic to $\Cleq(X, \R)$ for some compact ordered space $X$.
		\item \label{i:dist} The function $\disthom[M]$ is a metric and $M$ is Cauchy complete with respect to it.
		\item \label{i:ev} The canonical homomorphism $\ev \colon M \to \R^{\hom(M, \R)}$ is injective, $\hom(M, \R)$ is a compact ordered space, and the image of $\ev$ is $\Cleq(\hom(M, \R), \R)$.
	\end{enumerate}
\end{theorem}
\begin{proof}	
	Suppose \cref{i:some} holds.
	Then, by \cref{l:dist-finite}, $\disthom[\Cleq(X, \R)]$ is a metric; precisely, $\disthom[\Cleq(X, \R)]$ is the uniform metric.
	Since uniform limit of continuous functions is continuous and pointwise limits of order-preserving functions is order-preserving, it follows that $\Cleq(X, \R)$ is Cauchy complete with respect to $\disthom[\Cleq(X, \R)]$.
	\Cref{i:dist} follows.
	
	Suppose now \cref{i:dist} holds.
	By \cref{l:injective}, the function $\ev$ is injective.
	By \cref{l:char-compact-dist}, $\hom(M, \R)$ is a compact ordered space.
	By \cref{r:cont-mon-ev}, every element in the image of $\ev$ is an order-preserving continuous function.
	By \cref{l:order-separating}, the image of $\ev$ is order-separating.
	Therefore, by the ordered version of the Stone-Weierstrass theorem (\cref{t:StWe}), the closure under uniform convergence of the image of $\ev$ is $\Cleq(\hom(M, \R), \R)$.
	We claim that the function $\disthom[{\ev[M]}]$ is the uniform metric.
	It is enough to prove that, for every homomorphism $f \colon \ev[M] \to \R$, there exists $g \in \hom(M, \R)$ such that, for every $h \in \ev[M]$, we have $f(h) = h(g)$.
	Indeed, if $f \colon \ev[M] \to \R$ is a homomorphism, we define $g$ as the composite $f \circ \ev \colon M \to \R$; then, for every $h = \ev(x)$, we have 
	\[
		f(h) = f(\ev(x)) = (f \circ \ev)(x) = g(x) = \ev_x(g) = \ev(x)(g) = h(g).
	\]
	This proves our claim that the function $\disthom[{\ev[M]}]$ is the uniform metric.
	Since $M$ is Cauchy complete with respect to $\disthom[M]$, $\ev[M]$ is Cauchy complete with respect to $\disthom[{\ev[M]}]$, i.e.\ the uniform metric.
	Therefore, $\ev[M] = \Cleq(\hom(M, \R), \R)$.
	In conclusion, \cref{i:ev} holds.
	
	It is immediate that \cref{i:ev} implies \cref{i:some}.
\end{proof}

\begin{theorem} \label{t:one-char}
	The category of compact ordered spaces is dually equivalent to the category of algebras $M$ in the signature $\{+, \lor, \land\} \cup \Dyad$ such that $\disthom[M]$ is a metric and $M$ is Cauchy complete with respect to it.
\end{theorem}
\begin{proof}
	By \cref{r:dyad-full}, the category of compact ordered spaces is dually equivalent to the category of $(\{+, \lor, \land\} \cup \Dyad)$-algebras which are isomorphic to $\Cleq(X, \R)$ for some compact ordered space $X$.
	By the equivalence between \cref{i:some,i:dist} in \cref{t:char-algebras}, these algebras are precisely the algebras such that $\disthom[M]$ is a metric and $M$ is Cauchy complete with respect to it.
\end{proof}


\section{An intrinsic definition of the metric}

Next, we replace the function $\disthom$ with a function which has a more intrinsic definition.

\begin{definition}\label{d:dyadic-ell}
	A \emph{\dlm}\index{lattice-ordered!monoid!dyadic} is an algebra $\alge{M}$ in the signature $\{ + , \lor, \land\}\cup\Dyad$ (where $+$, $\lor$ and $\land$ have arity $2$, and each element of $\Dyad$ has arity $0$) with the following properties.
	\begin{enumerate}[label = DM\arabic*., ref = DM\arabic*, start = 0, leftmargin=\widthof{LTE4'.} + \labelsep]
		\item \label[axiom]{ax:R0} $\langle M; + , \lor, \land,0\rangle$ is {\alm} (see \cref{d:lm}).
		\item \label[axiom]{ax:R1} For all $\alpha, \beta \in \Dyad$ with $\alpha \leq \beta$ we have $\alpha^\alge{M} \leq \beta^\alge{M}$.
		\item \label[axiom]{ax:R2} For all $\alpha, \beta \in \Dyad$ we have $\alpha^\alge{M} +^\alge{M} \beta^\alge{M} = (\alpha +^\R \beta)^\alge{M}$.
		\item \label[axiom]{ax:R3} For all $x \in M$, there exist $\alpha, \beta \in \Dyad$ such that $\alpha^\alge{M} \leq x \leq \beta^\alge{M}$.
	\end{enumerate}
\end{definition}

\Cref{ax:R1,ax:R2,ax:R3} say that the constants in $\Dyad$ behave (with respect to $+$, $\lor$, $\land$) exactly as in the algebra $\R$ (or $\Dyad$).
In other words, \cref{ax:R1,ax:R2,ax:R3} are (equivalent to) the positive atomic diagram of $\Dyad$ (cf. e.g.\ \cite[Section~3.2, p.\ 211]{Prest2003}).
Notice that \cref{ax:R3} is equivalent to the \cref{ax:U3} of order unit for {\ulms} (\cref{d:lm}).

The sets $\R$ and $\Dyad$ are given the structure of {\adlm} in an obvious way.

We let $\Dyad_{\geq 0}$ denote the set $\Dyad \cap [0, \infty)$.

\begin{notation} \label{n:dist}
	Given {\adlm} $M$ and given $x,y \in M$ we set
	\[
		\distint[M](x, y) \df \inf\left\{t \in \Dyad_{\geq 0} \mid y + (-t)^M \leq x \leq y + t^M\right\}.
	\]
	When $M$ is understood, we write simply $\distint(x, y)$ for $\distint[M](x, y)$.
\end{notation}

\begin{remark} \label{r:dist-on-R}
	On $\R$, the function $\distint$ is the euclidean distance, i.e., for every $x, y \in \R$, we have
	\[
		\distint[\R](x,y) = \lvert x - y \rvert.
	\]
\end{remark}

\begin{remark}
	On algebras of real-valued functions, the function $\distint$ is the uniform metric.
	Indeed, if $X$ is a set, and $f$ and $g$ are functions from $X$ to $\R$, we have
	\begin{equation*}
		\begin{split}
			\inf \{t \in \Dyad_{\geq 0} \mid g - t \leq f \leq g + t\}
			& = \inf \{t \in \Dyad_{\geq 0} \mid -t \leq f -g \leq t\}\\
			& = \inf \{t \in \Dyad_{\geq 0} \mid \lvert f - g\rvert \leq t\}\\
			& = \sup \lvert f - g \rvert;
		\end{split}
	\end{equation*}
	therefore, if $M \seq \R^X$ is {\adlm}, and $f$ and $g$ are elements of $M$, we have
	\[
		\distint[M](f, g) = \sup_{x \in X}\lvert f(x) - g(x) \rvert.
	\]
\end{remark}

\begin{lemma}
	For every {\dlm} $M$, the function $\distint[M]$ is finite-valued.
\end{lemma}
\begin{proof}
	Let $x, y \in M$.
	By \cref{ax:R3}, there exists $\alpha \in \Dyad_{\geq 0}$ such that $(-\alpha)^M \leq x \leq \alpha^M$ and $(-\alpha)^M \leq y \leq \alpha^M$.
	Then, we have
	\[
		y + (-2\alpha)^M \leq \alpha^M + (-2 \alpha)^M = (- \alpha)^M \leq x \leq \alpha^M = (-\alpha)^M + (2 \alpha)^M \leq y + (2 \alpha)^M.
	\]
	Thus, $\distint[M](x) \leq 2 \alpha$.
\end{proof}

\begin{proposition} \label{p:Repn}
	Every subdirectly irreducible {\lm} is totally ordered.
\end{proposition}

\begin{proof}
	This is a corollary of \cite[Lemma 1.4]{Repn}, but already in \cite[Corollary 2]{Merlier} it is proved that any {\lm} is a subdirect product of totally ordered ones, and it is asserted, in Remark~3 of the same paper, that this was an unpublished result by L.\ Fuchs.
\end{proof}

\begin{corollary} \label{c:irreducible}
	All subdirectly irreducible {\dlms} are totally ordered.
\end{corollary}

\begin{proof}
	This is an immediate consequence of \cref{p:Repn}, since the constants do not play any role in subdirect irreducibility.
\end{proof}

\begin{lemma}\label{l:non-triv}
	Let $M$ be a non-trivial {\dlm}. For all $\alpha, \beta \in \Dyad$, we have
	\begin{align}
		\label{e:const-=}	\alpha^M = \beta^M		& \Leftrightarrow \alpha = \beta;\\
		\label{e:const-leq} \alpha^M \leq \beta^M	& \Leftrightarrow \alpha \leq \beta.
	\end{align}
\end{lemma}

\begin{proof}
	Let us first prove \cref{e:const-=}.
	Clearly, if $\alpha = \beta$, then $\alpha^M = \beta^M$.
	For the converse direction, let us suppose $\alpha^M = \beta^M$ and suppose, by way of contradiction, that we have $\alpha \neq \beta$.
	Without loss of generality, we may suppose $\alpha < \beta$.
	Set $t \df \beta - \alpha$.
	Then, $t > 0$.
	We have
	\[
		0^M = \alpha^M + (-\alpha)^M = \beta^M + (-\alpha)^M = (\beta - \alpha)^M = t^M.
	\]
	For every $s \in \Dyad_{\geq 0}$, there exists $n \in \N$ such that $s \leq n t$; then, we have 
	\[
		0^M \leq s^M \leq n t^M = n 0^M = 0^M.
	\]
	Hence, for every $s \in \Dyad_{\geq 0}$, we have $s^M = 0$.
	Analogously, we have $t^M = 0$ for every $t \in \Dyad$ with $t \leq 0$, as well.
	Thus, for every $t \in \Dyad$, we have $t^M = 0$.
	Let $x \in M$.
	By \cref{ax:R3}, there exist $\alpha, \beta \in \Dyad$ such that $\alpha^M \leq x \leq \beta^M$.
	Then, we have $0^M = \alpha^M \leq x \leq \beta^M = 0^M$, which implies $x = 0^M$.
	Therefore, every element of $M$ equals $0^M$.
	This contradicts the hypothesis of non-triviality of $M$.
	We have thus proved \cref{e:const-=}.
	
	By \cref{ax:R1}, if $\alpha \leq \beta$, then $\alpha^M \leq \beta^M$.
	For the converse direction, let us suppose $\alpha^M \leq \beta^M$, and suppose, by way of contradiction, that $\alpha > \beta$.
	Then, we would have $\alpha^M \geq \beta^M$.
	Thus, we would have $\alpha^M = \beta^M$.
	Therefore, by \cref{e:const-=}, we have $\alpha = \beta$: a contradiction.
\end{proof}

\begin{notation}
	For an element $x$ of a {\dlm} $M$ we set
	\begin{align*}
		I_x			& \df \left\{t \in \Dyad \mid t^M \leq x\right\};\\
		S_x			& \df \left\{t \in \Dyad \mid t^M \geq x\right\};\\
		\essinf x	& \df \sup I_x;\\
		\esssup x	& \df \inf S_x.
	\end{align*}
	When $\essinf x = \esssup x$, we let $\ess x$ denote this number.
\end{notation}

\begin{lemma} \label{l:ess}
	Given a totally ordered non-trivial {\dlm} $M$, the values $\essinf x$ and $\esssup x$ are finite and coinciding.
\end{lemma}

\begin{proof}
	Since every element is bounded from above and below by some dyadic rational (\cref{ax:R3} in \cref{d:dyadic-ell}), the sets $I_x$ and $S_x$ are non-empty.
	Since $M$ is totally ordered, $I_x \cup S_x = \Dyad$.
	Since $M$ is non-trivial, the intersection of $I_x$ and $S_x$ has at most an element, by \cref{l:non-triv}.
	Therefore, the values $\essinf x$ and $\esssup x$ are finite and coinciding.
\end{proof}

\begin{lemma} \label{l:constant-ess}
	Given a non-trivial {\dlm} $M$, for each $t \in \Dyad$ we have $\ess(t^M) = t$.
\end{lemma}

\begin{proof}
	By \cref{l:non-triv}, since $M$ is non-trivial, we have $I_{t^M} = \Dyad \cap (\infty, t]$, and therefore $\essinf t^M = t$.
	Analogously for $\esssup$.
\end{proof}

\begin{proposition} \label{p:hom-to-R}
	Given a totally ordered non-trivial {\dlm} $M$, there exists a unique homomorphism from $M$ to $\R$, namely $x \mapsto \ess x$.
\end{proposition}

\begin{proof}
	By \cref{l:ess}, for every $x \in M$ the values $\essinf x$ and $\esssup x$ are finite and coinciding.
	Hence, the function $\ess \colon M \to \R;\, x \mapsto \ess x$ is well defined.
	Let us prove that this function is a homomorphism.
	The function $\ess$ preserves every constant symbol in $\Dyad$ by \cref{l:constant-ess}.
	Let $x, y \in M$ and let $\otimes$ denote any operation among $\{+, \lor, \land\}$.
	Let $\alpha \in I_x$ (i.e.\ $\alpha^M \leq x$) and $\beta \in I_y$ (i.e.\ $\beta^M \leq y$).
	Then, by monotonicity of $\otimes$, we have $\alpha^M \otimes \beta^M \leq x \otimes y$.
	Since $\alpha^M \otimes \beta^M = (\alpha \otimes \beta)^M$, we then have $I_x \otimes I_y \seq I_{x \otimes y}$.
	Therefore, $\sup(I_x \otimes I_y) \leq \sup(I_{x \otimes y})$.
	Since $\otimes \colon \R^2 \rightarrow \R$ is continuous, we have, for every non-empty $U \seq \R$ and $V \seq \R$, that $\sup (U \otimes W) = (\sup U) \otimes (\sup W)$.
	By the axiom of order unit (\cref{ax:U3}), the set $I_x$ is non-empty.
	Therefore,
	\[
		\ess x \otimes \ess y = \sup I_x \otimes \sup I_y = \sup(I_x \otimes I_y) \leq \sup(I_{x \otimes y}) = \ess(x \otimes y).
	\] 
	Replacing, in the proof above, the set $I_x$ with $S_x$, the symbol $\sup$ with $\inf$ and reversing the inequalities, one obtains that the opposite inequality $\ess{x} \otimes \ess{y} \geq \ess(x \otimes y)$ holds.
	
	Uniqueness follows from the fact that every homomorphism preserves dyadic rationals and is order-preserving.
\end{proof}

\begin{lemma} \label{l:hom-geq}
	Let $M$ and $N$ be {\dlms}, and let $\varphi \colon M \to N$ be a homomorphism.
	Then, for all $x, y \in M$ we have
	\[
		\distint[M](x,y)  \geq \distint[N](\varphi(x), \varphi(y)).
	\]
\end{lemma}
\begin{proof}
	The proof is straightforward.
\end{proof}

\begin{lemma} \label{l:dist-ess}
	For all $x$ and $y$ in a totally ordered non-trivial {\dlm}, we have
	$\distint(x, y) = \lvert \ess{x} - \ess{y} \rvert$.
\end{lemma}

\begin{proof}
	Let $M$ be a totally ordered non-trivial {\dlm}, and let $x, y \in M$.
	It is easy to see that, given a homomorphism $\varphi \colon A \to B$ between {\dlms}, and given $a, b \in A$, we have $\distint[A](a, b) \geq \distint[B](\varphi(a),\varphi(b))$.
	By \cref{p:hom-to-R}, the function $\ess \colon M \to \R$ is a homomorphism, and therefore
	\[
		\distint[M](x, y) \stackrel{\text{\cref{l:hom-geq}}}{\geq} \distint[\R](\ess{x},\ess{y}) \stackrel{\text{\cref{r:dist-on-R}}}{=} \lvert \ess{x} - \ess{y} \rvert.
	\]
	
	Let us prove the opposite inequality.
	Since $\distint[M]$ is a pseudometric, we can apply the triangle inequality to obtain
	\begin{equation} \label{e:dist-ineq}
		\distint[M](x, y) \leq \distint[M]\mathopen{}\left(x, (\ess{x})^M\right)\mathclose{} + \distint[M]\mathopen{}\left((\ess{x})^M, (\ess{y})^M\right)\mathclose{} + \distint[M]\mathopen{}\left((\ess{y})^M, y\right)\mathclose{}.
	\end{equation}
	It is easily seen that we have 
	\[
		\distint[M]\mathopen{}\left(x, (\ess{x})^M\right)\mathclose{} = 0,
	\]
	\[
		\distint[M]\mathopen{}\left((\ess{x})^M, (\ess{y})^M\right)\mathclose{} = \lvert \ess{x} - \ess{y} \rvert,
	\]
	and
	\[
		\distint[M]\mathopen{}\left((\ess{y})^M, y\right)\mathclose{} = 0.
	\]
	Therefore, the right-hand side of \cref{e:dist-ineq} equals  $\lvert \ess{x} - \ess{y} \rvert$, and we obtain the desired inequality.
\end{proof}

\begin{theorem} \label{t:dist-two-ways}
	For every {\dlm} $M$, we have
	\[
		{\distint[M]} = {\disthom[M]}.
	\]
\end{theorem}
\begin{proof}
	Let $I$ be the set of congruences $\theta$ on $M$ such that $M/\theta$ is subdirectly irreducible.
	Let $\iota \colon M \to \prod_{\theta \in I} M/\theta$ be the canonical homomorphism.
	By Birkhoff's subdirect representation theorem, $\iota$ is injective.
	For each $\theta \in I$, let $\pi_\theta \colon M \to M/\theta$ denote the quotient map.
	Since $\iota$ is injective, for all $x, y \in M$ we have
	\begin{equation} \label{e:subdirect}
		\distint[M](x, y) = \sup_{\theta \in I}\distint[M/\theta](\pi_\theta(x),\pi_\theta(y)).
	\end{equation}
	For every $\theta \in I$, the algebra $M/\theta$ is subdirectly irreducible; hence $M/\theta$ is totally ordered (\cref{c:irreducible}) and non-trivial.
	By \cref{l:dist-ess}, for every $\theta \in I$ we have
	\begin{equation} \label{e:dist-pres}
		\distint[M/\theta](\pi_\theta(x),\pi_\theta(y)) = \lvert \ess{\pi_\theta(x)} - \ess{\pi_\theta(y)} \rvert.
	\end{equation}
	Hence, by \cref{e:subdirect,e:dist-pres}, we have
	\begin{equation} \label{e:supess}
		\distint[M](x, y) = \sup_{\theta \in I}\lvert \ess{\pi_\theta(x)} - \ess{\pi_\theta(y)} \rvert.
	\end{equation}
	By \cref{p:hom-to-R}, for every $\theta \in I$, the function ${\ess} \colon M/\theta \to \R$ is a morphism of {\dlms}.
	Hence, from \cref{e:supess} we deduce
	\begin{equation} \label{e:dist-leq}
		\distint[M](x, y)  \leq  \sup_{f \in \hom(M, \R)} \lvert f(x) - f(y) \rvert.
	\end{equation}
	For every homomorphism $f \colon M \to \R$ and all $x, y \in M$, we have
	\[
		\distint[M](x, y) \stackrel{\text{\cref{l:hom-geq}}}{\geq} \distint[\R](f(x),f(y)) \stackrel{\text{\cref{r:dist-on-R}}}{=} \lvert f(x) - f(y) \rvert;
	\]
	thus, the opposite inequality of \cref{e:dist-leq} holds.
\end{proof}

\begin{definition}
	An algebra $M$ in the signature $\{+, \lor, \land\} \cup \Dyad$ is \emph{Archimedean}\index{Archimedean}\index{lattice-ordered!monoid!Archimedean} if $M$ is isomorphic to a subalgebra of a power of $\R$ with obviously defined operations.
	The following are two other conditions which are easily seen to be equivalent.
	\begin{enumerate}
		\item The canonical homomorphism $M \to \R^{\hom(M, \R)}$ is injective.
		\item For all $x,y \in A$ with $x \neq y$ there exists a homomorphism $f \colon M \to \R$ such that $f(x) \neq f(y)$.
	\end{enumerate}
\end{definition}

\begin{theorem} \label{t:char-Arch-dlms}
	\xmakefirstuc{\adlm} $A$ is Archimedean if, and only if, for all distinct $x, y \in A$ we have $\distint(x, y) \neq 0$.
\end{theorem}
\begin{proof}
	By \cref{t:dist-two-ways,l:injective}.
\end{proof}

\begin{theorem}
	An algebra $M$ in the signature $\{+, \lor, \land\} \cup \Dyad$ is isomorphic to the algebra $\Cleq(X, \R)$ of real-valued order-preserving continuous functions on $X$ for some compact ordered space $X$ if, and only if, $M$ is a {\dlm} that satisfies $\distint[M](x,y) = 0 \Rightarrow x = y$ (so that $\distint[M]$ is a metric), and which is Cauchy complete with respect to $\distint[M]$.
\end{theorem}
\begin{proof}
	By \cref{t:dist-two-ways,t:char-algebras}.
\end{proof}

\begin{theorem}[Ordered Yosida duality] \label{t:Yosida}
	The category $\CompOrd$ of compact ordered spaces is dually equivalent to the category of {\dlms} $M$ which satisfy $\distint[M](x,y) = 0 \Rightarrow x = y$ (so that $\distint[M]$ is a metric), and which are Cauchy complete with respect to $\distint[M]$.
\end{theorem}
\begin{proof}
	By \cref{t:dist-two-ways,t:one-char}.
\end{proof}


\section{Conclusions}

We obtained an analogue of Yosida duality, where compact Hausdorff spaces are replaced by compact ordered spaces.
In the next chapter, we will obtain an explicit axiomatisation of a variety dual to $\CompOrd$.
The results in the present chapter should provide a useful intuitive ground to grasp the ideas behind that axiomatisation.


\chapter{Equational axiomatisation}\label{chap:axiomatisation}


\section{Introduction} \label{s:intro-axiom}

In \cref{chap:direct-proof} we proved that the opposite of the category of compact ordered spaces is equivalent to a variety of algebras.
To obtain a description of one such variety, we defined the signature $\SignCM$ consisting of all order-preserving continuous functions from powers of $[0,1]$ to $[0,1]$ itself.
Every element of this signature has an obvious interpretation on the set $[0,1]$---namely, itself---and so $[0,1]$ is a $\SignCM$-algebra in an obvious way.
The class 
\[
	\opS\opP\mathopen{}\left(\left\langle[0,1]; \SignCM\right\rangle\right)\mathclose{}
\]
was then shown to be a variety which is dually equivalent to the category of compact ordered spaces.

The aim of the present chapter is to provide an \emph{explicit equational axiomatisation} of $\CompOrdop$: in \cref{d:DMVMinfty} we describe the variety $\DMVMinfty$ consisting of what we call \emph{\dmvminftys}, and in \cref{t:axiomatisation} we prove that this variety is in fact dually equivalent to the category of compact ordered spaces.

The primitive operations of $\DMVMinfty$ are $\oplus$, $\odot$, $\lor$, $\land$, all dyadic rationals in $[0,1]$, and the operation of countably infinite arity $\limop$.
The finitary operations have a canonical interpretation on $[0,1]$:
\begin{align*}
	x \oplus y							& \df \min\{x + y, 1\},\\
	x \odot y						 	& \df \max\{x + y - 1, 0\},\\
	x \lor y 							& \df \max\{x,y\},\\
	x \land y							& \df \min\{x,y\},\\
	t \in \Dyadu						& \df \text{the element } t.
\end{align*}
Furthermore, we interpret the infinitary operation $\limop$ in $[0,1]$ as
\[
	\limop(x_1, x_2, x_3, \dots)	\df \lim_{n \to \infty} \aprj_n(x_1, \dots, x_n),
\]
where $\mu_n$ is defined inductively by setting
\[
	\aprj_1(x_1) \df x_1,
\]
and, for $n \geq 2$,
\begin{align*}
	&\aprj_n(x_1, \dots,x_n)\\
	& \df \max\left\{\min\left\{x_n, \aprj_{n-1}(x_1, \dots,x_{n-1}) + \frac{1}{2^{n-1}}\right\}, \aprj_{n-1}(x_1, \dots,x_{n-1}) - \frac{1}{2^{n-1}}\right\}.
\end{align*}

The main sources of inspiration for this chapter have been \cite{MarraReggio2017,HN2018,HNN2018}.

We warn the reader that the interpretation in $[0,1]$ of the operation of countably infinite arity $\limop$ here differs from the interpretation of the operation $\delta$ in \cite{Abbadini2019}, and we believe the axiomatisation in this chapter to be more elegant, one of the reasons being the self-duality of $\limop$.


\subsubsection{Sketch of the proof}

We sketch the proof of the main result of this chapter (\cref{t:axiomatisation}) which shows that $\CompOrd$ is dually equivalent to the variety $\DMVMinfty$. 

We make use of the sub-signature $\SignDMVM \df \{\oplus, \odot, \lor, \land\} \cup (\Dyadu)$ of $\SignCM$.
We define a \emph{\dmvm} as a $\SignDMVM$-algebra $\alge{A}$ which is {\amvm} such that, for each $\alpha, \beta \in \Dyadu$, we have $\alpha^\alge{A} \oplus^\alge{A} \beta^\alge{A} = (\alpha \oplus^\alge{\R} \beta)^\alge{A}$, and $\alpha^\alge{A} \odot^\alge{A} \beta^\alge{A} = (\alpha \odot^\alge{\R} \beta)^\alge{A}$ (\cref{d:dyadic-MVM}).
Using the subdirect representation theorem, we obtain that a $\SignDMVM$-algebra is isomorphic to a subalgebra of a power of $[0,1]$ if, and only if, it is {\admvm} which satisfies
\[
	\forall x\ \forall y\ \distint(x, y) = 0 \Rightarrow x = y
\]
(see \cref{n:distint-int} for the definition of $\distint$, and \cref{t:char-Arch} for the equivalence of the two conditions).

We then define the variety $\DMVMinfty$ of {\dmvminftys} (\cref{d:DMVMinfty}), in the signature $\SignDMVMinfty = \SignDMVM \cup \{\limop\}$, where $\limop \colon [0,1]^{\Np} \to [0,1]$ is an order-preserving continuous function.
We obtain the following results.
\begin{enumerate}
	\item The algebra $[0,1]$ with standard interpretations is {\admvminfty}.
	
	\item	The $\SignDMVM$-reduct of any {\dmvminfty} is {\admvm} which satisfies
	\[
		\forall x\ \forall y\ \distint(x, y) = 0 \Rightarrow x = y.
	\]
	
	\item For every {\dmvminfty} $A$, every $\SignDMVM$-homomorphism from $A$ to $[0,1]$ is a $\SignDMVMinfty$-homomorphism.
	
	\item For every cardinal $\kappa$, the term operations of arity $\kappa$ of the $\SignDMVMinfty$-algebra $[0,1]$ are the order-preserving continuous functions from $[0,1]^\kappa$ to $[0,1]$.
\end{enumerate}
Using these facts, we deduce that the variety $\DMVMinfty$ of {\dmvminftys} is term-equivalent to $\opS\opP\mathopen{}\left(\left\langle [0,1]; \SignCM \right\rangle\right)\mathclose{}$.
Since the latter is dually equivalent to $\CompOrd$, we obtain that $\DMVMinfty$ and $\CompOrd$ are dually equivalent\footnote{This result provides an alternative proof of the fact (already obtained in \cref{t:MAIN}) that $\CompOrd$ is dually equivalent to a variety of algebras; in this proof, we use the fact that $[0,1]$ is a regular injective regular cogenerator of $\CompOrd$. We do not need, instead, the fact that equivalence relations in $\CompOrdop$ are effective.}.


\section{Primitive operations and their interpretations}
\label{s:primitive}

In this section, we state a unit interval ordered version of the Stone-Weierstrass theorem and we use it to choose a convenient set of primitive operations for the variety dual to $\CompOrd$.


\subsection{Unit interval ordered Stone-Weierstrass theorem}

The ordered version of the Stone-Weierstrass theorem (\cref{t:StWe}) admits an analogous version with $[0, 1]$ instead of $\R$, whose proof---that we omit---could be either obtained analogously to \cref{t:StWe} or---in light of the equivalence established in \cref{t:G is equiv}---as a consequence of it.

\begin{theorem}[Unit interval ordered Stone-Weierstrass theorem]\label{t:StWe-unit}\index{Stone-Weierstrass theorem!ordered!unit interval}
	Let $X$ be a compact ordered space, let $L$ be an order-separating set of continuous order-preserving functions from $X$ to $[0,1]$ which is closed under $\oplus$, $\odot$, $\lor$, $\land$ and which contains every constant in $[0,1]$.
	Then, the closure of $L$ under uniform convergence is the set of continuous order-preserving functions from $X$ to $[0,1]$.
\end{theorem}

\Cref{t:StWe-unit} tells us something about sets of generating operations for the clone of order-preserving continuous functions on $[0,1]$.

\begin{lemma} \label{l:denseness}
	Let $\kappa$ be a cardinal, and let $L_\kappa$ be the set of operations from $[0,1]^\kappa$ to $[0,1]$ that are generated by $\{\oplus, \odot,\lor,\land\} \cup [0,1]$.
	Then, the set of order-preserving continuous functions from $[0,1]^\kappa$ to $[0,1]$ coincides with the closure of $L_\kappa$ under uniform convergence.
\end{lemma}
\begin{proof}
	We apply \cref{t:StWe-unit} with $X = [0,1]^\kappa$: note that $X$ is a compact ordered space and that $L_\kappa$ is order-separating because it contains the projections, which are easily seen to order-separate the elements of $[0,1]^\kappa$.
\end{proof}

\begin{lemma} \label{l:generation}
	Let $\alpha \colon [0,1]^{\Np} \to [0,1]$ be an order-preserving continuous function such that, if $L$ is a set of functions from a set $X$ to $[0,1]$, then the closure of $L$ under pointwise application of $\alpha$ contains the closure of $L$ under uniform limits.
	Let $D$ be a dense subset of $[0,1]$.
	Then $\{\oplus, \odot,\lor,\land\} \cup D \cup \{\alpha\}$ generates the clone on $[0,1]$ of order-preserving continuous functions.
\end{lemma}
\begin{proof}
	By \cref{l:denseness}.
\end{proof}

In the following we will look for a function $\limop \colon [0,1]^{\Np} \to [0,1]$ that satisfies the conditions in \cref{l:generation}, so that the set $\{\oplus, \odot, \land, \lor\} \cup \Dyad \cup \{\limop\}$ settles our search for a generating set of the clone of order-preserving continuous functions on $[0,1]$.

Ideally, one would like to take
\[
	\lim \colon [0,1]^{\Np} \to [0,1].
\]
The first thing we notice is that $\lim$ is not defined on the whole $[0,1]^{\Np}$ because not all sequences converge.
That said, the next ideal thing one would like to take is a function that maps each Cauchy sequence to its limits.
However, this is not possible because we want the function to be continuous, and every function $\alpha \colon [0,1]^{\Np} \to [0,1]$ that maps each Cauchy sequence to its limit is not continuous, as it does not commute with topological limits:
\[
	\alpha\big(\lim_{n \to \infty} (\underbrace{1, \dots, 1}_{n\ \text{times}}, 0, 0, \dots)\big) = \alpha(1, 1, 1, \dots) = 1,
\]
and
\[
	\lim_{n \to \infty} \alpha(\underbrace{1, \dots, 1}_{n\ \text{times}}, 0, 0, \dots) = \lim_{n \to \infty} 0 = 0.
\]
So, we look for an order-preserving continuous function $\limop \colon [0,1]^{\Np} \to [0,1]$ which maps just `enough' Cauchy sequences to their limit (and with no restriction on the value it takes on other sequences).
There does not seem to be a canonical candidate for $\limop$.
In \cite{HNN2018} a suitable operation denoted $\delta$ was described.
Here we use a slightly different function.
As a further remark showing the limitations in the choice of $\limop$, we point out that there exists no continuous function $\alpha \colon [0,1]^{\Np} \to [0,1]$ that satisfies the following identities, which would be natural for an operation which acts as a limit.
\begin{enumerate}
	\item $\alpha(x, x, x, \dots) = x$;
	\item $\alpha(x_1, x_2, x_3, \dots) = \alpha(x_2, x_3, x_4, \dots)$.
\end{enumerate}
Indeed, one such function would satisfy 
\[
	\alpha(1, 1, 1, \dots) = 1,
\]
and, for every $n$,
\[
	\alpha(\underbrace{1, \dots, 1}_{n\ \text{times}}, 0, 0, \dots) = 0,
\]
and so it would not be continuous by the previous discussion.

However, the function $\limop$ that we will describe seems to us quite a natural choice.

The idea in order to ensure continuity of $\limop$ is to define $\limop$ as the uniform limit of a sequence of continuous functions $\widetilde{\aprj}_n \colon [0,1]^{\Np} \to [0,1]$, where $\widetilde{\aprj}_n$ are functions generated by $\{\oplus, \odot, \land, \lor\} \cup \Dyad$.
The fact that the functions $\widetilde{\aprj}_n$ are generated by $\{\oplus, \odot, \land, \lor\} \cup \Dyad$ implies that they are order-preserving and continuous.
To be sure that $(\widetilde{\aprj}_n)_n$ converges uniformly, we will require $\dist(\widetilde{\aprj}_n, \widetilde{\aprj}_{n+1}) \leq \frac{1}{2^n}$, where $\dist$ is the uniform metric.
In this way, $(\widetilde{\aprj}_n)_n$ admits a uniform limit that we will denote with $\limop$.
Then, $\limop$ is guaranteed to be continuous.
Since the functions $\widetilde{\aprj}_n$ are order-preserving, the function $\limop$ is order-preserving, as well.
Moreover, we will look for functions $\widetilde{\aprj}_n$ which behave very much like projections onto the $n$-th coordinates: this means that, for `enough many sequences $(x_n)_n$', we have $\widetilde{\aprj}_n(x_1, x_2, x_3, \dots) = x_n$.
In this way, for `enough many sequences $(x_1, x_2, x_3, \dots)$', we have 
\[
	\limop(x_1, x_2, x_3, \dots) = \lim_{n \to \infty} \widetilde{\aprj}_n(x_1, x_2, x_3, \dots) = \lim_{n \to \infty} x_n.
\]
Then, $\limop$ will be very much like a limit operation\footnote{In fact, the symbol $\limop$ should be evocative of the word `\emph{l}imit'.}.


\subsection{Completion via \texorpdfstring{$2$}{2}-Cauchy sequences}

The need for $\limop$ to be continuous motivated the requirement $\dist(\widetilde{\aprj}_n, \widetilde{\aprj}_{n+1}) \leq \frac{1}{2^n}$.
This brings us to the following definition.

\begin{definition}\label{d:$2$-Cauchy}
	A sequence $(x_n)_{n \in \Np}$ in a metric space $(X,\dist)$ is called \emph{$2$-Cauchy}\index{$2$-Cauchy sequence!in a metric space}\index{sequence!$2$-Cauchy!in a metric space} if, for every $n \in \Np$, we have
	\[
		\dist(x_n,x_{n + 1}) \leq \frac{1}{2^n}.
	\]
\end{definition}

As shown in the following results, every $2$-Cauchy sequence is a Cauchy sequence, and Cauchy completeness is equivalent to convergence of all $2$-Cauchy sequences.

\begin{lemma}\label{l:supercauchy is cauchy}
	Let $(x_n)_{n \in \Np}$ be a $2$-Cauchy sequence in a metric space $(X,\dist)$.
	Then, for every $n,m \in \Np$ with $n \leq m$, we have
	\[
		\dist(x_n,x_m) < \frac{1}{2^{n-1}}.
	\]
\end{lemma}

\begin{proof} 
	By the triangle inequality, we have
	\[ 	\dist(x_n,x_m) \leq \sum_{i=n}^{m-1}\dist(x_i,x_{i + 1}) \leq \sum_{i=n}^{m-1}\frac{1}{2^{i}}< \sum_{i=n}^{ \infty}\frac{1}{2^i}= \frac{1}{2^{n-1}}.  \qedhere\]
\end{proof}

\begin{lemma}\label{l:$2$-Cauchy-implies-Cauchy}
	Every $2$-Cauchy sequence in a metric space is a Cauchy sequence.
\end{lemma}

\begin{proof}
	By \cref{l:supercauchy is cauchy}.
\end{proof}

\begin{lemma} \label{l:subsequence}
	Every Cauchy sequence in a metric space admits a $2$-Cauchy subsequence.
\end{lemma}

\begin{proof}
	Let $(x_n)_{n \in \Np}$ be a Cauchy sequence.
	Choose $n_1$ so that, for $k \geq n_1$, we have $\dist(x_{n_j}, x_k) \leq  \frac{1}{2}$, and then, iteratively, choose ${n_j}$ ($j= 2, 3, 4, \dots$) so that ${n_j} > n_{j-1}$ and, for every $k \geq n_j$, $\dist(x_{n_j}, x_k) \leq  \frac{1}{2^{j}}$.
	Then, $(x_{n_j})_{j \in \Np}$ is a $2$-Cauchy subsequence.
\end{proof}

\begin{lemma} \label{l:2-Cauchy-and-completeness}
	A metric space $X$ is Cauchy complete if, and only if, every $2$-Cauchy sequence in $X$ converges.
\end{lemma}

\begin{proof} 
	By \cref{l:$2$-Cauchy-implies-Cauchy}, every $2$-Cauchy sequence is a Cauchy sequence.
	Thus, if $X$ is Cauchy complete, then every $2$-Cauchy sequence in $X$ converges.
	For the converse implication, it is enough to notice that every Cauchy sequence in a metric space admits a $2$-Cauchy subsequence (\cref{l:subsequence}) and that a Cauchy sequence admitting a convergent subsequence converges.
\end{proof}

So, ensuring convergence of $2$-Cauchy sequences is enough to ensure Cauchy-completeness.

\begin{lemma} \label{l:closure}
	Let $\alpha \colon [0,1]^{\Np} \to [0,1]$ be a function that maps all $2$-Cauchy sequences to their limits.
	Let $X$ be a set and let $L$ be a set of functions from $X$ to $[0,1]$.
	Then, every uniform limit of sequences in $L$ belongs to the closure of $L$ under pointwise application of $\alpha$.
\end{lemma}

\begin{proof}
	Let $(f_n)_n$ be a sequence in $L$ that converges uniformly to a function $f$.
	Since $(f_n)_n$ converges, it is a Cauchy sequence.
	By \cref{l:subsequence}, $(f_n)_{n}$ admits a $2$-Cauchy subsequence $(f_{n_j})_{j}$.
	Then, for every $x \in X$, $(f_{n_j}(x))_{j}$ is a $2$-Cauchy sequence, and it converges to $f(x)$.
	Thus, $\alpha((f_{n_j}(x))_{j}) = f(x)$.
\end{proof}


\subsection{The function of countably infinite arity \texorpdfstring{$\limop$}{\textlambda}}

Let us roughly anticipate how $\limop \colon [0,1]^{\Np} \to [0,1]$ will be defined.
\begin{description}
	
	\item [Input]
			The input for $\limop$ is a sequence $(x_n)_{n \in \Np}$.
	
	\item [Step]
			We do an intermediate step in which the sequence $(x_n)_{n}$ is turned into a $2$-Cauchy-sequence $(y_n)_{n}$.
			This is done with as little modification as possible.
			In particular, if $(x_n)_{n}$ was already $2$-Cauchy, then $(y_n)_n = (x_n)_n$.
			For every $n \in \Np$, the element $y_n$ will depend on $x_1, \dots, x_n$, i.e.\ $y = \aprj_n(x_1, \dots, x_n)$ for some $\aprj_n \colon [0,1]^n \to [0,1]$.
	
	\item [Output]
			The output (i.e.\ the value $\limop(x_1, x_2, x_3, \dots)$) is $\lim_{n \to \infty} y_n$, which exists because $(y_n)_n$ is a Cauchy sequence.
	
\end{description}
For each $n$, the function $\aprj_n$ will be order-preserving, so that also $\limop$ is guaranteed to be so.
Moreover, we will have $\dist(\aprj_n, \aprj_{n+1}) \leq \frac{1}{2^n}$, so that $\limop$ is guaranteed to be continuous.
Note that, by construction, $\limop$ will map $2$-Cauchy sequences to their limit.

We first illustrate our choice of $\mu_n$ with an example, which shows how we turn a sequence in $[0,1]$ into a $2$-Cauchy sequence in $[0,1]$.
Consider a sequence beginning with
\begin{align*}
	x_1 & = 0.1,	& x_2	& = 0.5,	& x_3	& = 0,	& x_4	& = 0.3,	&&\dots.
	\intertext{Let us turn this sequence into a $2$-Cauchy one with as few modifications as possible:}
	y_1 & = 0.1,	& y_2	& = 0.5,	& y_3	& = 0.25,	& y_4	& = 0.3,	&&\dots.
\end{align*}
The first element of the original sequence, $0.1$, can be left unchanged in the new sequence: $y_1 \df x_1 = 0.1$.
The distance between the first and the second element of the new sequence must be less or equal than $\frac{1}{2}$.
Since the distance between $y_1 = 0.1$ and $x_2 = 0.5$ is less than $\frac{1}{2}$, the second element, $0.5$, can be left as it is in the new sequence: $y_2 \df x_2 = 0.5$.
The distance between the second and the third element of the new sequence must be less or equal than $\frac{1}{4}$.
Since the distance between $y_2 = 0.5$ and $x_3 = 0$ is strictly greater than $\frac{1}{4}$, we have to replace $0$ with another element: we take this new element $y_3$ as close to $x_3 = 0$ as possible, given the restriction $\dist(y_2, y_3) \leq 0.25$.
Thus we take $y_3 \df 0.25$.
The distance between the third and the fourth element of the new sequence must be less or equal than $\frac{1}{8}$ ($=0.125$).
Since the distance between $y_3 = 0.25$ and $x_4 = 0.3$ is less than $0.125$, the fourth element can be left unchanged: $y_4 \df x_4 = 0.3$.

The $n$-th element of the new sequence depends on the first $n$ elements of the old one; else said, $y_n = \aprj_n(x_1, \dots, x_n)$ for some function $\aprj_n \colon [0,1]^n \to [0,1]$.

\begin{notation} \label{n:apricot}
	Inductively on $n \in \Np$, we define the function $\aprj_n \colon [0,1]^n \to [0,1]$ of arity $n$.
	We set
	\[
		\aprj_1(x_1) \df x_1,
	\]
	and, for $n \geq 2$,
	\begin{align*}
		&\aprj_n(x_1, \dots,x_n)\\
		& \df \max\left\{\min\left\{x_n, \aprj_{n-1}(x_1, \dots,x_{n-1}) + \frac{1}{2^{n-1}}\right\}, \aprj_{n-1}(x_1, \dots,x_{n-1}) - \frac{1}{2^{n-1}}\right\}.
	\end{align*}
\end{notation}

The following three lemmas capture the main properties of the functions $(\aprj_n)_n$.

\begin{lemma} \label{l:ord-pres_cont}
	For every $n \in \Np$, the function $\aprj_n \colon [0,1]^n \to [0,1]$ is order-preserving and continuous.
\end{lemma}
\begin{proof}
	This is easily proved by induction, observing that the functions $+$, $\max$, $\min$ and all the constants are order-preserving and continuous with respect to the product order and topology.
\end{proof}

\begin{lemma}\label{l:2-Cauchy}
	For each sequence $(x_n)_{n \in \Np}$ of elements of $[0,1]$, the sequence 
	\[
		(\aprj_n(x_1, \dots,x_n))_{n \in \Np}
	\]
	is a $2$-Cauchy sequence.
\end{lemma}

\begin{proof}
	For all $x, y \in [0,1]$ and all $n \in \Np$, we have
	\begin{equation} \label{e:ineq-for-mu}
		 \dist\mathopen{}\left(y, \max \left\{\min\left\{x, y + \frac{1}{2^{n-1}}\right\}, y - \frac{1}{2^{n-1}}\right\}\right)\mathclose{} \leq \frac{1}{2^{n-1}}.
	\end{equation}
	If we set $x = x_n$ and $y = \aprj_{n-1}(x_1, \dots, x_{n-1})$ in \cref{e:ineq-for-mu} and we apply the inductive definition of $\aprj_n$, we obtain $\dist(\aprj_{n-1}(x_1, \dots, x_n), \aprj_{n}(x_1, \dots, x_{n-1})) \leq \frac{1}{2^{n-1}}$.
\end{proof}

\begin{lemma}\label{l:effect-of-rho}
	Given a $2$-Cauchy sequence $(x_n)_{n \in \Np}$ of elements of $[0,1]$, we have, for all $n \in \Np$, 
	\[
		\aprj_n(x_1, \dots,x_n) =x_n.
	\]
\end{lemma}

\begin{proof}
	We prove this by induction on $n \in \Np$.
	The statement holds for $n = 1$ by definition of $\aprj_1$.
	Suppose the statement holds for $n \in \Np$, and let us prove it holds for $n + 1$.
	By the inductive definition of $\aprj_n$ and by the inductive hypothesis, we have
	\begin{equation} \label{e:mu}
		\aprj_n(x_1, \dots, x_n) = \max\left\{\min\left\{x_n, x_{n-1} + \frac{1}{2^{n-1}}\right\}, x_{n-1} - \frac{1}{2^{n-1}}\right\}.
	\end{equation}
	Since $(x_n)_{n \in \Np}$ is $2$-Cauchy, we have $\dist(x, y) \leq \frac{1}{2^n}$, i.e.\ $x_{n-1} - \frac{1}{2^n} \leq x_n \leq \frac{1}{2^n}$; hence, the right-hand side of \cref{e:mu} coincides with $x_n$.
\end{proof}

\begin{notation} \label{n:delta}
	Let $(x_n)_{n \in \Np}$ be a sequence of elements of $\R$.
	By \cref{l:2-Cauchy}, the sequence $(\aprj_n(x_1, \dots,x_n))_{n \in \Np}$ is a $2$-Cauchy sequence and thus a Cauchy sequence by \cref{l:$2$-Cauchy-implies-Cauchy}.
	Since the metric space $[0,1]$ is complete, the Cauchy sequence $(\aprj_n(x_1, \dots,x_n))_{n \in \Np}$ admits a limit, that we denote with $\limop(x_1, x_2, x_3, \dots)$.
	This establishes a function%
	\footnote{
		The function $\limop$ defined here differs from the function $\delta$ from \cite{HNN2018,Abbadini2019}. There are two main advantages on the side of $\limop$.
		The first one is elegance: the function $\limop$ is self-dual: for every sequence $(x_n)_{n}$ of elements of $[0,1]$, we have $1 - \limop((1 - x_n)_n) = \limop((x_n)_n)$.
		The second advantage is that the closure under $\limop$ contains the closure under uniform limits for \emph{any} set of $[0,1]$-valued functions (see \cref{l:closure}).
	}
	\begin{equation*}
		\begin{split}
			\limop \colon	[0,1]^{\Np}				& \longrightarrow	[0,1] \\
								(x_1,x_2,x_3, \dots)	& \longmapsto		\lim_{n\rightarrow \infty}\aprj_n(x_1, \dots,x_n).	
		\end{split}
	\end{equation*}
\end{notation}

\begin{proposition}\label{p:delta-cont-mon}
	The function $\limop\colon {[0,1]}^{\Np}\rightarrow [0,1]$ is order-preserving and continuous (with respect to the product order and product topology).
\end{proposition}

\begin{proof} 
	For every $n \in \Np$, we set
	\begin{equation*}
		\begin{split}
			\widetilde{\aprj}_n \colon	{[0,1]}^{\Np}			& \longrightarrow	[0,1] \\
												(x_n)_{n \in \Np}		& \longmapsto		\aprj_n(x_1, \dots,x_{n}).
		\end{split}
	\end{equation*}
	Then, the sequence $(\widetilde{\aprj}_n)_{n \in \Np}$ converges uniformly to $\limop$.
	By \cref{l:ord-pres_cont}, for every $n \in \Np$, the function $\aprj_n\colon [0,1]^n\rightarrow [0,1]$ is order-preserving and continuous.
	Moreover, for every $i \in \Np$, the projection $\pi_i\colon [0,1]^{\Np}\rightarrow [0,1]$ onto the $i$-th coordinate is order-preserving and continuous.
	We have $\widetilde{\aprj}_n = \mu_n(\pi_1, \dots, \pi_n)$, which shows that $\widetilde{\aprj}_n$ is order-preserving and continuous.
	Since $\limop$ is the pointwise limit of $\widetilde{\aprj}_n$, $\limop$ is order-preserving, as well.
	Since $(\widetilde{\aprj}_n)_{n \in \Np}$ uniformly converges to $\limop$, the latter is continuous.
\end{proof}

\begin{lemma} \label{l:to-their-limit}
	The function $\limop \colon [0,1]^{\Np} \to [0,1]$ maps $2$-Cauchy sequences to their limit.
\end{lemma}
\begin{proof}
	For every $2$-Cauchy sequence $(x_n)_{n \in \Np}$ in $[0,1]$ we have
	\[
		\limop(x_1, x_2, x_3, \dots) = \lim_{n\rightarrow \infty}\aprj_n(x_1, \dots,x_n) \stackrel{\text{\cref{l:effect-of-rho}}}{=} \lim_{n \to \infty}x_n. \qedhere
	\]
\end{proof}

The set of primitive operations that we use is 
\[
	\{\oplus, \odot, \lor, \land\} \cup (\Dyadu) \cup \limop.
\]
The reason why we take only dyadic rationals instead of all elements in $[0,1]$ is because dyadic rationals will be useful during the study of the finite axiomatisation provided in \cref{chap:finite-axiomatisation}.
Further, we point out that this choice has the advantage to obtain only a countable set of primitive operations and a countable set of axioms.


\section{Dyadic MV-monoidal algebras}

\begin{definition}\label{d:dyadic-MVM}
	A \emph{\dmvm}\index{MV-monoidal algebra!dyadic} is an algebra $\alge{A}$ in the signature $\{\oplus, \odot, \lor, \land\}\cup(\Dyadu)$ (where $\oplus$, $\odot$, $\lor$ and $\land$ have arity $2$ and each element of $\Dyadu$ has arity $0$) with the following properties.
	\begin{enumerate}[label = DE\arabic*., ref = DE\arabic*, start = 0, leftmargin=\widthof{LTE4'.} + \labelsep]
		\item \label[axiom]{ax:DM0} $\langle A; \oplus^\alge{A}, \odot^\alge{A}, \lor^\alge{A}, \land^\alge{A},0^\alge{A},1^\alge{A}\rangle$ is {\amvm} (see \cref{d:MVM}).
		\item \label[axiom]{ax:DM1} For all $\alpha, \beta \in \Dyadu$, we have $\alpha^\alge{A} \oplus^\alge{A} \beta^\alge{A} = (\alpha \oplus^\alge{\R} \beta)^\alge{A}$.
		\item \label[axiom]{ax:DM2} For all $\alpha, \beta \in \Dyadu$, we have $\alpha^\alge{A} \odot^\alge{A} \beta^\alge{A} = (\alpha \odot^\alge{\R} \beta)^\alge{A}$.
	\end{enumerate}
\end{definition}

We let $\DMVM$ denote the category of {\dmvms} with homomorphisms.

\Cref{ax:DM1,ax:DM2} are (equivalent to) the positive atomic diagram of $\Dyadu$ (cf. e.g.\ \cite[Section~3.2, p.\ 211]{Prest2003} for the notion of positive atomic diagram); we have not included axioms regarding the lattice operations, because they are a consequence, as the following shows.
\begin{lemma}
	For every {\dmvm} $\alge{A}$ and every $\alpha, \beta \in \Dyadu$ with $\alpha \leq \beta$ (as real numbers), we have $\alpha^\alge{A} \leq^\alge{A} \beta^\alge{A}$.
\end{lemma}
\begin{proof}
	We have $\beta^{\alge{A}} = (\alpha \oplus^{\R} (\beta - \alpha))^\alge{A} = \alpha^{\alge{A}} \oplus^\alge{A} (\beta - \alpha)^\alge{A} \geq \alpha^{\alge{A}}$.
\end{proof}

\begin{remark} \label{r:equiv}
	Building on the equivalence established in \cref{t:G is equiv}, it is not difficult to prove that the category of {\dmvms} is equivalent to the category of {\dlms}.
	One functor maps {\adlm} $M$ to the {\dmvm} $\Gam(M)$, on which the constants are defined by restriction.
	The other functor maps {\admvm} $A$ to the {\dlm} $\X(A)$, on which a dyadic rational $t$ is interpreted as follows: denoting with $k$ the unique integer such that $t \in [k,k + 1)$, we set
	\begin{equation}
		\label{e:lambda-interpretation}
		\begin{split}
			t^{\X(A)} \colon	\Z	& \longrightarrow	A\\
			n	& \longmapsto		{
				\begin{cases}
					1						& \text{if } n < k;\\
					(t - k)^A	& \text{if } n = k;\\
					0						& \text{if } n > k.
				\end{cases}
			}
		\end{split}
	\end{equation}
\end{remark}

\begin{example} \label{ex:int-ok}
	The unit interval $[0,1]$ with standard interpretations is {\admvm}.
\end{example}

\begin{definition}
	We say that an algebra $A$ in the signature $\{\oplus, \odot, \lor, \land\} \cup (\Dyadu)$ is \emph{Archimedean}\index{Archimedean}\index{MV-monoidal algebra!Archimedean} if $A$ is isomorphic to a subalgebra of a power of $[0,1]$ with obviously defined operations%
	\footnote{
		The following are two other conditions which are easily seen to be equivalent.
		\begin{enumerate}
			\item The canonical homomorphism $A \to [0,1]^{\hom(A, [0,1])}$ is injective.
			\item For all $x,y \in A$ with $x \neq y$ there exists a homomorphism $f \colon A \to [0,1]$ such that $f(x) \neq f(y)$.
		\end{enumerate}
	}.
\end{definition}

As it was pointed out by one of the referees, the Archimedean algebras are those algebras so that $[0, 1]$ acts as a cogenerator.
As it is explained in \cite{PorstTholen1991}, this is an essential property to obtain a natural duality.

\begin{notation} \label{n:distint-int}
	In analogy with \cref{n:dist}, given {\adlm} $A$, and given $x, y \in A$, we set
	\[
		\distint[A](x, y) \df \inf\left\{t \in \Dyadu \mid y \odot (1-t)^A \leq x \leq y \oplus t^A\right\}.
	\]
	When $A$ is understood, we write simply $\distint(x, y)$ for $\distint[A](x, y)$.
\end{notation}

It is clear that every algebra in the signature $\{\oplus, \odot, \lor, \land\} \cup (\Dyadu)$ which is Archimedean is {\admvm}.
The following theorem settles the problem of identifying which {\dmvms} are Archimedean.
We omit the proof, since---in light of \cref{r:equiv}---it is analogous to \cref{t:char-Arch-dlms}.

\begin{theorem} \label{t:char-Arch}
	\xmakefirstuc{\admvm} $A$ is Archimedean if, and only if, for all distinct $x, y \in A$ we have $\distint(x, y) \neq 0$.
\end{theorem}


\section{Equational axiomatisation}

\begin{notation} \label{n:trunc}
	For every $n \in \N$, we define a binary term $\trunc_n$ in the language of {\dmvms}:
	\[
		\trunc_{n}(x, y) \df \left(x \land \left(y \oplus \frac{1}{2^n}\right)\right) \lor \left(y \odot \left(1 - \frac{1}{2^n}\right)\right).
	\]
\end{notation}

For example, for $x, y \in [0,1]$, we have
\[
	\trunc_{n}^\R(x, y) = \max\left\{\min\left\{x, y + \frac{1}{2^n}\right\}, y - \frac{1}{2^n}\right\}.
\]

\begin{notation} \label{n:aprj}
	Inductively on $n \in \Np$, we define a term $\aprj_n$ of arity $n$ in the language of {\dmvms}:
	\begin{align*}
		\aprj_1(x_1)					\df{}	& 	x_1; \\
		\begin{split}
			\aprj_n(x_1, \dots,x_n)	\df{}	& 	\trunc_{n-1}\mathopen{}\big(x_n,\aprj_{n-1}(x_1, \dots,x_{n-1})\big)\mathclose{}\\
											={}	& 	\left(x_n \land \left(\aprj_{n-1}(x_1, \dots,x_{n-1}) \oplus \frac{1}{2^{n-1}}\right)\right)\\
													&	\lor \left(\aprj_{n-1}(x_1, \dots,x_{n-1}) \odot \left(1 - \frac{1}{2^{n-1}}\right)\right).
		\end{split}
	\end{align*}
\end{notation}

\begin{remark}
	The interpretation of $\aprj_n$ on the unit interval $[0,1]$ is precisely the function $\aprj_n \colon [0,1]^n \to [0,1]$ defined inductively in \cref{n:apricot}.
	Thus, the overlapping notation should not be a source of problems.
\end{remark}

We identify a variety of algebras which we will show to be dual to the category of compact ordered spaces.

\begin{definition}\label{d:DMVMinfty}
	A \emph{\dmvminfty}\index{MV-monoidal algebra!limit dyadic} is an algebra $\alge{A}$ in the language $\{\oplus, \odot, \lor, \land\}\cup (\Dyadu) \cup \{\limop\}$---where $\oplus$, $\odot$, $\lor$ and $\land$ have arity $2$, every element of $\Dyadu$ has arity $0$, and $\limop$ has countably infinite arity---with the following properties.
	\begin{enumerate} [label = LDE\arabic*., ref = LDE\arabic*, start = 0, leftmargin=\widthof{LTE4'.} + \labelsep]
	
		\item \label[axiom]{ax:inf}
				The ($\{\oplus, \odot, \lor, \land\}\cup( \Dyadu)$)-reduct of $\alge{A}$ is {\admvm} (see \cref{d:dyadic-MVM}).
	
		\item \label[axiom]{ax:constant-ax-inf}
				$\limop(x,x,x, \dots) =x$.
				
		\item \label[axiom]{ax:delta-ominus-ax-inf}
				$\limop\mathopen{}\big(\trunc_{0}(x, y),\trunc_{1}(x, y),\trunc_{2}(x, y), \dots\big)\mathclose{} = y$. (See \cref{n:trunc} for the definition of $\trunc_n$.)
				
		\item \label[axiom]{ax:delta-sandwich-ax-inf}
				For every $n \in \Np$ we have
				\[
					\aprj_n(x_1, \dots,x_{n}) \odot \left(1 - \frac{1}{2^{n-1}}\right) \leq \limop(x_1,x_2,x_3, \dots) \leq \aprj_n(x_1, \dots,x_{n}) \oplus \frac{1}{2^{n-1}}.
				\]
				(See \cref{n:aprj} for the definition of $\aprj_n$.)
	\end{enumerate}
\end{definition}

\Cref{ax:constant-ax-inf,ax:delta-ominus-ax-inf} guarantee (given \cref{ax:inf}) that the algebra is Archimedean (see \cref{p:W is Archimedean} below);
\cref{ax:delta-sandwich-ax-inf} forces $\limop(x_1,x_2,x_3,\dots)$ to be the limit of $(\aprj_n(x_1, \dots,x_n))_{n \in \Np}$ (see \cref{l:full-on-int} below).
\begin{lemma}\label{l:properties of W}
	The unit interval $[0,1]$, with standard interpretations of the operation symbols, is {\admvminfty}.
\end{lemma}
\begin{proof}
	As already observed in \cref{ex:int-ok}, the unit interval $[0,1]$ is {\admvm}, so \cref{ax:inf} holds.

	The sequence $(x,x,x, \dots)$ is $2$-Cauchy, and its limit is $x$; thus $\limop(x, x, x, \dots) = x$.
	Thus, \cref{ax:constant-ax-inf} holds.
	
	Let us prove \cref{ax:delta-ominus-ax-inf}.
	\begin{claim} \label{cl:trunc-cauchy}
		For all $x, y \in [0,1]$, the sequence $(\trunc_{0}(x, y),\trunc_{1}(x, y),\trunc_{2}(x, y), \dots)$ is $2$-Cauchy.
	\end{claim}
	\begin{claimproof}
		Let $n \in \Np$.
		If $x \in \left[y - \frac{1}{2^{n}}, y + \frac{1}{2^{n}}\right]$, then $x = \tau_{n-1}(x, y) = \tau_n(x, y)$.
		Thus, in this case, we have 
		\[
			\dist\mathopen{}\big(\tau_{n-1}(x, y), \tau_n(x, y)\big)\mathclose{} = \dist(x,x) = 0 \leq \frac{1}{2^{n-1}}.
		\]
		If $x \geq y + \frac{1}{2^{n}}$, then both $\tau_{n-1}(x, y)$  and $\tau_n(x, y)$ belong to $\left[y + \frac{1}{2^{n}}, y + \frac{1}{2^{n-1}}\right]$, and therefore $\dist(\tau_{n-1}(x, y), \tau_n(x, y)) \leq \frac{1}{2^{n-1}}$.
		Analogously if $x \leq y - \frac{1}{2^{n}}$.
	\end{claimproof}
	\begin{claim} \label{cl:converges}
		For all $x, y \in [0,1]$, the sequence $(\trunc_{0}(x, y),\trunc_{1}(x, y),\trunc_{2}(x, y), \dots)$ converges to $y$.
	\end{claim}
	\begin{claimproof}
		The sequence $(\trunc_{0}(x, y),\trunc_{1}(x, y),\trunc_{2}(x, y), \dots)$ is bounded from below by $\left(y - \frac{1}{2^0}, y - \frac{1}{2^1}, y - \frac{1}{2^2}, \dots\right)$ and from above by $\left(y + \frac{1}{2^0},y + \frac{1}{2^1},x + \frac{1}{2^2}, \dots\right)$, and both these two sequences converge to $y$.
		Thus, $(\trunc_{0}(x, y),\trunc_{1}(x, y),\trunc_{2}(x, y), \dots)$ converges to $y$.
	\end{claimproof}
	For all $x,y\in [0,1]$, by \cref{cl:trunc-cauchy}, the sequence $(\trunc_{0}(x, y),\trunc_{1}(x, y),\trunc_{2}(x, y), \dots)$ is $2$-Cauchy; thus, by \cref{l:to-their-limit}, $\limop(\trunc_{0}(x, y),\trunc_{1}(x, y),\trunc_{2}(x, y), \dots)$ is the limit of the sequence $(\trunc_{0}(x, y),\trunc_{1}(x, y),\trunc_{2}(x, y), \dots)$, which, by \cref{cl:converges}, is $y$.
	Hence, \cref{ax:delta-ominus-ax-inf} holds.
	
	Let us prove \cref{ax:delta-sandwich-ax-inf}.	
	By \cref{l:2-Cauchy}, the sequence $(\aprj_n(x_1, \dots,x_n))_{n \in \Np}$ is a $2$-Cauchy sequence.
	By \cref{l:supercauchy is cauchy}, for every $n,m \in \Np$ with $n \leq m$, we have 
	\[
		\dist\mathopen{}\big(\aprj_n(x_1, \dots,x_n),\aprj_m(x_1, \dots,x_m)\big)\mathclose{} < \frac{1}{2^{n-1}}.
	\]
	Fixing $n$ and letting $m$ tend to $ \infty$, we obtain 
	\begin{align*}
		\dist(\aprj_n(x_1, \dots,x_n),\limop(x_1,x_2,x_3, \dots)) & = \dist\mathopen{}\left(\aprj_n(x_1, \dots, x_n), \lim_{m\to \infty}\aprj_m(x_1, \dots, x_m)\mathclose{}\right) \\
		& \leq	\frac{1}{2^{n-1}}. \qedhere
	\end{align*}
\end{proof} 

\begin{lemma}\label{l:t_n}
	Let $A$ be {\admvm}, let $x,y \in A$, and suppuse $\distint(x, y) = 0$.
	Then, for all $n \in \N$, we have $\trunc_{n}(x, y) =x$.
\end{lemma}
\begin{proof}
	If $\distint(x, y) = 0$, then, for every $n \in \N$, we have $x  \odot \left(1 - \frac{1}{2^{n}}\right) \leq y \leq x \oplus \frac{1}{2^n}$, which implies $\trunc_n(x, y) = x$.
\end{proof}
\begin{lemma} \label{l:arch-by-equations}
	Let $A$ be {\admvm}, and suppose that there exists a function $\alpha \colon A^{\Np} \to A$ such that, for all $x, y \in A$, the following conditions hold.
	\begin{enumerate}
		\item \label{i:lim-const} $\alpha(x,x,x, \dots) =x$.
		\item \label{i:shrink} $\alpha\mathopen{}\big(\trunc_{0}(x, y),\trunc_{1}(x, y),\trunc_{2}(x, y), \dots\big)\mathclose{} = y$.
	\end{enumerate}
	Then, $A$ is Archimedean.
\end{lemma}
\begin{proof}
	Let $x,y \in A$ be such that $\distint(x, y) = 0$.
	Then
	\begin{align*}
		x	&	=	\alpha(x,x,x, \dots)																	&& \text{(\cref{i:lim-const})}\\
			&	=	\alpha\mathopen{}\big(\trunc_{0}(x, y),\trunc_{1}(x, y),\trunc_{2}(x, y), \dots\big)\mathclose{} 	&& \text{(\cref{l:t_n})}		\\
			&	=	y.																							&& \text{(\cref{i:shrink})}	
	\end{align*}
	By \cref{t:char-Arch}, this implies that $A$ is Archimedean.
\end{proof}

\begin{proposition}\label{p:W is Archimedean}
	The $(\{\oplus, \odot, \lor, \land\} \cup (\Dyadu))$-reduct of any {\dmvminfty} is Archimedean.
\end{proposition}
\begin{proof}
	As proved in \cref{l:arch-by-equations}, this follows from \cref{ax:constant-ax-inf,ax:delta-ominus-ax-inf}.
\end{proof}

\begin{lemma} \label{l:full-on-int}
	Every function from {\admvminfty} to $[0,1]$ which preserves every operation symbol in $\{\oplus, \odot, \lor, \land\} \cup \Dyad$ preserves also $\limop$.
\end{lemma}
\begin{proof}
	Let $A$ be {\admvminfty}, and let $f \colon A \to [0,1]$ be a function that preserves every operation symbol in $\{\oplus, \odot, \lor, \land\} \cup \Dyad$.
	Let $(x_1,x_2, x_3, \dots)$ be a sequence of elements of $A$.
	By \cref{ax:delta-sandwich-ax-inf}, for every $n \in \Np$, we have
	\[
		\aprj_n(x_1, \dots,x_{n}) \odot \left(1 - \frac{1}{2^{n-1}}\right)  \leq  \limop(x_1, x_2, x_3, \dots)  \leq  \aprj_n(x_1, \dots,x_{n}) \oplus \frac{1}{2^{n-1}}.
	\]
	Since $f$ preserves every operation symbol in $\{\oplus, \odot, \lor, \land\} \cup \Dyad$ we have, for every $n \in \Np$,
	\begin{align*}
		\aprj_n(f(x_1), \dots,f(x_{n})) \odot \left(1 - \frac{1}{2^{n-1}}\right)	& \leq  f(\limop(x_1, x_2, x_3, \dots))\\
																											& \leq  \aprj_n(f(x_1), \dots,f(x_{n})) \oplus \frac{1}{2^{n-1}}.
	\end{align*}
	It follows that, for every $n \in \Np$, we have
	\[
		\lvert f(\limop(x_1, x_2, x_3, \dots)) - \aprj_n(f(x_1), \dots, f(x_n)) \rvert \leq \frac{1}{2^n}.
	\]
	It follows that
	\[
		f(\limop(x_1, x_2, x_3, \dots))  =  \lim_{n \to \infty} \aprj_n(f(x_1), \dots, f(x_n)) = \limop(f(x_1, x_2, x_3, \dots)). \qedhere
	\]
\end{proof}

\begin{proposition} \label{p:isp-dmvminfty}
	We have
	\[
		\DMVMinfty = \opS\opP([0,1]).
	\]
\end{proposition}
\begin{proof}
	Let us first prove that every {\dmvminfty} $A$ is isomorphic to a subalgebra of a power of the algebra $[0,1]$.
	By \cref{p:W is Archimedean}, the reduct to the signature $\{\oplus, \odot, \lor, \land\}\cup (\Dyadu)$ of $A$ is isomorphic to a subalgebra of $[0,1]^{\kappa}$, for some cardinal $\kappa$.
	Let $\iota \colon A \hookrightarrow [0,1]^\kappa$ denote the corresponding inclusion.
	We claim that $\iota$ preserves also $\limop$.
	By \cref{l:full-on-int}, every function from $A$ to $[0,1]$ which preserves every operation symbol in $\{\oplus, \odot, \lor, \land\} \cup \Dyad$ preserves also $\limop$.
	Thus, given any $i \in \kappa$, the composite $A \xhookrightarrow{\iota} [0,1]^{\kappa} \xrightarrow{\pi_i} [0,1]$---where $\pi_i$ denotes the $i$-th projection---preserves $\limop$.
	Therefore, $\iota$ preserves $\limop$, settling our claim, and thus $A$ is isomorphic to a subalgebra of a power of the algebra $[0,1]$.
	
	The converse implication is guaranteed by the following facts.
	\begin{enumerate}
		\item
			The algebra $[0,1]$ in the signature $\{\oplus, \odot, \lor, \land\}\cup (\Dyadu) \cup \{\limop\}$ with standard interpretation of the operation symbols is {\admvminfty} by \cref{l:properties of W}.
		\item
			The class of algebras $\DMVMinfty$ is a variety, and so it is closed under products and subalgebras. \qedhere
	\end{enumerate}
\end{proof}

\begin{lemma} \label{l:clone-mon-cont}
	For every cardinal $\kappa$, the set of interpretations of the term operations of the algebra $[0,1]$ in the signature $\{\oplus, \odot, \lor, \land\}\cup (\Dyadu) \cup \{\limop\}$ is the set of order-preserving continuous functions from $[0,1]^\kappa$ to $[0,1]$.
\end{lemma}
\begin{proof}
	Let $\kappa$ be a cardinal, and let $L_\kappa$ be the set of term operations of $[0,1]$ of arity $\kappa$.
	We now apply \cref{t:StWe-unit}, with $X = [0,1]^\kappa$: note that $X$ is a compact ordered space and that $L_\kappa$ is order-separating because it contains the projections, which are easily seen to order-separate the elements of $[0,1]^\kappa$.
	Therefore, the set of order-preserving continuous functions from $[0,1]^\kappa$ to $[0,1]$ coincides with the closure of $L_\kappa$ under uniform convergence.
	By \cref{l:closure}, using the fact that $L_\kappa$ is closed under $\limop$, we obtain that the closure of $L_\kappa$ under uniform convergence is $L_\kappa$ itself.
\end{proof} 

\begin{remark} \label{r:ISP-te}
	Let $\FF$ and $\GG$ be signatures, and let $\alge{A}$ and $\alge{B}$ be algebras in signatures $\FF$ and $\GG$ with the same underlying set. Suppose the clone on $\alge{A}$ equals the clone of $\alge{B}$. Then, the quasivarieties $\mathrm{ISP}(\alge{A})$ and $\mathrm{ISP}(\alge{B})$ are term-equivalent.
\end{remark}

Let $\CMiso$ denote the class of $\SignCM$-algebras which are (isomorphic to) a subalgebra of a power of the $\SignCM$-algebra $[0,1]$ with standard interpretation of the operation symbols, i.e.\ 
\[
	\CMiso \df \opS\opP\mathopen{}\left(\left\langle [0,1]; \SignCM \right\rangle\right)\mathclose{}.
\]

\begin{theorem} \label{t:term-equiv}
	The classes $\CMiso$ and $\DMVMinfty$ are term-equivalent varieties.
\end{theorem}
\begin{proof}
	By \cref{p:isp-dmvminfty}, the class $\DMVMinfty$ consists of the algebras in the signature $\{\oplus, \odot, \lor, \land\} \cup (\Dyadu) \cup \{\limop\}$ which are isomorphic to a subalgebra of a power of $[0,1]$.
	By definition of $\CMiso$, the class $\CMiso$ consists of the $\SignCM$-algebras which are isomorphic to a subalgebra of a power of $[0,1]$.
	The clone of term operations of the $\SignCM$-algebra $[0,1]$ consists of the order-preserving continuous functions.
	By \cref{l:clone-mon-cont}, the interpretations of the term operations of the algebra $[0,1]$ in the signature $\{\oplus, \odot, \lor, \land\} \cup (\Dyadu) \cup \{\limop\}$ are the order-preserving continuous functions.
	By \cref{r:ISP-te}, the class $\CMiso$ is term-equivalent to $\DMVMinfty$.
	Since the class $\DMVMinfty$ is a variety of algebras, also the class $\CMiso$ is a variety of algebras.
\end{proof}

\begin{theorem} \label{t:axiomatisation}
	The category $\CompOrd$ of compact ordered spaces is dually equivalent to the variety $\DMVMinfty$ of {\dmvminftys} (see \cref{d:DMVMinfty}).
\end{theorem}
\begin{proof}
	By \cref{t:term-equiv}, the varieties $\DMVMinfty$ and $\CMiso$ are term-equivalent.
	By \cref{t:duality-not-explicit}, the categories $\CompOrd$ and $\CMiso$ are dually equivalent.
\end{proof}

We describe two contravariant functors which witness the equivalence in \cref{t:axiomatisation}.
One contravariant functor is
\[
	\Cleq (-,[0,1])\colon \CompOrd \to \DMVMinfty.
\]
In the opposite direction, we have the contravariant functor 
\[
	\hom(-,[0,1]) \colon \DMVMinfty \to \CompOrd,
\]
defined as follows.
Given a {\dmvminfty} $A$, we let $\hom(A, [0,1])$ denote the set of $(\{\oplus, \odot, \lor, \land\}\cup (\Dyadu) \cup \{\limop\})$-homomorphisms from $A$ to $[0,1]$, or equivalently (by \cref{l:full-on-int}), the set of $(\{\oplus, \odot, \lor, \land\}\cup (\Dyadu))$-homomorphisms from $A$ to $[0,1]$.
We equip $\hom(A, [0,1])$ with the initial order and the initial topology with respect to the structured source of evaluation maps
\[
	(\ev_x \colon \hom(A, [0,1]) \to [0,1])_{x \in A},
\]
or, equivalently, the induced topology and order with respect to the inclusion 
\[
	\hom(A, [0,1]) \seq [0,1]^A,
\]
or, equivalently, as follows. 
For $f, g \in \hom(A, [0,1]) $ we set $f \leq g$ if, and only if, for all $x \in X$, we have $f(x) \leq g(x)$.
Furthermore, we endow $\hom(A, [0,1])$ with the smallest topology that contains, for every element $x \in A$ and every open subset $O$ of $[0,1]$, the set $\{f \in \hom(A, [0,1]) \mid f(x) \in O\}$.


\section{Conclusions}

We finally obtained an equational axiomatisation of the dual of the category of compact ordered spaces.
One final question arises, to which our next chapter will be devoted: Does there exist a \emph{finite} equational axiomatisation of $\CompOrdop$?


\chapter{Finite equational axiomatisation}\label{chap:finite-axiomatisation}


\section{Introduction}

In the previous chapter we obtained an explicit equational axiomatisation of the dual of $\CompOrd$.
In this chapter we take a further step by providing a \emph{finite} equational axiomatisation, meaning that we use only finitely many function symbols and finitely many equational axioms to present the variety.
To the best of the author's knowledge, the existence of such a finite axiomatisation is a new result.

Recall that in \cref{chap:axiomatisation} we obtained an equational axiomatisation of $\CompOrdop$ in the signature consisting of $\oplus$, $\odot$, $\lor$, $\land$, all dyadic rationals in $[0,1]$ and $\limop$.
Since we now want only finitely many primitive operations, we cannot include in the signature all the dyadic rationals in $[0,1]$.
So, we replace them with the constants $0$ and $1$, together with the unary operation $\hh$ of division by $2$, and (for the sake of elegance) its `dual' operation $\jj$ defined on $[0,1]$ by $x \mapsto 1 - (1-\hh(x)) = \frac{1}{2} + \frac{x}{2}$.
The primitive operations are then $\oplus$, $\odot$, $\lor$, $\land$, $0$, $1$, $\hh$, $\jj$, and $\limop$.


\section{Term-equivalent alternatives for algebras with dyadic constants}

The algebras of the following section---called \emph{\tmvms}---have {\admvm} as a reduct: this will allow us to use the results of the previous chapter.
The fact that {\atmvm} has {\admvm} as a reduct is easier to observe if we introduce a term-equivalent alternative for {\dmvms}: instead of all the constants in $\Dyadu$, we consider only the constants in $\{\frac{1}{2^n} \mid n \in \N\} \cup \{1 - \frac{1}{2^n} \mid n \in \N\}$.
To help the intuition, we first obtain a term-equivalence for {\dlms}.


\subsubsection{Term-equivalent alternative for dyadic commutative distributive \texorpdfstring{$\ell$}{\unichar{"02113}}-monoids}

\begin{definition}[Term-equivalent alternative to \cref{d:dyadic-ell}] \label{d:lm-alt}
	A \emph{\dlm} is an algebra $\alge{M}$ in the signature 
	\[
		\{ + , \lor, \land, 0\} \cup \left\{\frac{1}{2^n} \mid n \in \N\right\} \cup \left\{-\frac{1}{2^n} \mid n \in \N\right\}
	\]
	(where the operations $+$, $\lor$ and $\land$ have arity $2$, and the element $0$ and every element of $\left\{\frac{1}{2^n} \mid n \in \N\right\} \cup \left\{-\frac{1}{2^n} \mid n \in \N\right\}$ have arity $0$) with the following properties.
	\begin{enumerate}[label = DM'\arabic*., ref = DM'\arabic*, start = 0, leftmargin=\widthof{LTE4'.} + \labelsep]
		\item \label{i:F0} $\langle M; + , \lor, \land,0\rangle$ is {\alm} (see \cref{d:lm}).
		\item \label{i:F1} For all $n \in \Np$, $\frac{1}{2^n} + \frac{1}{2^n} = \frac{1}{2^{n-1}}$.
		\item \label{i:F2} For all $n \in \Np$, $\left(-\frac{1}{2^n}\right) + \left(-\frac{1}{2^n}\right) = -\frac{1}{2^{n-1}}$.
		\item \label{i:F3} For all $n \in \N$, $-\frac{1}{2^n} + \frac{1}{2^n}= 0$.
		\item \label{i:F4} For all $n \in \N$, $- \frac{1}{2^n} \leq 0 \leq \frac{1}{2^n}$.
		\item \label{i:F5} For all $x \in M$, there exists $n \in \N$ such that $ n(-1) \leq x \leq n1$.
	\end{enumerate}
\end{definition}

We claim that the classes of algebras described in \cref{d:lm,d:lm-alt} under the common name of `{\dlms}' are term-equivalent.

Indeed, we first note that \cref{i:F0,i:F1,i:F2,i:F3,i:F4,i:F5} holds for every {\dlm}.
For the opposite direction, if $M$ satisfies \cref{i:F1,i:F2,i:F3,i:F4,i:F5}, then, for every $k \in \Np$ and $n \in \N$, we denote with $\frac{k}{2^n}$ the element 
\[
	\underbrace{\frac{1}{2^n} + \dots + \frac{1}{2^n}}_{k \text{ times}},
\]
and we denote with $-\frac{k}{2^n}$ the element
\[
	\underbrace{\left(-\frac{1}{2^n}\right) + \dots + \left(-\frac{1}{2^n}\right)}_{k \text{ times}}.
\]
In this way, to every dyadic rational is associated an element of $M$; this association is well given for the following reason: if a strictly positive dyadic rational is both equal to $\frac{k}{2^n}$ and $\frac{k'}{2^{n'}}$ for $k, k' \in \Np$ and $n, n' \in \N$, then the elements $\underbrace{\frac{1}{2^n} + \dots + \frac{1}{2^n}}_{k \text{ times}}$ and $\underbrace{\frac{1}{2^{n'}} + \dots + \frac{1}{2^{n'}}}_{k' \text{ times}}$ are the same by \cref{i:F1}; an analogous statement holds for strictly negative dyadic rationals, using \cref{i:F2}.
\Cref{ax:R1} holds by \cref{i:F4}, using the monotonicity of $+$.
\Cref{ax:R2} holds by \cref{i:F4}.
\cref{ax:R3} holds by \cref{i:F5}.

The two classes of algebras are then term-equivalent.


\subsubsection{Term-equivalent alternative for dyadic MV-monoidal algebras}

\begin{definition}[Term-equivalent alternative to \cref{d:dyadic-MVM}] \label{d:dmvm-alt}
	An algebra $\alge{A}$ in the signature $\{\oplus, \odot, \lor, \land\} \cup \{\dd_n \mid n \in \N \} \cup \{\uu_n \mid n \in \N\}$, where $\oplus$, $\odot$, $\lor$ and $\land$ have arity $2$ and each $\dd_n$ and $\uu_n$ has arity $0$, is a \emph{\dmvm} provided it satisfies the following properties.
	\begin{enumerate}[label = DE'\arabic*., ref = DE'\arabic*, start = 0, leftmargin=\widthof{LTE4'.} + \labelsep]
		\item \label[axiom]{ax:dmvm-alt:0}			$\langle A; \oplus, \odot, \lor, \land, \uu_0, \dd_0\rangle$ is {\amvm} (see \cref{d:MVM}).
		\item \label[axiom]{ax:dmvm-alt:d-oplus} 	For every $n \in \Np$, $\dd_{n} \oplus \dd_{n} = \dd_{n-1}$.
		\item \label[axiom]{ax:dmvm-alt:u-odot} 	For every $n \in \Np$, $\uu_{n} \odot \uu_{n} = \uu_{n-1}$.
		\item \label[axiom]{ax:dmvm-alt:d-odot} 	For every $n \in \Np$, $\dd_{n} \odot \dd_{n} = 0$.
		\item \label[axiom]{ax:dmvm-alt:u-oplus} 	For every $n \in \Np$, $\uu_{n} \oplus \uu_{n} = 1$.
		\item \label[axiom]{ax:dmvm-alt:oplus} 	For every $n \in \N$, $\dd_n \oplus \uu_n = 1$.
		\item \label[axiom]{ax:dmvm-alt:odot} 		For every $n \in \N$, $\dd_n \odot \uu_n = 0$.
	\end{enumerate}
\end{definition}
The conjunction of \cref{ax:dmvm-alt:d-oplus,ax:dmvm-alt:d-odot} is equivalent (given \cref{ax:dmvm-alt:0}) to $\dd_n + \dd_n = \dd_{n-1}$ in the enveloping {\ulm} of $\alge{A}$, and---loosely speaking---it corresponds to \cref{i:F1} in \cref{d:lm-alt}.
Analogously, the conjunction of \cref{ax:dmvm-alt:u-oplus,ax:dmvm-alt:u-odot} is equivalent to $\uu_n + \uu_n - 1 = \uu_{n-1}$, and it corresponds to \cref{i:F2} in \cref{d:lm-alt}.
\Cref{ax:dmvm-alt:oplus,ax:dmvm-alt:odot} are equivalent to $\dd_n + \uu_n = 1$ and correspond to \cref{i:F3}.

We show that the classes of algebras described in \cref{d:dyadic-MVM,d:dmvm-alt} are term-equivalent.
First, we observe that \cref{i:F1,i:F2,i:F3,i:F4,i:F5} hold for every {\dmvm} in the sense of \cref{d:dyadic-MVM}, with $\dd_n = \frac{1}{2^n}$ and $\uu_n = 1 - \frac{1}{2^n}$.

For the converse direction, we make use of the following result.
\begin{lemma}
	Let $A$ be {\admvm} in the sense of \cref{d:dmvm-alt}.
	Then, for every $n \in \N$, and every $k \in \{0, \dots, 2^n\}$, we have
	\[
		\underbrace{\dd_n \oplus \dots \oplus \dd_n}_{k \text{ times}} = \underbrace{\uu_n \odot \dots \odot \uu_n}_{2^n - k \text{ times}}
	\]
\end{lemma}
\begin{proof}
	By \cref{t:G is equiv}, it is enough to show that, in the {\ulm} that envelops $A$, we have
	\begin{equation} \label{e:d-u}
		\underbrace{\dd_n + \dots + \dd_n}_{k \text{ times}} = \underbrace{\uu_n + \dots + \uu_n}_{2^n - k \text{ times}} - (2^n - k - 1)
	\end{equation}
	To prove \cref{e:d-u}, it is enough to add the $+$-invertible element $\underbrace{\dd_n + \dots + \dd_n}_{2^n - k \text{ times}}$ on both sides.
\end{proof}
So, suppose that $A$ satisfies \cref{i:F1,i:F2,i:F3,i:F4,i:F5}. Then, for every $n \in \N$ and $k \in \{0, \dots, 2^n\}$, we let $\frac{k}{2^n}$ denote the element $\underbrace{\dd_n \oplus \dots \oplus \dd_n}_{k \text{ times}}$, or equivalently the element $\underbrace{\uu_n \odot \dots \odot \uu_n}_{2^n - k \text{ times}}$.
In this way, to every dyadic rational is associated an element of $A$; it is not difficult to see that this association is well given.

The two classes of algebras are then term-isomorphic.


\section{MV-monoidal algebras with division by \texorpdfstring{$2$}{2}}

\begin{definition} \label{d:tmvm}
	A \emph{\tmvm}\index{MV-monoidal algebra!$2$-divisible} $\langle A; \oplus, \odot, \lor, \land, 0, 1, \hh, \jj \rangle$ (arities $2$, $2$, $2$, $2$, $0$, $1$, $1$) is an algebra with the following properties.
	\begin{enumerate} [label = TE\arabic*., ref = TE\arabic*, start = 0, leftmargin=\widthof{LTE4'.} + \labelsep]
		
		\item \label[axiom]{ax:tmvm:mvm}
				$\langle A; \oplus, \odot, \lor, \land, 0, 1\rangle$ is {\amvm} (see \cref{d:MVM}).
				
		\item \label[axiom]{ax:tmvm:cohf-via-hf}
				$\jj(x) = \hh(1) \oplus \hh(x)$.
		
		\item \label[axiom]{ax:tmvm:hf-via-cohf}
				$\hh(x) = \jj(0) \odot \jj(x)$.
		
		\item \label[axiom]{ax:tmvm:double-oplus}
				$\hh(x) \oplus \hh(x) = x$.
		
		\item \label[axiom]{ax:tmvm:double-odot}
				$\jj(x) \odot \jj(x) = x$.
		
		\item \label[axiom]{ax:tmvm:double-h}
				$\hh(\hh(x) \oplus \hh(y)) = \hh(\hh(x)) \oplus \hh(\hh(y))$.
				
		\item \label[axiom]{ax:tmvm:double-j}
				$\jj(\jj(x) \odot \jj(y)) = \jj(\jj(x)) \odot \jj(\jj(y))$.
		
	\end{enumerate}
\end{definition}

The axioms have been chosen so that, for every {\tmvm} $A$, (the appropiate reduct of) $A$ is {\admvm} (\cref{l:reduct}) and every function $f \colon A \to [0,1]$ that preserves the operations of {\dmvms} preserves also $\hh$ and $\jj$ (\cref{l:pres-also-h}).
This is enough for our purposes.

\begin{remark}
	To give an intuition about the axioms, we state (without a proof) that the category of {\tmvms} is equivalent to the category of what we might call \emph{\tulms}\index{lattice-ordered!monoid!commutative distributive!$2$-divisible!unital!}, i.e.\ algebras $\langle M; + , \lor, \land, 0, 1, -1, \frac{\cdot}{2} \rangle$ (arities $2$, $2$, $2$, $0$, $0$, $0$, $1$) with the following properties (we write $\frac{x}{2}$ for $\frac{\cdot}{2}(x)$).
	\begin{enumerate}
		\item $\langle M; + , \lor, \land, 0, 1, -1 \rangle$ is {\aulm}.
		\item If $x \geq 0$, then $\half{x} \geq 0$.
		\item If $x \leq 0$, then $\half{x} \leq 0$.
		\item $\half{x} + \half{x} = x$.
		\item $\half{x} + \half{y} = \half{x + y}$.
	\end{enumerate}
	One functor maps {\atulm} $M$ to the {\mvm} $\Gam(M)$ on which the interpretation of $\hh$ is $\half{x}$, and the interpretation of $\jj$ is $\half{1} + \half{x}$.
	The functor in the opposite direction maps {\amvm} $A$ to the {\tulm} $\X(A)$, on which the interpretation of $\half{\cdot}$ is as follows: for $\gs{x} \in \X(A)$ and $k \in \Z$, we set
	\[
		\half{\gs{x}}(k) = \hh\mathopen{}\big(\gs{x}(2k)\big)\mathclose{} \oplus \hh\mathopen{}\big(\gs{x}(2k + 1)\big)\mathclose{},
	\]
	or, equivalently,
	\[
		\half{\gs{x}}(k) = \jj\mathopen{}\big(\gs{x}(2k)\big)\mathclose{} \odot \jj\mathopen{}\big(\gs{x}(2k + 1)\big)\mathclose{}.
	\]
\end{remark}

\begin{lemma} \label{l:no-exceed}
	For every element $x$ in {\atmvm}, we have $\hh(x) \odot \hh(x) = 0$ and $\jj(x) \oplus \jj(x) = 1$.
\end{lemma}
\begin{proof}
	We have 
	\begin{align*}
		\hh(x) \odot \hh(x)	& = \jj(0) \odot \jj(x) \odot \jj(0) \odot \jj(x)	&& \text{(\cref{ax:tmvm:hf-via-cohf})}\\
									& = \jj(0) \odot \jj(0) \odot \jj(x) \odot \jj(x)	&& \\
									& = 0 \odot x 													&& \text{(\cref{ax:tmvm:double-odot})}\\
									& = 0,															&& \by{\cref{l:absorbing}}
	\end{align*}
	and, dually, $\jj(x) \oplus \jj(x) = 1$.
\end{proof}

\begin{remark}
	By \cref{ax:tmvm:double-oplus,l:no-exceed}, we have $\hh(x) + \hh(x) = x$.
	Analogously, we have $\jj(x) + \jj(x) - 1 = x$.
\end{remark}

\begin{lemma} \label{l:hf-0}
	In {\atmvm} we have $\hh(0) = 0$ and $\jj(1) = 1$.
\end{lemma}

\begin{proof}
	By \cref{ax:tmvm:double-oplus}, we have $\hh(0) \oplus \hh(0) = 0$.
	Thus, by \cref{l:increasing}, $\hh(0) \leq 0$, which implies, by \cref{l:bounded-lattice}, $\hh(0) = 0$.
	Dually, $\jj(1) = 1$.
\end{proof}

\begin{lemma} \label{l:two-ways-for-one-half}
	In every {\dmvm} we have $\hh(1) = \jj(0)$.
\end{lemma}

\begin{proof}
	We have
	\begin{equation*}
		\jj(0) \stackrel{\text{\cref{ax:tmvm:cohf-via-hf}}}{=} \hh(1) \oplus \hh(0) \stackrel{\text{\cref{l:hf-0}}}{=} \hh(1) \oplus 0 = \hh(1). \qedhere
	\end{equation*}
\end{proof}

We use the convention $\hh^0(x) = \jj^0(x) = x$.

\begin{lemma} \label{l:j-many-sum}
	For every $n \in \Np$ and every $x$ in {\atmvm} we have 
	\begin{equation} \label{e:j-as-sum}
		\jj^n(x) = \hh^1(1) \oplus \hh^2(1) \oplus \hh^3(1) \oplus \dots \oplus \hh^n(1) \oplus \hh^{n}(x)
	\end{equation}
	and
	\begin{equation} \label{e:h-as-sum}
		\hh^n(x) = \jj^1(0) \odot \jj^2(0) \odot \jj^3(0) \odot \dots \odot \jj^n(0) \odot \jj^{n}(x).
	\end{equation}
\end{lemma}

\begin{proof}
	We prove \cref{e:j-as-sum} by induction on $n$.
	The case $n = 1$ reads as $\jj(x) = \hh(1) \oplus \hh(x)$, which is just \cref{ax:tmvm:cohf-via-hf}.
	Let us suppose that \cref{e:j-as-sum} holds for a fixed $n \in \N$, and let us prove that it holds for $n + 1$.
	We have
	\begin{align*}
		\jj^{n+1}(0)	& = \jj^{n}\mathopen{}\big(\jj(x)\big)\mathclose{}																																	&&\\
							& = \hh^1(1) \oplus \hh^2(1) \oplus \hh^3(1) \oplus \dots \oplus \hh^n(1) \oplus \hh^{n}\mathopen{}\big(\jj(x)\big)\mathclose{}							&& \text{(ind.\ hyp.)}\\
							& = \hh^1(1) \oplus \hh^2(1) \oplus \hh^3(1) \oplus \dots \oplus \hh^n(1) \oplus \hh^{n}\mathopen{}\big(\hh(1) \oplus \hh(x)\big)\mathclose{}			&& \text{(\cref{ax:tmvm:cohf-via-hf})}\\
							& = \hh^1(1) \oplus \hh^2(1) \oplus \hh^3(1) \oplus \dots \oplus \hh^n(1) \oplus \hh^{n+1}(1) \oplus \hh^{n+1}(x).	&& \text{(\cref{ax:tmvm:double-h})}
	\end{align*}
	Dually we have \cref{e:h-as-sum}.
\end{proof}

\begin{lemma} \label{l:one-step-up}
	In {\atmvm}, for every $n \in \Np$, we have
	\[
		\hh^n(x) \oplus \hh^n(x) = \hh^{n-1}(x),
	\]
	and
	\[
		\jj^n(x) \odot \jj^n(x) = \jj^{n-1}(x).
	\]
\end{lemma}

\begin{proof}
	The first equation holds by \cref{ax:tmvm:double-oplus}, and the second equation is dual.
\end{proof}

\begin{lemma} \label{l:complement}
	In {\atmvm}, for every $n \in \N$ we have
	\[
		\hh^n(1) \oplus \jj^n(0) = 1
	\]
	and
	\[
		\hh^n(1) \odot \jj^n(0) = 0.
	\]
\end{lemma}

\begin{proof}
	We have
	\begin{align*}
		\hh^n(1) \oplus \jj^n(0) 	& = \hh^n(1) \oplus \big(\hh^n(1) \oplus \hh^{n-1}(1) \oplus \dots \oplus \hh^1(1) \oplus \hh^n(0)\big)	&& \text{(\cref{l:j-many-sum})}\\
											& = \hh^n(1) \oplus \hh^n(1) \oplus \hh^{n-1}(1) \oplus \dots \oplus \hh^1(1).					&& \text{(\cref{l:hf-0})}
	\end{align*}
	By \cref{l:one-step-up}, we have 
	\begin{gather*}
		\hh^{n}(1) \oplus \hh^n(1) = \hh^{n-1}(x),\\
		\hh^{n-1}(1) \oplus \hh^{n-1}(1) = \hh^{n-2}(x),\\
		\vdots\\
		\hh^1(1)\oplus \hh^1(1) = 1.
	\end{gather*}
	Hence,
	\[
		\hh^n(1) \oplus \hh^n(1) \oplus \hh^{n-1}(1) \oplus \dots \oplus \hh^1(1) = 1.
	\]
	So, $\hh^n(1) \oplus \jj^n(0) = 1$.
	Dually, $\hh^n(1) \odot \jj^n(0) = 0$.
\end{proof}

\begin{lemma} \label{l:reduct}
	Every {\tmvm} has a reduct which is {\admvm} (in the sense of \cref{d:dmvm-alt}), obtained by setting, for each $n \in \N$, $\dd_n \df \hh^n(1)$ and $\uu_n \df \jj^n(0)$.
\end{lemma}

\begin{proof}
	By convention, we have $\dd_0 = \hh^0(1) = 1$ and $\uu_0 = \jj^n(0) = 0$.
	\Cref{ax:dmvm-alt:0} holds because, by \cref{ax:tmvm:mvm}, $\langle A; \oplus, \odot, \lor, \land, \uu_0, \dd_0\rangle$ is {\amvm}.
	\Cref{ax:dmvm-alt:d-oplus} holds because, by \cref{l:one-step-up}, for every $n \in \Np$ we have 
	\[
		\dd_n \oplus \dd_n = \hh^n(1) \oplus \hh^n(1) = \hh^{n-1} = \dd_{n-1}.
	\]
	\Cref{ax:dmvm-alt:u-odot} is dual.
	For every $n \in \Np$, we have
	\[
		\dd_n = \hh^n(1)\stackrel{\text{\cref{l:one-step-up}}}{\leq} \hh(1) \stackrel{\text{\cref{l:two-ways-for-one-half}}}{=} \jj(0).
	\]
	Hence, we have
	\[
		\dd_n \odot \dd_n \leq \jj(0) \odot \jj(0) \stackrel{\text{\cref{ax:tmvm:double-odot}}}{=} 0.
	\]
	Hence, by \cref{l:bounded-lattice}, we have $\dd_n \odot \dd_n = 0$. Thus, \cref{ax:dmvm-alt:d-odot} holds.
	Dually, \cref{ax:dmvm-alt:d-oplus} holds.
	\Cref{ax:dmvm-alt:oplus} holds because, by \cref{l:complement}, for every $n \in \Np$, we have $\dd_n \oplus \uu_n = \hh^n(1) \oplus \jj^n(0)$.
	Dually, \cref{ax:dmvm-alt:odot} holds.
\end{proof}

\begin{lemma} \label{l:pres-also-h}
	Let $A$ be {\atmvm}, and let $f \colon A \to [0,1]$ be a function that preserves $\oplus$, $\odot$, $0$ and $1$.
	Then $f$ preserves also $\hh$ and $\jj$.
\end{lemma}
\begin{proof}
	For every $x \in A$ we have $\hh(x) \oplus \hh(x) = x$ (by \cref{ax:tmvm:double-oplus}) and $\hh(x) \odot \hh(x) = 0$ (by \cref{l:no-exceed}).
	Since $f$ preserves $\oplus$, $\odot$, and $0$, we have $f(\hh(x)) \oplus f(\hh(x))  =  f(x)$ and $f(\hh(x)) \odot f(\hh(x)) =  0$.
	Therefore $f(\hh(x)) = \half{f(x)}= \hh(f(x))$.
	Dually for $\jj$.
\end{proof}


\section{Finite equational axiomatisation}

By \cref{l:reduct}, every {\tmvm} has a reduct which is {\admvm}.
Therefore, we inherit the notation for the binary term $\trunc_n$ and the $n$-ary term $\aprj_n$.
In the language of {\tmvms}, these terms can be expressed as follows.

\begin{notation} \label{n:trunc-t}
	For every $n \in \N$, we define a binary term $\trunc_n$ in the language of {\tmvms}:
	\[
		\trunc_{n}(x, y) \df \big(x \land (y \oplus \hh^n(1))\big) \lor \big(y \odot \jj^n(0)\big).
	\]
\end{notation}

\begin{notation} \label{n:aprj-t}
	Inductively on $n \in \Np$, we define a term $\aprj_n$ of arity $n$ in the language of {\tmvms}:
	\begin{align*}
		\aprj_1(x_1)					& \df x_1; \\
		\begin{split}
			\aprj_n(x_1, \dots,x_n)	& \df	\trunc_{n-1}\mathopen{}\big(x_n,\aprj_{n-1}(x_1, \dots,x_{n-1})\big)\mathclose{}\\
											& = \big(x_n \land \left(\aprj_{n-1}(x_1, \dots,x_{n-1}) \oplus \hh^n(1)\right)\big)	\lor \big(\aprj_{n-1}(x_1, \dots,x_{n-1}) \odot \jj^n(0)\big).
		\end{split}
	\end{align*}
\end{notation}

\begin{definition}\label{d:TMVMinfty}
	 A \emph{\tmvminfty}\index{MV-monoidal algebra!limit $2$-divisible} $\langle A; \oplus, \odot, \lor, \land, 0, 1, \hh, \jj, \limop \rangle$ (arities $2$, $2$, $2$, $2$, $0$, $0$, $1$, $1$, $\omega$) is an algebra  with the following properties.
	\begin{enumerate} [label = LTE\arabic*., ref = LTE\arabic*, start = 0, leftmargin=\widthof{LTE4'.} + \labelsep]
		\item \label[axiom]{ax:atmv-infty}
				The algebra $\langle A; \oplus, \odot, \lor, \land, 0, 1, \hh, \jj \rangle$ is {\atmvm} (see \cref{d:tmvm}).
		\item \label[axiom]{ax:constant} 
				$\limop(x, x, x, \dots) = x$.
		\item \label[axiom]{ax:delta-ominus} 
				$\limop\mathopen{}\big(\trunc_{0}(x, y), \trunc_{1}(x, y), \trunc_{2}(x, y), \dots\big)\mathclose{} = y$.
		\item \label[axiom]{ax:rhoify}
				$\limop(x_1, x_2, x_3, \dots) = \limop(\aprj_1(x_1), \aprj_2(x_1, x_2), \aprj_3(x_1, x_2, x_3), \dots)$.
		\item \label[axiom]{ax:base-case}
				$\aprj_2(x_1, x_2) \odot \jj(0)  \leq  \limop(x_1, x_2, x_3, \dots)  \leq  \aprj_2(x_1, x_2) \oplus \hh(1)$.
		\item \label[axiom]{ax:inductive-step-0} 
				$\limop(x_1, x_2, x_3, \dots) \oplus \limop(x_1, x_2, x_3, \dots)$\\
				$= \limop\mathopen{}\big(\aprj_2(x_1, x_2) \oplus \aprj_2(x_1, x_2), \aprj_3(x_1, x_2, x_3) \oplus \aprj_3(x_1, x_2, x_3), \dots\big)\mathclose{}$.
		\item \label[axiom]{ax:inductive-step-1} 
				$\limop(x_1, x_2, x_3, \dots) \odot \limop(x_1, x_2, x_3,\dots)$\\ 
				$= \limop\mathopen{}\big(\aprj_2(x_1,x_2) \odot \aprj_2(x_1, x_2), \aprj_3(x_1, x_2, x_3) \odot \aprj_3(x_1, x_2, x_3), \dots\big)\mathclose{}$.
	\end{enumerate}
\end{definition}
We let $\TMVMinfty$ denote the category of {\tmvminftys} with homomorphisms.

\begin{definition}
	A \emph{$2$-Cauchy sequence in {\admvm}}\index{$2$-Cauchy sequence!in {\admvm}}\index{sequence!$2$-Cauchy!in {\admvm}} $A$ is a sequence $(x_n)_{n \in \Np}$ in $A$ such that, for every $n \in \Np$, we have
	\[
		x_n \odot \left(1 - \frac{1}{2^n}\right) \leq x_{n + 1} \leq x_n \oplus \frac{1}{2^n}.
	\] 
\end{definition}

\begin{lemma}
	Let $(x_1, x_2, x_3, \dots)$ be a sequence in {\admvm}.
	Then $(\aprj_1(x_1), \aprj_2(x_1, x_2), \aprj_3(x_1, x_2, x_3), \dots)$ is a $2$-Cauchy sequence.
\end{lemma}
\begin{proof}
	Immediate from the definition of $\aprj_n$.
\end{proof}

\begin{lemma}
	Let $(x_1, x_2, x_3, \dots)$ be a $2$-Cauchy sequence in {\admvm}.
	Then, for every $n \in \Np$, we have $\aprj_n(x_1, \dots, x_n) = x_n$.
\end{lemma}
\begin{proof}
	Immediate from the definition of $\aprj_n$.
\end{proof}

\begin{lemma} \label{l:ineq-for-2C}
	Let $A$ be {\amvm}, and let $x, y, z, w \in A$.
	Then
	\[
		(x \odot y) \oplus (z \odot w) \geq (x \oplus z) \odot y \odot w.
	\]
\end{lemma}
\begin{proof}
	Recall, from \cref{l:almost associative}, that for all $a,b,c \in A$ we have $a \odot (b \oplus c) \leq (a \odot b) \oplus c$.
	Therefore we have $(x \odot y) \oplus (z \odot w) \geq ((x \odot y) \oplus z) \odot w$.
	Using again \cref{l:almost associative}, we have $((x \odot y) \oplus z) \odot w \geq (x \oplus z) \odot y \odot w$.
\end{proof}

\begin{lemma}\label{l:oplus-and-odot-are-$2$-Cauchy}
	Given a $2$-Cauchy sequence $(x_n)_{n \in \Np}$ in {\admvm}, the sequences $(x_{n+1} \oplus x_{n+1})_{n \in \Np}$ and $(x_{n+1} \odot x_{n+1})_{n \in \Np}$ are $2$-Cauchy.
\end{lemma}
\begin{proof}
	Let $k \in \Np$.
	Since $(x_n)_{n \in \Np}$ is $2$-Cauchy, we have
	\[
		x_{k+1} \odot \left(1 - \frac{1}{2^{k+1}}\right)  \leq  x_{k + 2} \leq x_{k+1} \oplus \frac{1}{2^{k+1}}.
	\]
	Therefore we have
	\begin{align*}
		x_{k+2} \oplus x_{k+2}	& \geq	\left(x_{k+1} \odot \left(1 - \frac{1}{2^{k+1}}\right)\right) \oplus \left(x_{k+1} \odot \left(1 - \frac{1}{2^{k+1}}\right)\right)	&&\\
										& \geq	(x_{k+1} \oplus x_{k+1}) \odot \left(\left(1 - \frac{1}{2^{k+1}}\right) \odot \left(1 - \frac{1}{2^{k+1}}\right)\right)					&& \text{(\cref{l:ineq-for-2C})}\\
										& =		(x_{k+1} \oplus x_{k+1}) \odot \left(1 - \frac{1}{2^{k}}\right).																							&&
	\end{align*}
	Moreover, we have
	\begin{align*}
		x_{k+2} \oplus x_{k+2}	& \leq  \left(x_{k+1} \oplus \frac{1}{2^{k+1}}\right) \oplus  \left(x_{k+1} \oplus \frac{1}{2^{k+1}}\right)\\
										& =		(x_{k+2} \oplus x_{k+2}) \oplus \frac{1}{2^k}.
	\end{align*}
	It follows that $(x_{n+1} \oplus x_{n+1})_{n \in \Np}$ is a $2$-Cauchy sequence.
	Dually for $(x_{n+1} \odot x_{n+1})_{n \in \Np}$.
\end{proof}

The reason why \cref{ax:base-case,ax:inductive-step-0,ax:inductive-step-1} are written in the way they are written is to make it clear that they are equational.
However, this presentation is not the clearest one.
So, we point out that these axioms, are equivalent, given \cref{ax:atmv-infty,ax:rhoify}, to the following statements.
\begin{enumerate} [label = LTE\arabic*'., ref = LTE\arabic*', start = 4, leftmargin=\widthof{LTE4'.} + \labelsep]
	\item 
			If $(x_n)_{n \in \Np}$ is a $2$-Cauchy sequence, then 
			\[
				x_2 \odot \jj(0)  \leq  \limop(x_1, x_2, x_3, \dots)  \leq  x_2 \oplus \hh(1).
			\]
	\item \label{ax:ind-step0'}
			If $(x_n)_{n \in \Np}$ is a $2$-Cauchy sequence, then 
			\[
				\limop(x_1, x_2, x_3, \dots) \oplus \limop(x_1, x_2, x_3, \dots) =  \limop(x_2 \oplus x_2, x_3 \oplus x_3, x_4 \oplus x_4, \dots).
			\]
	\item \label{ax:ind-step1'}
			If $(x_n)_{n \in \Np}$ is a $2$-Cauchy sequence, then 
			\[
				\limop(x_1, x_2, x_3, \dots) \odot \limop(x_1, x_2, x_3, \dots) = \limop(x_2 \odot x_2, x_3 \odot x_3, x_4 \odot x_4, \dots).
			\]
\end{enumerate}
\begin{lemma} \label{l:int-is-t-infty}
	The algebra $[0,1]$ with obvious interpretation of the operation symbols is {\atmvminfty}.
\end{lemma}
\begin{proof}
	The fact that $[0,1]$ satisfies \cref{ax:atmv-infty,ax:constant,ax:delta-ominus} is proved in \cref{l:properties of W}.
	\Cref{ax:rhoify} holds by the definition of $\limop$, together with \cref{l:2-Cauchy,l:effect-of-rho}.
	\Cref{ax:base-case} is the case $n=2$ of \cref{ax:delta-sandwich-ax-inf} in \cref{d:DMVMinfty}, which was proved in \cref{l:properties of W} to hold in $[0,1]$.
	Let us now prove \cref{ax:ind-step0'}.
	Let $(x_1, x_2, x_3, \dots)$ be a $2$-Cauchy sequence.
	By \cref{l:oplus-and-odot-are-$2$-Cauchy}, the sequence $(x_2 \oplus x_2, x_3 \oplus x_3, x_4 \oplus x_4, \dots)$ is $2$-Cauchy.
	By \cref{l:to-their-limit}, we have $\limop(x_1, x_2, x_3, \dots) = \lim_{n \to \infty} x_n$ and 
	\[
		\limop(x_2 \odot x_2, x_3 \odot x_3, x_4 \odot x_4, \dots) = \lim_{n \to \infty} x_n \oplus x_n,
	\]
	and this last number, by continuity of $\oplus \colon [0,1]^2 \to [0,1]$, coincides with 
	\[
		\left(\lim_{n \to \infty} x_n\right) \oplus \left(\lim_{n\to \infty} x_n\right).
	\]
	So, both $\limop(x_1, x_2, x_3, \dots) \oplus \limop(x_1, x_2, x_3, \dots)$ and $\limop(x_2 \oplus x_2, x_3 \oplus x_3, x_4 \oplus x_4, \dots)$ coincide with $(\lim_{n \to \infty} x_n) \oplus (\lim_{n\to \infty} x_n)$.
	This proves \cref{ax:ind-step0'}.
	Analogously for \cref{ax:ind-step1'}.
\end{proof}
\begin{proposition}\label{p:W is Archimedean-t}
	The $(\{\oplus, \odot, \lor, \land\} \cup (\Dyadu))$-reduct of any {\tmvminfty} is Archimedean.
\end{proposition}
\begin{proof}
	As proved in \cref{l:arch-by-equations}, this follows from \cref{ax:constant,ax:delta-ominus}.
\end{proof}

\begin{lemma}\label{l:iterating tau-1}
	Let $n \in \Np$ and, for each $i \in \{1, \dots, n\}$, let $\alpha_i$ be either the term operation $x \mapsto x \oplus x$ or the term operation $x \mapsto x \odot x$.
	Then, for every $2$-Cauchy sequence $(x_n)_{n \in \Np}$ in {\atmvminfty}, we have
	\[
		\alpha_n \dots \alpha_1(x_{n+2}) \odot \jj(0)  \leq  \alpha_n \dots \alpha_1(\limop(x_1, x_2, x_3,\dots))  \leq  \alpha_n \dots \alpha_1(x_{n+2}) \oplus \hh(1).
	\]
\end{lemma}
\begin{proof}
	By iterated application of \cref{ax:inductive-step-0} and \cref{ax:inductive-step-1}, which is possible by \cref{l:oplus-and-odot-are-$2$-Cauchy}, we obtain
	\[
		\alpha_n \dots \alpha_1 (\limop(x_1, x_2, x_3, \dots)) = \limop\mathopen{}\big(\alpha_n \dots \alpha_1(x_{n+1}), \alpha_n \dots \alpha_1(x_{n+2}), \alpha_n \dots \alpha_1(x_{n+3}), \dots\big)\mathclose{}.
	\]
	By \cref{ax:base-case} we have 
	\begin{align*}
		&\alpha_n \dots \alpha_1(x_{n+1}) \odot \jj(0)\\
		&\leq  \limop\mathopen{}\big(\alpha_n \dots \alpha_1(x_{n+2}), \alpha_n \dots \alpha_1(x_{n+2}), \alpha_n \dots \alpha_1(x_{n+3}), \dots\big)\mathclose{} \\
		&\leq  \alpha_n \dots \alpha_1(x_{n+2}) \oplus \hh(1).
	\end{align*}
	The desired statement follows.
\end{proof}

\begin{lemma} \label{l:piecewise-linear}
	For every $n \in \Np$, the following two conditions are equivalent for a function $f \colon [0,1] \to [0,1]$.
	\begin{enumerate}
		\item
				There exists an $n$-tuple $(\alpha_1, \dots, \alpha_n)$ of functions from $[0,1]$ to $[0,1]$ belonging to $\{x \mapsto x \oplus x, x \mapsto x \odot x \}$ such that $f = \alpha_n \circ \dots \circ \alpha_1$.
		\item \label{i:interpolant}
				There exists $k \in \{0,\dots, 2^n-1\}$ such that $f$ is the linear interpolant of $(0, 0)$, $(\frac{k}{2^n}, 0)$, $(\frac{k + 1}{2^n}, 1)$, $(1, 1)$, i.e.,  for every $x \in [0,1]$, we have
				\[
					f(x) =	{
								\begin{cases}
									0				& \text{if } x \in [0, \frac{k}{2^n}];\\
									2^n x - k	& \text{if } x \in [\frac{k}{2^n}, \frac{k+1}{2^n}];\\
									1				& \text{if } x \in [\frac{k+1}{2^n}, 1].
								\end{cases}
								}
				\]
				(See \cref{fig:piecewise} for a plot of $f$ for $n = 2$ and $k \in \{0, 1, 2, 3\}$.)
	\end{enumerate}
\end{lemma}
\begin{proof}
	This can be proved by induction on $n$.
\end{proof}
\begin{figure}[ht]
	\centering
	\begin{subfigure}{0.2\textwidth}
		\includegraphics{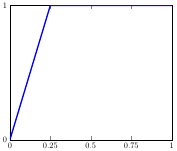}
		\caption{$k = 0$}
	\end{subfigure}
	\quad
	\begin{subfigure}{0.2\textwidth}
		\includegraphics{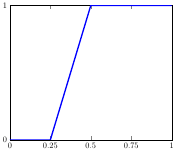}
		\caption{$k = 1$}
	\end{subfigure}
	\quad
	\begin{subfigure}{0.2\textwidth}
		\includegraphics{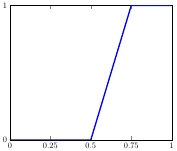}
		\caption{$k = 2$}
	\end{subfigure}
	\quad
	\begin{subfigure}{0.2\textwidth}
		\includegraphics{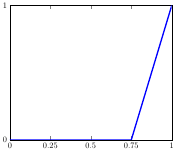}
		\caption{$k = 3$}
	\end{subfigure}
	\caption{The plot of the function $f$ of \cref{i:interpolant} in \cref{l:piecewise-linear} for $n = 2$ and $k \in \{0, 1, 2, 3\}$.}
	\label{fig:piecewise}
\end{figure}
\begin{lemma}\label{l:distance by tau}
	Let $a, b \in [0,1]$, let $n \in \Np$ and suppose that, for every $n$-tuple of functions $\alpha_1, \dots, \alpha_n \colon [0,1] \to [0,1]$ belonging to $\{x \mapsto x \oplus x, x \mapsto x \odot x \}$, we have
	\[
		\lvert \alpha_n\dots\alpha_1(a) - \alpha_n\dots\alpha_1(b)\rvert < 1.
	\]
	Then $\lvert a - b \rvert < \frac{1}{2^{n-1}}$.
\end{lemma}
\begin{proof}
	We proceed by contraposition: suppose $\lvert a- b \rvert  \geq  \frac{1}{2^{n-1}}$.
	Then, there exists $k \in \{0,\dots,2^n-1\}$ such that $\left\{\frac{k}{2^n}, \frac{k+1}{2^n}\right\} \seq [a, b]$.
	Consider the function 
	\begin{align*}
		f \colon	[0,1]	& \longrightarrow	[0,1]\\
					x		& \longmapsto		{
													\begin{cases}
														0				& \text{if } x \in [0, \frac{k}{2^n}]\\
														2^n x - k	& \text{if } x \in [\frac{k}{2^n}, \frac{k+1}{2^n}]\\
														1				& \text{if } x \in [\frac{k+1}{2^n}, 1]
													\end{cases}
													}
	\end{align*}
	By \cref{l:piecewise-linear}, there exists an $n$-tuple $(\alpha_1, \dots, \alpha_n)$ of functions from $[0,1]$ to $[0,1]$ belonging to $\{x \mapsto x \oplus x, x \mapsto x \odot x \}$ such that $f = \alpha_n \circ \dots \circ \alpha_1$.
	Hence we have
	\[
		\lvert {\alpha}_n \dots {\alpha}_1(a) - {\alpha}_n\dots {\alpha}_1(b) \rvert  =  \lvert f(a) - f(b) \rvert  =  \lvert 0 - 1 \rvert = 1,
	\]
	which concludes the proof.
\end{proof}

\begin{lemma} \label{l:full-on-int-t}
	Every function from {\atmvminfty} to $[0,1]$ that preserves $\oplus$, $\odot$, $\lor$, $\land$, $0$, $1$, $\hh$ and $\jj$ preserves also $\limop$.
\end{lemma}
\begin{proof}
	Let $A$ be {\atmvminfty} and let $f$ be a function from $A$ to $[0,1]$ that preserves $\oplus$, $\odot$, $\lor$, $\land$, $0$, $1$, $\hh$ and $\jj$.
	\begin{claim} \label{cl:for-2Cauchy}
		for every $2$-Cauchy sequence $(x_n)_{n \in \Np}$ in $A$, we have $f(\limop(x_1, x_2, x_3, \dots)) = \limop(f(x_1), f(x_2), f(x_3), \dots)$.
	\end{claim}
	\begin{claimproof}
		Let $(x_n)_{n \in \Np}$ be a $2$-Cauchy sequence in $A$.
		Let $n \in \Np$.
		By \cref{l:iterating tau-1}, for every $n$-tuple $(\alpha_1, \dots, \alpha_n)$ of term operations in $\{x \mapsto x \oplus x, x \mapsto x \odot x\}$, we have
		\[
			\alpha_n \dots \alpha_1(x_{n+2}) \odot \jj(0)  \leq  \alpha_n \dots \alpha_1(\limop(x_1, x_2, x_3,\dots))  \leq  \alpha_n \dots \alpha_1(x_{n+2}) \oplus \hh(1),
		\]
		and thus, using the fact that $f$ preserves $\oplus$, $\odot$, $\lor$, $\land$, $0$, $1$, $\hh$ and $\jj$, we have
		\[
			\alpha_n \dots \alpha_1(f(x_{n+2})) \odot \frac{1}{2}  \leq  \alpha_n \dots \alpha_1(f(\limop(x_1, x_2, x_3,\dots)))  \leq  \alpha_n \dots \alpha_1(f(x_{n+2})) \oplus \frac{1}{2},
		\]
		i.e.\
		\[
			\lvert \alpha_n \dots \alpha_1(f(\limop(x_1, x_2, x_3,\dots))) - \alpha_n \dots \alpha_1(f(x_{n+2})) \rvert \leq \frac{1}{2}.
		\]
		Therefore, by \cref{l:distance by tau}, we have
		\[
			\lvert f(\limop(x_1, x_2, x_3,\dots)) - f(x_{n+2}) \rvert  \leq  \frac{1}{2^{n-1}}.
		\]
		It follows that $f(\limop(x_1, x_2, x_3,\dots)) = \lim_{n \to \infty} f(x_n)$.
		It is easy to see that the sequence $(f(x_1), f(x_2), f(x_3), \dots)$ is $2$-Cauchy.
		Therefore, by \cref{l:to-their-limit},
		\[
			\lim_{n \to \infty} f(x_n) = \limop(f(x_1), f(x_2), f(x_3), \dots).
		\]
		It follows that 
		\[
			f(\limop(x_1, x_2, x_3,\dots)) = \limop(f(x_1), f(x_2), f(x_3), \dots),
		\]
		settling our claim.
	\end{claimproof}
	Let now $(x_n)_{n \in \Np}$ be an arbitrary $2$-Cauchy sequence in $A$.
	Then, we have
	\begin{align*}
		& f\mathopen{}\big(\limop(x_1, x_2, x_3,\dots)\big)\mathclose{}\\
		& = f\mathopen{}\big(\limop(\aprj_1(x_1), \aprj_2(x_1,x_2), \aprj_3(x_1, x_2, x_3),\dots)\big)\mathclose{}						&& \by{\cref{ax:rhoify}}\\
		& = \limop\mathopen{}\big(f(\aprj_1(x_1)), f(\aprj_2(x_1,x_2)), f(\aprj_3(x_1, x_2, x_3)),\dots\big)\mathclose{}				&& \by{\cref{cl:for-2Cauchy}} \\
		& = \limop\mathopen{}\big(\aprj_1(f(x_1)), \aprj_2(f(x_1),f(x_2)), \aprj_3(f(x_1), f(x_2), f(x_3)),\dots\big)\mathclose{}	&& \\
		& = \limop\mathopen{}\big(f(x_1), f(x_2), f(x_3), \dots\big)\mathclose{}. 																	&& \qedhere
	\end{align*}
\end{proof}

\begin{proposition}\label{p:isp-t}
	We have
	\[
		\TMVMinfty = \opS\opP([0,1]).
	\]
\end{proposition}
\begin{proof}
	We first prove $\TMVMinfty \seq \opS\opP([0,1])$.
	By \cref{p:W is Archimedean-t}, the reduct to the signature $\{\oplus, \odot, \lor, \land\}\cup (\Dyadu)$ of any {\tmvminfty} is isomorphic to a subalgebra of a power of the algebra $[0,1]$ with standard interpretation of the operations.
	Let $\iota \colon A \hookrightarrow [0,1]^\kappa$ denote the corresponding inclusion.
	We claim that $\iota$ preserves also $\hh$, $\jj$, and $\limop$.
	By \cref{l:pres-also-h}, every function from {\atmvminfty} to $[0,1]$ which preserves $\oplus$, $\odot$, $\lor$, $\land$, $0$ and $1$ preserves also $\hh$ and $\jj$; by \cref{l:full-on-int-t}, it preserves also $\limop$.
	Thus, given any $i \in \kappa$, the composite $A \xhookrightarrow{\iota} [0,1]^{\kappa} \xrightarrow{\pi_i} [0,1]$---where $\pi_i$ denotes the $i$-th projection---preserves $\hh$, $\jj$, and $\limop$.
	Therefore, $\iota$ preserves also $\limop$, settling our claim, and thus $A$ is isomorphic to a subalgebra of a power of the algebra $[0,1]$.
	Therefore, $\TMVMinfty \seq \opS\opP([0,1])$.
	
	The opposite inclusion $\TMVMinfty \supseteq \opS\opP([0,1])$ is guaranteed by the following facts.
	\begin{enumerate}
		\item
			The algebra $[0,1]$ in the signature $\{\oplus, \odot, \lor, 0, 1, \land, \hh, \jj, \limop\}$ with standard interpretation of the operations is {\atmvminfty} by \cref{l:int-is-t-infty}.
		\item
			The class of algebras $\TMVMinfty$ is a variety, and so it is closed under products and subalgebras. \qedhere
	\end{enumerate}
\end{proof}

\begin{lemma} \label{l:clone-mon-cont-t}
	For every cardinal $\kappa$, the set of interpretations of the term operations of arity $\kappa$ on the algebra $[0,1]$ in the signature $\{\oplus, \odot, \lor, \land, 0, 1, \hh, \jj, \limop\}$ is the set of order-preserving continuous functions from $[0,1]^\kappa$ to $[0, 1]$.
\end{lemma}
\begin{proof}
	Let $\kappa$ be a cardinal, and let $L_\kappa$ be the set of interpretations on $[0,1]$ of the term operations of arity $\kappa$.
	We now apply \cref{t:StWe-unit} with $X = [0,1]^\kappa$: note that $X$ is a compact ordered space and $L_\kappa$ is order-separating because it contains the projections.
	Therefore, the set of order-preserving continuous function from $[0,1]^\kappa$ to $[0,1]$ coincides with the closure of $L_\kappa$ under uniform convergence.
	By \cref{l:closure}, using the fact that $L_\kappa$ is closed under $\limop$, we obtain that the closure of $L_\kappa$ under uniform convergence is $L_\kappa$ itself.
\end{proof} 

\begin{theorem} \label{t:term-equiv-t}
	The classes $\CMiso$ and $\TMVMinfty$ are term-equivalent varieties of algebras.
\end{theorem}
\begin{proof}
	By \cref{p:isp-t}, the class $\TMVMinfty$ consists of the algebras in the signature $\{\oplus, \odot, \lor, \land, 0, 1, \hh, \jj, \limop\}$ which are isomorphic to a subalgebra of a power of $[0,1]$.
	By definition of $\CMiso$, the class $\CMiso$ consists of the $\SignCM$-algebras which are isomorphic to a subalgebra of a power of $[0,1]$.
	The clone of term operations of the $\SignCM$-algebra $[0,1]$ consists of the order-preserving continuous functions.
	By \cref{l:clone-mon-cont-t}, the clone of term operations of the algebra $[0,1]$ in the signature $\{\oplus, \odot, \lor, \land, 0, 1, \hh, \jj, \limop\}$ consists of the order-preserving continuous functions.
	The class $\TMVMinfty$ is clearly a variety of algebras.
	By \cref{r:ISP-te}, the class $\CMiso$ is a variety which is term-equivalent to $\TMVMinfty$.
\end{proof}

\begin{theorem} \label{t:axiomatisation-t}
	The category $\CompOrd$ of compact ordered spaces is dually equivalent to the variety $\TMVMinfty$ of {\tmvminftys} (see \cref{d:TMVMinfty}).
\end{theorem}
\begin{proof}
	The category $\TMVMinfty$ is isomorphic to $\CMiso$ by \cref{t:term-equiv}, and $\CMiso$ is dually equivalent to $\CompOrd$ by \cref{t:duality-not-explicit}.
\end{proof}


\section{Conclusions}

We showed that the dual of the category of compact ordered spaces admits a finite equational axiomatisation.
This concludes the main development of our work.


\chapter{Conclusions}\label{chap:conclusions}

We have concluded our journey into the axiomatisability of the dual of the category of compact ordered spaces.
We started by motivating compact ordered spaces as the correct solution for $X$ in the equation
\begin{equation*}
	\text{Stone spaces : Priestley spaces = Compact Hausdorff spaces : $X$.}
\end{equation*}
Then, we observed that Stone spaces, Priestley spaces, and compact Hausdorff spaces all have a dual which is equivalent to a variety of (possibly infinitary) algebras, and we raised a question, which had been left open in \cite{HNN2018}: does the same happen for compact ordered spaces?
We showed that this is the case: the category $\CompOrd$ of compact ordered spaces and order-preserving continuous maps is dually equivalent to the variety $\CMiso$---in the signature $\SignCM$ of order-preserving continuous functions from powers of $[0,1]$ to $[0,1]$---consisting of the subalgebras of the powers of the $\SignCM$-algebra $[0,1]$.
Clearly, $[0,1]$ could be replaced by any of its isomorphic copies in $\CompOrd$.
We also observed that each operation in the theory of this variety depends on at most countably many coordinates.

Moreover, we proved that the countable bound on the arities is the best possible: $\CompOrd$ is not dually equivalent to any variety of finitary algebras.

After these results, we addressed the problem of establishing an explicit equational axiomatisation.
We pushed our investigation to the point of obtaining a finite equational axiomatisation, which established an ordered version of the main result of \cite{MarraReggio2017}.
From a historical point of view, our choice of the primitive operations is very natural: it is based on the lattice operations and on the addition of real numbers, following the tradition of several dualities for compact Hausdorff spaces \cite{KreinKrein1940,Yosida1941,Stone1941,Kakutani1941}.
Moreover, MV-algebras were at the base of the axiomatisation of the dual of compact Hausdorff spaces in \cite{MarraReggio2017}, so we found it is reasonable to base our axiomatisation on the order-preserving term-operations of MV-algebras, which led us to the notion of MV-monoidal algebra.

However, beyond the historical motivation, our choice of the generating set of operations remains somewhat arbitrary, and we have left completely unaddressed the problem of identifying which other choices of primitive operations give rise to an adequate duality for $\CompOrd$.
To the best of the author's understanding, one of the reasons why the interval $[-\infty, + \infty]$ is usually disregarded in dualities for compact Hausdorff spaces is because of the non-existence of a continuous extension of the binary addition at infinity.
This has lead, in the dualities available in the literature, to either a loss of first-order definability, or the employment of truncated addition, which carries axioms that some find unwieldy.
However, in our discussion on the Stone-Weierstrass theorem, we presented the results of M.\ H.\ Stone with special attention to his characterisation of the topological closure of any given lattice of continuous real functions over a compact space.
This result seems to suggests that the binary addition could be replaced, for example, by the set of affine unary functions, which, instead, admit continuous extensions at infinity.
In the ordered case, only the order-preserving ones are to be considered.
The author wonders:
May a simpler description of $\CompOrdop$ be obtained starting from the order-preserving affine unary functions on $[-\infty, + \infty]$, together with the lattice operations?

In closing, we indicate how the results in this thesis can be used to strengthen one of the results by \cite{HNN2018} about coalgebras\index{coalgebra} for the Vietoris functor on the category of compact ordered spaces.
In fact, the theory of coalgebras was one of the motivations for the algebraic study of $\CompOrdop$ in \cite{HNN2018}.
It is well known that the category of modal algebras is dually equivalent to the category of coalgebras for the Vietoris endofunctor on the category of Stone spaces\index{Vietoris functor!for Stone spaces}\index{functor!Vietoris|see{Vietoris functor}}; for more details, see \cite{KupkeKurzVenema2004}.
A similar study based on the Vietoris functor\index{Vietoris functor!for Priestley spaces} on the category of Priestley spaces and monotone continuous maps can be found in \cite{CignoliLafalcePetrovich1991,Petrovich1996,BonsangueKurzRewitzky2007}.
Dualities for the Vietoris endofunctor\index{Vietoris functor!for compact Hausdorff spaces} on the category of compact Hausdorff spaces appear in \cite{BezhanishviliBezhanishviliHarding2015,BezhanishviliBezhanishviliHarding2015a,BezhanishviliCaraiMorandi}.

A similar approach can be carried out for compact ordered spaces.
We recall (see e.g.\ \cite[Section~4]{HN2018}) that the \emph{Vietoris functor for compact ordered spaces}\index{Vietoris functor!for compact ordered spaces} $V \colon \CompOrd \to \CompOrd$ sends a compact ordered space $X$ to the space $V(X)$ of all closed up-sets of $X$, ordered by reverse inclusion $\supseteq$, and equipped with the topology generated by the sets
\begin{align*}
	\{A \seq X \mid A \text{ closed up-set and }A \cap U \neq \emptyset\}\ \  	&\ \  \text{($U \seq X$ open down-set)},\\
	\{A \seq X \mid A \text{ closed up-set and }A \cap K = \emptyset\}\ \ 		&\ \  \text{($K \seq X$ closed down-set)}.
\end{align*}
Given a map $f \colon X \to Y$ in $\CompOrd$, the functor returns the map $V(f)$ that sends a closed up-set $A \seq X$ to the up-closure $\upset f[A]$ of $f[A]$.
In \cite[Theorem~4.2]{HNN2018}, using the fact that $\CompOrd$ is dually equivalent to an $\aleph_1$-ary quasivariety, it was proved that the category $\CoAlg(V)$ of coalgebras for the endofunctor $V$ is dually equivalent to a $\aleph_1$-ary quasivariety, as well.
Such a quasivariety is described by adding to the theory of $\CMiso$ (dual of $\CompOrd$) a unary operation $\lozenge$, subject to the axioms
\begin{enumerate}
	\item \label{i:diam1} $\lozenge 0 = 0$;
	\item \label{i:diam2} $\lozenge (x \lor y) = \lozenge x \lor \lozenge y$;
	\item \label{i:diam3} for all $t \in [0,1]$, $\lozenge (x \odot t) = \lozenge x \odot t$;
	\item \label{i:diam4} $\lozenge (x \odot y) \leq \lozenge x \odot \lozenge y$.
\end{enumerate}
Given a coalgebra $r \colon X \to V(X)$, the unary operation $\lozenge$ is interpreted on $\Cleq(X, [0,1])$ by setting, for each $f \in \Cleq(X, [0,1])$ and each $x \in X$,
\[
	(\lozenge f)(x) \df \sup_{y \in r(x)} f(y).
\]
Since \cref{i:diam1,i:diam2,i:diam3,i:diam4} are equational, using the fact that $\CMiso$ is a variety, we obtain that the quasivariety $\CoAlg(V)$ described in \cite[Theorem~4.2]{HNN2018} is actually a variety.
In summary, then, we have:

\begin{theorem}
	The category $\CoAlg(V)$ of coalgebras and homomorphisms for the Vietoris functor $V \colon \CompOrd \to \CompOrd$ is dually equivalent to a variety, with operations of at most countable arity.
\end{theorem}

\cleardoublepage
\phantomsection
\addcontentsline{toc}{chapter}{Bibliography}

\printunsrtglossary[type = category]
\printunsrtglossary[type = symbol]

\cleardoublepage
\phantomsection
\addcontentsline{toc}{chapter}{Index}
\printindex

\end{document}